\documentclass{article}

\usepackage{arxiv}
\usepackage{xcolor}
\usepackage[utf8]{inputenc} % allow utf-8 input
\usepackage[T1]{fontenc}    % use 8-bit T1 fonts
\usepackage{hyperref}       % hyperlinks
\usepackage{url}            % simple URL typesetting
\usepackage{booktabs}       % professional-quality tables
\usepackage{multirow}       % multi-row cells
\usepackage{adjustbox}      % resizing tables
\usepackage{amsfonts}       % blackboard math symbols
\usepackage{amsthm}         % for proof environment
\usepackage{nicefrac}       % compact symbols for 1/2, etc.
\usepackage{subcaption} % recommended modern replacement for subfigure
\usepackage{microtype}      % microtypography
\usepackage{lipsum}		% Can be removed after putting your 
\usepackage{amssymb}
\usepackage{algorithm}
\usepackage{algpseudocode}
\usepackage{amsmath}
\usepackage{graphicx}
\usepackage[numbers,sort&compress]{natbib}
\usepackage{doi}
\usepackage{comment}
\usepackage{color, colortbl}
\usepackage{subcaption}
\newtheorem{remark}{Remark}

\newtheorem{theorem}{Theorem}
\newtheorem{prop}{Proposition}
\newtheorem{definition}{Definition}
\newtheorem{corollary}{Corollary}

\DeclareMathOperator*{\argmin}{argmin}

\title{RANDSMAPs: Random-Feature/multi-Scale Neural Decoders with Mass Preservation} 

\author{ \href{https://orcid.org/0000-0001-9840-0018}{\includegraphics[scale=0.06]{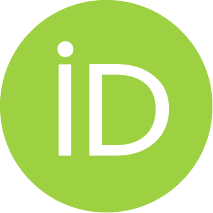}\hspace{1mm}Dimitrios G. Patsatzis} \\
	Modelling and Engineering Risk and Complexity \\
	Scuola Superiore Medirionale\\
	Naples, Italy \\
	\texttt{d.patsatzis@ssmeridionale.it} \\
	\And
	\href{https://orcid.org/0000-0003-2989-4397}{\includegraphics[scale=0.06]{orcid.pdf}\hspace{1mm}Alessandro Della Pia} \\
	Modelling and Engineering Risk and Complexity\\
	Scuola Superiore Medirionale\\
    Naples, Italy \\
	\texttt{a.dellapia@ssmeridionale.it} \\
     \And
	\href{https://orcid.org/0000-0003-4888-467X}{\includegraphics[scale=0.06]{orcid.pdf}\hspace{1mm}Lucia Russo} \\
	Institute of Sciences and Technologies for \\ Sustainable Energy and Mobility \\
	Consiglio Nazionale delle Ricerche\\
    Naples, Italy \\
	\texttt{lucia.russo@stems.cnr.it} \\
	\And
\href{https://orcid.org/0000-0002-9568-3355}{\includegraphics[scale=0.06]{orcid.pdf}\hspace{1mm}Constantinos Siettos}\thanks{Corresponding author} \\
	Dipartimento di Matematica e Applicazioni \\ ``Renato Caccioppoli''\\
	Università degli Studi di Napoli Federico II\\
    Naples, Italy \\
	\texttt{constantinos.siettos@unina.it} \\
}

\renewcommand{\shorttitle}{\textit{arXiv} Template}

\begin{document}
\maketitle

\begin{abstract}
We introduce RANDSMAPs (Random-feature/multi-scale neural decoders with Mass Preservation), numerical analysis-informed, explainable neural decoders designed to explicitly respect conservation laws when solving the challenging ill-posed pre-image problem in manifold learning. We start by proving the equivalence of vanilla random Fourier feature neural networks to Radial Basis Function interpolation and the double Diffusion Maps (based on Geometric Harmonics) decoders in the deterministic limit. We then establish the theoretical foundations for RANDSMAP and introduce its multiscale variant to capture structures across multiple scales. We formulate and derive the closed-form solution of the corresponding constrained optimization problem and prove the mass preservation property. Numerically, we assess the performance of RANDSMAP on three benchmark problems/datasets with mass preservation obtained by the Lighthill–Whitham–Richards traffic flow PDE with shock waves, 2D rotated MRI brain images, and the Hughes crowd dynamics PDEs. We demonstrate that RANDSMAPs yield high reconstruction accuracy at low computational cost and maintain mass conservation at single-machine precision. In its vanilla formulation, the scheme remains applicable to the classical pre-image problem, i.e., when mass-preservation constraints are not imposed.
\end{abstract}
% keywords can be removed
\section{Introduction}
Manifold learning is of utmost importance in many fields for uncovering the low-dimensional structure that underlies high-dimensional data and for enabling effective dimensionality reduction. For example, in machine learning it can be used to dramatically reduce sample and computational complexity when training neural networks and neural operators, thus dealing with the ``curse of dimensionality'' \cite{lusch2018deep,lee2020coarse,kemeth2022learning,galaris2022numerical,fabiani2024task,kontolati2024learning,fabiani2025enabling}. The same latent representations are ideal substrates for constructing manifold-informed reduced-order models (ROMs), facilitating training and improving accuracy and stability over longer time periods \cite{deane1991low,rowley2004model,lassila2014model,brunton2016discovering,papaioannou2022time,romor2023non,koronaki2024nonlinear,kaszas2025globalizing}. Overall, manifold methods turn high-dimensional learning and simulation problems into much simpler and more robust tasks by exposing ``the geometry that truly matters'' \cite{evangelou2022parameter,sroczynski2024learning}.

Methods for manifold learning (encoding) and for the solution of the corresponding pre-image problem (decoding) can be generally categorized into two broad classes. The first class contains numerical analysis and linear algebra-based methods, which build explicit, often theoretically grounded restriction and reconstruction maps \cite{belkin2003laplacian,belkin2006convergence,coifman2005geometric,coifman2006geometric,belkin2008towards}. In this class belong (a) linear methods, such as the Singular Value Decomposition (SVD) (or as otherwise called Proper Orthogonal Decomposition (POD)), \cite{deane1991low,graham1996alternative,shvartsman1998nonlinear,rowley2004model,sirisup2005stability}, the related to dynamical systems (Extended) Dynamic Mode Decomposition (DMD) \cite{tu2014dynamic,li2017extended}, and linear random projections \cite{hegde2007random,halko2009finding,erichson2019randomized2}, which are theoretically supported by the celebrated Johnson-Lindenstrauss embedding theorem \cite{johnson1984extensions}, and, (b) nonlinear manifold learning methods, such as kernel-Principal Component Analysis \cite{scholkopf1998nonlinear}, Local Linear Embedding \cite{roweis2000nonlinear}, ISOMAP \cite{tenenbaum2000global}, Laplacian Eigenmaps \cite{belkin2003laplacian,belkin2008towards}, Diffusion Maps (DM) \cite{coifman2006diffusion,coifman2006geometric,nadler2006diffusion,coifman2008diffusion,dsilva2015data}.
%and the Koopman Operator \cite{mezic2005spectral,williams2015data,arbabi2017ergodic}. 
For linear manifold learning methods, the pre-image (lifting) problem is solved in a straightforward way as it is implemented via linear projections onto a finite-dimensional subspace, thus solving a linear (regularized) least-squares problem. The resulting linear system admits a closed-form solution via for example, the Moore–Penrose pseudo-inverse, or, the Tikhonov regularization. These linear methods are computationally efficient and yield globally optimal
%as demonstrated by the Eckart-Young-Mirsky theorem \cite{eckart1936approximation,mirsky1960symmetric}, 
reconstructions within the chosen subspace, when the manifold--and therefore the lifting operator--can be reasonably assumed to be (approximately) linear. Their limitation, of course, is that they cannot capture nonlinear manifolds and their pre-image mapping. On the other hand, the solution of the pre-image problem in nonlinear manifold learning is ill-posed, and therefore small perturbations in latent coordinates can produce large changes in the reconstructed images. For this reason, a variety of independent approaches, including parametric surrogate models such as Radial Basis Functions (RBFs) and kernel/Ridge regression, local interpolation schemes such as $k$-Nearest Neighbors ($k$-NN) with convex interpolation \cite{fix1985discriminatory,Patsatzis_2023}, spectral/Nystr\"{o}m-based reconstructions such as the ``double" Diffusion Maps (DDM) approach that is based on Geometric Harmonics (GHs) \cite{coifman2006geometric,evangelou2022double} have been proposed (for a review and comparative analysis see \cite{chiavazzo2014reduced,papaioannou2022time}). The choice among these approaches is dictated, among others, by the available data and their density, interpretability, and computational cost. These numerical analysis-based methods for the solution of the pre-image problem are attractive because they are interpretable, often with provable approximation properties, but they can be sensitive to hyperparameter choice, suffer in very high ambient dimensions, and, in some cases, require solving large nonlinear optimization problems. Various manifold learning methods and corresponding decoders are implemented in the Python package \textit{datafold}  \cite{lehmberg2020datafold}.

The second class consists of deep neural networks-based learning methods, and in particular various autoencoder architectures \cite{hinton2006reducing,masci2011stacked,vlachas2022multiscale,romor2023non, peterfreund2020local} that learn encoder and decoder maps jointly end-to-end.
%Theoretically,  due to the universal approximation theorems of (deep) neural networks \cite{cybenko1989approximation,hornik1989multilayer,barron1991universal}, they are extremely expressive, and can capture highly-dimensional nonlinear pre-images. 
Various structures of autoencoders can be implemented and trained within the \textit{PyTorch} framework \cite{paszke2019pytorch}. While convenient for black-box encoding/decoding, their training is data-hungry, suffering from the ``curse of dimensionality” and they lack interpretability. Importantly, in autoencoder-based approaches, fundamental physical properties such as mass preservation are typically imposed only as soft constraints through penalty terms in the loss function rather than being explicitly enforced at the architectural level, which can further compromise generalization performance and numerical accuracy. 

Here, we propose RANDSMAP, an explainable numerical analysis-informed neural network decoder, based on random nonlinear features/projections \cite{rahimi2008uniform,gorban2016approximation,bach2017equivalence,liu2021random,nelsen2021random,surasinghe2021randomized} to solve the pre-image problem while explicitly preserving mass whenever such a conservation law is apparent in the data. RANDSMAPs follow the same architecture as random feature/projection neural networks (RFNNs) \cite{igelnik1995stochastic,huang2006extreme,scardapane2017randomness,cao2018review,fabiani2021numerical,galaris2022numerical,fabiani2023parsimonious,bollt2024neural,bolager2023sampling,dong2023method,sun2024local,fabiani2025randonets,fabiani2025random}, extending their computational efficiency and explainability capabilities,  for the solution of the mass-reserving pre-image problem. %RANDSMAPs is built on a single hidden–layer neural network, where the hidden weights and biases are fixed (typically drawn at random), and only the output weights and biases are trained. This reduces learning to a regularized linear least–squares problem, yielding a computationally efficient scheme with global optimum via established linear–algebra methods %(Tikhonov/Ridge regularization, QR, QLP \cite{stewart1999qlp,falini2024approximated} or SVD factorizations), 
%including iterative methods in the Krylov subspace, with well-studied approximation and convergence properties. 
The explainability of RANDSMAP stems from the fact that (a) the output is an explicit linear combination of fixed (random feature) basis functions, implying that the contribution of each feature to the prediction is directly given by the learned linear weights, and (b) as we rigorously prove, in the deterministic limit they  asymptotically recover the RBFs reconstruction.
%The novelty of RANDSMAPs lies in the formulation and the close-form solution that preserves the mass via random features and the introduction of the corresponding multi-scale scheme. %Instead of solving an unconstrained regression problem, RANDSMAPs solves a constrained linear least-squares problem, where the solution is explicitly required to satisfy the sum-to-one invariant present in the data. 
%By deriving the regularized closed-form solution via Lagrange multipliers, RANDSMAP guarantees exact conservation by construction while maintaining the reconstruction accuracy and efficiency of the ``inherited'' by the random features decoder.

The contributions of this work are threefold. %First, we provide a theoretical analysis of the mass-conservation property of three widely used numerical decoders. 
%We prove that the linear POD/SVD and the nonlinear $k$-NN decoders produce mass-preserving reconstructions by construction, whereas the nonlinear GHs-based DDM decoder does not. 
First, we introduce and theoretically analyze RANDSMAPs, deriving a closed-form solution for the pre-image problem that guarantees, as we rigorously prove, mass-preservation. Second, we provide a theoretical insight establishing the equivalence between the numerical analysis-based DDM decoder, radial basis functions (RBFs), and the vanilla RFNN decoder in the deterministic limit, when the latter is equipped with random Fourier features. Third, we derive conservation error bounds for the truncated solutions of RANDSMAP, and extend it with multi-scale random Fourier features for capturing structures across multiple scales on the manifold; we prove that the latter features are equivalent to  multi-Gaussian kernels in the deterministic limit.  

RANDSMAPs are evaluated against three mass-preserving benchmark examples with varying scalability: a moderate-dimensional 1D traffic density dataset from the Lighthill-Whitham-Richards (LWR) hyperbolic PDE; a high-dimensional 2D gray-scale MRI image dataset with sparse observations; and a high-dimensional 2D pedestrian density dataset governed by the Hughes crowd dynamics PDEs. Additionally, two classic low-dimensional manifold learning examples without inherent mass conservation (the classic $3$-dimensional Swiss roll and S-curve projected in $20$ dimensions) are included to establish a baseline for RFNN decoder performance. In all examples, RANDSMAPs are compared against the DDM and $k$-NN decoders. To ensure fair assessment, all decoders operate on the same low-dimensional embedding/encoding, obtained here via the DM algorithm, though each decoder is, in principle, compatible with any embedding.     

The structure of the paper is as follows. Section~\ref{sec:Prelim} formally states the constrained pre-image problem and briefly reviews current state-of-the-art methods in the field. Section~\ref{sec:NA_PI_cons} analyzes the conservation properties of the numerical analysis-based decoders considered. Section~\ref{sec:decoderRFNN} introduces the RANDSMAP theoretical foundations: we first establish the connection between the DDM and RBF kernel-based approximations and the random Fourier features in the deterministic limit; then, we present the multi-scale random Fourier features, and finally derive the RANDSMAP's closed-form solution that guarantees mass-preservation, and provide error bounds for its truncated solutions. Section~\ref{sec:numerics} provides an extensive numerical evaluation on the three mass-preserving benchmark problems; results on the non-mass-preserving low-dimensional examples are detailed in Appendices~\ref{app:SR} and \ref{app:SC20}. We conclude in Section~\ref{sec:conclusions} discussing potential future research directions.

\section{Preliminaries and Problem Statement} \label{sec:Prelim}
Let us start with the definition of a $d$-dimensional manifold (see e.g., in \cite{kuhnel2015differential,wang2012geometric}).
\begin{definition}[Manifold]
\textit{A set $\mathcal{M}$ is called a $d$-dimensional manifold (of class $C^{k}$) if for each point $p \in \mathcal{M}$, there is an open set $\mathcal{W} \subset \mathcal{M}$ such that there exists a $C^{k}$ bijection $f: \mathcal{U} \rightarrow \mathcal{W}$, $\mathcal{U} \subset \mathbb{R}^d$ open, with a $C^{k}$ inverse $f^{-1}: \mathcal{W} \rightarrow \mathcal{U}$ (i.e. $\mathcal{W}$ is $C^{k}$-diffeomorphic to $\mathcal{U}$). The bijection $f$ is called a local parametrization of $\mathcal{M}$, while $f^{-1}$ is called a local coordinate mapping, and the pair $(\mathcal{W},f^{-1})$ is called a chart on $\mathcal{M}$.}
\label{eq:defman}
\end{definition}
A point $p \in \mathcal{W}$ can be expressed, via the local parameterization $f$, as 
\begin{equation}
p=f(v_1,v_2,\dots, v_d), 
\end{equation}
where $v=(v_1,v_2,\dots, v_d)=f^{-1}(p)\in \mathcal{U}$. This definition endows $\mathcal{M}$ with the structure of a Hausdorff and second-countable topological space; the standard requirements ensuring that the manifold is separable/$T_2$ (every pair of distinct points on $\mathcal{M}$ has disjoint open neighborhoods) and admits partitions of unity \cite{kuhnel2015differential,wang2012geometric}. Thus, a manifold can be summarized as a Hausdorff, second-countable topological space in which every point has a neighborhood diffeomorphic to an open subset of $\mathbb{R}^d$.

\begin{definition}[Tangent space and tangent bundle to a manifold]
\textit{Let $\mathcal{M} \subset \mathbb{R}^M$ be a d-dimensional manifold embedded in ambient space and $p \in \mathcal{M}$. The tangent space to $\mathcal{M}$ at $p$ is the vector subspace $T_{p}\mathcal{M}$ formed by the tangent vectors at $p$, defined by
$$ T_{p}\mathcal{M} = D f_{v}(T_{v}\mathbb{R}^{d}),$$
where $T_{v}\mathbb{R}^{d}$ is the tangent space at $v$ and $Df_v:T_{v}\mathbb{R}^{d}\to \mathbb{R}^M$ is the derivative of $f$ at $v$. Using the standard bases for $T_v\mathbb{R}^d$ and $\mathbb{R}^M$, the matrix of $Df_v$ is the $M \times d$ Jacobian $[\partial f_i/\partial v_j]$ with $i=1,\ldots,M$, $j=1,\ldots,d$, which has rank $d$. Thus, the vectors $\left\{ \frac{\partial f}{\partial v_1}, \frac{\partial f}{\partial v_2}, \dots, \frac{\partial f}{\partial v_d} \right\}$ form a basis for $T_{p}\mathcal{M}$. The union of tangent spaces $$T \mathcal{M} = \bigcup_{p \in \mathcal{M}} T_{p}\mathcal{M}$$ is called the tangent bundle of $\mathcal{M}$.} 
\end{definition}
\begin{definition}[Riemannian manifold and metric]
\textit{A Riemannian manifold is a manifold $\mathcal{M}$ endowed with a positive definite inner product $g_{p}:T_{p}\mathcal{M} \times T_{p}\mathcal{M}\to\mathbb{R}$, defined on the tangent space $T_{p}\mathcal{M}$ at each point $p\in\mathcal{M}$. The family of inner products is called a Riemannian metric $g$, and the Riemannian manifold is denoted $(\mathcal{M},g)$.}
\end{definition}
The above definitions describe the continuum limit setting of a manifold with infinitely many points. However, in practice, we only have a finite set of observations, and the analytic charts (as well as the Riemannian metric that defines the manifold's geometry) are not directly accessible. Manifold learning addresses this gap by constructing, from finite data, numerical approximations that converge, in a probabilistic sense, to the continuum objects as the sample size $N \to \infty$. To facilitate this discussion, we adopt the standard convention in this field: we denote the continuum coordinate chart by $\Psi$ (so that $\Psi = f^{-1}$), and the continuum parametrization by $\Psi^{-1}$ (so that $\Psi^{-1} = f$, where $f$ is the parametrization from Definition~\ref{eq:defman}). 
\paragraph{Manifold learning and encoding.} Given $\mathcal{X}=\{x_i\}_{i=1}^N\subset\mathbb{R}^M$ sampled i.i.d. from a smooth Riemannian manifold $(\mathcal{M},g)$, manifold learning consists of constructing empirical operators $\mathcal{R}_{N,\theta}$ (e.g. a linear operator based on POD/SVD or DMD, or a nonlinear one based on ISOMAP, DM, etc.), where $\theta$ denotes a possible set of parameters (such as kernel shape parameters). From the eigendecomposition of such an operator, one obtains a low-dimensional embedding
\begin{equation}
\Psi_N: \mathcal{X} \subset{R}^M \to \mathcal{Y} \subset \mathbb{R}^d,
\label{eq:PsiN}
\end{equation}
which assigns to each data point $x_i \in \mathcal{X}$ a coordinate vector $y_i=\Psi_N(x_i) \in \mathbb{R}^d$. 

Under suitable conditions (e.g., uniform sampling from a smooth Riemannian manifold and appropriate choice of kernel parameters), the empirical embedding $\Psi_N$ converges (pointwise or uniformly) to the continuum coordinate chart $\Psi: \mathcal{W}\subset \mathcal{M}\to \mathcal{U} \subset \mathbb{R}^d$ as $N\to\infty$ (since $\mathcal{X}\to \mathcal{W}$ and $\mathcal{Y}\to \mathcal{U}$), while preserving the underlying geometry (e.g., Euclidean, geodesics or diffusion distances). This geometric consistency follows from the fact that a surrogate operator $\mathcal{R}_{N,\theta}$ converges to a continuum operator on the $\mathcal{M}$. For example, with a properly normalized Gaussian affinity kernel, the graph Laplacian constructed in Laplacian Eigenmaps or DM converges to the Laplace–Beltrami operator on $\mathcal{M}$ \cite{belkin2003laplacian,coifman2005geometric,belkin2006convergence,nadler2006diffusion,belkin2008towards}. This allows us to consider the eigenvectors of $\mathcal{R}_{N,\theta}$ corresponding to the largest few eigenvalues as discrete approximations of the principal eigenfunctions of the operator, rendering the resulting embedding $\Psi_{N}$ to inherit the operator's convergence. The construction described above yields an embedding $\Psi_{N}$ for the available data in $\mathcal{X}$. To complete the encoding step, one must also be able to embed new, unseen points $x^* \notin \mathcal{X}$ (the ``out-of-sample extension'' problem). A standard approach for this is the Nystr\"{o}m extension \cite{coifman2006geometric,papaioannou2022time,evangelou2022double,Patsatzis_2023}, which uses the already computed eigenvectors to evaluate the embedding at arbitrary points. %Thus, the full embedding/encoding map, applicable to both the training set and new inputs, is obtained.
\paragraph{The Pre-Image (decoding) problem.} As discussed, in the continuum limit, there exists a continuous inverse of $\Psi$
\begin{equation}
\Psi^{-1} : \mathcal{U} \subset \mathbb{R}^d \longrightarrow \mathcal{W} \subset \mathcal{M}.
\end{equation}
With only finite data, however, the map $\Psi_N$ obtained from the empirical operator is not exactly invertible. Recovering a data point from its latent coordinates, referred to as the out-of-sample pre-image problem, is therefore ill-posed, and one must resort to numerical approximations. A generic approach for finding the reconstruction of an out-of-sample point $\hat{x}^* \in \mathbb{R}^M$ is to solve the optimization problem
\begin{equation}
\hat{x}^*= \argmin_{x \in \mathbb{R}^M} \big\| \Psi_{N}(x) - y^* \big\|^2,
\label{eq:PI_opt}
\end{equation}
where $y^* \in \mathbb{R}^d$ is a given latent point, and $\Psi_N(x)$ is the empirical/surrogate embedding/encoding map in Eq.~\eqref{eq:PsiN}. In general, $y^*$ does not correspond to the image of an ``in-sample'' point in $\mathcal{X}$, hence the out-of-sample symbolism $y^*$.

Several classical numerical methods have been proposed for this task, including  RBF-based interpolation methods, the $k$-NN algorithm with convex interpolation \cite{fix1985discriminatory} and DDM based on GHs \cite{evangelou2022double,coifman2006geometric}; for a review and comparison, see \cite{chiavazzo2014reduced,papaioannou2022time}. Some of these methods, such as $k$‑NN with convex interpolation or linear POD via SVD, preserve mass by construction when these are apparent in the data \cite{viola2025pde,alvarez2025next}, while, the nonlinear DDM approach does not, as we prove in Section~\ref{sec:NA_PI_cons}. % In particular, numerical analysis-based pre-image methods, generally ignore physical constraints, which can lead to systematic violation of physically meaningful quantities in the reconstructed field %\cite{aitchison1982statistical,egozcue2003isometric}.
%\paragraph{Neural networks as universal approximators.} 
On the other hand, autoencoders, \cite{hinton2006reducing,makhzani2015adversarial,peterfreund2020local,vlachas2022multiscale} provide  a black-box, numerical analysis-agnostic framework for the ``end-to-end'' solution of both the encoding (manifold learning) and decoding (pre‑image) problems.
%as they can be designed to approximate the continuum nonlinear operators involved. Classical universal approximation theorems for operators (see e.g., Chen \& Chen \cite{chen1995universal}) show that suitably chosen network architectures can approximate broad classes of nonlinear operators to arbitrary accuracy. More recently, operator learning architectures, such as DeepONets \cite{lu2021learning} and Fourier neural operators \cite{li2020fourier}, implement this idea by learning maps between functional spaces. In particular, random feature/projection NNs (RandONets) have been shown to combine accuracy with low computational cost when approximating linear and nonlinear continuum operators \cite{fabiani2025randonets}.
%Autoencoders offer a ``one-size-fits-all'', black-box, numerical analysis-agnostic framework that can learn both encoding and decoding maps end-to-end. 
Compared to classical numerical-analysis algorithms, however, they suffer from several drawbacks: they are data-hungry (due to the ``curse of dimensionality''), lack interpretability, and usually cannot explicitly enforce problem-specific constraints; these are usually imposed as soft constraints in the loss function.

%It is important to note that true operator learning assumes access to an (effectively) continuous sampling across the underlying manifold. In practice, we only have a finite set of pointwise measurements. Consequently, one cannot identify the full infinite-dimensional operator, whose action outside the sampled subspace is undetermined, but only a finite-dimensional realization; i.e., a functional map that approximates the operator on the observed data. NNs are well-suited to learn such realizations/functionals. However, even in this restricted setting, they typically fail to respect conservation laws exactly; these are usually integrated as soft constraints in the loss training function, unless specifically designed in an ad-hoc/taylor-made, problem-specific manner (see, for example, \cite{patsatzis2025gorinns} for hyperbolic PDEs).

\paragraph{Problem statement.} Existing machine-learning–based decoders typically overlook intrinsic physical constraints, such as mass conservation. Designing a general decoding framework that combines machine learning with numerical analysis that explicitly guarantees mass preservation while maintaining reconstruction accuracy, therefore, remains an open challenge. Toward this goal, we present a numerical-analysis–informed neural network decoder that produces a reconstruction $\hat{x}^*$ for a given latent point $y^*$, which both approximates the solution of the pre-image optimization problem in Eq.~\eqref{eq:PI_opt}, and explicitly preserves the sum-to-one invariant (the mass) in the reconstruction as $\sum_{j=1}^M\hat{x}^*_{j}=1$. 
%This implies that the desired decoder should provide a solution to the constrained optimization problem
%\begin{equation}
%\hat{x}^*= \argmin_{x \in \mathbb{R}^M} \big\| \Psi_N(x) - y^* \big\|^2 \quad \text{subject to} \quad \sum_{j=1}^M x_j=1.
%\label{eq:PI_opt_cons}
%\end{equation}
%The decoder must solve this problem by construction, ensuring the constraint is satisfied exactly for any $y^*\in\mathbb{R}^d$, rather than being imposed a posteriori or approximated by soft constraints.
%For retaining the expressive power and flexibility of data-driven frameworks, we adopt a NN-based decoder. However, to overcome the limitations of this decoder class, our goal is to render it interpretable and computationally efficient, which can be done by structuring the decoder so that its training reduces to a linear least-squares problem. This will enable the use of classical numerical analysis and linear algebra tools, such as robust solvers (SVD, QR, preconditioners, iterative methods), and provide error estimate guarantees based on low-rank approximations. 
%In summary, the core problem is to develop a NN-based decoder for solving Eq.~\eqref{eq:PI_opt_cons} that is simultaneously accurate, exactly conservative, and interpretable, thereby combining the expressive power of data-driven models with the stability and analytical guarantees of numerical analysis frameworks.

\section{Numerical analysis-based decoders and conservation properties}
\label{sec:NA_PI_cons}
We begin by briefly reviewing three standard numerical analysis-based decoders and examining their conservation properties. 

First, let us consider a conservative dataset $\mathcal{X}=\{x_i\}_{i=1}^N\subset\mathbb{R}^M$ in the high-dimensional space, where each data vector satisfies the (normalized) mass conservation, implying a sum-to-one invariant $\sum_{j=1}^M x_{ji}=1$ for any $i=1,\ldots,N$. Let $1_N$ and $1_M$ denote the $N$- and $M$-dimensional column vectors of ones, respectively, and let all the data points be collected in the matrix 
$X = \begin{bmatrix} x_1, & \ldots, &x_N \end{bmatrix} \in \mathbb{R}^{M \times N}$. Then, the sum-to-one invariant, apparent in the data, can be equivalently written in the compact form 
\begin{equation}
    1_M^\top X = 1_N^\top.
    \label{eq:s2one_cons}
\end{equation}

Now, consider the encoded coordinates of $\mathcal{X}$ in the dataset $\mathcal{Y}=\{ y_i \}_{i=1}^N \subset \mathbb{R}^d$, where $y_i=\Psi_N(x_i)$, obtained by a low-dimensional embedding $\Psi_N$ in Eq.~\eqref{eq:PsiN}; for example by a linear POD/SVD embedding or a nonlinear DM one. Then, any decoder $\Psi_N^{-1}$ solving the optimization problem in Eq.~\eqref{eq:PI_opt} for the latent points in $\mathcal{Y}$ is mass-conserving, when the sum-to-one invariant
\begin{equation}
    1_M^\top \widehat{X} = 1_N^\top,
    \label{eq:s2one_cons_recon}
\end{equation}
holds for the reconstructions $\widehat{X}= \begin{bmatrix} \hat{x}_1, & \ldots, &\hat{x}_N \end{bmatrix} \in \mathbb{R}^{M \times N}$, where $\hat{x}_i=\Psi_N^{-1}(y_i)$. Note that the sum-to-one invariant should also hold for reconstructions $\widehat{X}^*=\begin{bmatrix} \hat{x}^*_1, & \ldots, &\hat{x}^*_L \end{bmatrix}\in \mathbb{R}^{M \times L}$ of out-of-sample, unseen points in $\mathcal{Y}^*=\{ y^*_l\}_{l=1}^L\subset \mathbb{R}^d$, where $\hat{x}^*_i=\Psi_N^{-1}(y^*_i)$.

Given the above mass-conservation criterion, let us now review three standard numerical analysis-based decoders: the POD/SVD decoder, the $k$-NN decoder with convex interpolation, and the DDM decoder.

\subsection{Decoding with POD/SVD}
The POD/SVD method \cite{deane1991low,graham1996alternative,shvartsman1998nonlinear,rowley2004model,sirisup2005stability} is a linear manifold learning method that encodes high-dimensional data onto the optimal low-dimensional subspace minimizing the (Frobenius) norm of the reconstruction error. Given the data matrix $X$, the decoder reconstructs a point from its latent coordinates via projection onto the leading $d$ singular vectors. As we have proven in our recent works \cite{alvarez2025next,viola2025pde}, due to the apparent sum-to-one invariant in the data, the POD basis is orthogonal to the constant vector $1_M$, rendering the POD reconstruction mass-preserving by construction. A complete proof is also provided in Appendix~\ref{app:POD}. While POD offers a computationally efficient, exactly conservative decoder, its expressivity is inherently limited to linear manifolds, restricting its applicability to nonlinear embeddings.     

\subsection{Decoding with k-NN}
The $k$-NN decoder is a nonlinear, nonparametric method that reconstructs a point by interpolating its $k$ nearest neighbors in the latent space as a convex combination \cite{fix1985discriminatory}. Given a latent point $y^*$, the decoder solves a constrained optimization to find weights $\alpha_k$ such that $\sum_{i=1}^k\alpha_i=1$, and returns the reconstruction $x^*=\sum_{i=1}^k \alpha_i x_{S(i)}$, where $x_{S(i)}$ are high-dimensional, known, pre-images of the $k$ neighbors. The convex interpolation structure inherently guarantees mass-conservation, as when the $k$ neighbors satisfy the sum-to-one invariant, so does a convex interpolation of them. A proof of conservation, along with the detailed $k$-NN decoder implementation, is provided in Appendix~\ref{app:kNN}.

Despite this exact conservation guarantee, the method faces significant practical limitations. It suffers from the ``curse of dimensionality'' in neighbor search, requires computationally expensive pointwise optimization, and its accuracy is sensitive to data sparsity and the choice of $k$; see Remark~\ref{rmk:kNN}. These factors make it impractical for large-scale or batch decoding, where efficiency is paramount. A detailed discussion of these challenges can be found in \cite{halder2024enhancing}.

\subsection{Decoding with DDM/GH}\label{sec:decodingwithDMS}
The DDM decoder is a nonlinear kernel-based pre-image method, based on GHs  \cite{evangelou2022double,coifman2006geometric}. Given the set of ambient points $\mathcal{X}$ and their corresponding low-dimensional embedding $\mathcal{Y}$ (obtained, for example, from a DM encoder; see Appendix~\ref{app:DDM_dec}), the DDM decoder constructs a reconstruction mapping $\mathcal{Y}\to\mathcal{X}$ in two steps. First, it performs a ``second'' DM-like step (hence ``double'' in DDM) applied directly to the latent coordinates $\mathcal{Y}$ to build an improved basis for kernel interpolation in the latent space \cite{evangelou2022double,Patsatzis_2023}. This yields eigenvectors and eigenvalues collected in matrices $V_r\in\mathbb{R}^{N\times r}$ and $\Lambda_r \in \mathbb{R}^{r\times r}$, respectively, where $r$ is the number of retained DDM eigenmodes. Next, the DDM decoder employs out-of-sample reconstruction via GHs \cite{coifman2005geometric,coifman2006geometric} by projecting a new latent point $y^*$ onto the basis $\{V_r,\, \Lambda_r\}$ via kernel evaluations, producing reconstruction weights that are linearly combined with the training data $X$. A detailed presentation of the DDM decoder is provided Appendix~\ref{app:DDM_dec}, where the DM encoder is also presented for completeness. 

For a set of unseen latent points $\mathcal{Y}^*=\{ y^*_l\}_{l=1}^L$, the complete batch reconstruction formula of the DDM decoder (see Eq.~\eqref{eq:DoubleDMs_mat}) is
\begin{equation}
\widehat{X}^* = X\,V_r\, \Lambda_r^{-1} \, V_r^\top \, K^{*\top},
    \label{eq:DoubleDMs_mat_main}
\end{equation}
where $X$ is the training data matrix and $K^*\in \mathbb{R}^{L \times N}$ is the ``second'' DM-like kernel matrix computed between the new points in $\mathcal{Y}^*$ and the training set $\mathcal{Y}$ in the latent space. %This ``second'' DM-like kernel with elements $K^*_{lj}=k_2(y^*_l,y_j; \epsilon_2)$ is now computed on the new latent points $y^*_l\in\mathcal{Y^*}$ and the ``first'' DM coordinates. $y_j\in\mathcal{Y}$%, which are retrieved by the latent points via the diagonal scaling $z^*_l=\Xi_d^{-1}y^*_l$ and $z_j=\Xi_d^{-1}y_j$, respectively, where $\Xi_d=\operatorname{diag}(\xi_1,\ldots,\xi_d)$ contains the eigenvalues of the ``first'' DM transition matrix (see Appendix~\ref{app:DDM_dec} for details). 
The matrix product $X\,V_r\, \Lambda_r^{-1} \, V_r^\top$ in Eq.~\eqref{eq:DoubleDMs_mat_main} can be precomputed, enabling efficient batch reconstruction without pointwise optimization (unlike, e.g., the $k$-NN decoder; see Remark~\ref{rmk:kNN}). However, while efficient, the kernel-based DDM decoder does not preserve mass in general, as we prove next. 

\begin{prop} \label{prop:DDM_cons}
Suppose the dataset $\mathcal{X}=\{x_i\}_{i=1}^N\subset\mathbb{R}^M$ is mass-preserving, i.e., the data matrix $X \in \mathbb{R}^{M \times N}$ satisfies the (normalized) sum-to-one constraint $1_M^\top X = 1_N^\top$. Let the matrix $\widehat{X}^* \in \mathbb{R}^{M \times L}$ contain the reconstructed fields of a set of new latent points $\mathcal{Y}^*=\{ y^*_l\}_{l=1}^L\subset\mathbb{R}^d$, obtained via the DDM decoder as
\begin{equation*}
    \widehat{X}^* = X\,V_r\, \Lambda_r^{-1} \, V_r^\top \, K^{*\top},
\end{equation*}
where $V_r \in \mathbb{R}^{N \times r}$ and $\Lambda_r \in \mathbb{R}^{r \times r}$ denote matrices collecting the ``second'' DM eigenvectors and eigenvalues, and $K^* \in \mathbb{R}^{L \times N}$ denotes the kernel matrix between new latent points $\mathcal{Y}^*$ and the training embeddings $\mathcal{Y}$. Then, in general, the reconstructed fields $\widehat{X}^*$ are not mass-preserving, i.e., $1_M^\top \widehat{X}^* \neq 1_L^\top$.
\end{prop}
\begin{proof}
Multiplying the reconstruction of the DDM decoder on the left with the all-ones vector $1_M^\top$ and using the constraint $1_M^\top X = 1_N^\top$ yields
\begin{equation}
1_M^\top  \widehat{X}^{*}= 1_M^\top X \,V_r\, \Lambda_r^{-1}\, V_r^\top \, K^{*\top}=1_N^\top V_r
\,\Lambda_r^{-1}\, V_r^\top \, K^{*\top}.
\end{equation}
Clearly, $1_M^\top  \widehat{X}^{*}\neq 1_{L}^\top$ as, in the general case, $V_r \,\Lambda_r^{-1}\, V_r^\top $ is not the pseudo-inverse of $K^{*\top}$. In fact, it is the truncated pseudo-inverse of the kernel matrix $K^{(2)}\in\mathbb{R}^{N \times N}$ computed over the training embeddings $\mathcal{Y}$; see Eq.~\eqref{eq:DDMkernel}. Therefore, the DDM decoder is not, in general, mass-preserving.
\end{proof}

\section{The RANDSMAP decoder: Random-features Neural Decoder with Mass Preservation}
\label{sec:decoderRFNN}
The analysis in Section~\ref{sec:NA_PI_cons} shows that whether a classical numerical analysis-based pre‑image method preserves the mass depends on its specific architecture: linear POD/SVD and $k$‑NN with convex interpolation do so (albeit with significant limitations), as a consequence of their linear or convex structure, while nonlinear kernel-based methods such as DDM violate it. On the other hand, deep NN-based decoders (autoencoders) operate as numerical analysis–agnostic black boxes. In such a case, the conservation laws are typically enforced only through soft penalties during training, rather than being built into the model. This highlights a fundamental gap as most of the machine-learning-based nonlinear decoders are not designed to \textit{explicitly enforce} physical invariants. To overcome this, we introduce a framework that can guarantee exact conservation for flexible, nonlinear representations, while preserving both interpretability and numerical analysis rigor.

Here, towards this aim, we introduce a nonlinear neural decoder based on random nonlinear features/projections \cite{rahimi2008uniform,gorban2016approximation,bach2017equivalence,liu2021random,nelsen2021random,surasinghe2021randomized} to solve the pre-image problem while explicitly preserving mass whenever such a conservation law is apparent in the data. Following the architecture of single-hidden layer RFNNs, RANDSMAPs are inherently interpretable: the hidden layer defines an explicit, data-driven set of feature maps, and the decoder’s parameters are simply the linear output weights that combine these features. Critically in RANDSMAPs, instead of solving an unconstrained regression problem (as in RFNNs), we solve the constrained (by the sum-to-one invariant) optimization problem, which, due to the RFNN-like structure, reduces to a linearly‑constrained least‑squares problem. This reformulation allows us to enforce exact conservation \emph{by construction}, without resorting to soft penalty terms that only approximate the constraint. Importantly, this reformulation of an NN-based decoder allows the employment of state‑of‑the‑art numerical-analysis and linear‑algebra tools (e.g., SVD, QR factorizations, preconditioned iterative solvers) and the derivation error estimates based on low‑rank approximations; thereby importing the stability and analytical guarantees of classical numerical schemes into a data‑driven framework.

To introduce RANDSMAPs, we proceed as follows. In Section~\ref{sb:RFNNdec}, we present the vanilla RFNN decoder for general (unconstrained) data and provide a closed-form solution. We prove that with a particular choice of activation functions (Fourier features), the RFNN decoder recovers (in the deterministic limit) the DDM decoder, thus providing a connection to classical RBFs interpolation. In Section~\ref{sb:RFNNdec_ms}, we present the multi-scale random Fourier features, which can be used for enhancing the decoder’s capacity to capture structures in multiple scales/frequencies. We show that the induced feature kernel recovers in the deterministic limit a multi-Gaussian kernel. Finally, in Section~\ref{sb:RANDSMAP} we formally introduce RANDSMAP, derive their mass-preserving closed-form solution via Lagrange multipliers, and provide bound guarantees of the conservation error in the case of truncated solutions. 

\subsection{Vanilla Random Feature Decoders and their equivalence to the DDM decoder and RBFs interpolation in the deterministic limit} \label{sb:RFNNdec}
The concept of random feature embeddings can be traced back to the celebrated Johnson-Lindenstrauss theorem \cite{johnson1984extensions} for linear systems (see also the discussion in \cite{galaris2022numerical,fabiani2023parsimonious}). Regarding NNs, the idea is to fix internal weights and biases of a single hidden-layer, thereby constructing a randomized nonlinear basis. The network's output is then a linear projection onto this fixed basis. This architecture reduces training to a linear least-squares problem for the output weights, while maintaining universal approximation capabilities. In particular, it has been proven that such RFNNs, with an appropriate class of activation functions and i.i.d. sampled parameters, can approximate continuous functions on compact sets arbitrarily well \cite{igelnik1995stochastic}, with established convergence rates \cite{makovoz1996random,rahimi2007random,rahimi2008uniform,gorban2016approximation}. These properties extend to the approximation of nonlinear operators \cite{fabiani2025randonets}, ensuring the framework is suitable for learning maps between spaces.

%Regarding neural networks with random features, Igelnik and Pao \cite{igelnik1995stochastic} have proven that the set of random features model   
%$f_{P}(x ; \omega)=\sum_{j=1}^P w_j \phi (a_j \cdot x +b_j)$  for activation functions $\phi \in L^{2}(\mathbb{R})$ or $H^{1}(\mathbb{R})$  and i.i.d random  parameters $a_j$, $b_j$ can approximate with arbitrary accuracy real continuous functions on compact sets  $K \subset \mathbb{R}^d$,  with a rate of convergence $O(P^{-1/2})$; see also in \cite{makovoz1996random,rahimi2007random,gorban2016approximation} for the universal approximation properties of random features, and in \cite{fabiani2025randonets} for the approximation of nonlinear operators.

%The idea is to fix the weights between the input and the hidden layer, fix the biases for the nodes of the hidden layer, while the output of the network is projected linearly onto the functional subspace spanned by the randomized nonlinear basis functions of the hidden layer. Under this construction, the only remaining unknowns  are the weights between the hidden and the output layer, whose estimation can be done by solving a linear regularized least squares problem. 
%Various works have demonstrated that RFNNs are universal approximators of continuous nonlinear functions \cite{igelnik1995stochastic,rahimi2007random,rahimi2008uniform} and operators \cite{fabiani2025randonets}, thus implying also point-wise approximation.

Within our context, an RFNN provides a randomized nonlinear map of a point in the latent space $y\in \mathbb{R}^d$ to its pre-image in the input space $x\in \mathbb{R}^M$, offering a regularized approximation of the (generally unknown) reconstruction map $\Psi^{-1}$. We denote the map approximation by $\Psi_{N,P}^{-1}$, where $N$ is the number of training data and $P$ is the number of random features, and model the RFNN, as shown in Fig.~\ref{fig:Schematic}, in the general form (see e.g., \cite{lowe1988multivariable})  
\begin{equation}
x^{(j)} = \Psi^{-1}_{N,P}(y)= \sum_{k=1}^{P} \alpha_{jk}%[\Psi^{-1}_N]
\phi(y; c_k) + \beta_j%[\Psi^{-1}_N]
, \quad \, j=1,2,\dots, M,
\label{eq:RFNN}
\end{equation}
where $x^{(j)}$ is the $j$-th coordinate of output vector $x$, $\phi(y; c_k)$ is a parametric family of random activation functions uniformly bounded in $\mathbb{R}^{d} \times \mathbb{R}^{\ell}$ and $c_k \in \mathbb{R}^{\ell}$ are $\ell$ random shape parameters drawn from proper probability distributions. The output weights $\alpha_{jk}$ and biases $\beta_j$ are the only unknowns to be learned from data. 

Given the universal approximation properties of random-feature networks \cite{igelnik1995stochastic,rahimi2007random,rahimi2008uniform,gorban2016approximation,fabiani2025randonets,fabiani2025random}, for any compact set $\mathcal{K} \subset \mathbb{R}^d$ and for every map $\Psi^{-1}$ in a bounded subset $\mathcal{B}$ of $C(\mathcal{K})$, there exists an RFNN of the form in Eq.~\eqref{eq:RFNN}, with a sufficiently large number of features $P$ and appropriately chosen weights and activation function, such that 
\begin{equation}
\big\| {\Psi^{-1}_{N,P}} - \Psi^{-1} \big\|^2_{L^2(\mu)}<\epsilon,
\label{eq:RFNNapprox}
\end{equation}
holds for any $\epsilon>0$, where $\mu$ is a suitable probability measure on $\mathcal{K}$. In practice, we learn the weights $\alpha_{jk}$ and biases $\beta_j$ from a finite training dataset $\{(x_i,y_i)\}_{i=1}^N$ to obtain the empirical approximation of $\Psi^{-1}$.

Let $\mathcal{X}=\{x_i\}_{i=1}^N\subset\mathbb{R}^M$ be the training dataset in the ambient space and $\mathcal{Y}=\{y_i\}_{i=1}^N\subset\mathbb{R}^d$ the corresponding latent points. When $N$ is very large, one may downsample by randomly selecting an index subset $\mathcal{S}=\{s_1,\dots,s_n\}\subseteq\{1,\dots,N\}$ with $|\mathcal{S}|=n\le N$, yielding the (down)sampled training sets of targets $\mathcal{X}_S=\{x_{s_i}\}_{i=1}^n$ and inputs $\mathcal{Y}_S=\{y_{s_i}\}_{i=1}^n$. For casting the RFNN formula in a matrix form on the training set, we employ the parameterized activation function in Eq.~\eqref{eq:RFNN} to obtain the $n\times (P+1)$ (down)sampled feature matrix 
\begin{equation}
\Phi_S = \big[1_n \, | \, \widetilde{\Phi}_S \big]
\in\mathbb{R}^{n\times (P+1)}, \quad \widetilde{\Phi}_{S_{ik}}=\phi(y_{s_i}; c_{k}),
\label{eq:RFNNBasis}
\end{equation} 
for $i=1,\dots,n$ and $k=1,\dots,P$. Let us further stack the (down)sampled target outputs in the data matrix $X_S = \begin{bmatrix} x_{s_1}, & \ldots, &x_{s_n} \end{bmatrix} \in \mathbb{R}^{M \times n}$. Then, by collecting the unknown coefficients (output weights and biases) into the row-wise $(P+1) \times M$ matrix
\begin{equation}
    A_S := 
\begin{bmatrix}
\beta &
\alpha_1 &
\dots &
\alpha_P
\end{bmatrix}^{\top}
\in\mathbb{R}^{(P+1) \times M}, \quad 
\beta = \bigl[\beta_1,\, \ldots, \, \beta_M \bigr]^\top, \quad
\alpha_k=\bigl[\alpha_{k1},\, \ldots, \,\alpha_{kM}\bigr]^\top, 
\end{equation}
for $k=1,\ldots,P$, we obtain the RFNN formula in the compact matrix form 
\begin{equation}
X^{\top}_{S} = \Phi_S A_S.
\label{eq:RFNNdecoder}
\end{equation}
We remark that although Eq.~\eqref{eq:RFNNdecoder} is derived for the training set (since $\Phi_S$ is evaluated on $\mathcal{Y}_S$ and $X_S$ contains the corresponding targets), the same linear relationship via $A_S$ can be used directly for inference on new latent points. In what follows, we will denote such out‑of‑sample predictions explicitly; otherwise, for notational compactness, we work with the training set formulation.

The linear form of the RFNN decoder in Eq.~\eqref{eq:RFNNdecoder} reduces training to a linear least-squares problem for computing $A_S$, the solution of which can be sought for example using the vanilla Tikhonov regularization
\begin{equation}\label{eq:obj-row}
\min_{A_S\in\mathbb{R}^{P+1\times M}} \; 
\big\|X^{\top}_S - \Phi_S A_S \big\|_2^2 + \lambda\,\|A_S\|_2^2,
\qquad \lambda>0,
\end{equation}
which provides the closed-form solution
\begin{equation}\label{eq:vanillaTikhonov}
A_S = (\Phi_S^\top \Phi_S + \lambda I_{P+1})^{-1}\, \Phi_S^\top \, X^{\top}_S \in\mathbb{R}^{(P+1)\times M},
\end{equation}
and its dual formulation
\begin{equation}
A_S = \Phi^\top_S (\Phi_S \Phi^\top_S + \lambda I_n)^{-1}\, X^{\top}_S \in\mathbb{R}^{(P+1)\times M},
\label{eq:dualform}
\end{equation}
where $I_{P+1}$ and $I_n$ denote the $(P+1)$- and $n$-dimensional unity matrices. The form in Eq.~\eqref{eq:vanillaTikhonov} is appropriate for overdetermined or square systems ($n\ge P+1$), while the one in Eq.~\eqref{eq:dualform} is for appropriate undetermined systems ($n< P+1$) for avoiding numerical instabilities arising from the inversion of singular matrices. 

Another way to solve the system in Eq.~\eqref{eq:RFNNdecoder} is via SVD of the feature matrix $\Phi_S$. Letting $r=rank(\Phi_S)$, the economy-size SVD gives
\begin{equation}
\Phi_S = U_r \,\Sigma_r\, V_r^\top, \qquad U_r \in \mathbb{R}^{n\times r}, \qquad
V_r \in \mathbb{R}^{(P+1)\times r}, \qquad
\Sigma_r = \operatorname{diag}(\sigma_1,\dots,\sigma_r) \in \mathbb{R}^{r\times r},
\end{equation}
where the columns of $U_r$ and $V_r$ include the left and right singular vectors $u_i \in \mathbb{R}^n$ and $v_i \in \mathbb{R}^{P+1}$, respectively, and $\sigma_i>0$ are the singular values of $\Phi_S$ for $i=1,\ldots,r$. Substituting this decomposition into Eq.~\eqref{eq:vanillaTikhonov} (or equivalently into Eq.~\eqref{eq:dualform}) yields the spectral filter form of the Tikhonov solution
\begin{equation}
A_S=V_r\,(\Sigma_r^{2}+\lambda I_r)^{-1}\,\Sigma_r\, U_r^{\top} X_S^{\top},
\label{eq:TikhSVD}
\end{equation}
or, written as a sum of rank-1 terms
\begin{equation}
A_S=\sum_{i=1}^{r} \frac{\sigma_i}{\sigma_i^2 + \lambda}\; v_i \,\bigl(u_i^{\!\top} X^{\top}_S\bigr).
\end{equation}

For $\lambda \to 0$, we recover the Moore–Penrose pseudo-inverse-based solution
\begin{equation}
A_S\xrightarrow[\lambda\to 0]{} V_r\,\Sigma_r^{-1}\,U_r^\top\, X^{\top}_S= 
\Phi_S^{+}\,X^{\top}_S,
\label{eq:MoorePenrose}
\end{equation}
where $\Phi_S^{+}$ denotes the pseudo-inverse of $\Phi_S$.

Having derived the solution of $A_S$, we finally obtain the reconstruction provided by the RFNN decoder. For a set of new latent points $\mathcal{Y}^*=\{ y^*_l\}_{l=1}^L\subset\mathbb{R}^d$, we first form the corresponding feature matrix
\begin{equation}
    \Phi^* = \big[1_L \, | \, \widetilde{\Phi}^* \big]
\in\mathbb{R}^{L\times (P+1)}, \quad \widetilde{\Phi}_{lj}^*=\phi(y^*_{l}; c_{j}),
\end{equation}
and then compute the reconstructed fields $\hat{x}^*_l$ in the matrix
\begin{equation}
\widehat{X}^* = A^{\top}_S \Phi^{*\top} \in \mathbb{R}^{M \times L},
\label{eq:RFNN_dec}
\end{equation}
using the closed-form solutions of $A_S$ in Eq.~\eqref{eq:vanillaTikhonov}, \eqref{eq:dualform}, \eqref{eq:TikhSVD}, or \eqref{eq:MoorePenrose}.

In a NN representation, one may choose among various forms for $\phi$, including random Fourier features derived from Bochner's theorem, or standard activation functions such as sigmoid or hyperbolic tangent, each yielding different approximation properties. %In what follows, we consider three specific choices for the feature map $\phi$: (i) random Fourier cosine features derived from Bochner's theorem; (ii) random multiscale Fourier cosine features sampled from a mixture of Gaussian distributions with different bandwidths; and (iii) random features using the standard sigmoid activation function.
%\subsection{The equivalence of vanilla random Fourier (RFNN) and DDM decoders} \label{sb:RFNNequiv}
At this point, we demonstrate that choosing vanilla random Fourier features in RFNN decoders is equivalent, in the deterministic limit, to the DDM decoder and RBFs interpolation.

\begin{prop}
Consider the dataset $\mathcal{X}_S=\{ x_{s_i}\}_{i=1}^n\subset\mathbb{R}^M$ in the ambient space, the corresponding set of latent coordinates $\mathcal{Y}_S=\{ y_{s_i}\}_{i=1}^n\subset\mathbb{R}^d$. Further, consider the RFNN decoder in Eq.~\eqref{eq:RFNN} without output bias (i.e., with $\beta_j=0$ for every $j=1,\ldots,M$) and with $P$ random Fourier features
\begin{equation}
\phi(y;c_k) = \sqrt{\frac{2}{P}}\cos(w_k^\top y + b_k), \quad c_k=(w_k,b_k), \quad k=1,\ldots,P
\label{eq:fourierrand}
\end{equation}
with $w_k \in \mathbb{R}^d$ sampled i.i.d. from the multivariate Gaussian distribution $w_k\sim \mathcal{N}(0,\Sigma_w)$, where $\Sigma_w$ is a $d\times d$ covariance matrix, and $b_k$ a random phase sampled i.i.d. from the uniform distribution $b_k\sim \mathcal{U}[0,2\pi)$. Then, for a set of new latent points $\mathcal{Y}^*=\{ y_l^*\}_{l=1}^L\subset\mathbb{R}^d$, the corresponding reconstructed fields provided by the RFNN decoder, when the RFNN pre-image is computed using the Moore-Penrose pseudo-inverse solution in Eq.~\eqref{eq:MoorePenrose}, converges as $P\to \infty$ to those provided by the DDM decoder in Eq.~\eqref{eq:DoubleDMs_mat_main} for $P=N$. %Moreover, the kernel induced by the RFNN random Fourier features approximates a Gaussian kernel in the deterministic limit $P \to \infty$; hence, the RFNN reconstruction converges, in that limit, to that of the DDM decoder. 
\end{prop}
\begin{proof}
Assuming the RFNN decoder in Eq.~\eqref{eq:RFNN} without output bias, we obtain the reduced feature matrix $\Phi_S=\widetilde{\Phi}_S$ with elements $\Phi_{S_{ik}}=\phi(y_i;c_k)$; see Eq.~\eqref{eq:RFNNBasis}. Using the normalized Fourier random features in Eq.~\eqref{eq:fourierrand}, the kernel induced by the RFNN decoder is
\begin{equation}
    K_S=\Phi_S\, \Phi_S^\top \in \mathbb{R}^{n \times n}, \quad K_{S_{ij}}=\big(\Phi_S \, \Phi_S^\top \big)_{ij}=\dfrac{2}{P}\sum_{k=1}^P \cos(w_k^\top y_i + b_k)\cos(w_k^\top y_j + b_k).
\end{equation}
By the Bochner's theorem \cite{rahimi2007random,rahimi2008uniform}, for $w_k\sim \mathcal{N}(0,\Sigma_w)$ and $b_k\sim \mathcal{U}[0,2\pi)$, the expectation of each element is
\begin{align}
\mathbb{E}_{w,b}\big[K_{S_{ij}}\big]
= \mathbb{E}_{w}\big[\cos \big(w^\top (y_i - y_j)\big)\big]
=\Re\Big\{\int_{\mathbb{R}^d} e^{i w^{\top} (y_i - y_j)}\, p(w)\, dw \Big\}= \nonumber \\
\exp\Big(-\frac{1}{2}(y_i - y_j)^\top \Sigma_w (y_i - y_j)\Big)=\bar{K}_S(y_i,y_j),
\label{eq:expectedcosines}
\end{align}
where $p(w)=\mathcal{N}(0,\Sigma_w)$ and $\bar{K}_S$ is the Gaussian kernel with covariance $\Sigma_w$. 

The above result holds in expectation, and therefore the RFNN kernel $K_S$ approximates the true shift-invariant kernel $\bar K_S$ only on average over the random draws of $w_k$ and $b_k$; for finite $P$, $K_S$ is a Monte Carlo approximation of the Gaussian kernel. In the isotropic case where $\Sigma_w = \sigma_w^2 I_d$, the expected kernel becomes the standard RBF kernel
\begin{equation}
  \bar K_S(y_i,y_j) = \exp\Big(-\dfrac{1}{2}\sigma_w^2\|y_i - y_j\|^2\Big),  
  \label{eq:RBF_kernel}
\end{equation}
with bandwidth $1/\sigma_w$, and the approximation error of $K_S$ can be bounded in spectral norm via the matrix Bernstein inequality \cite{tropp2015introduction}
\begin{equation}
\mathbb{E}\,\|K_S - \bar K_S\| \le
2 \sqrt{\frac{n \|\bar K_S\| \log(2n)}{P}} +
\frac{4 n \log(2n)}{3P},\label{eq:Tropp-Bernstein1}
\end{equation}
where $\|\cdot\|$ denotes the spectral norm. From Eq.~\eqref{eq:Tropp-Bernstein1}, it is evident that as $P \to \infty$, the spectral norm $\|K_S - \bar K_S\| \to 0$, i.e., the RFNN-induced kernel converges to the exact Gaussian (RBF) kernel. Therefore, for sufficiently large $P$, $K_S$ can be treated as a high-fidelity approximation of $\bar{K}_S$, justifying its use in the reproducing kernel Hilbert space (RKHS) pre-image solution that follows.

By the representer theorem, %\cite{kimeldorf1971some}, 
the solution of the Tikhonov regularized least-squares problem in the RKHS $\mathcal{H}_k$ associated with the RFNN-induced kernel $K_S$, admits an $n$-finite expansion in terms of kernel section evaluated at the training points. Its solution yields the corresponding RKHS coefficients \cite{bollt2020regularized}
\begin{equation}
    A^{\mathcal{H}_k}_S=(K_S+\lambda I_n)^{-1}\,X^{\top}_S \in \mathbb{R}^{n\times M}, \quad \lambda>0.
    \label{eq:ridge3}
\end{equation}
Lifting this solution to the feature space results to RFNN feature-space coefficients
\begin{equation}
    A_S = \Phi_S^\top A^{\mathcal{H}_k}_S= \Phi^\top_S (\Phi_S \Phi^\top_S + \lambda I_n)^{-1}\, X^{\top}_S,
    \label{eq:RFNNcoefs}
\end{equation}
which matches the dual form Tikhonov regularized RFNN solution given in Eq.~\eqref{eq:dualform}. 
 
Given the set of out-of-sample points $\mathcal{Y}^*$ to be reconstructed in the ambient space, we now define the cross-kernel matrix between $\mathcal{Y}^*$ and the training points in $\mathcal{Y}$ as
\begin{equation}
    K^*= \Phi^*\, \Phi_S^\top \in \mathbb{R}^{L \times n}, \quad 
    \Phi^*_{lk}=\phi(y^*_l;c_k) \in \mathbb{R}^{L \times P}.
\end{equation}
Substituting the coefficients in Eq.~\eqref{eq:RFNNcoefs} to the  RFNN reconstruction formula in Eq.~\eqref{eq:RFNN_dec}, we obtain the reconstructed fields $\hat{x}^*_l$ for $l=1,\ldots,L$ in the matrix
\begin{equation}
    \widehat{X}^* = A^{\top}_S \Phi^{*\top} = X_S \big(\Phi_S\, \Phi_S^\top + \lambda I_n \big)^{-1} \Phi_S \Phi^{*\top} = X_S \big(K_S + \lambda I_n \big)^{-1} K^{*\top} \in \mathbb{R}^{M \times L}.
    \label{eq:RFNN_kernel_reg_sol}
\end{equation}
By the representer theorem, the optimal regularized solution lives in the span $\mathrm{span}\{\,k_S(y_{s_i},\cdot)\,\}_{i=1}^n$, which is at most $n$-dimensional, whereas the finite RFNN restricts solutions to the $P$-dimensional subspace $\mathrm{span}\{\, \phi(\cdot; c_k) \,\}_{k=1}^P$. Whenever the feature-induced space contains the kernel  span, the feature-based RFNN solution coincides with the kernel ridge regression solution associated with $K_S$, as shown in Eq.~\eqref{eq:RFNN_kernel_reg_sol}. In this case, the difference between the two formulations is computational, governed by the relative sizes of $P$ and $n$, the choice of which balances training cost and reconstruction accuracy.

Finally, letting $\lambda\to0$ in Eq.~\eqref{eq:RFNN_kernel_reg_sol}, we obtain the Moore-Penrose pseudo-inverse-based reconstruction
\begin{equation}
    \widehat{X}^* = X_S V_r \Lambda_r^{-1}V_r^\top K^{*\top},
\end{equation}
where $K_S^{+}=V_r \Lambda_r^{-1}V_r^\top$ is the rank-$r$ pseudo-inverse of $K_S$. This reconstruction coincides with that of the DDM decoder in Eq.~\eqref{eq:DoubleDMs_mat_main} whenever the RFNN-induced kernel $K_S$ matches the DDM kernel $K$. As the RFNN-induced kernel $K_S$ is a finite approximation of the Gaussian kernel $\bar{K}_S$, which is also used in the DDM decoder, $K_S$ converges to $K$ as $P\to \infty$ for $P=N$ and matching kernel shape parameters $\sigma_w$ and $\epsilon_2$.
\end{proof}

\begin{remark}
While a classical Radial Basis Function (RBF) decoder is not considered in this work, it is conceptually related to the RFNN-induced kernel approach, as both methods define a feature map that embeds the latent coordinates into a RKHS. In particular, the RFNN-induced kernel can be seen as a Monte Carlo approximation of a shift-invariant RBF kernel, as shown for the isotropic case in Eq.~\eqref{eq:RBF_kernel}. However, while both the RFNN and DDM decoders reduce into a linear, convex pre-image problem (for computing the output weights and biases, and the basis $\{V_r, \Lambda_r\}$, respectively), in the classic RBF setting it is often necessary to optimize both the hidden layer coefficients and the kernel widths. This makes the RBF pre-image problem nonlinear and non-convex, and it can be further complicated by the ``curse of dimensionality'', potentially leading to degraded generalization performance.
\end{remark}
\subsection{Multi-scale random Fourier features} \label{sb:RFNNdec_ms}
We now consider a multi-scale variant of RFNNs, where the random Fourier features exhibit frequencies drawn from a mixture of Gaussians with varying bandwidths (see Fig.~\ref{fig:Schematic} for a schematic). This construction yields an induced kernel that is a continuous mixture of Gaussian kernels, providing multi-scale representation capabilities. The following theorem formalizes this connection.
\begin{figure}[!h]
    \centering
    \includegraphics[width=0.65\linewidth]{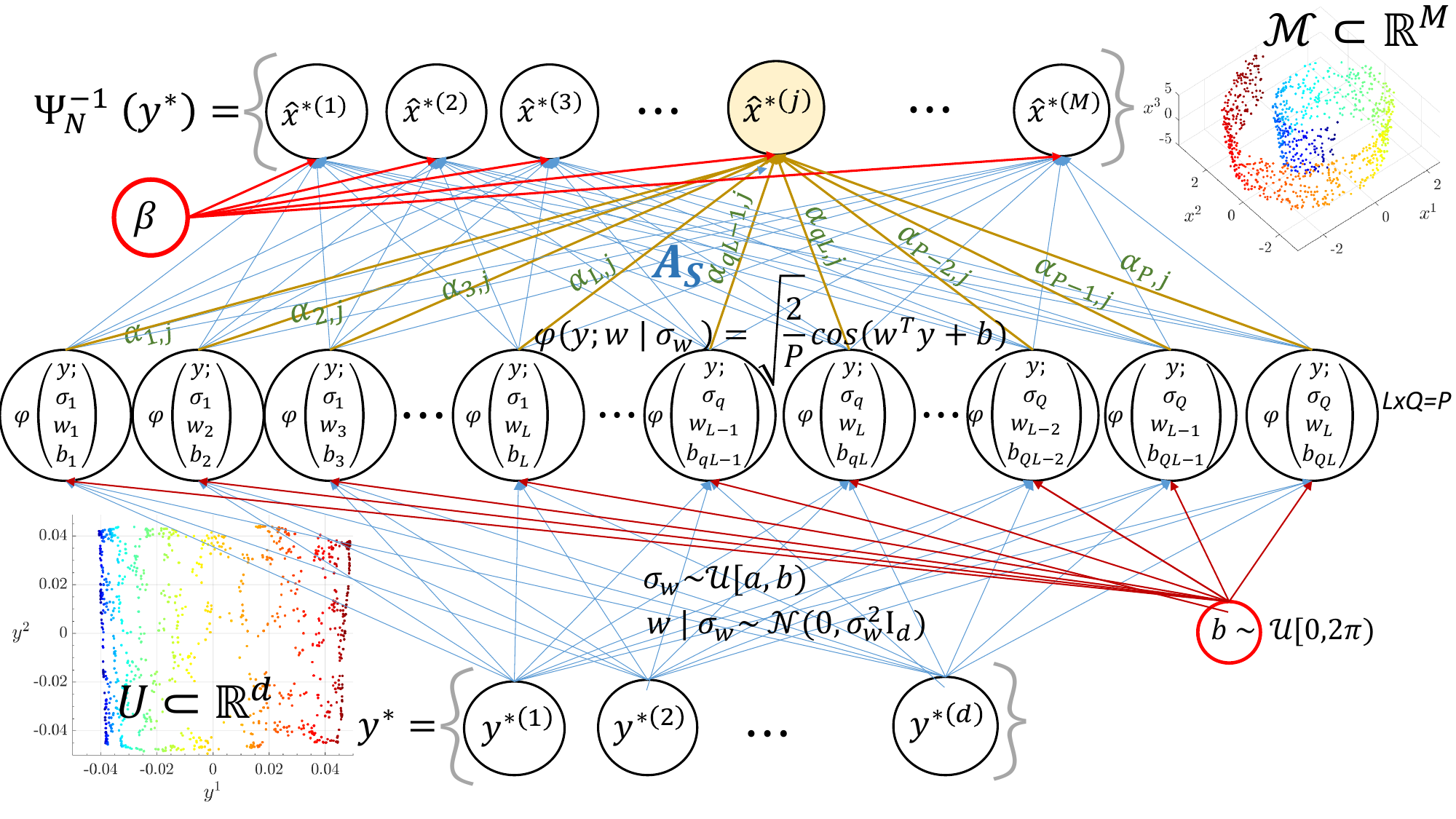}
    \caption{Schematic of the multi-scale Random Feature Neural Network decoder with random Fourier features; the features follow Eq.~\eqref{eq:fourierrand} with phase parameter $b_k \sim \mathcal{U}[0,2\pi)$ and frequency vectors $w_k$ sampled from the conditional probability distribution in Eq.~\eqref{eq:conditionalprobab}, i.e., sampling $\sigma_w \sim \mathcal{U}[a,b)$ and then draw $w \mid \sigma_w \sim \mathcal{N}(0, \sigma_w^2 I_d)$.}
    \label{fig:Schematic}
\end{figure}
%Let's take the Gaussian kernel
%\begin{equation}
%%k_{G}(y_i,y_j)=e^{-\alpha^2 \|y_i-y_j\|_2^2},
%\end{equation}
%per se, in which case the scale parameter $\alpha$ is drawn from a uniform distribution in $(a,b)$ \cite{fabiani2021numerical,fabiani2023parsimonious}. In this case, we demonstrate the following Theorem.
\begin{theorem}
Consider the random Fourier features $\phi(y;c_k) = \sqrt{2/P}\cos(w_k^\top y + b_k)$, $k=1,\ldots,P$ as defined in Eq.~\eqref{eq:fourierrand} with $b_k \sim \mathcal{U}[0,2\pi)$. Let the frequency vectors $w_k$ be sampled from the conditional probability distribution
\begin{equation}
p(w\mid\sigma_w)
=(2\pi\sigma_w^2)^{-d/2}\exp\Big(-\frac{\|w\|_2^2}{2\sigma_w^2}\Big),
\label{eq:conditionalprobab}
\end{equation}
conditioned on $\sigma_w$ that are sampled from a uniform distribution in $(a,b)$, i.e., $\sigma_w\sim\mathcal{U}[a,b)$. Then, the resulting RFNN induced kernel (without output bias) is the multi-Gaussian kernel 
\begin{equation}
    \mathbb{E}_{w,b}\big[(\Phi_S\, \Phi_S^\top)_{ij}\big] = \bar{K}_{\Sigma_w}(y_i,y_j)= \frac{1}{b-a} \int_a^b e^{-\frac{1}{2}\sigma_w^2 \|y_i - y_j\|_2^2} d\sigma_w.
\end{equation}
Moreover, the marginal distribution of $w$ obtained by this hierarchical sampling is
\begin{equation}
    p(w)=\frac{(2\pi)^{-d/2}}{b-a} \int_a^b \sigma_w^{-d}\exp \left(-\frac{\|w\|_2^2}{2\sigma_w^2}\right)\, d\sigma_w.
\end{equation}
\end{theorem}
\begin{proof}
We start by noting that the expected value of the Gaussian kernel $K_{\sigma_w}(y_i, y_j)=e^{-\frac{1}{2}\sigma_w^2 \|y_i - y_j\|_2^2}$, with respect to the shape parameter $\sigma_w$, when $\sigma_w\sim\mathcal{U}[a,b)$ is given by
\begin{equation}
\mathbb{E}_{\sigma_w} \left[e^{-\frac{1}{2}\sigma_w^2 \|y_i - y_j\|_2^2}\right]
= \frac{1}{b-a} \int_a^b e^{-\frac{1}{2}\sigma_w^2 \|y_i - y_j\|_2^2} d\sigma_w=\bar{K}_{\Sigma_w}(y_i,y_j).
\label{eq:mixedgaussians}
\end{equation}
It is easy to show that the resulting multi-Gaussian kernel $\bar{K}_{\Sigma_w}$ is a symmetric (and shift-invariant), continuous positive definite kernel on $\mathbb{R}^d$.

%By the Bochner's theorem \cite{bochner1933monotone}, we have that there exists a proper probability density $p(w)$ on $\mathbb{R}^d$ such that:
%\begin{equation}
%\bar{k}_{C_w}(\delta)=\int_{\mathbb{R}^d} e^{i w^\top \delta}\, p(w)\,dw, \quad \delta:=y_i-y_j.
%\end{equation}

Given a fixed value $\sigma_w$, we know by Bochner's theorem, that the single Gaussian kernel $K_{\sigma_w}(y_i, y_j)$ has a non negative spectral density
\begin{equation}
p(w\mid\sigma_w)
=(2\pi\sigma_w^2)^{-d/2}\exp\Big(-\frac{\|w\|_2^2}{2\sigma_w^2}\Big),
\quad \int_{\mathbb{R}^d} p(w\mid\sigma_w)\,dw = 1.
\end{equation}
Its marginal spectral density integrating over $\sigma_w$ is given by
\begin{equation}
p(w) = \int_a^b p(w \mid \sigma_w) \, p(\sigma_w) \, d\sigma_w=\frac{1}{b-a}\int_a^b p(w\mid \sigma_w)\,d\sigma_w
=\frac{(2\pi)^{-d/2}}{b-a}
\int_a^b \sigma_w^{-d}\exp \left(-\frac{\|w\|_2^2}{2\sigma_w^2}\right)\,d\sigma_w,
\label{eq:marginal}
\end{equation}
for which $\int_{\mathbb{R}^d} p(w)\,dw = 1.$

Hence, for the corresponding RKHS, based on the RFNN induced kernel without output bias for the random Fourier features in Eq.~\eqref{eq:fourierrand}, and Bochner's theorem, we get
\begin{align} 
& \mathbb{E}_{w,b}\big[(\Phi_S\Phi_S^\top)_{ij}\big]
= \mathbb{E}_{w}\big[\cos \big(w^\top (y_i - y_j)\big)\big] = \Re\Big\{\int_{\mathbb{R}^d} e^{i w^{\top} (y_i - y_j)}\, p(w)\, dw \Big\}= \nonumber \\ 
& \int_{\mathbb{R}^d} e^{i w^\top (y_i - y_j)}\,p(w)\,dw
=\frac{1}{b-a}\int_a^b \Big(\int_{\mathbb{R}^d} e^{i w^\top \delta} p(w\mid \sigma_w)\,dw\Big) d\sigma_w,
\label{eq:mixtures2}
\end{align}
where, for the interchange of the order of integration, the Fubini's theorem \cite{rudin1987real} applies as the integrand $e^{i w^\top(y_i - y_j)}\,p(w)$ is $L^1 \big(\mathbb{R}^d\times(a,b)\big)$ integrable. Hence, from Eq.~\eqref{eq:mixtures2}, we get
\begin{equation}
\mathbb{E}_{w,b}\big[(\Phi_S\Phi_S^\top)_{ij}\big]=\frac{1}{b-a}\int_a^b e^{-\frac{1}{2}\sigma_w^2 \|y_i - y_j\|_2^2}\,d\sigma_w,
\end{equation}
thus, recovering $\bar{K}_{\Sigma_w}(y_i, y_j)$.

Therefore, the marginal distribution $p(w)$ in Eq.~\eqref{eq:marginal} is the spectral density of the multi-Gaussian kernel $\bar{K}_{\Sigma_w}$, obtained by mixing Gaussians $\mathcal{N}\left(0, \sigma_w^2I_d\right)$ with $\sigma_w \sim \mathcal{U}[a,b)$.
\end{proof}
The above theorem yields that an unbiased estimation of the multi‑Gaussian kernel $\bar{K}_{\Sigma_w}$ is achieved by implementing the following simple hierarchical sampling scheme:  first sampling a scale parameter $\sigma_w \sim \mathcal{U}[a,b)$, then sampling the frequency vector $w$, conditioned on $\sigma_w$, from $w \mid \sigma_w \sim \mathcal{N}(0,\,\sigma_w^2 I_d)$, along with sampling an independent random phase $b \sim \mathcal{U}[0,2\pi)$.

Note that in the above scheme, when we aproximate the integral with a finite number of random features, we get a (uniform) mixture of Gaussian kernels, each one with its own shape parameter $\sigma_w$ (see Eq.~\eqref{eq:mixedgaussians}), while the vanilla random Fourier features give as expected value a single Gaussian kernel (see Eq.~\eqref{eq:expectedcosines}). Actually, this approach falls within the multiple kernel learning framework \cite{bach2008exploring,lin2008dimensionality,gonen2011multiple}%,wilson2013gaussian}.

The Monte Carlo implementation of the above scheme is performed here by drawing
\begin{equation}
\sigma_q \sim \mathcal{U}[a,b),\quad q=1,\dots,Q,
\end{equation}
and for each $\sigma_q$ draw
\begin{equation}
w_{q\ell}\sim \mathcal{N}(0,\sigma_q^2 I_d),\qquad
b_{q\ell}\sim \mathcal{U}[0,2\pi),\quad \ell=1,\dots,L,
\end{equation}
where $Q$ is the number of scales and $L$ is the number of features per scale. With $P=Q L$ total features, the Monte-Carlo estimator of the expected kernel $\bar{K}_{\Sigma_w,Q}(y_i,y_j)$, reads
\begin{equation}\label{eq:double-sum}
K_{\Sigma_w,Q}(y_i,y_j)
=\frac{1}{Q}\sum_{q=1}^{Q}\left(\frac{1}{L}\sum_{\ell=1}^{L}
2\cos(w_{q\ell}^\top y_i+b_{q\ell})\cos(w_{q\ell}^\top y_j+b_{q\ell})\right) =\frac{1}{Q} \sum_{q=1}^{Q} K_{\sigma_q}(y_i, y_j)\,
\end{equation}
thus corresponding to the concatenation of the individual random feature kernels, where
\begin{equation}
K_{\sigma_q}(y_i, y_j)=
\frac{1}{L}\sum_{\ell=1}^{L}
2\,\cos\big(w_{q\ell}^\top y_i + b_{q\ell}\big)\, \cos\big(w_{q\ell}^\top y_j + b_{q\ell}\big),
\end{equation}
is an unbiased Monte-Carlo estimator of the expected Gaussian kernel at the scale $\sigma_q$,
\begin{equation}
\bar{K}_{\sigma_q}(y_i, y_j)=
e^{-\tfrac{1}{2}\sigma_q^2 \|y_i - y_j\|_2^2}.
\end{equation}
Based on the Bernstein inequality \cite{tropp2015introduction} in Eq.~\eqref{eq:Tropp-Bernstein1}, we will now demonstrate the following Corollary.
\begin{corollary}
The expected approximation bound of the finite multi-scale induced by RFNN kernel $K_{\Sigma_w,Q}$ in Eq.~\eqref{eq:double-sum} for $P=Q\, L$, with respect to the expected kernel, $\bar{K}_{\Sigma_w,Q}$ when $L\to \infty$ is given by:
\begin{equation}
\mathbb{E} \big\| K_{\Sigma_w,Q} - \bar{K}_{\Sigma_w,Q} \big\| \le
2\sqrt{\frac{n \, c\, Q\, \log(2n)}{P}}
+
\frac{4 n\,Q\, \log(2n)}{3 P}, \, c>0.
\label{eq:boundmultikernel}
\end{equation}
\end{corollary}
\begin{proof}
For each scale $q=1,\ldots,Q$ with $\sigma_q$,  Eq.~\eqref{eq:Tropp-Bernstein1} implies
\begin{equation}
\mathbb{E} \big\| K_{\sigma_q} - \bar{K}_{\sigma_q} \big\| \le 
2\sqrt{\frac{n \|\bar{K}_{\sigma_q}\|\, \log(2n)}{L}}
+
\frac{4 n\, \log(2n)}{3 L}.
\label{eq:Qstate_bound}
\end{equation}
By averaging over the $Q$ scales, that is
$K_{\Sigma_w,Q}=\frac{1}{Q}\sum_{q=1}^Q K_{\sigma_q}$ and $ \bar{K}_{\Sigma_w,Q}=\frac{1}{Q}\sum_{q=1}^Q 
\bar{K}_{\sigma_q} $, we obtain from Eq.~\eqref{eq:Qstate_bound} the bound of the multi-scale approximation 
\begin{equation}
\mathbb{E} \big\| K_{\Sigma_w,Q} - \bar{K}_{\Sigma_w,Q} \big\| = \frac{1}{Q}\sum_{q=1}^Q\left(\mathbb{E} \big\| K_{\sigma_q} - \bar{K}_{\sigma_q} \big\| \right)
\le 
\frac{1}{Q}\sum_{q=1}^Q\left(
2\sqrt{\frac{n \|\bar{K}_{\sigma_q}\|\, \log(2n)}{L}}
+
\frac{4 n\, \log(2n)}{3 L}\right).
\end{equation}
Under the assumption that there exists a $c\in\mathbb{R}$ such that $ \|\bar{K}_{\sigma_q}\|\le c$ for every scale $q=1,2,\dots Q$, and substituting the number of features per scale by $L=P/Q$, yields the bound in Eq.~\eqref{eq:boundmultikernel}.
\end{proof}

\subsection{RANDSMAPS: Random and multi-scale feature neural decoders with mass-conservation} \label{sb:RANDSMAP}

While the vanilla RFNN decoders of Eq.~\eqref{eq:RFNN}, discussed in Section~\ref{sb:RFNNdec} can approximate arbitrary fields via Eq.~\eqref{eq:RFNN_dec}, they do not inherently preserve physical constraints such as mass conservation. When such constraints are present in the original data, a reconstruction that violates them may be physically inadmissible. To address this limitation, we introduce RANDSMAP (Random-feature Neural Decoders with Mass Preservation), a family of RFNN‑based (and multi-scale) decoders that explicitly respect the conservation of mass \emph{by construction}. 

RANDSMAPs retain the same decoding architecture as RFNNs in Eq.~\eqref{eq:RFNN}. However, instead of computing the unknown coefficients $A_S$ via the closed-form solutions of the unconstrained least-squares problem in Eq.~\eqref{eq:vanillaTikhonov}, \eqref{eq:dualform}, \eqref{eq:TikhSVD}, or \eqref{eq:MoorePenrose}, we now solve a constrained optimization problem. This formulation augments the standard $L_2$-norm minimization with a linear equality constraint that enforces the sum‑to‑one invariant required of the reconstructed fields in Eq.~\eqref{eq:s2one_cons_recon}. To obtain a closed‑form solution, we employ the method of Lagrange multipliers, which guarantees mass-preserving reconstructions. We first formalize this procedure for the (down)sampled input set $\mathcal{Y}_S$ and target set $\mathcal{X}_S$ in the following proposition.

\begin{prop} \label{prp:RANDSMAP}
Suppose the (down)sampled dataset $\mathcal{X}_S=\{x_{s_i}\}_{i=1}^n\subset\mathbb{R}^M$ is mass-preserving, i.e., the data matrix $X_S \in \mathbb{R}^{M \times n}$ satisfies the (normalized) sum-to-one constraint $1_M^\top X_S = 1_n^\top$. Let $\Phi_S = U_r\Sigma_r V_r^\top$ be the SVD of the feature matrix $\Phi_S\in\mathbb{R}^{n\times (P+1)}$ with $\operatorname{rank}(\Phi_S)=r$. Then, the coefficient matrix $A_S\in\mathbb{R}^{(P+1)\times M}$ given by 
\begin{equation}
A_S=V_r(\Sigma^2_r+\lambda I_r)^{-1}\Sigma_r U^{\top}_r\big(X^{\top}_S-\frac{1}{M}\big[I_n-U_r(I_r + \lambda \Sigma^{-2}_r)U^{\top}_r\big]1_n\,1_M^\top\big),
    \label{eq:ShurAfinal0}
\end{equation}
or, in the vanishing regularization limit $\lambda \to 0$, by the Moore-Penrose pseudo-inverse-based solution
\begin{equation} A_S=V_r\Sigma^{-1}_rU^{\top}_rX^{\top}_S=\Phi_S^{+}X^{\top}_S,
    \label{eq:ShurAfinal0s}
\end{equation}
ensures that the reconstructed fields $\widehat{X}_S = A^{\top}_S \Phi^\top$ provided by RANDSMAP decoder are mass-preserving, satisfying the sum-to-one constraint $1_M^\top\widehat{X}_S=1_n^\top$.
\end{prop}
\begin{proof}
When the (down)sampled data satisfy mass conservation, the RANDSMAP reconstruction $\widehat{X}_S \in \mathbb{R}^{M \times n}$ must preserve it as in Eq.~\eqref{eq:s2one_cons_recon}, i.e., 
\begin{equation}
1_M^\top \widehat{X}_S = 1_n^\top \Leftrightarrow 1_M^\top A_S^\top \Phi_S^\top  = 1_n^\top \Leftrightarrow \Phi_S A_S\,1_M = 1_n.%\quad \text{or} \quad \frac{1}{n}1^{\top}_n \, K_S A_S\,1_M = 1,
\label{eq:masscond}
\end{equation}
To enforce this constraint, we will first derive the solution in Eq.~\eqref{eq:ShurAfinal0} for the regularized least‑squares problem in Eq.~\eqref{eq:obj-row} subject to the linear condition in Eq.~\eqref{eq:masscond}. The minimum $L_2$-norm solution in Eq.~\eqref{eq:ShurAfinal0s} is obtained in the vanishing-limit of the Tikhonov regularization solution.

Introducing a Lagrange multiplier $\mu\in\mathbb{R}^n$ for the $n$ equality constraints yields the Lagrangian
\begin{equation}
\mathcal{L}(A_S,\mu) = \frac{1}{2}\|X^{\top}_S-\Phi_S A_S\|^2_2 + \frac{\lambda}{2}\|A_S\|^2_2 + \mu^\top (\Phi_S A_S\,1_M -1_n), \, \lambda>0,
\end{equation}
which is minimized w.r.t. $A_S$ and $\mu$. Assuming the standard convexity hypothesis of $\nabla^2_{A_S A_S}\mathcal{L}(A,\mu)=\Phi_S^\top \Phi_S + \lambda I_{P+1}$ being positive definite, which is satisfied under the mild condition $\lambda>0$, and setting the derivatives of $\mathcal{L}(A_S,\mu)$ w.r.t. $A_S$ and $\mu$ to zero gives the system  
\begin{equation}
\label{eq:KKT-matrix}
\begin{aligned}
&(\Phi_S^\top \Phi_S + \lambda I_{P+1}) A_S + \Phi_S^\top \mu\,1_M^\top = \Phi_S^\top X^{\top}_S,\\
& \Phi_S A_S\,1_M =1_n.
\end{aligned}
\end{equation}
Equation \eqref{eq:KKT-matrix} is a linear system coupling unknowns $A_S\in\mathbb{R}^{(P+1)\times M}$ and $\mu\in\mathbb{R}^n$, which can be solved using the Shur complement, as $\Phi_S^\top \Phi_S + \lambda I_{P+1}$ is invertible. Solving first for $A_S$ implies
\begin{equation}
    A_S=(\Phi_S^\top \Phi_S + \lambda I_{P+1})^{-1}(\Phi_S^\top X^{\top}_S-\Phi_S^\top \mu\,1_M^\top),
    \label{eq:ShurA}
\end{equation}
which, upon substitution into $\Phi_S A_S\,1_M =1_n$, yields the reduced equation
\begin{align}
    \Phi_S(K_S^\top \Phi_S + \lambda I_{P+1})^{-1}(\Phi_S^\top X^{\top}_S-\Phi_S^\top \mu\,1_M^\top)1_M & = 1_n, \, \mbox{or} \nonumber \\
    \Phi_S(\Phi_S^\top \Phi_S + \lambda I_{P+1})^{-1}(\Phi_S^\top 1_n-M\,\Phi_S^\top \mu) & = 1_n, \, \mbox{or} \nonumber \\
    \Phi_S(\Phi_S^\top \Phi_S + \lambda I_{P+1})^{-1}\Phi_S^\top( 1_n-M\,\mu) & =1_n,
   \label{eq:shur3}
\end{align}
where the (transposed) constrain $X_S^\top 1_M=1_n$ was used in the second step. Using the dual form identity
\begin{equation}
    (\Phi_S^\top \Phi_S + \lambda I_{P+1})^{-1}\Phi_S^\top=\Phi_S^\top(\Phi_S \Phi_S^\top + \lambda I_n)^{-1},
    \label{eq:dual_identity}
\end{equation}
the expression in Eq.~\eqref{eq:shur3} becomes
\begin{equation}
\Phi_S\Phi_S^\top(\Phi_S \Phi_S^\top + \lambda I_n)^{-1}( 1_n-M\,\mu) =1_n,
   \label{eq:shur4}
\end{equation}

Now considering the SVD of the feature matrix $\Phi_S=U_r\Sigma_r V^{\top}_r$, where $U_r \in \mathbb{R}^{n\times r}$, $\Sigma_r\in \mathbb{R}^{r \times r}$ and $V_r \in \mathbb{R}^{(P+1)\times r}$ with $r=\operatorname{rank}(\Phi_S)$, substitution in Eq.~\eqref{eq:shur4} yields the solution:
\begin{equation}
    \mu =\frac{1}{M}\big[I_n-U_r(I_r + \lambda \Sigma_r^{-2})U_r^{\top}\big]1_n.
    \label{eq:eqmu}
\end{equation}
Finally, substituting the expression for $\mu$ in Eq.~\eqref{eq:eqmu} into the Shur component for $A_S$ in Eq.~\eqref{eq:ShurA} and using the SVD of $\Phi_S$, we obtain the closed-form mass preserving solution
\begin{equation}
A_S=V_r(\Sigma^2_r+\lambda I_r)^{-1}\Sigma_r U^{\top}_r\big(X^{\top}_S-\frac{1}{M}\big[I_n-U_r(I_r + \lambda \Sigma^{-2}_r)U^{\top}_r\big]1_n\,1_M^\top\big),
    \label{eq:ShurAfinal}
\end{equation}
recovering the expression in Eq.~\eqref{eq:ShurAfinal0}. We note that at the limit $M\to \infty$, the above solution simplifies to the Tikhonov spectral filter solution in Eq.~\eqref{eq:TikhSVD} for the unconstrained problem.

Let us now verify that the solution for $A_S$ in Eq.~\eqref{eq:ShurAfinal} is indeed mass preserving, by evaluating the sum-to-one constraint $\Phi_S A_S\,1_M=1_n$ in Eq.~\eqref{eq:masscond}. Starting from the left-hand side, substitution of $A_S$ and the SVD of $\Phi_S$ yields
\begin{align}
    \Phi_S A_S1_M = U_r\Sigma_rV_r^{\top }V_r(\Sigma^2_r+\lambda I_r)^{-1}\Sigma_r U^{\top}_r\big(X^{\top}_S-\frac{1}{M}\big[I_n-U_r(I_r + \lambda \Sigma^{-2}_r)U^{\top}_r\big]1_n\,1_M^\top\big)1_M &= \nonumber \\ 
    U_r\Sigma_r(\Sigma^2_r+\lambda I_r)^{-1}\Sigma_r U^{\top}_r\big(X^{\top}_S\, 1_M-\big[I_n-U_r(I_r + \lambda \Sigma^{-2}_r)U^{\top}_r\big]1_n\big) &= \nonumber\\ 
    U_r\Sigma_r(\Sigma^2_r+\lambda I_r)^{-1}\Sigma_r U^{\top}_r\big(1_n-\big[I_n-U_r(I_r + \lambda \Sigma^{-2}_r)U^{\top}_r\big]1_n\big) &= \nonumber \\ 
    %U_r\Sigma_r(\Sigma^2_r+\lambda I_r)^{-1}\Sigma_r U^{\top}_r\big[I_n-I_n+U_r(I_r + \lambda \Sigma^{-2}_r)U^{\top}_r\big]1_n=\\ \nonumber
     U_r\Sigma_r(\Sigma^2_r+\lambda I_r)^{-1}\Sigma_r U^{\top}_r\big(U_r(I_r + \lambda \Sigma^{-2}_r)U^{\top}_r\big)1_n &= \nonumber\\ 
     U_r\Sigma_r(\Sigma^2_r+\lambda I_r)^{-1}\Sigma_r (I_r + \lambda \Sigma^{-2}_r)U^{\top}_r\,1_n=U_r\, U^{\top}_r\,1_n, &
     \label{eq:ShurAconst}
\end{align}
where the (transposed) constrain $X_S^\top 1_M=1_n$ was used in the third step. Now, recall that due to the presence of output biases, the feature matrix is written as $\Phi_S=\big[1_n \, | \, \widetilde{\Phi}_S \big]$; see Eq.~\eqref{eq:RFNNBasis}. Making the trivial assumption that $1_n$ does not, in general, belong in the column space of $\widetilde{\Phi}_S$, i.e., $1_n \notin \mathcal{R}(\widetilde{\Phi}_S)$, the inclusion of output bias implies 
\[\Phi_S \, e_{bias}=1_n.
\] 
where $e_{bias}^\top=[1,0,\ldots,0]\in \mathbb{R}^{P+1}$; therefore implying that $1_n$ belongs in the column space of $\Phi_S$, i.e., $1_n \in \mathcal{R}(\Phi_S)$. But $U_r$ spans $\mathcal{R}(\Phi_S)$ and thus, 
\begin{equation}
U_rU^{\top}_r\,1_n=1_n.
\label{eq:UrUrt}
\end{equation}
Substitution of the above to Eq.~\eqref{eq:ShurAconst} results in $\Phi_S A_S\,1_M=1_n$, so that the condition for the preservation of the the mass is satisfied for the reconstructions $\widehat{X}_S$.

We finally note that the spectral form solution in Eq.~\eqref{eq:ShurAfinal} recovers, in the vanishing regularization limit $\lambda\to0$, the Moore-Penrose pseudo-inverse based solution
\begin{equation}
A_S=V_r\Sigma^{-1}_rU^{\top}_r\,X^{\top}_S=\Phi_S^{+}\, X^{\top},
\end{equation}
which also satisfies the mass preservation condition, due to the result in Eq.~\eqref{eq:UrUrt}.
\end{proof}

As shown in Proposition~\ref{prp:RANDSMAP}, the inclusion of an output bias column in the feature matrix is the key ingredient for mass‑preserving reconstructions on the training data. It is trivial to show (based on the fact that $1_n\in \mathcal{R}(\Phi_S)$ and $1_n \notin \mathcal{R}(\widetilde{\Phi}_S)$) that the constrained solution $A_S$ in Eq.~\eqref{eq:ShurAfinal0} and \eqref{eq:ShurAfinal0s} also provides mass-preserving reconstructions for unseen, out-of-sample data.
However, while the RANDSMAP decoders are mass-preserving, this cannot be achieved by the RFNN decoders, except when the feature matrix is full rank, as we prove in the following Corollary.
\begin{corollary}
The coefficient matrix $A_S$ obtained by vanilla Tikhonov regularization in Eq.~\eqref{eq:vanillaTikhonov}, or its dual form in Eq.~\eqref{eq:dualform}, renders the reconstructed fields provided by the RFNN decoder mass-preserving only when $\Phi_S \in \mathbb{R}^{n\times (P+1)}$ is full rank.
\end{corollary}
\begin{proof}
Substituting the vanilla Tikhonov regularization solution $A_S$  in Eq.~\eqref{eq:vanillaTikhonov} into the left-hand side of the mass-preserving condition in Eq.~\eqref{eq:masscond} yields
\begin{equation}
   \Phi_S A_S 1_M=\Phi_S (\Phi_S^\top \Phi_S+\lambda I_{P+1})^{-1}\Phi_S^\top X^{\top}_S\,1_M=\Phi_S \Phi_S^\top (\Phi_S\,\Phi_S^\top +\lambda I_n)^{-1} 1_n,
\end{equation}
where the dual form identity in Eq.~\eqref{eq:dual_identity} and the sum-to-one constraint of the training data set $X^{\top}_S\,1_M=1_n$ was used on the second step. Clearly, for the above expression to be equal to $1_n$, $\Phi_S$ must be full rank; in such a case, the regularization parameter $\lambda$ is no longer required and therefore can be set to zero.
\end{proof}
%\begin{corollary}
%The vanilla spectral form of the Tikhonov regularization solution given by (\ref{eq:TikhSVD}) preserves the mass setting $\lambda=0$.
%\end{corollary}
%\begin{proof}
%Substituting, (\ref{eq:TikhSVD}) into (\ref{eq:masscond}), we get:
%\begin{equation}
%U\Sigma_r V^{\top}_r V_r(\Sigma_r^{2}+\lambda I_r)^{-1}\Sigma_r U_r^{\top} X_S^{\top}\,1_M=U\Sigma_r (\Sigma_r^{2}+\lambda I_r)^{-1}\Sigma_r U_r^{\top} 1_n,
%\end{equation}
%Clearly, for the above expression to be equal to $1_n$, we should set $\lambda=0$.
%\end{proof}
The exact conversation of the mass for the RANDSMAP reconstructions relies on the condition $U_r U_r^\top 1_n = 1_n$ (see Eqs.~\eqref{eq:ShurAconst} and \eqref{eq:UrUrt} in Proposition~\ref{prp:RANDSMAP}), which holds when all non-zero singular values of $\Phi_S$ are retained, i.e., when  $r = rank(\Phi_S)$. In practice, however, a truncated SVD is often preferred to reduce computational cost, suppress noise, and improve numerical stability. For a truncated SVD, where only the leading $tr < r=rank(\Phi_S)$ modes are kept, the equality $U_r U_r^\top 1_n = 1_n$ becomes approximate, leaving a residual projection error. For such cases, we present the following Corollary, which provides a bound on the resulting conservation error.
\begin{corollary} \label{cor:cons_err_bound}
Let $U_{tr}$ contain the first $tr<r=rank(\Phi_S)$ left singular vectors of the SVD of $\Phi_S$, in which $1_n\in \mathcal{R}(\Phi_S)$. Then, the residual $e=(I_n - U_{tr} U_{tr}^\top)1_n$ of the mass conservation condition in Eq.~\eqref{eq:masscond} is upper bounded by
\begin{equation}
    \|e\|_2 \le \sigma_{tr+1},
\end{equation}
where $\sigma_{tr+1}$ is the first omitted singular value of $\Phi_S$.
\end{corollary}
\begin{proof}
Since $1_n\in \mathcal{R}(\Phi_S)$, it can be expressed via the SVD of $\Phi_S=U_r \Sigma_r V_r^\top$, yielding
\begin{equation}
    1_n = \Phi_S e_j = U_r \Sigma_r V_r^\top e_j = \sum_{i=1}^r \sigma_i v_{ij}\, u_i, \,  v_{ij}=(V_r^\top e_j)_i, 
\end{equation}
where $j$ is the index of the column of $\Phi_S$ coinciding with $1_n$; here, by construction of the RANDSMAP feature matrix $j=1$. The projection onto the truncated basis $U_{tr}$ is $U_{tr}U_{tr}^\top1_n=\sum_{i=1}^{tr} \sigma_i v_{ij}\, u_i$, which gives the the residual of the mass conservation condition in Eq.~\eqref{eq:masscond} as
\begin{equation}
e=(I_n - U_{tr} U_{tr}^\top)1_n=\sum_{i=tr+1}^r \sigma_i v_{ij} u_i.
\end{equation}
Because the left singular vectors $\{u_i\}^r_{i=1}$ are orthonormal and the singular values are in decreasing order (i.e., $\sigma_i\le\sigma_{tr+1}$ for every $i\ge tr+1$), the residual is bounded by
\begin{equation}
    \|e\|_2^2 = \sum_{i=tr+1}^r (\sigma_i v_{ij})^2 \le \sigma_{tr+1}^2 \sum_{i=tr+1}^r v_{ij}^2.
\end{equation}
Finally, the columns of the matrix $V_r$ are orthonormal, so that $\sum_{i=tr+1}^r v_{ij}^2 \le \sum_{i=1}^r v_{ij}^2 \le 1$. Hence, the residual is bounded by
\begin{equation}
   \|e\|_2^2 \le \sigma_{tr+1}^2 \Rightarrow \|e\|_2 \le \sigma_{tr+1}.
\end{equation}
\end{proof}

\begin{remark}
We hereby note that for large-scale systems (when $M\gg 1$ and the sample size $n$ or feature dimensions $P$ are large; roughly for $n,P\ge\mathcal{O}(10^4)$), direct methods such as SVD become computationally and memory-wise prohibitive. In such cases, it is preferable to compute the coefficient matrices $A_S$ of RFNNs and RANDSMAP using iterative methods in the Krylov subspace. 

\end{remark}

\section{Numerical results}\label{sec:numerics}
Here, we evaluate the performance of the RANDSMAP decoder for pre-image reconstruction, comparing it with the standard RFNN decoder and the numerical analysis-based DDM and $k$-NN decoders. The assessment focuses on three mass-preserving physical systems of increasing dimensionality and complexity (Sections~\ref{sb:LWR}-\ref{sb:Hughes}). 
%To test also scalability, we consider three progressively higher‑dimensional benchmarks: a moderate‑dimensional 1D traffic density system ($M=400$), a $128 \times 128$ gray‑scale MRI image reconstruction task ($M=16384$), and a 2D pedestrian density system discretized into $200\times 50$ grid points ($M=10000$).Implementation details common to all examples are provided in Section~\ref{sb:ImpDet}. 
To establish a performance baseline in a general non-conservative setting, we also include two classic manifold learning examples, the Swiss roll and the S-curve projected in $20$-dimensions; results for these examples are provided in Appendices~\ref{app:SR} and \ref{app:SC20}.

For both the RANDSMAP (conservative) and RFNN (non-conservative) decoders, we examine three variants distinguished by their random feature mapping: random Fourier features (RFF), multi‑scale random Fourier features (MS‑RFF), and sigmoidal features (Sig). All decoders reconstruct ambient space states from low‑dimensional embeddings obtained via the DM algorithm. 

For each problem, the performance of each decoder is evaluated based on (i) reconstruction accuracy, (ii) computational cost, and (iii) conservation accuracy. We quantify the reconstruction accuracy of each decoder by its relative $L_2$ and $L_\infty$ reconstruction error across the training and testing sets, $\mathcal{X}_{tr}$ and $\mathcal{X}_{ts}$, respectively. To assess computational cost, we recorded the computational times required for both training and inference. Training time includes encoder training (DM algorithm), hyperparameter tuning (of a single parameter for each decoder) over a held-out validation set $\mathcal{X}_{vl}$, and decoder training. Inference time covers the complete end-to-end pipeline on the testing set $\mathcal{X}_{ts}$, including the encoding to obtain the embedding and the decoding to reconstruct the entire testing set. To assess conservation accuracy, we quantify the pointwise deviation from the exact sum-to-one invariant in Eq.~\eqref{eq:s2one_cons_recon}. As the RANDSMAP decoder is designed to preserve the sum-to-one invariant by construction, we verify this numerically and demonstrate that it maintains competitive reconstruction performance. To account for the stochasticity inherent to the random feature construction, all reported results for the RANDSMAP and RFNN decoders are averaged over 100 independent runs. The complete implementation details, including dataset generation, train/validation/test splits, feature construction sampling procedures, hyperparameter grids, and solver specifications, are provided in Appendix~\ref{app:ImpDet}.

All simulations were performed on a workstation equipped with Intel$^{\text{\textregistered}}$ Xeon$^{\text{\textregistered}}$ E5-2630 v4 CPU (2.20GHz, 2 processors) and 64GB of RAM, using MATLAB R2025b, without parallel computing to ensure consistent comparison across decoders. 

\subsection{Case Study 1: The Lighthill-Whitham-Richards 1D Traffic Model Dataset (\texorpdfstring{$M=400$}{M=400})} \label{sb:LWR}
For our first mass-preserving benchmark, we construct a dataset governed by a dynamical system of traffic flow on a one-dimensional road. The physics is described by the Lighthill-Whitham-Richards (LWR) model %\cite{lighthill1955kinematic,richards1956shock}, 
a standard conservation law 
\begin{equation}
    \partial_t \rho + \partial_x (v_{max}\rho(1-\rho/\rho_{max})) = 0,
\end{equation}
where $\rho \equiv \rho(t,x)\in [0, \rho_{max}]$ is the vehicle density, $v_{max}=2$ is the free-flow speed, and $\rho_{max}=1$ is the maximum jam density. All quantities are expressed in physical units: spatial coordinates in meters, time in seconds, velocity in meters per second, and density in vehicles per meter. We solve the PDE of the LWR model on a periodic spatial domain $\Omega \in [-5,5]$. As this is a hyperbolic PDE, even smooth initial data can give rise to shock waves.

To preserve the intrinsic conservation property, we integrated the PDE using a conservative Finite Volume numerical scheme. In particular, we employed a high-resolution Godunov method with a Roe approximate Riemann solver and the van Leer flux limiter, ensuring accurate shock capturing while avoiding spurious oscillations (see \cite{patsatzis2025gorinns} for details). The domain is discretized into $M=400$ uniform cells. For constructing the dataset, we generated $100$ distinct trajectories by drawing random initial conditions, each being a Gaussian density bump, $\rho(0,x) = a \, \exp\left( -(x - x_c)^2 / (2w^2) \right)$, with uniformly sampled parameters of  amplitude $a \sim \mathcal{U}[0.5, 1.5)$, width $w \sim \mathcal{U}[0.2, 0.8)$, and center $x_c \sim \mathcal{U}[-3, 3)$. To introduce realistic measurement noise, we added i.i.d. Gaussian noise $\mathcal{N}(0, \sigma^2)$ with $\sigma=0.02$. We then enforced physical feasibility by setting any negative densities (from noise) to zero, and normalize the total mass to unity so that $\sum_{j=1}^{400} \rho(0, x_j) = 1$. Each initial condition is then integrated for $t_{end} = 20$ with a time step $\delta t = 0.005$. As expected from the LWR model dynamics, smooth bumps evolve into right-going traveling shock waves. Crucially, the numerical scheme preserves the total mass in time. Finally, we randomly select $120$ snapshots from each trajectory, resulting in a full dataset of $N_{obs} = 12000$ density profiles $x_i \in \mathbb{R}^{400}$.

From the full dataset, we randomly split the data into training and validation sets $\mathcal{X}_{tr}$ and $\mathcal{X}_{vl}$, each containing $N = 2000$ points, with the remaining $8000$ points reserved for the testing set $\mathcal{X}_{ts}$; a visualization of the density profiles in $\mathcal{X}_{tr}$ is shown in  Fig.~\ref{fig:LWR_data}. Since the ambient space where the data lie is $400$-dimensional, we project them onto three summary statistics of the density distributions: the mean position $\mu_x$, variance $\sigma_x^2$, and skewness $\gamma_x$. This low-dimensional projection reveals an underlying manifold structure arising from the conservation law.

We next employed the DM algorithm with $\alpha=0$ and $w_1=1$ to the training set to obtain a $d=2$-dimensional embedding. The resulting DM coordinates $y_i \in \mathbb{R}^2$ are shown in Fig.~\ref{fig:LWR_DMcoords}, colored by the mean position $\mu_x$. The dimensionality $d=2$ was selected based on a clear spectral gap after the second eigenvalue (data not shown) and its sufficiency for achieving accurate reconstruction, as demonstrated below.

Next, we trained the proposed RANDSMAP decoders (RANDSMAP-RFF with random Fourier features, RANDSMAP-MS-RFF with multi-scale random Fourier features, and RANDSMAP-Sig with sigmoidal random features) with $P = N, N/2, N/4$ features (using all $n=N$ data points) alongside the deterministic DDM and $k$-NN decoders, as outlined in Appendix~\ref{app:ImpDet}. For a direct comparison of mass conservation, we also train the non-mass-preserving RFNN-Sig decoder. Hyperparameters for the stochastic decoders were tuned via a grid search over 10 values of $\sigma_w \in [0.1, 1]$, $\sigma_{UB} \in [2, 15]$, and $c \in [1, 20]$. The optimal configuration for $P=N$ was $\sigma_w = 0.16$, $\sigma_{UB} = 8.5$, and $c = 11$ (for both RANDSMAP-Sig and RFNN-Sig). For DDM and $k$-NN, we tuned $w_2 \in [0.2, 1]$ and $k \in [2, 11]$, finding optimal values of $w_2 = 0.6$ and $k = 6$.

The full quantitative performance on the training and testing sets is detailed in Tables~\ref{tab:LWR_dec_tr_all} and \ref{tab:LWR_dec_ts_all} in Appendix~\ref{app:LWR}. The tables report the mean relative $L_2$ and $L_\infty$ reconstruction errors, the mass conservation error, and the computational times for training and inference. A concise visual summary of the accuracy vs. computational time trade-off is presented in Fig.~\ref{fig:LWR_Dec}.
\begin{figure}[!htbp]
    \centering
    \begin{subfigure}[b]{0.32\textwidth}
        \centering
        \includegraphics[width=\textwidth]{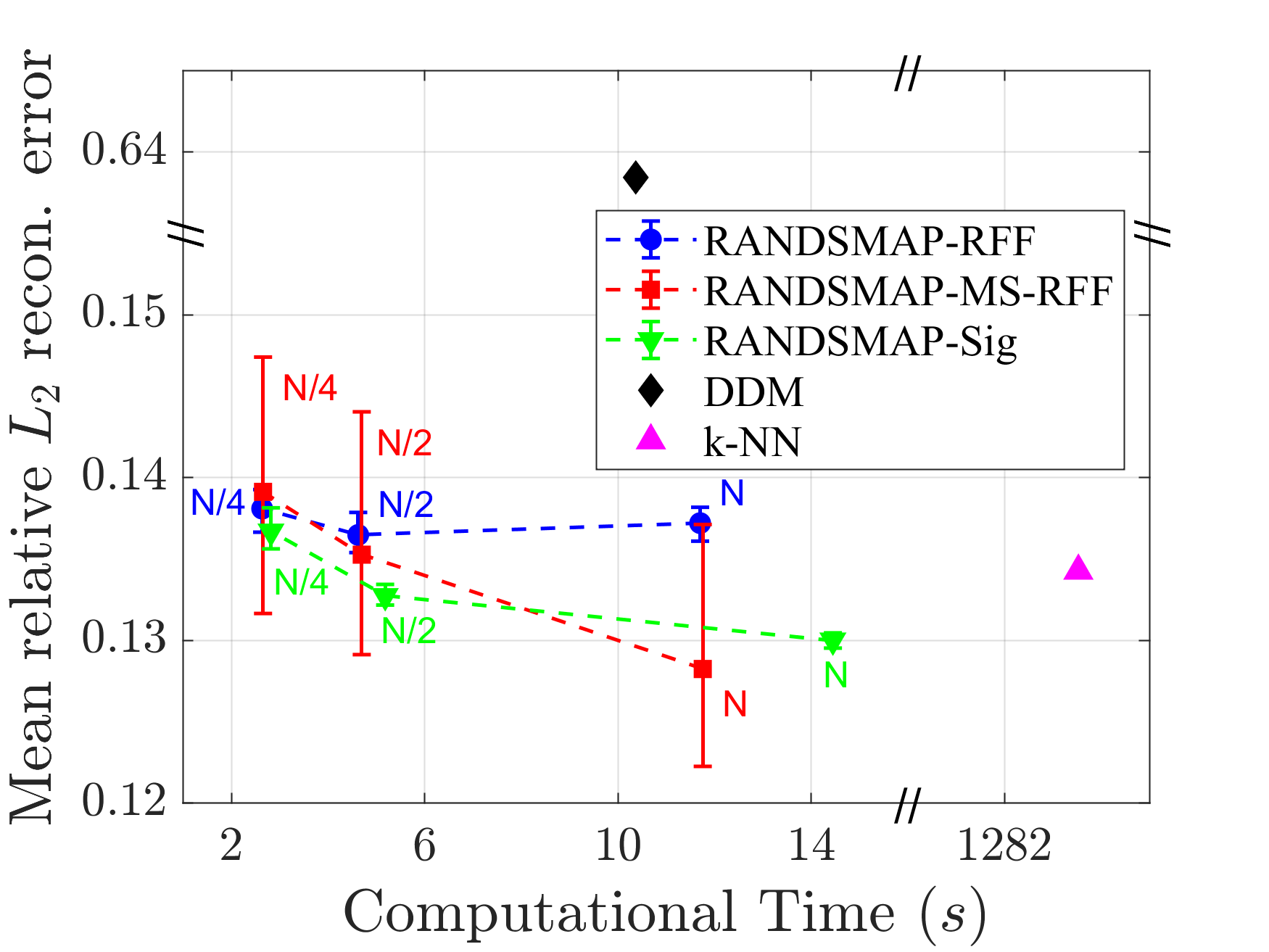}
        \caption{Training, $L_2$ error}
        \label{fig:LWR_dec_tr}
    \end{subfigure}
    \hspace{2pt}
    \begin{subfigure}[b]{0.32\textwidth}
        \centering
        \includegraphics[width=\textwidth]{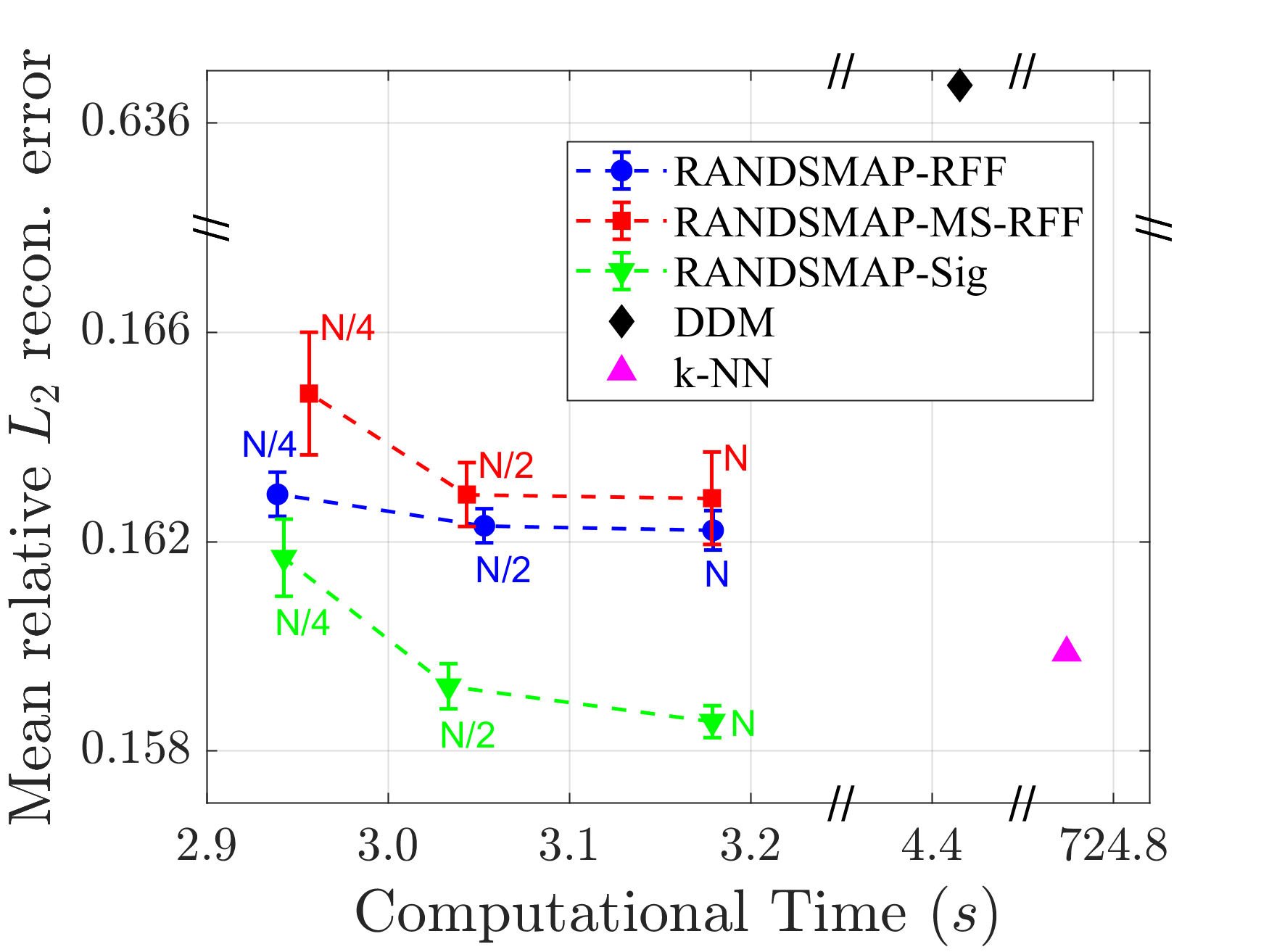}
        \caption{Testing, $L_2$ error}
        \label{fig:LWR_dec_ts}
    \end{subfigure}
    \hspace{2pt}
    \begin{subfigure}[b]{0.32\textwidth}
        \centering
        \includegraphics[width=\textwidth]{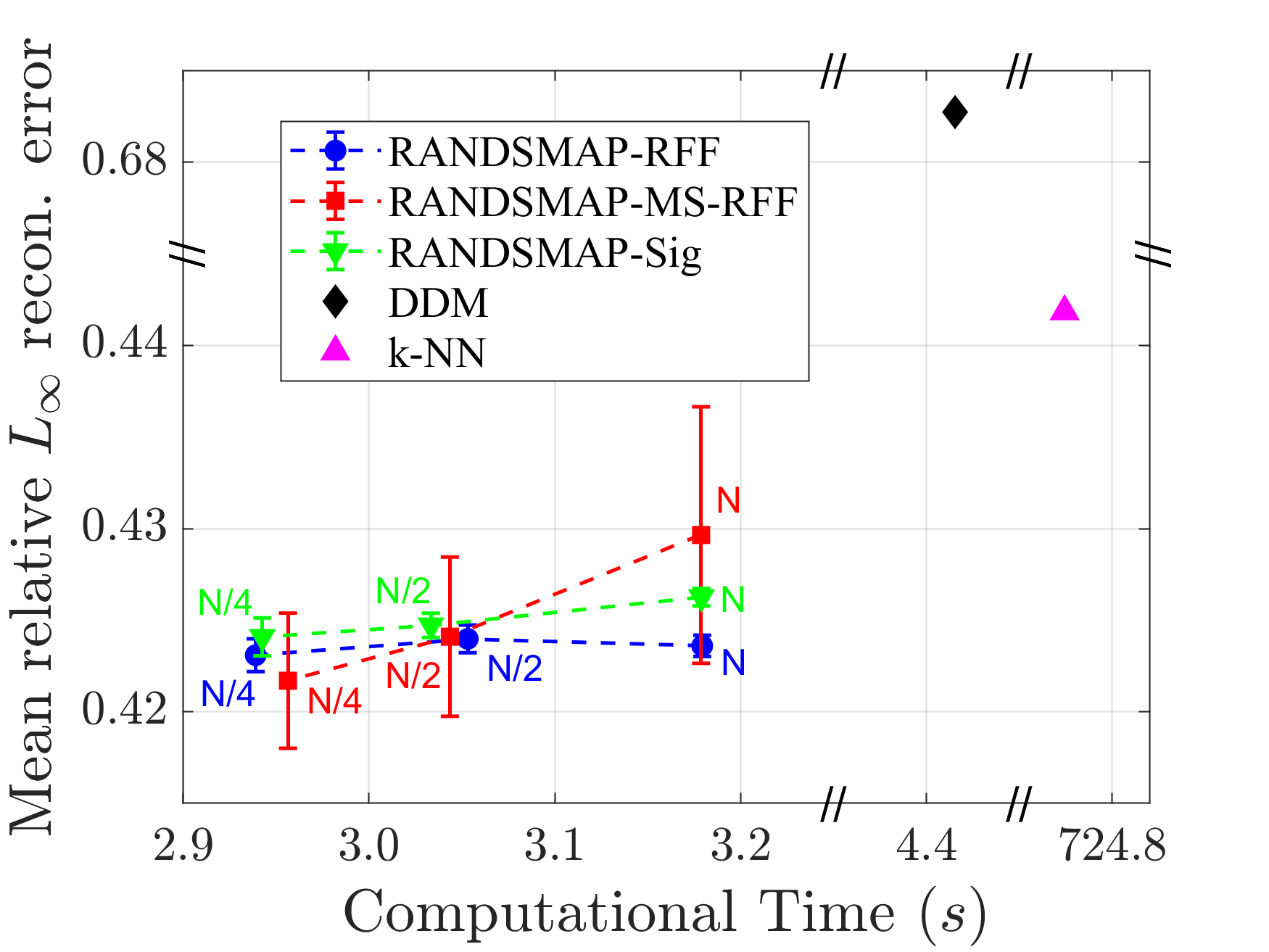}
        \caption{Testing, $L_\infty$ error}
        \label{fig:LWR_dec_ts_linf}
    \end{subfigure}
    \caption{Reconstruction error vs. computational time for the LWR 1D traffic model dataset ($M=400$). Results are shown for $N=2000$ training points. Panel~\ref{fig:LWR_dec_tr} shows mean relative $L_2$ error on the training set vs. training time. Panels~\ref{fig:LWR_dec_ts} and \ref{fig:LWR_dec_ts_linf} show mean relative $L_2$ and $L_\infty$ errors, respectively, on the testing set vs. inference time. For the stochastic RANDSMAP decoders (RANDSMAP-RFF, RANDSMAP-MS-RFF, RANDSMAP-Sig) with $P=N,N/2,N/4$, points show the median over 100 random initializations and error bars indicate the 5–95\% percentile range. Deterministic decoders (DDM, $k$-NN) are shown without error bars. Detailed numerical results are provided in Tables~\ref{tab:LWR_dec_tr_all} and \ref{tab:LWR_dec_ts_all}.}
    \label{fig:LWR_Dec}
\end{figure}
Figure~\ref{fig:LWR_Dec} reveals several key findings. The RANDSMAP decoders are consistently orders of magnitude faster than the $k$-NN decoder while achieving comparable or superior reconstruction accuracy. In contrast, the DDM decoder performs poorly, with a mean relative $L_2$ error of $\sim 64\%$ (training), compared to less than $17\%$ for all other methods. As expected, RANDSMAP variants with fewer random features ($P=N/4$) are less accurate; accuracy improves as $P$ increases, most notably for the multi-scale Fourier (RANDSMAP-MS-RFF) and sigmoidal (RANDSMAP-Sig) features. On the training set, all RANDSMAP decoders surpass the $k$-NN decoder in $L_2$ accuracy. On the testing set, only RANDSMAP-Sig achieves a lower $L_2$ error than $k$-NN. However, for the $L_\infty$ error (which penalizes the worst-case pointwise deviation), all RANDSMAP variants outperform $k$-NN. Notably, the best-performing variant depends on $P$: RANDSMAP-MS-RFF is most accurate for $P=N/4$, while RANDSMAP-RFF is best for $P=N$. This suggests that RANDSMAP-MS-RFF is a favorable choice when robust, worst-case-avoiding reconstruction is prioritized, which is a crucial property for capturing sharp features like shock waves, as evident in the sample reconstruction of Fig.~\ref{fig:LWR_recon}.

Similar trends were also reported for the non-conservative RFNN decoders, when compared to the DDM and $k$-NN decoders on the non-mass-preserving examples in Appendices~\ref{app:SR} and \ref{app:SC20}; they matched or surpassed reconstruction accuracy at substantially lower computational cost, with accuracy improving as $P$ increases. Crucially, however, the poor performance of the DDM decoder in this example is fundamental, not an artifact of poor hyperparameter tuning. As shown in Fig.~\ref{fig:LWR_recon}, the DDM decoder produces an oversmoothed reconstructed profile incapable of resolving the shock discontinuity. An exhaustive 2D hyperparameter search over the DDM weight $w_2$ and the number of eigenvectors $r$ (Fig.~\ref{fig:LWR_DM_GH_tuning}) confirms that the optimal set $(w_2, r) = (0.6, 5)$ still yields a poor reconstruction (Fig.~\ref{fig:LWR_DM_GH_k5}). Attempting to improve shock resolution by increasing to $r=25$ eigenvectors (Fig.~\ref{fig:LWR_DM_GH_k25}) introduces numerical instability and poor global reconstruction, highlighting that the DDM’s smooth, global basis is inherently unsuitable for representing local discontinuities.

Critically, the RANDSMAP decoders enforce exact mass conservation, as proven in Proposition~\ref{prp:RANDSMAP}. This is hereby confirmed numerically as their mean conservation errors are on the order of the single machine precision ($10^{-8}$) (Tables~\ref{tab:LWR_dec_tr_all} and \ref{tab:LWR_dec_ts_all} in the Appendix). In contrast, the non-mass-conserving RFNN-Sig decoder, while similarly accurate in reconstruction, exhibits conservation errors three orders of magnitude larger ($\sim 10^{-5}$). This conservation guarantee comes at a computational cost, since RANDSMAP decoders are approximately $2-4$ times slower to train than their RFNN counterparts. As expected, the $k$-NN decoder also preserves mass to machine precision (Proposition~\ref{prop:kNN_cons}), while DDM does not (Proposition~\ref{prop:DDM_cons}), with errors exceeding $10^6$. It is important here to highlight that these results numerically confirm the conservation error bound provided in Corollary~\ref{cor:cons_err_bound}. Specifically, for the RANDSMAP-MS-RFF variant with $P=N$, an independent run yielded a measured conservation error of approximately $2.4\times 10^{-8}$. This aligns closely with the corresponding feature matrix's trailing singular value $\sigma_{tr+1}=2.6\times 10^{-8}$, thereby satisfying the inequality $\|e\|_2<\sigma_{tr+1}$ as predicted.

Finally, Fig.~\ref{fig:LWR} visualizes the testing set reconstructions for all decoders (excluding RFNN) projected onto the space of summary statistics (mean $\mu_x$, variance $\sigma_x^2$, skewness $\gamma_x$). This projection of the $400$-dimensional density profiles indicates that all but the DDM decoder successfully recover the underlying manifold.
\subsection{Case Study 2: 2D Rotated MRI Images Dataset (\texorpdfstring{$M=128\times128$}{M=128 x 128})} \label{sb:MRI}
For our next benchmark, we assess the proposed decoders on a high-dimensional, image reconstruction task under data scarcity. We generate a dataset from a single $M=128\times128=16384$ gray-scale MRI image (the 4th image from MATLAB's MRI dataset). By applying $N_{obs}=3600$ in-plane rotations with angles $\theta_i\in[0, 2\pi)$, we create a data manifold where the sole degree of freedom is the rotation angle. To enforce a mass-preserving constraint analogous to the traffic flow benchmark, we normalize each rotated image so that its total pixel intensity is sum-to-one invariant, i.e., $\sum_{j=1}^Mx_j=1$. Representative samples are shown in Fig.~\ref{fig:MRIimage_data} in Appendix~\ref{app:MRI}.

We randomly split the data set, treating each flattened image as a vector $x_i\in \mathbb{R}^{16384}$, into a training and a validation set $\mathcal{X}_{tr}$ and $\mathcal{X}_{vl}$, each with $N = 720$ points. The remaining $2160$ points form the testing set $\mathcal{X}_{ts}$. Employing the DM algorithm with $\alpha=1$ and $w_1=0.5$ to the training set yields a $d=2$-dimensional embedding; the resulting DM coordinates $y_i\in \mathbb{R}^2$ are shown on Fig.~\ref{fig:MRI_DMcoords}, colored by the rotation angle $\theta_i$. Following the implementation procedure detailed in Appendix~\ref{app:ImpDet}, we train all decoders (RANDSMAP variants with $P=N, N/2, N/4$, DDM, $k$-NN and the non-mass-preserving RFNN-Sig variant) and tune their hyperparameters via grid search over the following ranges: $\sigma_w\in[0.02,0.1]$ for RANDSMAP-RFF, $\sigma_{UB}\in[35,65]$ for RANDSMAP-MS-RFF, $c\in[45,120]$ for RANDSMAP-Sig and RFNN-Sig, $w_2\in[0.2,1]$ for DDM, and $k\in[2,11]$ for $k$-NN. For the stochastic decoders with $P=N$, the identified optimal values were $\sigma_w=0.03$, $\sigma_{UB}=60$ and $c=115$, while for the deterministic ones were $w_2=0.4$ and $k=2$.  
As in the previous benchmark, the full quantitative performance of all decoders is detailed in Tables~\ref{tab:MRI_dec_tr_all} and \ref{tab:MRI_dec_ts_all} of Appendix~\ref{app:MRI}. A summary of the reconstruction accuracy versus the computational time trade-off is visualized in Fig.~\ref{fig:MRI_Dec}.
\begin{figure}[!htbp]
    \centering
    \begin{subfigure}[b]{0.32\textwidth}
        \centering
        \includegraphics[width=\textwidth]{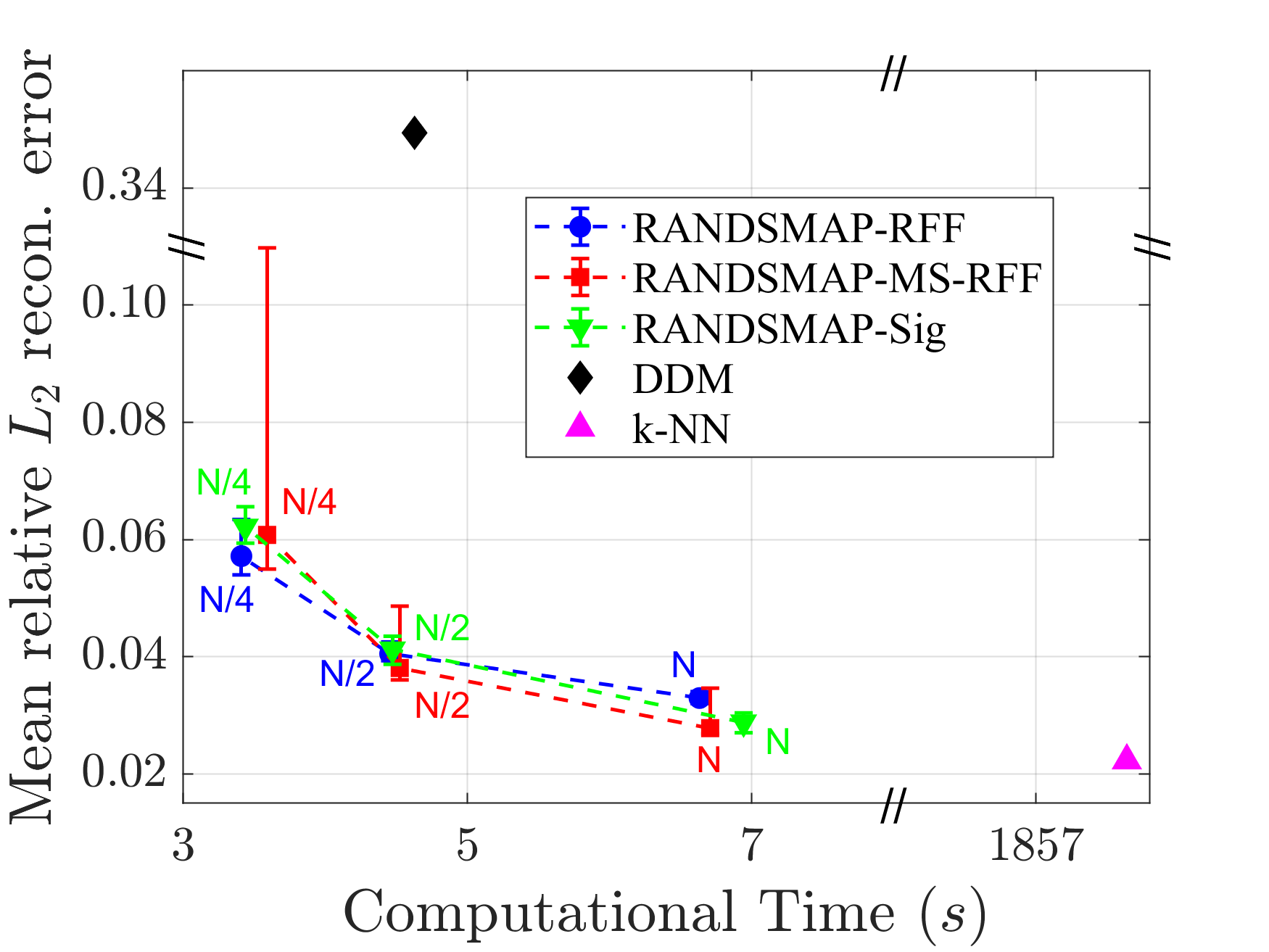}
        \caption{Training, $L_2$ error}
        \label{fig:MRI_dec_tr}
    \end{subfigure}
    \hspace{2pt}
    \begin{subfigure}[b]{0.32\textwidth}
        \centering
        \includegraphics[width=\textwidth]{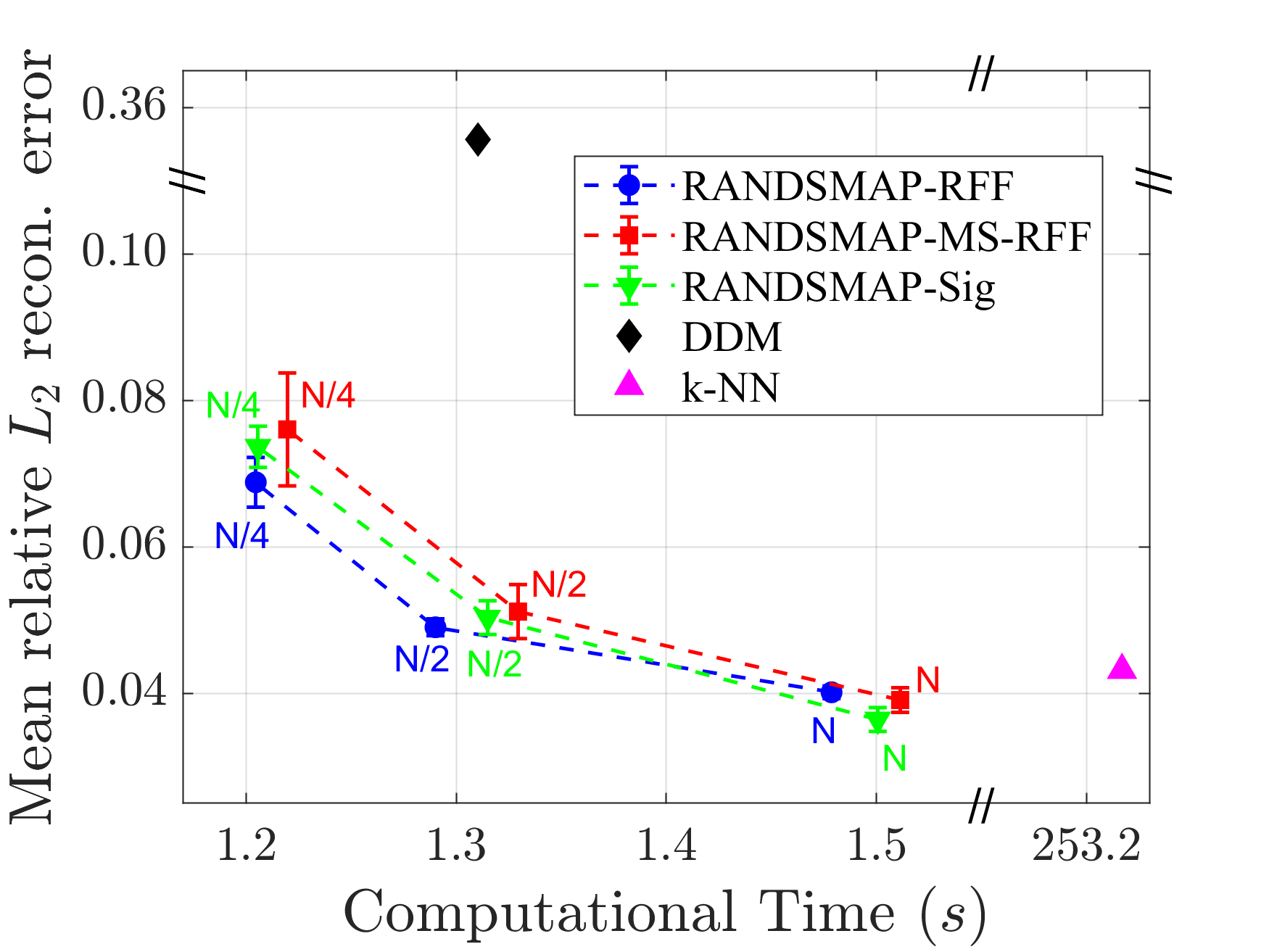}
        \caption{Testing, $L_2$ error}
        \label{fig:MRI_dec_ts}
    \end{subfigure}
    \hspace{2pt}
    \begin{subfigure}[b]{0.32\textwidth}
        \centering
        \includegraphics[width=\textwidth]{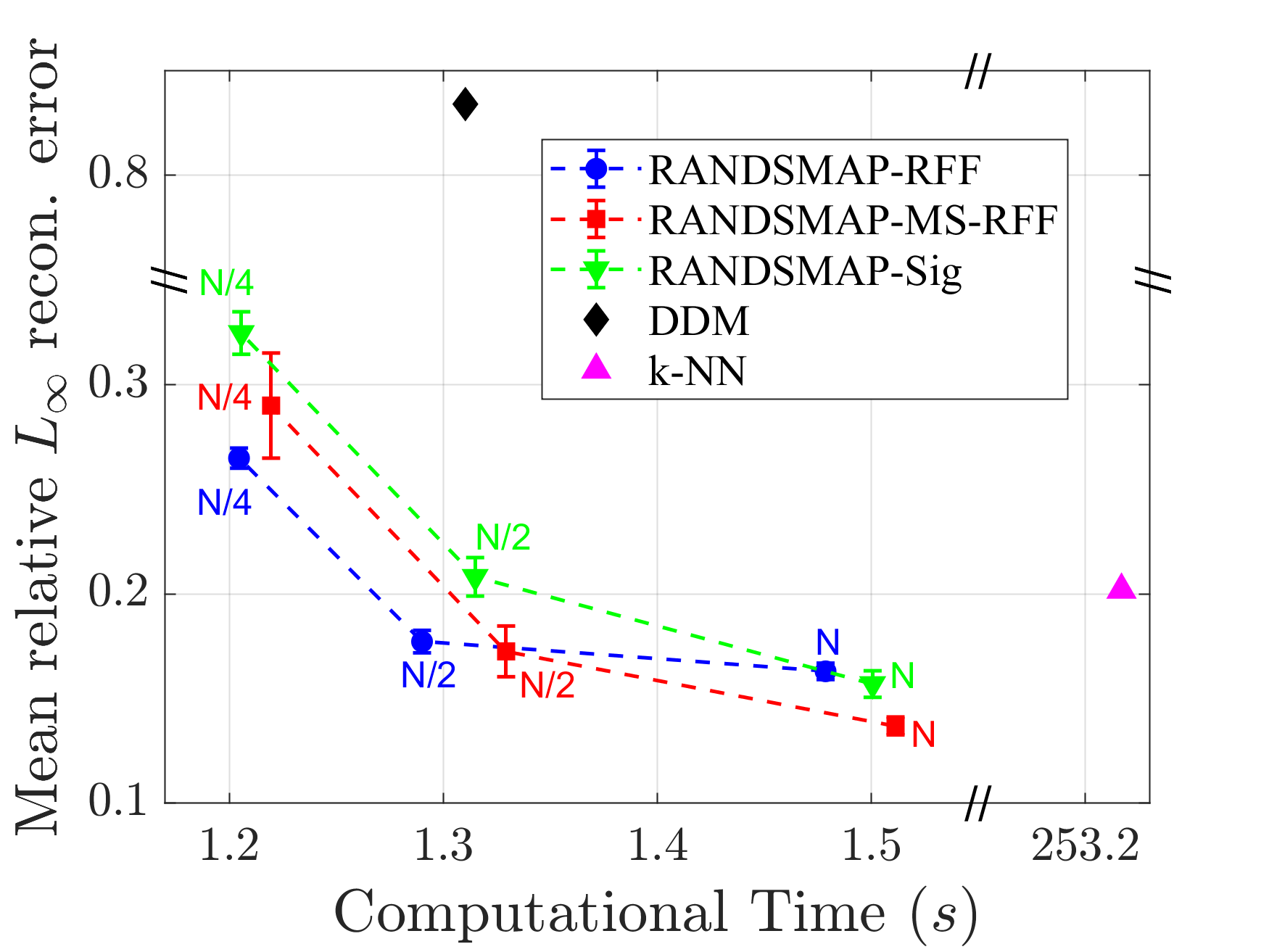}
        \caption{Testing, $L_\infty$ error}
        \label{fig:MRI_dec_ts_linf}
    \end{subfigure}
    \caption{Reconstruction error vs. computational time for the 2D rotated MRI images dataset ($M=128 \times 128$). Results are shown for $N=720$ training points. Panel~\ref{fig:MRI_dec_tr} shows mean relative $L_2$ error on the training set vs. training time. Panels~\ref{fig:MRI_dec_ts} and \ref{fig:MRI_dec_ts_linf} show mean $L_2$ and $L_\infty$ errors, respectively, on the testing set vs. inference time. For the stochastic RANDSMAP decoders (RANDSMAP-RFF, RANDSMAP-MS-RFF, RANDSMAP-Sig) with $P=N/4,N/2,N$, points show the median over 100 random initializations and error bars indicate the 5–95\% percentile range. Deterministic decoders (DDM, $k$-NN) are shown without error bars. Detailed numerical results are provided in Tables~\ref{tab:MRI_dec_tr_all} and \ref{tab:MRI_dec_ts_all}.}
    \label{fig:MRI_Dec}
\end{figure}
Consistent with the traffic flow benchmark, all RANDSMAP decoders significantly outperform the $k$-NN decoder in speed, while matching or surpassing its reconstruction accuracy on both training and testing sets. The DDM decoder again performs poorly due to its inherent limitation in representing sharp image edges, not due to poor tuning (see Fig.~\ref{fig:MRI_DM_GH} for validation). As the number of features $P$ increases, the accuracy of RANDSMAP variants improves at a higher computational cost. In contrast to the traffic flow benchmark, here the most accurate error RANDSMAP variant is that with vanilla Fourier features for low $P=N/4$. However, for $P=N$, all variants achieve similar $L_2$ errors, but RANDSMAP-MS-RFF attains the best $L_\infty$ error, presenting again a favorable choice for worst-case-avoiding reconstruction.       

Mass conservation is again enforced exactly by the RANDSMAP decoders, with errors on the order of $10^{-7}$, while the non-mass-conserving RFNN-Sig decoder exhibits errors three orders of magnitude larger. Interestingly, conservation error improves as $P$ decreases (see Tables~\ref{tab:MRI_dec_tr_all} and \ref{tab:MRI_dec_ts_all}), which we attribute to the smaller, better-conditioned constrained least-squares problem in Eq.~\eqref{eq:KKT-matrix}.
 
For completeness, Fig.~\ref{fig:MRI} visualizes the per-image reconstruction errors for all decoders, and Fig.~\ref{fig:MRI_recon} shows a representative reconstructed slice alongside absolute error maps. These visualizations confirm that the RANDSMAP decoders (particularly the -MS-RFF and -Sig variants) achieve more accurate reconstructions than DDM or $k$-NN. Notably, the RANDSMAP reconstructions appear sharper and less spatially diffused than those of the locally averaging $k$-NN decoder.
\subsection{Case Study 3: The Hughes 2D Pedestrian Model Dataset (\texorpdfstring{$M=200\times50$}{M=200 x 50})} \label{sb:Hughes}
For our last benchmark, we consider a two-dimensional crowd dynamics problem described by Hughes's continuum model \cite{hughes2002continuum}, a system of PDEs coupling pedestrian density $\rho(t,x,y)$ and a potential field $\phi(t,x,y)$
\begin{equation}
    \partial_t \rho - \partial_x\left( \rho f(\rho) \dfrac{\partial_x\phi}{\|\nabla\phi\|}\right) - \partial_y\left( \rho f(\rho) \dfrac{\partial_y\phi}{\|\nabla\phi\|}\right) = 0, \qquad 
    \|\nabla\phi\| %=\sqrt{\left(\partial_x\phi\right)^2 + \left(\partial_y\phi\right)^2}
    =\dfrac{1}{f(\rho)g(\rho)}, 
\end{equation}
where $\|\nabla\phi\|=\sqrt{(\partial_x\phi)^2 + (\partial_y\phi)^2}$, $f(\rho)$ is the pedestrian speed, and $g(\rho)$ is a pedestrians discomfort factor. All quantities are expressed in physical units: spatial coordinates in meters, time in seconds, velocity in meters per second, and density in persons per square meter. Following \cite{viola2025pde}, we adopt a linear density-velocity relation $f(\rho)=v_{max}(1-\rho/\rho_{max})$ with $v_{max}=1$, $\rho_{max}=5$, and set $g(\rho)=1$. The system is solved on a rectangular corridor domain $\Omega = [0,20]\times[0,5]$ containing a square obstacle $\Omega_{obs}=[10,11]\times[2,3]$. Periodic boundary conditions are imposed on the left and right walls, while homogeneous Neumann (zero-flux) conditions are enforced on the top and bottom walls and on the obstacle boundaries.

To preserve mass exactly, we discretize the model using a conservative Finite Volume scheme. At each time step, we first compute the potential $\phi$ via the Fast Sweeping Method \cite{zhao2005fast} and then update the density $\rho$ using a first-order upwind Godunov scheme in both spatial directions (see \cite{viola2025pde} for details). The domain is discretized into $M=200\times 50$ uniform cells. To construct the dataset, we generate $80$ distinct trajectories from random initial conditions. Each initial condition is a 2D Gaussian density bump $\rho(0,x,y)=a \exp \left( - (x-x_c)^2/(2w_x^2) - (y-y_c)^2/(2w_y^2) \right)$ with uniformly sampled parameters of amplitude $a \sim \mathcal{U}[1.2,2.1)$, widths $w_x,w_y \sim \mathcal{U} [1.6,2)$, and centers $x_c, y_c \sim \mathcal{U} [1.5, 3.5)$. Each profile is normalized so that the total initial mass is unity; i.e., $\sum_{j=1}^{200}\sum_{k=1}^{50} \rho(0,x_j,y_k)=1$. Every initial condition is then integrated up to $t_{end}= 70$ with time step $\delta t=0.025$. The resulting flows evolve to the right while smoothly diverting around the obstacle, and the numerical scheme conserves the total mass exactly in time. From each trajectory, we randomly subsample $375$ snapshots, yielding a full dataset of $N_{obs}=30000$ 2D density profiles, each represented as a vector $x_i\in \mathbb{R}^{10000}$. 

We randomly splitted the dataset, treating each density profile as a vector $x_i\in\mathbb{R}^{10000}$, into training and validation sets $\mathcal{X}_{tr}$ and $\mathcal{X}_{vl}$ (each with $N=5000$ points), reserving the remaining $20000$ points for testing ($\mathcal{X}_{ts}$). Figure~\ref{fig:Hughes_data} visualizes $\mathcal{X}_{tr}$ by projecting the high-dimensional density profiles onto three summary statistics: the mean positions $(\mu_x,\mu_y)$  and the variance $\sigma_x^2$. 

We then applied DM with $\alpha=1$ and $w_1=0.4$ to $\mathcal{X}_{tr}$. Unlike previous benchmarks, the eigenvalue spectrum lacks a clear gap. We select $d=10$ DM coordinates as this dimension yields accurate reconstruction and was deemed sufficient in \cite{viola2025pde}. The first three components of the embedding $y_i\in\mathbb{R}^{10}$ are shown in Fig.~\ref{fig:Hughes_DMcoords}, colored by $\mu_x$. Following our standard procedure, we train all decoders (RANDSMAP variants with $P=N, N/2, N/4$, DDM, $k$-NN and the non-mass-preserving RFNN-Sig variant) and tune their hyperparameters via grid search over the following ranges: $\sigma_w\in[0.2,1]$ for RANDSMAP-RFF, $\sigma_{UB}\in[1,10]$ for RANDSMAP-MS-RFF, $c\in[1,20]$ for RANDSMAP-Sig and RFNN-Sig, $w_2\in[0.1,1]$ for DDM, and $k\in[2,18]$ for $k$-NN. For the stochastic decoders with $P=N$, the optimal values were $\sigma_w=0.4$, $\sigma_{UB}=2.5$ and $c=4$, while for the deterministic decoders $w_2=1.45$ and $k=10$.  

The full numerical results are provided in Tables~\ref{tab:Hughes_dec_tr_all} and \ref{tab:Hughes_dec_ts_all} (Appendix~\ref{app:Hughes}); a summary of the reconstruction accuracy vs. computational time trade-off is shown in Fig.~\ref{fig:Hughes_Dec}.
\begin{figure}[!htbp]
    \centering
    \begin{subfigure}[b]{0.32\textwidth}
        \centering
        \includegraphics[width=\textwidth]{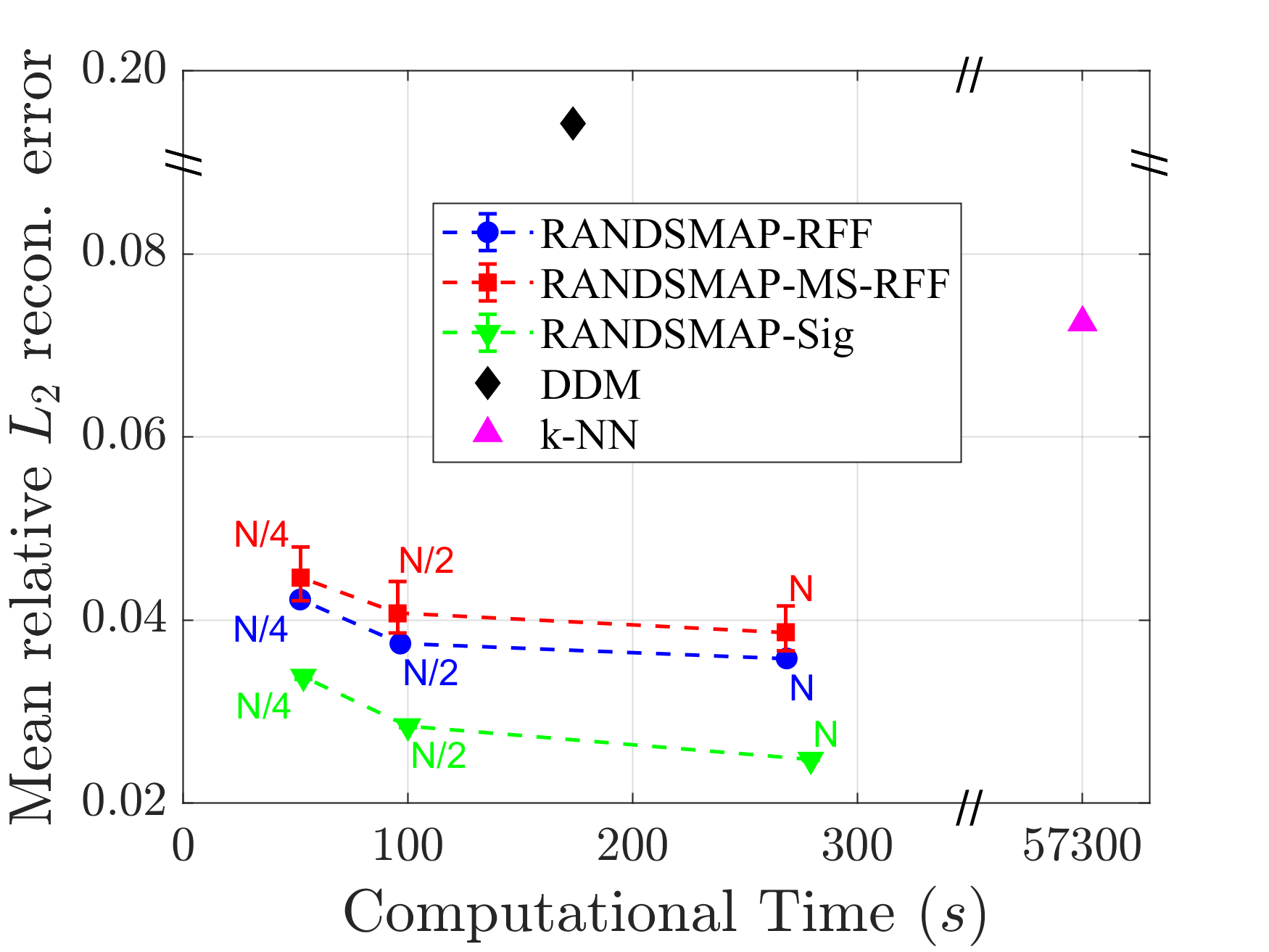}
        \caption{Training, $L_2$ error}
        \label{fig:Hughes_dec_tr}
    \end{subfigure}
    \hspace{2pt}
    \begin{subfigure}[b]{0.32\textwidth}
        \centering
        \includegraphics[width=\textwidth]{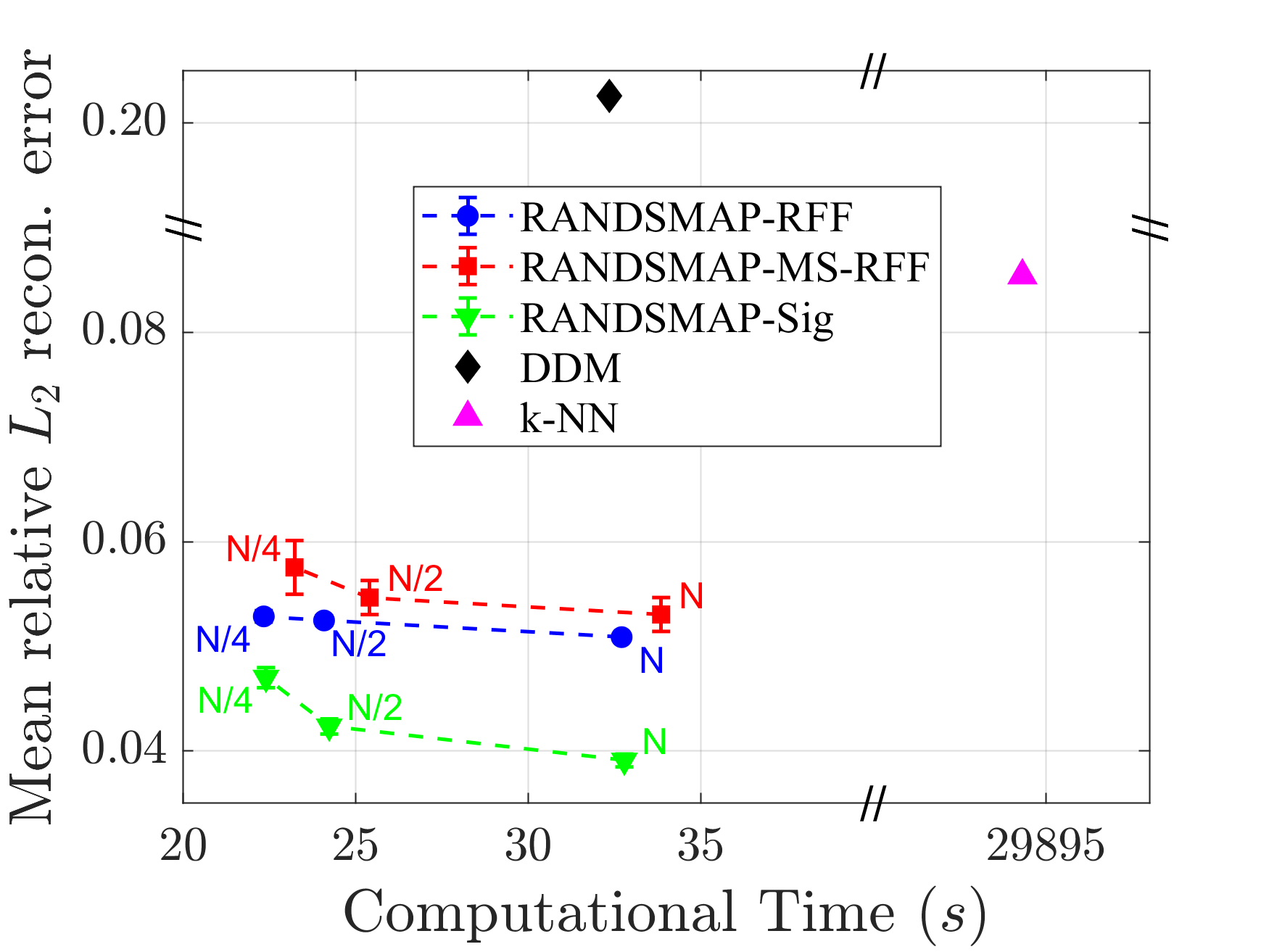}
        \caption{Testing, $L_2$ error}
        \label{fig:Hughes_dec_ts}
    \end{subfigure}
    \hspace{2pt}
    \begin{subfigure}[b]{0.32\textwidth}
        \centering
        \includegraphics[width=\textwidth]{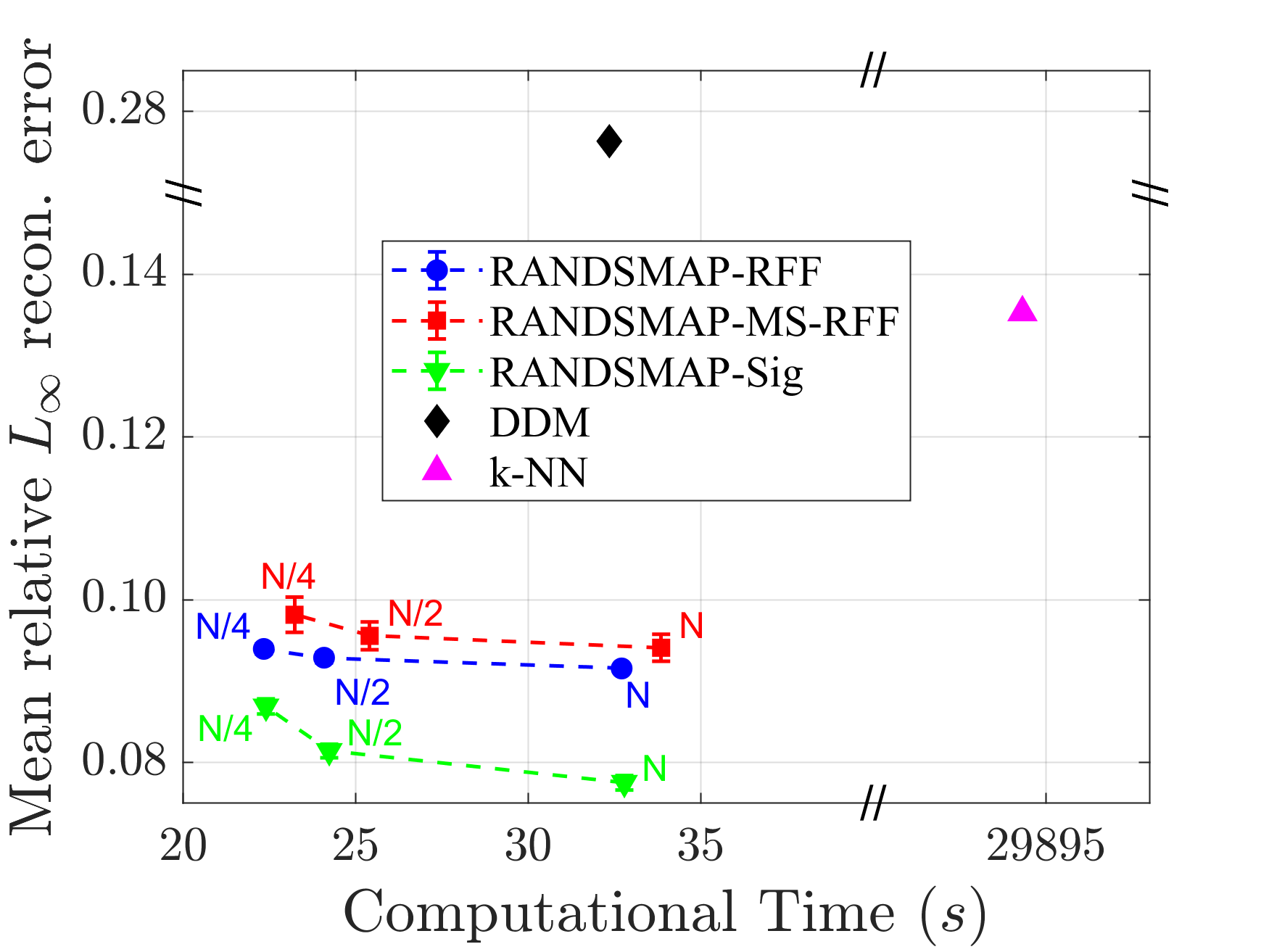}
        \caption{Testing, $L_\infty$ error}
        \label{fig:Hughes_dec_ts_linf}
    \end{subfigure}
    \caption{Reconstruction error vs. computational time for the Hughes 2D pedestrian model dataset ($M=200 \times 50$). Results are shown for $N=5000$ training points. Panel~\ref{fig:Hughes_dec_tr} shows the mean relative $L_2$ error on the training set vs. training time. Panels~\ref{fig:Hughes_dec_ts} and \ref{fig:Hughes_dec_ts_linf} show mean $L_2$ and $L_\infty$ errors, respectively, on the testing set vs. inference time. For the stochastic RANDSMAP decoders (RANDSMAP-RFF, RANDSMAP-MS-RFF, RANDSMAP-Sig) with $P=N/4,N/2,N$, points show the median over 100 random initializations and error bars indicate the 5–95\% percentile range. Deterministic decoders (DDM, $k$-NN) are shown without error bars. Detailed numerical results are provided in Tables~\ref{tab:Hughes_dec_tr_all} and \ref{tab:Hughes_dec_ts_all}.}
    \label{fig:Hughes_Dec}
\end{figure}
Consistently across previous benchmarks, all RANDSMAP decoders are significantly faster than $k$-NN both in training and inference. Crucially, here they also surpass it in reconstruction accuracy, achieving $<6\%$ relative $L_2$ testing error versus $>8\%$ for $k$-NN. The DDM decoder performs fairly ($L_2$ testing error $\sim20\%$) at a cost comparable to RANDSMAP, showcasing a notable improvement over its poor performance in previous benchmarks. This improved performance is likely due to the smoother flows in this problem, which are better approximated by DDM's smooth global kernel basis. An exhaustive hyperparameter search (see Fig.~\ref{fig:Hughes_DM_GH}) confirms this $\sim20\%$ error is optimal for DDM on this manifold. Figure~\ref{fig:Hughes_Dec} further shows that with adequate training data, the RANDSMAP decoders exhibit excellent generalization: accuracy improves with $P$ (at higher computational cost), and $L_2$ and $L_\infty$ errors are similarly low, indicating uniformly good pointwise reconstruction. The RANDSMAP-Sig variant is slightly more accurate than the Fourier feature variants in this case.

Importantly, the RANDSMAP decoders conserve mass to high precision; the mean of pointwise reconstruction is smaller than $10^{-6}$ (see Tables~\ref{tab:Hughes_dec_tr_all} and \ref{tab:Hughes_dec_ts_all}). Interestingly, the non-mass-conserving RFNN-Sig decoder achieves similar conservation error here, except for low $P$ where RANDSMAP variants (especially the Fourier-based RFF and MS-RFF ones) outperform it by 3–5 orders of magnitude. For completeness, we visualize the reconstruction errors onto the summary-statistics space in Fig.~\ref{fig:Hughes}, confirming that all decoders accurately recover the manifold (with DDM being least accurate). Finally, Fig.~\ref{fig:Hughes_recon} shows a representative snapshot where the crowd diverts around the obstacle. The RANDSMAP reconstructions better capture fine details and sharp gradients compared to other decoders.   
\section{Conclusions}\label{sec:conclusions}
This work introduces RANDSMAP, a framework that merges classical numerical analysis with neural feature maps to guarantee exact mass preservation in manifold reconstruction--a gap in both purely numerical and NN‑based approaches. Unlike purely numerical analysis-based methods, which often ignore conservation constraints, or black‑box autoencoders, which cannot enforce conservation laws exactly, RANDSMAP combines random‑feature neural decoders with a numerically stable, linearly‑constrained least‑squares formulation. This ensures exact mass conservation by construction, without sacrificing reconstruction accuracy or interpretability. By importing established numerical tools such as SVD and Cholesky factorizations into a randomized‑feature framework, the method inherits both the expressivity of neural representations and the stability, efficiency, and convergence guarantees of numerical analysis-based schemes.

%Our theoretical analysis proves that widely used numerical-analysis based decoders behave differently with respect to mass preservation: linear POD/SVD and k‑NN methods conserve mass incidentally, while kernel‑based DDM (via GHs) does not. 
We establish the theoretical foundation of RANDSMAP, deriving its closed‑form solution, proving exact mass preservation, and demonstrating its equivalence to deterministic kernel methods in expectation in the deterministic limit. Furthermore, we provide conservation error bounds for truncated solutions and extend the framework to multi‑scale random Fourier features, showing their equivalence to multi‑Gaussian kernels in the deterministic limit. Numerical experiments across a range of benchmark problems, including traffic flow with shock waves, MRI reconstruction, and crowd dynamics, validate the theoretical guarantees and demonstrate the practical advantages of RANDSMAP. The method consistently achieves orders of magnitude faster inference and training than $k$‑NN and outperforms the kernel‑based DDM decoder in speed while matching or exceeding their accuracy. Critically, RANDSMAP enforces exact mass preservation, with conservation errors as low as single machine precision, and successfully captures multiscale features that single‑scale kernel methods fail to represent without introducing unphysical oscillations; a limitation rooted in the Runge phenomenon inherent to flat‑limit RBF interpolation \cite{fornberg2007runge}. Furthermore, RANDSMAP scales robustly to high‑dimensional ambient spaces, confirming its suitability for large‑scale, physics‑aware decoding tasks.

RANDSMAP provides a robust and exact framework for integrating hard physical constraints into neural decoders, bridging structured numerical linear algebra with data‑driven learning. Future work could extend this combination to encode broader physical invariants, such as symmetries, differential constraints, or non‑negativity, directly into the feature construction, enabling a new class of physics‑informed, interpretable, and scalable scientific machine learning decoders. A comparison with autoencoders in similar problems is left for future work.
\section*{Acknowledgements}
C. S. acknowledges partial support from the PNRR MUR, projects PE0000013-Future Artificial Intelligence Research-FAIR \& CN0000013 CN HPC - National Centre for HPC, Big Data and Quantum Computing, Gruppo Nazionale Calcolo Scientifico-Istituto Nazionale di Alta Matematica (GNCS-INdAM).

\bibliography{references}

@article{Patsatzis_2023,
	title={Data-driven control of agent-based models: An Equation/Variable-free machine learning approach},
	author={Patsatzis, D. and Russo, L. and Kevrekidis, I. G. and Siettos, C.},
	journal={Journal of Computational Physics},
	volume={478},
	number={111953},
	year={2023},
}

@article{evangelou2022double,
	author    = {Evangelou, N. and Dietrich, F. and Chiavazzo, E. and Lehmberg, D. and Meila, M. and Kevrekidis, I. G.},
	title     = {Double diffusion maps and their latent harmonics for scientific computations in latent space},
	journal   = {Journal of Computational Physics},
	volume    = {485},
	number    = {112072},
	year      = {2023},
}

@article{johnson1984extensions,
  title={Extensions of {L}ipschitz mappings into a {H}ilbert space},
  author={Johnson, William B. and Lindenstrauss, Joram},
  journal={Contemporary Mathematics},
  volume={26},
  pages={189-206},
  number={1},
  year={1984}
}

@article{gorban2016approximation,
  title={Approximation with random bases: Pro et contra},
  author={Gorban, Alexander N and Tyukin, Ivan Yu and Prokhorov, Danil V and Sofeikov, Konstantin I},
  journal={Information Sciences},
  volume={364},
  pages={129--145},
  year={2016},
  publisher={Elsevier}
}

@article{igelnik1995stochastic,
  title={Stochastic choice of basis functions in adaptive function approximation and the functional-link net},
  author={Igelnik, Boris and Pao, Yoh-Han},
  journal={IEEE Transactions on Neural Networks},
  volume={6},
  number={6},
  pages={1320--1329},
  year={1995},
  publisher={IEEE}
}

@article{rahimi2007random,
  title={Random features for large-scale kernel machines},
  author={Rahimi, Ali and Recht, Benjamin},
  journal={Advances in neural information processing systems},
  volume={20},
  year={2007}
}

@article{makovoz1996random,
  title={Random approximants and neural networks},
  author={Makovoz, Yuly},
  journal={Journal of Approximation Theory},
  volume={85},
  number={1},
  pages={98--109},
  year={1996},
  publisher={Elsevier}
}

@article{fabiani2023parsimonious,
  title={Parsimonious physics-informed random projection neural networks for initial value problems of ODEs and index-1 DAEs},
  author={Fabiani, Gianluca and Galaris, Evangelos and Russo, Lucia and Siettos, Constantinos},
  journal={Chaos: An Interdisciplinary Journal of Nonlinear Science},
  volume={33},
  number={4},
  year={2023},
  publisher={AIP Publishing}
}

@article{galaris2022numerical,
  title={Numerical Bifurcation Analysis of PDEs From Lattice Boltzmann Model Simulations: a Parsimonious Machine Learning Approach},
  author={Galaris, Evangelos and Fabiani, Gianluca and Gallos, Ioannis and Kevrekidis, Ioannis and Siettos, Constantinos},
  journal={Journal of Scientific Computing},
  volume={92},
  number={2},
  pages={1--30},
  year={2022},
  publisher={Springer}
}

@article{fabiani2025randonets,
  title={RandONets: Shallow networks with random projections for learning linear and nonlinear operators},
  author={Fabiani, Gianluca and Kevrekidis, Ioannis G and Siettos, Constantinos and Yannacopoulos, Athanasios N},
  journal={Journal of Computational Physics},
  volume={520},
  pages={113433},
  year={2025},
  publisher={Elsevier}
}

@article{hegde2007random,
  title={Random projections for manifold learning},
  author={Hegde, Chinmay and Wakin, Michael and Baraniuk, Richard},
  journal={Advances in neural information processing systems},
  volume={20},
  year={2007}
}

@article{chiavazzo2014reduced,
  title={Reduced models in chemical kinetics via nonlinear data-mining},
  author={Chiavazzo, Eliodoro and Gear, Charles W and Dsilva, Carmeline J and Rabin, Neta and Kevrekidis, Ioannis G},
  journal={Processes},
  volume={2},
  number={1},
  pages={112--140},
  year={2014},
  publisher={Multidisciplinary Digital Publishing Institute}
}

@article{papaioannou2022time,
  title={Time-series forecasting using manifold learning, radial basis function interpolation, and geometric harmonics},
  author={Papaioannou, Panagiotis G and Talmon, Ronen and Kevrekidis, Ioannis G and Siettos, Constantinos},
  journal={Chaos: An Interdisciplinary Journal of Nonlinear Science},
  volume={32},
  number={8},
  year={2022},
  publisher={AIP Publishing}
}

@article{scholkopf1998nonlinear,
  title={Nonlinear component analysis as a kernel eigenvalue problem},
  author={Sch{\"o}lkopf, Bernhard and Smola, Alexander and M{\"u}ller, Klaus-Robert},
  journal={Neural computation},
  volume={10},
  number={5},
  pages={1299--1319},
  year={1998},
  publisher={MIT Press}
}

@article{roweis2000nonlinear,
  title={Nonlinear dimensionality reduction by locally linear embedding},
  author={Roweis, Sam T and Saul, Lawrence K},
  journal={science},
  volume={290},
  number={5500},
  pages={2323--2326},
  year={2000},
  publisher={American Association for the Advancement of Science}
}

@article{tenenbaum2000global,
  title={A global geometric framework for nonlinear dimensionality reduction},
  author={Tenenbaum, Joshua B and De Silva, Vin and Langford, John C},
  journal={science},
  volume={290},
  number={5500},
  pages={2319--2323},
  year={2000},
  publisher={American Association for the Advancement of Science}
}

@article{belkin2003laplacian,
  title={Laplacian eigenmaps for dimensionality reduction and data representation},
  author={Belkin, Mikhail and Niyogi, Partha},
  journal={Neural computation},
  volume={15},
  number={6},
  pages={1373--1396},
  year={2003},
  publisher={MIT Press}
}

@article{coifman2005geometric,
  title={Geometric diffusions as a tool for harmonic analysis and structure definition of data: diffusion maps},
  author={Coifman, Ronald R and Lafon, Stephane and Lee, Ann B and Maggioni, Mauro and Nadler, Boaz and Warner, Frederick and Zucker, Steven W},
  journal={Proceedings of the National Academy of Sciences},
  volume={102},
  number={21},
  pages={7426--7431},
  year={2005},
  publisher={National Academy of Sciences}
}

@article{coifman2006diffusion,
  title={Diffusion maps},
  author={Coifman, Ronald R and Lafon, St{\'e}phane},
  journal={Applied and computational harmonic analysis},
  volume={21},
  number={1},
  pages={5--30},
  year={2006},
  publisher={Elsevier}
}

@article{coifman2008diffusion,
  title={Diffusion maps, reduction coordinates, and low dimensional representation of stochastic systems},
  author={Coifman, Ronald R and Kevrekidis, Ioannis G and Lafon, St{\'e}phane and Maggioni, Mauro and Nadler, Boaz},
  journal={Multiscale Modeling \& Simulation},
  volume={7},
  number={2},
  pages={842--864},
  year={2008},
  publisher={SIAM}
}

@article{dsilva2015data,
  title={Data-driven reduction for multiscale stochastic dynamical systems},
  author={Dsilva, Carmeline J and Talmon, Ronen and Gear, C William and Coifman, Ronald R and Kevrekidis, Ioannis G},
  journal={arXiv preprint arXiv:1501.05195},
  year={2015}
}

@article{nadler2006diffusion,
  title={Diffusion maps, spectral clustering and reaction coordinates of dynamical systems},
  author={Nadler, Boaz and Lafon, St{\'e}phane and Coifman, Ronald R and Kevrekidis, Ioannis G},
  journal={Applied and Computational Harmonic Analysis},
  volume={21},
  number={1},
  pages={113--127},
  year={2006},
  publisher={Elsevier}
}

@article{coifman2006geometric,
  title={Geometric harmonics: a novel tool for multiscale out-of-sample extension of empirical functions},
  author={Coifman, Ronald R and Lafon, St{\'e}phane},
  journal={Applied and Computational Harmonic Analysis},
  volume={21},
  number={1},
  pages={31--52},
  year={2006},
  publisher={Elsevier}
}

@article{hinton2006reducing,
  title={Reducing the dimensionality of data with neural networks},
  author={Hinton, Geoffrey E and Salakhutdinov, Ruslan R},
  journal={science},
  volume={313},
  number={5786},
  pages={504--507},
  year={2006},
  publisher={American Association for the Advancement of Science}
}

@article{vlachas2022multiscale,
  title={Multiscale simulations of complex systems by learning their effective dynamics},
  author={Vlachas, Pantelis R and Arampatzis, Georgios and Uhler, Caroline and Koumoutsakos, Petros},
  journal={Nature Machine Intelligence},
  volume={4},
  number={4},
  pages={359--366},
  year={2022},
  publisher={Nature Publishing Group UK London}
}

@inproceedings{masci2011stacked,
  title={Stacked convolutional auto-encoders for hierarchical feature extraction},
  author={Masci, Jonathan and Meier, Ueli and Cire{\c{s}}an, Dan and Schmidhuber, J{\"u}rgen},
  booktitle={International conference on artificial neural networks},
  pages={52--59},
  year={2011},
  organization={Springer}
}

@article{makhzani2015adversarial,
  title={Adversarial autoencoders},
  author={Makhzani, Alireza and Shlens, Jonathon and Jaitly, Navdeep and Goodfellow, Ian and Frey, Brendan},
  journal={arXiv preprint arXiv:1511.05644},
  year={2015}
}

@article{peterfreund2020local,
  title={Local conformal autoencoder for standardized data coordinates},
  author={Peterfreund, Erez and Lindenbaum, Ofir and Dietrich, Felix and Bertalan, Tom and Gavish, Matan and Kevrekidis, Ioannis G and Coifman, Ronald R},
  journal={Proceedings of the National Academy of Sciences},
  volume={117},
  number={49},
  pages={30918--30927},
  year={2020},
  publisher={National Academy of Sciences}
}

@article{lehmberg2020datafold,
  title={Datafold: Data-driven models for point clouds and time series on manifolds},
  author={Lehmberg, Daniel and Dietrich, Felix and K{\"o}ster, Gerta and Bungartz, Hans-Joachim},
  journal={Journal of Open Source Software},
  volume={5},
  number={51},
  pages={2283},
  year={2020}
}

@article{paszke2019pytorch,
  title={Pytorch: An imperative style, high-performance deep learning library},
  author={Paszke, Adam and Gross, Sam and Massa, Francisco and Lerer, Adam and Bradbury, James and Chanan, Gregory and Killeen, Trevor and Lin, Zeming and Gimelshein, Natalia and Antiga, Luca and others},
  journal={Advances in neural information processing systems},
  volume={32},
  year={2019}
}

@article{halko2009finding,
  title={Finding structure with randomness: Stochastic algorithms for constructing approximate matrix decompositions},
  author={Halko, Nathan and Martinsson, Per-Gunnar and Tropp, Joel A},
  journal={arXiv preprint arXiv:0909.4061},
  volume={909},
  year={2009}
}

@article{erichson2019randomized2,
  title={Randomized dynamic mode decomposition},
  author={Erichson, N Benjamin and Mathelin, Lionel and Kutz, J Nathan and Brunton, Steven L},
  journal={SIAM Journal on Applied Dynamical Systems},
  volume={18},
  number={4},
  pages={1867--1891},
  year={2019},
  publisher={SIAM}
}

@article{huang2006extreme,
  title={Extreme learning machine: theory and applications},
  author={Huang, Guang-Bin and Zhu, Qin-Yu and Siew, Chee-Kheong},
  journal={Neurocomputing},
  volume={70},
  number={1-3},
  pages={489--501},
  year={2006},
  publisher={Elsevier}
}

@article{bach2017equivalence,
  title={On the equivalence between kernel quadrature rules and random feature expansions},
  author={Bach, Francis},
  journal={Journal of machine learning research},
  volume={18},
  number={21},
  pages={1--38},
  year={2017}
}

@article{patsatzis2025gorinns,
  title={GoRINNs: Godunov-Riemann informed neural networks for learning hyperbolic conservation laws},
  author={Patsatzis, Dimitrios G and di Bernardo, Mario and Russo, Lucia and Siettos, Constantinos},
  journal={Journal of Computational Physics},
  volume={534},
  pages={114002},
  year={2025},
  publisher={Elsevier}
}

@article{belkin2006convergence,
  title={Convergence of Laplacian eigenmaps},
  author={Belkin, Mikhail and Niyogi, Partha},
  journal={Advances in neural information processing systems},
  volume={19},
  year={2006}
}

@article{belkin2008towards,
  title={Towards a theoretical foundation for Laplacian-based manifold methods},
  author={Belkin, Mikhail and Niyogi, Partha},
  journal={Journal of Computer and System Sciences},
  volume={74},
  number={8},
  pages={1289--1308},
  year={2008},
  publisher={Elsevier}
}

@article{evangelou2022parameter,
  title={On the parameter combinations that matter and on those that do not: data-driven studies of parameter (non) identifiability},
  author={Evangelou, Nikolaos and Wichrowski, Noah J and Kevrekidis, George A and Dietrich, Felix and Kooshkbaghi, Mahdi and McFann, Sarah and Kevrekidis, Ioannis G},
  journal={PNAS nexus},
  volume={1},
  number={4},
  pages={pgac154},
  year={2022},
  publisher={Oxford University Press}
}

@article{kontolati2024learning,
  title={Learning nonlinear operators in latent spaces for real-time predictions of complex dynamics in physical systems},
  author={Kontolati, Katiana and Goswami, Somdatta and Em Karniadakis, George and Shields, Michael D},
  journal={Nature Communications},
  volume={15},
  number={1},
  pages={5101},
  year={2024},
  publisher={Nature Publishing Group UK London}
}

@article{fabiani2024task,
  title={Task-oriented machine learning surrogates for tipping points of agent-based models},
  author={Fabiani, Gianluca and Evangelou, Nikolaos and Cui, Tianqi and Bello-Rivas, Juan M and Martin-Linares, Cristina P and Siettos, Constantinos and Kevrekidis, Ioannis G},
  journal={Nature communications},
  volume={15},
  number={1},
  pages={4117},
  year={2024},
  publisher={Nature Publishing Group UK London}
}

@article{lee2020coarse,
  title={Coarse-scale PDEs from fine-scale observations via machine learning},
  author={Lee, Seungjoon and Kooshkbaghi, Mahdi and Spiliotis, Konstantinos and Siettos, Constantinos I and Kevrekidis, Ioannis G},
  journal={Chaos: An Interdisciplinary Journal of Nonlinear Science},
  volume={30},
  number={1},
  year={2020},
  publisher={AIP Publishing}
}

@article{fabiani2025enabling,
  title={Enabling Local Neural Operators to perform Equation-Free System-Level Analysis},
  author={Fabiani, Gianluca and Vandecasteele, Hannes and Goswami, Somdatta and Siettos, Constantinos and Kevrekidis, Ioannis G},
  journal={arXiv preprint arXiv:2505.02308},
  year={2025}
}

@article{kemeth2022learning,
  title={Learning emergent partial differential equations in a learned emergent space},
  author={Kemeth, Felix P and Bertalan, Tom and Thiem, Thomas and Dietrich, Felix and Moon, Sung Joon and Laing, Carlo R and Kevrekidis, Ioannis G},
  journal={Nature communications},
  volume={13},
  number={1},
  pages={3318},
  year={2022},
  publisher={Nature Publishing Group UK London}
}

@article{kaszas2025globalizing,
  title={Globalizing manifold-based reduced models for equations and data},
  author={Kasz{\'a}s, B{\'a}lint and Haller, George},
  journal={Nature Communications},
  volume={16},
  number={1},
  pages={5722},
  year={2025},
  publisher={Nature Publishing Group UK London}
}

@article{sroczynski2024learning,
  title={On learning what to learn: Heterogeneous observations of dynamics and establishing possibly causal relations among them},
  author={Sroczynski, David W and Dietrich, Felix and Koronaki, Eleni D and Talmon, Ronen and Coifman, Ronald R and Bollt, Erik and Kevrekidis, Ioannis G},
  journal={PNAS nexus},
  volume={3},
  number={12},
  pages={pgae494},
  year={2024},
  publisher={Oxford University Press US}
}

@article{romor2023non,
  title={Non-linear manifold reduced-order models with convolutional autoencoders and reduced over-collocation method},
  author={Romor, Francesco and Stabile, Giovanni and Rozza, Gianluigi},
  journal={Journal of Scientific Computing},
  volume={94},
  number={3},
  pages={74},
  year={2023},
  publisher={Springer}
}

@article{lassila2014model,
  title={Model order reduction in fluid dynamics: challenges and perspectives},
  author={Lassila, Toni and Manzoni, Andrea and Quarteroni, Alfio and Rozza, Gianluigi},
  journal={Reduced Order Methods for modeling and computational reduction},
  pages={235--273},
  year={2014},
  publisher={Springer}
}

@article{koronaki2024nonlinear,
  title={Nonlinear dimensionality reduction then and now: AIMs for dissipative PDEs in the ML era},
  author={Koronaki, Eleni D and Evangelou, Nikolaos and Martin-Linares, Cristina P and Titi, Edriss S and Kevrekidis, Ioannis G},
  journal={Journal of Computational Physics},
  volume={506},
  pages={112910},
  year={2024},
  publisher={Elsevier}
}

@article{brunton2016discovering,
  title={Discovering governing equations from data by sparse identification of nonlinear dynamical systems},
  author={Brunton, Steven L and Proctor, Joshua L and Kutz, J Nathan},
  journal={Proceedings of the national academy of sciences},
  volume={113},
  number={15},
  pages={3932--3937},
  year={2016},
  publisher={National Academy of Sciences}
}

@article{lusch2018deep,
  title={Deep learning for universal linear embeddings of nonlinear dynamics},
  author={Lusch, Bethany and Kutz, J Nathan and Brunton, Steven L},
  journal={Nature communications},
  volume={9},
  number={1},
  pages={4950},
  year={2018},
  publisher={Nature Publishing Group UK London}
}

@article{tu2014dynamic,
  title={On dynamic mode decomposition: Theory and applications},
  author={Tu, Jonathan H and Rowley, Clarence Worth and Luchtenburg, Dirk M and Brunton, Steven L and Kutz, J Nathan},
  journal={Journal of Computational Dynamics},
  volume={1},
  number={2},
  pages={391--421},
  year={2014},
  publisher={American Institute of Mathematical Sciences}
}

@article{rowley2004model,
  title={Model reduction for compressible flows using POD and Galerkin projection},
  author={Rowley, Clarence W and Colonius, Tim and Murray, Richard M},
  journal={Physica D: Nonlinear Phenomena},
  volume={189},
  number={1-2},
  pages={115--129},
  year={2004},
  publisher={Elsevier}
}

@article{deane1991low,
  title={Low-dimensional models for complex geometry flows: application to grooved channels and circular cylinders},
  author={Deane, Anil E and Kevrekidis, Ioannis G and Karniadakis, George Em and Orszag, SA0746},
  journal={Physics of Fluids A: Fluid Dynamics},
  volume={3},
  number={10},
  pages={2337--2354},
  year={1991},
  publisher={American Institute of Physics}
}

@article{li2017extended,
  title={Extended dynamic mode decomposition with dictionary learning: A data-driven adaptive spectral decomposition of the Koopman operator},
  author={Li, Qianxiao and Dietrich, Felix and Bollt, Erik M and Kevrekidis, Ioannis G},
  journal={Chaos: An Interdisciplinary Journal of Nonlinear Science},
  volume={27},
  number={10},
  year={2017},
  publisher={AIP Publishing}
}

@article{graham1996alternative,
  title={Alternative approaches to the Karhunen-Loeve decomposition for model reduction and data analysis},
  author={Graham, Michael D and Kevrekidis, Ioannis G},
  journal={Computers \& chemical engineering},
  volume={20},
  number={5},
  pages={495--506},
  year={1996},
  publisher={Elsevier}
}

@article{shvartsman1998nonlinear,
  title={Nonlinear model reduction for control of distributed systems: A computer-assisted study},
  author={Shvartsman, Stanislav Y and Kevrekidis, Ioannis G},
  journal={AIChE Journal},
  volume={44},
  number={7},
  pages={1579--1595},
  year={1998},
  publisher={Wiley Online Library}
}

@article{sirisup2005stability,
  title={Stability and accuracy of periodic flow solutions obtained by a POD-penalty method},
  author={Sirisup, Sirod and Karniadakis, George Em},
  journal={Physica D: Nonlinear Phenomena},
  volume={202},
  number={3-4},
  pages={218--237},
  year={2005},
  publisher={Elsevier}
}

@article{halder2024enhancing,
  title={Enhancing K-nearest neighbor algorithm: a comprehensive review and performance analysis of modifications},
  author={Halder, Rajib Kumar and Uddin, Mohammed Nasir and Uddin, Md Ashraf and Aryal, Sunil and Khraisat, Ansam},
  journal={Journal of Big Data},
  volume={11},
  number={1},
  pages={113},
  year={2024},
  publisher={Springer}
}

@article{fabiani2025random,
  title={Random projection neural networks of best approximation: Convergence theory and practical applications},
  author={Fabiani, Gianluca},
  journal={SIAM Journal on Mathematics of Data Science},
  volume={7},
  number={2},
  pages={385--409},
  year={2025},
  publisher={SIAM}
}

@article{lowe1988multivariable,
  title={Multivariable functional interpolation and adaptive networks},
  author={Lowe, David and Broomhead, D},
  journal={Complex systems},
  volume={2},
  number={3},
  pages={321--355},
  year={1988}
}

@article{gonen2011multiple,
  title={Multiple kernel learning algorithms},
  author={G{\"o}nen, Mehmet and Alpayd{\i}n, Ethem},
  journal={The Journal of Machine Learning Research},
  volume={12},
  pages={2211--2268},
  year={2011},
  publisher={JMLR. org}
}

@inproceedings{rahimi2008uniform,
  title={Uniform approximation of functions with random bases},
  author={Rahimi, Ali and Recht, Benjamin},
  booktitle={2008 46th annual allerton conference on communication, control, and computing},
  pages={555--561},
  year={2008},
  organization={IEEE}
}

@book{rudin1987real,
  title={Real and complex analysis},
  author={Rudin, Walter},
  year={1987},
  publisher={McGraw-Hill, Inc.}
}

@article{bach2008exploring,
  title={Exploring large feature spaces with hierarchical multiple kernel learning},
  author={Bach, Francis},
  journal={Advances in neural information processing systems},
  volume={21},
  year={2008}
}

@article{lin2008dimensionality,
  title={Dimensionality reduction for data in multiple feature representations},
  author={Lin, Yen-Yu and Liu, Tyng-Luh and Fuh, Chiou-Shann},
  journal={Advances in Neural Information Processing Systems},
  volume={21},
  year={2008}
}

@article{liu2021random,
  title={Random features for kernel approximation: A survey on algorithms, theory, and beyond},
  author={Liu, Fanghui and Huang, Xiaolin and Chen, Yudong and Suykens, Johan AK},
  journal={IEEE Transactions on Pattern Analysis and Machine Intelligence},
  volume={44},
  number={10},
  pages={7128--7148},
  year={2021},
  publisher={IEEE}
}

@article{nelsen2021random,
  title={The random feature model for input-output maps between banach spaces},
  author={Nelsen, Nicholas H and Stuart, Andrew M},
  journal={SIAM Journal on Scientific Computing},
  volume={43},
  number={5},
  pages={A3212--A3243},
  year={2021},
  publisher={SIAM}
}

@article{bollt2024neural,
  title={How neural networks work: Unraveling the mystery of randomized neural networks for functions and chaotic dynamical systems},
  author={Bollt, Erik},
  journal={Chaos: An Interdisciplinary Journal of Nonlinear Science},
  volume={34},
  number={12},
  year={2024},
  publisher={AIP Publishing}
}

@article{surasinghe2021randomized,
  title={Randomized projection learning method for dynamic mode decomposition},
  author={Surasinghe, Sudam and M. Bollt, Erik},
  journal={Mathematics},
  volume={9},
  number={21},
  pages={2803},
  year={2021},
  publisher={MDPI}
}

@article{bollt2020regularized,
  title={Regularized kernel machine learning for data driven forecasting of chaos},
  author={Bollt, Erik},
  journal={Annual Review of Chaos Theory, Bifurcations and Dynamical Systems},
  volume={9},
  pages={1--26},
  year={2020}
}

@article{bolager2023sampling,
  title={Sampling weights of deep neural networks},
  author={Bolager, Erik L and Burak, Iryna and Datar, Chinmay and Sun, Qing and Dietrich, Felix},
  journal={Advances in Neural Information Processing Systems},
  volume={36},
  pages={63075--63116},
  year={2023}
}

@article{sun2024local,
  title={Local randomized neural networks with discontinuous Galerkin methods for partial differential equations},
  author={Sun, Jingbo and Dong, Suchuan and Wang, Fei},
  journal={Journal of Computational and Applied Mathematics},
  volume={445},
  pages={115830},
  year={2024},
  publisher={Elsevier}
}

@article{dong2023method,
  title={A method for computing inverse parametric PDE problems with random-weight neural networks},
  author={Dong, Suchuan and Wang, Yiran},
  journal={Journal of Computational Physics},
  volume={489},
  pages={112263},
  year={2023},
  publisher={Elsevier}
}

@article{fabiani2021numerical,
  title={Numerical solution and bifurcation analysis of nonlinear partial differential equations with extreme learning machines},
  author={Fabiani, Gianluca and Calabr{\`o}, Francesco and Russo, Lucia and Siettos, Constantinos},
  journal={Journal of Scientific Computing},
  volume={89},
  number={2},
  pages={44},
  year={2021},
  publisher={Springer}
}

@article{scardapane2017randomness,
  title={Randomness in neural networks: an overview},
  author={Scardapane, Simone and Wang, Dianhui},
  journal={Wiley Interdisciplinary Reviews: Data Mining and Knowledge Discovery},
  volume={7},
  number={2},
  pages={e1200},
  year={2017},
  publisher={Wiley Online Library}
}

@article{cao2018review,
  title={A review on neural networks with random weights},
  author={Cao, Weipeng and Wang, Xizhao and Ming, Zhong and Gao, Jinzhu},
  journal={Neurocomputing},
  volume={275},
  pages={278--287},
  year={2018},
  publisher={Elsevier}
}

@book{kuhnel2015differential,
  title={Differential geometry},
  author={K{\"u}hnel, Wolfgang},
  volume={77},
  year={2015},
  publisher={American Mathematical Soc.}
}

@book{wang2012geometric,
  title={Geometric structure of high-dimensional data and dimensionality reduction},
  author={Wang, Jianzhong},
  volume={5},
  year={2012},
  publisher={Springer}
}

@article{tropp2015introduction,
  title={An introduction to matrix concentration inequalities},
  author={Tropp, Joel A and others},
  journal={Foundations and Trends{\textregistered} in Machine Learning},
  volume={8},
  number={1-2},
  pages={1--230},
  year={2015},
  publisher={Now Publishers, Inc.}
}

@article{coleman1996interior,
  title={An interior trust region approach for nonlinear minimization subject to bounds},
  author={Coleman, Thomas F and Li, Yuying},
  journal={SIAM Journal on optimization},
  volume={6},
  number={2},
  pages={418--445},
  year={1996},
  publisher={SIAM}
}

@article{hughes2002continuum,
  title={A continuum theory for the flow of pedestrians},
  author={Hughes, Roger L},
  journal={Transportation Research Part B: Methodological},
  volume={36},
  number={6},
  pages={507--535},
  year={2002},
  publisher={Elsevier}
}

@article{alvarez2025next,
  title={Next generation equation-free multiscale modelling of crowd dynamics via machine learning},
  author={Alvarez, Hector Vargas and Patsatzis, Dimitrios G and Russo, Lucia and Kevrekidis, Ioannis and Siettos, Constantinos},
  journal={arXiv preprint arXiv:2508.03926},
  year={2025}
}

@article{viola2025pde,
  title={PDE-Free Mass-Constrained Learning of Complex Systems with Hidden States: The crowd dynamics case},
  author={Viola, Gianmaria and Della Pia, Alessandro and Russo, Lucia and Kevrekidis, Ioannis and Siettos, Constantinos},
  journal={arXiv preprint arXiv:2510.17657},
  year={2025}
}

@article{zhao2005fast,
  title={A fast sweeping method for eikonal equations},
  author={Zhao, Hongkai},
  journal={Mathematics of computation},
  volume={74},
  number={250},
  pages={603--627},
  year={2005}
}

@book{fix1985discriminatory,
  title={Discriminatory analysis: nonparametric discrimination, consistency properties},
  author={Fix, Evelyn},
  volume={1},
  year={1985},
  publisher={USAF school of Aviation Medicine}
}

@article{fornberg2007runge,
  title={The Runge phenomenon and spatially variable shape parameters in RBF interpolation},
  author={Fornberg, Bengt and Zuev, Julia},
  journal={Computers \& Mathematics with Applications},
  volume={54},
  number={3},
  pages={379--398},
  year={2007},
  publisher={Elsevier}
}
\bibliographystyle{ieeetr}

%\clearpage
%\newpage
\section*{APPENDIX}
\appendix
\renewcommand{\theequation}{A.\arabic{equation}}
\renewcommand{\thefigure}{A.\arabic{figure}}
\renewcommand{\thetable}{A.\arabic{table}}
\renewcommand{\theprop}{A.\arabic{prop}}
\setcounter{equation}{0}
\setcounter{figure}{0}
\setcounter{table}{0}
\setcounter{prop}{0}

\section{Mass conservation of the POD/SVD decoder}
\label{app:POD}

Here, we prove that the decoding operator resulting from POD/SVD method is mass-preserving.
\begin{prop} \label{prop:POD}
Suppose the dataset \(X \in \mathbb{R}^{M \times N}\) is mass-preserving, i.e., satisfies the (normalized) sum-to-one constraint $1_M^\top X = 1_N^\top$. Let $\tilde{X} = XH$ be the column-centered data matrix with
\begin{equation}
H = I_N - \frac{1}{N} 1_N 1_N^\top,
\label{eq:Hcenter}
\end{equation}
and let $U_d = [u_1,\ldots,u_d] \in \mathbb{R}^{M \times d}$ denote the orthonormal basis formed by the first $d$ left singular vectors of the SVD of $\tilde{X}$. Then, the reconstructed/decoded field computed by POD
\begin{equation}
\widehat{X} = U_d U_d^\top \tilde{X} \;+\; X (I_N - H),
\label{eq:recon_theorem}
\end{equation}
is also mass-preserving, satisfying the sum-to-one constraint
$1_M^\top \widehat{X} = 1_N^\top$.
\end{prop}
\begin{proof}
Using the constraint $1_M^\top X = 1_N^\top$ and the expression of $H$ in Eq.~\eqref{eq:Hcenter}, the multiplication of the centered matrix $\tilde{X} = X H$ by $1_M^\top$ implies
\begin{equation*}
1_M^\top \tilde{X} = 1^{\top}_M X H = 1_N^\top H = 1_N^\top - \frac{1}{N} 1_N^\top 1_N 1_N^\top = 1_N^\top - 1_N^\top = 0_N^\top,
\end{equation*}
since  $1_N^\top 1_N = N$. This implies that $1_{M}$ is orthogonal to the column space of $\tilde{X}$, and consequently to the left singular vectors $U_d$, which provide an orthonormal basis of this space. Therefore,
\begin{equation*}
1^{\top}_{M} U_d = 0^{\top}_d,
\end{equation*}
where $d=1,2,\dots r$ and $r$ is the rank of $X$. Then, multiplication of the POD reconstruction in Eq.~\eqref{eq:recon_theorem} by $1_{M}^\top$ yields
\begin{equation}
1_M^\top \widehat{X} = 1_M^\top U_d U_d^\top \tilde{X} + 1_M^\top X (I_N - H) = 0_N^\top + 1_M^\top X (I_N - H).
\label{eq:P1eq1}
\end{equation}
Using the constraint $1_M^\top X = 1_N^\top$ and the expression of $H$ in Eq.~\eqref{eq:Hcenter}, we get
\begin{equation}
1_M^\top X (I_N - H)  =  1_N^\top (I_N - H) =  \frac{1}{N} 1_N^\top 1_N 1_N^\top = 1_N^\top,
\label{eq:P1eq2}
\end{equation}
since  $1_N^\top 1_N = N$. Substituting Eq.~\eqref{eq:P1eq2} in Eq.~\eqref{eq:P1eq1}, we finally get  
\begin{equation*}
    1_M^\top \widehat{X} = 1_N^\top,
\end{equation*}
so that the reconstructed fields satisfy the sum-to-one constraint, i.e., POD preserves the total mass.
\end{proof}

Note that Proposition~\ref{prop:POD} also applies for, say $L$, unseen points $\mathcal{X}^*=\{x^*_l\}_{l=1}^L\subset \mathbb{R}^M$ collected in the matrix $X^* = \begin{bmatrix} x_1, & \ldots, &x_L \end{bmatrix} \in \mathbb{R}^{M \times L}$, using the POD basis $U_d$ of the known dataset matrix $X$. This stems from Eqs.~\eqref{eq:P1eq1} and \eqref{eq:P1eq2}, which similarly reduce to $1_L^\top$ for the reconstructed field of unseen points $\widehat{X}^*$.

\renewcommand{\theequation}{B.\arabic{equation}}
\renewcommand{\thefigure}{B.\arabic{figure}}
\renewcommand{\thetable}{B.\arabic{table}}
\renewcommand{\theprop}{B.\arabic{prop}}
\renewcommand{\theremark}{B.\arabic{remark}}
\setcounter{equation}{0}
\setcounter{figure}{0}
\setcounter{table}{0}
\setcounter{prop}{0}
\setcounter{remark}{0}
\section{The k-NN decoder and its conservation properties}\label{app:kNN}
Here, we present the nonlinear $k$-NN decoder with convex interpolation among neighbors and prove that its exact mass-preserving property. Let us first assume that the encoding map $\Psi_N$ in Eq.~\eqref{eq:PsiN} is already known. Given a new latent point $y^*$, the $k$-NN algorithm approximates the reconstructed field $\hat{x}^*$ as a linear combination of $K$ nearest neighbors \cite{fix1985discriminatory,Patsatzis_2023}, i.e., as
\begin{equation}
    \hat{x}^*=\Psi^{-1}_N (y^*) = \sum_{k=1}^K \alpha_k x_{S(k)},
    \label{eq:convexinterpol}
\end{equation}
where $x_{S(k)}\in \mathcal{X}$ are data points with known, via the encoding map $\Psi_N$, embeddings $y_{S(k)}=\Psi_N(x_{S(k)})$. The $k$ neighbors of $y^*$ are identified in the latent space as
\begin{equation}
\mathcal{N}_K(y^*) 
= \argmin_{\substack{S \subset \{1,2,\dots,N\} \\ |S|=K}}
\sum_{k \in S}  \| y^* - y_k \|_2 , \quad  y_{k} = \Psi_N (x_{k}).
\end{equation} 

The weights $\alpha_1, \ldots, \alpha_K \in \mathbb{R}$ are then estimated by solving the following optimization problem in the latent space
\begin{equation}
    \left( \alpha_{1}, \ldots, \alpha_{K} \right) = \argmin_{(\alpha_1, \ldots, \alpha_K)\in [0,1]^K} \Bigg\{ \Bigg\| y^* - \Psi_N\left( \sum_{k=1}^K \alpha_k x_{S(k)} \right)\Bigg\| : \sum_{k=1}^K \alpha_k =1 \Bigg\},
    \label{eq:LiftkNN}
\end{equation}
where $S(k) \in \mathcal{N}_K(y^*)$ are the indices of the nearest neighbors. The constraint $\sum_{k=1}^K \alpha_k =1$ guarantees that the reconstructed field in Eq.~\eqref{eq:convexinterpol} is a convex interpolant of its neighboring samples on the data manifold. The following proposition proves that the decoder resulting from the described $k$-NN algorithm is mass-preserving.

\begin{prop} \label{prop:kNN_cons}
Suppose the dataset $\mathcal{X}=\{x_i\}_{i=1}^N\subset\mathbb{R}^M$ is mass-preserving, i.e., satisfies the (normalized) sum-to-one constraint $1_M^\top x_i = 1$ for each $x_i$, $i=1,\ldots,N$. Let $\hat{x}^*$ denote the reconstructed field of a new latent point $y^*$, provided by the $k$-NN decoder and defined as a convex combination of $K$ neighboring data points with indices $S(k)$ by
\begin{equation*}
    \hat{x}^*= \sum_{k=1}^K \alpha_k x_{S(k)}, \quad x_{S(k)}\in \mathcal{X},
\end{equation*}
where the coefficients $\alpha_k$ satisfy $\sum_{k=1}^K \alpha_k =1$. Then, the reconstructed/decoded field $\hat{x}^*$ is mass-preserving, satisfying the sum-to-one constraint $1_M^\top \hat{x}^* = 1$.
\end{prop}
\begin{proof}
Since the decoder reconstructs $\hat{x}^*$ as a convex interpolant of data points in $\mathcal{X}$, and each neighboring data point $x_{S(k)}\in\mathcal{X}$ satisfies the sum-to-one constraint $1_M^\top x_{S(k)} = 1$, we compute
\begin{equation}
    1^{\top}_M \hat{x}^*=1_M^\top \sum_{k=1}^{K}\alpha_k \, x_{S(k)}=\sum_{k=1}^{K}\alpha_k \, 1_M^\top \, x_{S(k)}=\sum_{k=1}^{K}\alpha_k\, \cdot 1.
\end{equation}
Using the convexity constraint $\sum_{k=1}^K \alpha_k =1$, we obtain 
\begin{equation*}
    1_M^\top \hat{x}^* = 1
\end{equation*}
which shows that the $k$-NN decoder is mass-preserving.
\end{proof}

\begin{remark} \label{rmk:kNN}
The classical $k$-NN algorithm that has been used to solve the pre-image problem in the context of nonlinear manifolds, e.g., with Diffusion Maps \cite{fix1985discriminatory,Patsatzis_2023} is a non-parametric and numerical analysis-agnostic method based on two operations: computing pairwise distances between all points and sorting to identify the nearest neighbors. As a consequence, it suffers from the ``curse of dimensionality'', since both computational cost and memory requirements scale poorly with the number of points $N$; this limitation can be mitigated via GPU parallelization or approximate nearest-neighbor search algorithms. More importantly, the accuracy of $k$-NN deteriorates when the data are sparse, and selecting an appropriate number of neighbors is challenging given the computational complexity. \\
Constructing the decoder via the $k$-NN interpolation in Eq.~\eqref{eq:convexinterpol} requires solving the constrained optimization problem in Eq.~\eqref{eq:LiftkNN}, which can be done with trust-region, projected gradient, or active-set methods. While these approaches explicitly satisfy the convexity constraints, they become computationally inefficient when many points need to be reconstructed, as the optimization must be performed point by point. Alternative implementations in the literature often replace the optimization with approximate barycentric or locally linear weights in the latent space to improve efficiency, at the cost of reduced reconstruction accuracy. A detailed review of these challenges and performance considerations can be found in \cite{halder2024enhancing}.

\end{remark}

\renewcommand{\theequation}{C.\arabic{equation}}
\renewcommand{\thefigure}{C.\arabic{figure}}
\renewcommand{\thetable}{C.\arabic{table}}
\renewcommand{\theprop}{C.\arabic{prop}}
\renewcommand{\theremark}{C.\arabic{remark}}
\setcounter{equation}{0}
\setcounter{figure}{0}
\setcounter{table}{0}
\setcounter{prop}{0}
\setcounter{remark}{0}
\section{The ``Double'' Diffusion Maps (DDM) decoder}
\label{app:DDM_dec}

We present the construction of the ``double'' Diffusion Maps (DDM) decoder \cite{evangelou2022double}, which combines DM embeddings \cite{coifman2005geometric,coifman2006diffusion} with kernel-based interpolation via GHs \cite{coifman2006geometric}. We briefly review these components, as the DM embedding is also used for encoding.  

Consider the dataset $\mathcal{X}=\{ x_i\}_{i=1}^N\subset\mathbb{R}^M$ of $N$ points in the ambient space, sampled i.i.d. from a smooth Riemannian manifold $(\mathcal{M},g)$. Let
\begin{equation}
    X = \begin{bmatrix} x_1, & \ldots, & x_N \end{bmatrix} \in\mathbb{R}^{M \times N},
\end{equation}
be the corresponding data matrix, stacking column-wise the data points $x_i$. The standard DM algorithm \cite{coifman2005geometric,coifman2006diffusion} begins by forming a Gaussian kernel matrix 
\begin{equation}
K^{(1)} \in \mathbb{R}^{N \times N}, \quad K_{ij}^{(1)} = k_1(x_i, x_j;\epsilon_1) = \operatorname{exp}\left(-\dfrac{\lVert x_i-x_j\rVert_2^2}{\epsilon_1^2}\right) \ge 0,
\end{equation}
where $\epsilon_1$ is the shape parameter. Following the $a$-parameterized symmetric normalization, we obtain the $a$-normalized kernel
\begin{equation}
    K^{(1,a)} = \big(D^{(1)}\big)^{-a} K^{(1)} \big(D^{(1)}\big)^{-a}, \quad
    D^{(1)} = \operatorname{diag}\Big(\sum_{j=1}^N K_{ij}^{(1)}\Big), 
\end{equation}
where $a\in[0,1]$. Row-normalization of $K^{(1,a)}$ then yields the Markov $a$-normalized transition matrix
\begin{equation}
T^{(a)}= \big(D^{(1,a)}\big)^{-1} K^{(1,a)}, \quad 
D^{(1,a)} = \operatorname{diag}\Big(\sum_{j=1}^N K_{ij}^{(1,a)}\Big)
\label{eq:transitionmatrix}
\end{equation}
whose $i$-th row represents the transition probabilities from $x_i$ point to all other points. We hereby note that the infinitesimal generator of the diffusion process defined by $T^{(a)}$ converges, in the limit of large $N$ and small $\epsilon_1$, to the Fokker–Planck operator for $a=0$ and to the Laplace–Beltrami operator for $a=1$ \cite{coifman2006diffusion}. In practice, choosing $a=1$ yields a density-invariant DM embedding that eliminates the effect of non-uniform sampling.%, which is often desirable in manifold learning.

The matrix $T^{(a)}$ in Eq.~\eqref{eq:transitionmatrix} is similar to the symmetric, positive semidefinite matrix $S=\left(D^{(1)}\right)^{-1/2} K^{(1)} \left(D^{(1)}\right)^{-1/2}$ and, thus, admits a decomposition
\begin{equation}
T=\sum_{i=1}^N \xi_i u_i v_{i}^\top, 
\end{equation}
where $\xi_i \in \mathbb{R}$ are the (positive) eigenvalues, and $u_i \in \mathbb{R}^N$ and $v_i \in \mathbb{R}^N$ are the corresponding left and right eigenvectors, respectively, which can be chosen to form a biorthonormal basis; i.e., $\langle u_i,v_{j} \rangle = \delta_{ij}$. It can be easily shown that $T^{(a)}$ has a trivial all-ones eigenvector corresponding to eigenvalue one. Excluding the trivial all-ones eigenvector (with eigenvalue one), the best $d$-dimensional Euclidean approximation of the row space of $T^{(a)}$ is given by the $d$ right eigenvectors corresponding to the largest eigenvalues %\cite{eckart1936approximation,mirsky1960symmetric}. This leads to the DM embedding of a data point $x_i$
\begin{equation}
y_i = (\xi_1 v_{1i}, \dots, \xi_d v_{di})^\top  \in \mathbb{R}^d.
\label{eq:DMcoords}
\end{equation}
It has been shown that the Euclidean distance between such coordinates provides the best approximation of the diffusion distance on the original manifold \cite{nadler2006diffusion}. 

For employing the DM algorithm, one needs to tune/select the hyperparameters $\epsilon_1$ and $\alpha$. Here, instead of tuning $\epsilon_1$, we set $\epsilon_1=w_1\, \text{median}(\|x_i-x_j\|_2)$ and tune $w_1$, while we set $\alpha$ depending on the benchmark example. For new points $x^*\in\mathbb{R}^M$, we further employ  out-of-sample extension via the standard Nystr\"{o}m extension \cite{coifman2006geometric} applied to the eigenfunctions of $T^{(a)}$; we revisit this interpolation in the context of decoding through GHs.

Before proceeding further, let us denote the set of DM embeddings by $\mathcal{Y}=\{ y_i \}_{i=1}^N \subset \mathbb{R}^d$ and the corresponding data matrix $Y = \begin{bmatrix} y_1, & \ldots, & y_N \end{bmatrix} \in \mathbb{R}^{d \times N}$. %It is also useful to define the unscaled eigenvector coordinates
%\begin{equation}
%    z_i = (v_{1i}, \ldots, v_{di})^\top \in \mathbb{R}^d
%\end{equation}
%with associated set $\mathcal{Z}=\{ z_i \}_{i=1}^N \subset \mathbb{R}^d$ and matrix $Z = \begin{bmatrix} z_1, & \ldots, & z_N \end{bmatrix} \in \mathbb{R}^{d \times N}$. From the definition of Eq.~\eqref{eq:DMcoords}, we have $y_i = \Xi_d z_i$, where $\Xi_d = \operatorname{diag}(\xi_1,\ldots,\xi_d)$, so that $Y = \Xi_d Z$ and $Z = \Xi_d^{-1} Y$.
The set $\mathcal{Y}$ serves as the input for the ``second'' DM step, as described next.

\paragraph{``Second'' DM-like step for pre-image reconstruction} To solve the pre-image problem, a ``second'' DM-like step is applied directly to $\mathcal{Y}$ \cite{evangelou2022double}. This constructs a new kernel and its eigenfunctions purely in latent space, providing a better basis for out-of-sample reconstruction via GHs \cite{coifman2005geometric,coifman2006geometric,evangelou2022double}.

Using the points $y_i \in \mathcal{Y}$, we form a Gaussian kernel matrix with shape parameter $\epsilon_2$
\begin{equation}
K^{(2)} \in \mathbb{R}^{N \times N}, \quad K_{ij}^{(2)} = k_2(y_i, y_j; \epsilon_2) = exp\left(-\dfrac{\lVert y_i-y_j\rVert_2^2}{\epsilon_2^2}\right) \ge 0.
\label{eq:DDMkernel}
\end{equation}
Being symmetric and positive semidefinite, $K^{(2)}$ admits an eigendecomposition
\begin{equation}
K^{(2)} \phi_{i} = \lambda_i \phi_{i}, \quad i = 1, \dots, r,
\end{equation}
where $\lambda_i\ge0$ are the eigenvalues, sorted in descending order and $\phi_i\in\mathbb{R}^N$ are the corresponding orthonormal eigenvectors. Retaining the first $r$ dominant eigenmodes gives the eigenvector matrix
\begin{equation}
V_r = \begin{bmatrix}
    \phi_{1}, & \ldots, & \phi_{r}
\end{bmatrix} \in \mathbb{R}^{N \times r},
\end{equation}
and the diagonal eigenvalue matrix $\Lambda_r = \operatorname{diag}(\lambda_1,\ldots,\lambda_r) \in \mathbb{R}^{r \times r}$. The columns of $V_r$ form an orthonormal basis for the subspace used to reconstruct functions in the high-dimensional space via the DDM decoder. 

Any function $f \in \mathbb{R}^N$ defined on the training set $\mathcal{X}$ can be approximated in the truncated DDM subspace as
\begin{equation}
f \approx V_r V_r^\top f.
\label{eq:proj_train}
\end{equation}
Using the truncated eigendecomposition $K^{(2)} \approx V_r \Lambda_r V_r^\top$, the projection operator can be rewritten as
\begin{equation}
V_r V_r^\top
= V_r \Lambda_r^{-1} V_r^\top K^{(2)},
\label{eq:proj_kernel}
\end{equation}
where $V_r \Lambda_r^{-1} V_r^\top$ is the truncated pseudo-inverse of $K^{(2)}$. Then, for functions defined on the training set, we obtain the reconstruction
\begin{equation}
f \approx V_r \Lambda_r^{-1} V_r^\top K^{(2)} f.
\label{eq:recon_train}
\end{equation}
Clearly, to reconstruct with high accuracy the ambient space, $r$ should be taken large enough \cite{evangelou2022double,Patsatzis_2023}, with exact reconstruction attained for $r=N$.

Equation \eqref{eq:recon_train} gives the projection of a function $f$ defined on the training set. For reconstruction of new points, however, $f$ is not available. Instead, we rely on the kernel, which can be evaluated between new unseen points and training ones. This leads to a Nystr\"{o}m-type extension known as Geometric Harmonics \cite{coifman2006geometric}. Specifically, consider a new latent point $y^* \in \mathbb{R}^{d}$ obtained via the (extended) DM embedding in Eq.~\eqref{eq:DMcoords}. %We first extract its corresponding eigenvector coordinate via the diagonal scaling $z^*=\Xi_d^{-1} y^*$. 
The kernel vector between $y^*$ and DM coordinates of the training set $\mathcal{Y}$ is  
\begin{equation}
    k^*=
\begin{bmatrix}
k_2(y^*, y_1; \epsilon_2), & \ldots, & k_2(y^*, y_N; \epsilon_2)
\end{bmatrix}\in \mathbb{R}^N.
\label{eq:2DMkernel}
\end{equation}
In analogy to Eq.~\eqref{eq:recon_train}, the projection coefficients $\alpha \in \mathbb{R}^N$ for this new point are be obtained by replacing $K^{(2)} f$ with $k^{*\top}$. This yields the GHs extension formula
\begin{equation}
    \alpha= V_r \, \Lambda_r^{-1} \, V_r^\top k^{*\top} \in \mathbb{R}^N.
    \label{eq:2DM2}
\end{equation}
The vector $\alpha$ represents the coordinates of the new point in the DDM subspace, effectively extending the projection operator to out-of-sample data.

For the pre-image problem, we seek the corresponding high-dimensional reconstructed state $\hat{x}^*$. Following standard kernel-based pre-image methods, we assume $\hat{x}^*$ can be approximated as a linear combination of the available training samples
\begin{equation}
\hat{x}^* = X\alpha.
\label{eq:2DM1}
\end{equation}
Substituting Eq.~\eqref{eq:2DM2} into \eqref{eq:2DM1} gives the compact reconstruction of the DDM decoder for a single latent point
\begin{equation}
    \hat{x}^* = X\,V_r\, \Lambda_r^{-1} \, V_r^\top \, k^{*\top}.
\label{eq:DoubleDMs}
\end{equation}
We note that this formulation enables the simultaneous processing of multiple new latent points. Specifically, for a set of unseen latent points $\mathcal{Y}^*=\{ y^*_l\}_{l=1}^L$, %we first compute the DM eigenvector coordinates via the diagonal scaling $z^*_l=\Xi_d^{-1} y^*_l$ for every $l=1,\ldots,L$ and then 
we first evaluate the kernel matrix $K^*\in \mathbb{R}^{L \times N}$ in Eq.~\eqref{eq:2DMkernel}, with elements $K^*_{lj} = k_2(y^*_l,y_i,\epsilon_2)$. The batch reconstruction then follows directly from Eq.~\eqref{eq:DoubleDMs}, yielding
\begin{equation}
    \widehat{X}^* = X\,V_r\, \Lambda_r^{-1} \, V_r^\top \, K^{*\top},
    \label{eq:DoubleDMs_mat}
\end{equation}
where $\widehat{X}^*=\begin{bmatrix} x^*_1, & \ldots, & x^*_L \end{bmatrix} \in \mathbb{R}^{M \times L}$ collects column-wise all the reconstructed fields $x_l^*$ for $l=1,\ldots,L$. We note that for employing the DDM encoder, one needs to tune/select the shape parameter $\epsilon_2$ and the number of DDM eigenvectors $r$. Here, in analogy to the DM algorithm, we set $\epsilon_2=w_2\, \text{median}(\|y_i-y_j\|_2)$ and tune $w_2$ instead of $\epsilon_2$.%, while we select $r$ depending on the benchmark example.

Finally we note that the ``first'' DM embedding can be extended to new ambient points $\mathcal{X}^* = \{ x^*_l\}_{l=1}^L\subset \mathbb{R}^M$ via the standard Nystr\"{o}m extension, analogous in form the GHs-based decoder in Eq.~\eqref{eq:DoubleDMs_mat}, but applied in the opposite direction. This procedure projects the kernel evaluations $k_1(x^*_l,x_i,\epsilon_1)$ onto the first $d$ right eigenvectors of the transition matrix $T^{(a)}$ (scaled by their eigenvalues) to produce the latent coordinates $y^*_l$, which are collected in $\mathcal{Y}^* = \{ y^*_l\}_{l=1}^L\subset \mathbb{R}^d$, thus completing the DM encoder. Hence, the Nystr\"{o}m extension maps $\mathcal{X}^* \to \mathcal{Y}^*$ (encoding), while the DDM decoder maps $\mathcal{Y}^* \to \mathcal{X}^*$ (decoding).

\begin{remark}
    In the classic DDM algorithm \cite{evangelou2022double}, the ``second'' DM-like step is not performed on the DM coordinates $y_i$ themselves, but on the unscaled DM eigenvector coordinates
    \begin{equation}
        z_i = (v_{1i}, \ldots, v_{di})^\top \in \mathbb{R}^d,
    \end{equation}
    which are related to the DM coordinates via $z_i=\Xi_d^{-1} y_i$ where $\Xi_d = \operatorname{diag}(\xi_1,\ldots,\xi_d)$ (see Eq.~\eqref{eq:DMcoords}). Denoting $\mathcal{Z}=\{ z_i \}_{i=1}^N \subset \mathbb{R}^d$ and $Z = \begin{bmatrix} z_1, & \ldots, & z_N \end{bmatrix} \in \mathbb{R}^{d \times N}$, we have $Y = \Xi_d Z$ and $Z = \Xi_d^{-1} Y$. Consequently, the kernel matrix $K^{(2)}$ in Eq.~\eqref{eq:DDMkernel} for known points and $k^*$ in Eq.~\eqref{eq:2DMkernel} for an unknown point $y^*$ are evaluated on $\mathcal{Z}$, i.e., using $k_2(z_i, z_j; \tilde{\epsilon}_2)$ and $k_2(z^*, z_j; \tilde{\epsilon}_2)$, rather than on the original DM coordinates $\mathcal{Y}$. 

    In our presentation, we form the second kernel directly on the DM coordinates $y_i\in \mathcal{Y}$ for two reasons: (i) to allow the DDM decoder to be applied to any latent embedding $\mathcal{Y}$, not only those obtained from DM algorithm with known eigendecomposition; and (ii) to align the notation with the RFNN, RANDSMAP, and $k$‑NN decoders, enabling a unified comparison of conservation properties. The diagonal scaling $\Xi_d$ is effectively absorbed into the kernel bandwidth, so that $\tilde{\epsilon}_2$ in the classic formulation corresponds to a suitably chosen $\epsilon_2$ in Eqs.~\eqref{eq:DDMkernel}, \eqref{eq:2DMkernel} and \eqref{eq:DoubleDMs_mat}.
\end{remark}

\renewcommand{\theequation}{D.\arabic{equation}}
\renewcommand{\thefigure}{D.\arabic{figure}}
\renewcommand{\thetable}{D.\arabic{table}}
\setcounter{equation}{0}
\setcounter{figure}{0}
\setcounter{table}{0}
\section{Implementation details} \label{app:ImpDet}
Here, we provide the implementation details for the training, tuning, and assessment of the considered decoders.

For all benchmark problems, we generated synthetic datasets $\mathcal{X}=\{x_i\}_{i=1}^{N_{obs}}\subset \mathbb{R}^M$ with $N_{obs}$ data points. Each dataset was randomly split into training ($\mathcal{X}_{tr}$), validation ($\mathcal{X}_{vl}$) and testing ($\mathcal{X}_{ts}$) subsets. The splitting ratios varied between 10/10/80\% and 20/20/60\% to explore the effect of sparsity in the training data. In every configuration, the training and validation sets contained the same number of observations $N$, while the testing sets contained $L$ unseen observations.

To obtain the low-dimensional embedding $\Psi_N$ in Eq.~\eqref{eq:PsiN}, we employed the DM algorithm (see Appendix~\ref{app:DDM_dec}) on the training data $\mathcal{X}_{tr}$, tuning the DM weight $w_1$ for each benchmark. The resulting DM encoder, augmented by the Nystr\"{o}m extension for out-of-sample data, produced the training ($\mathcal{Y}_{tr}$), validation ($\mathcal{Y}_{vl}$) and testing ($\mathcal{Y}_{ts}$) subsets of low-dimensional DM coordinates $y\in\mathbb{R}^d$, where the intrinsic dimension $d$ varied across the benchmarks. %; $\mathcal{Y}_{tr}$ is used for training the decoders, $\mathcal{Y}_{vl}$ for tuning the related hyperparameters and $\mathcal{Y}_{ts}$ to evaluate their performance and generalization in unseen data.

For the pre-image problem mapping $\Psi_N^{-1}$, we constructed three variants for each RANDSMAP decoder, distinguished by their feature mapping: vanilla random Fourier features (RANDSMAP-RFF), multi‑scale random Fourier features (RANDSMAP-MS‑RFF), and sigmoidal features (RANDSMAP-Sig). The corresponding non‑conservative RFNN decoders follow the same feature choices and are labeled as RFNN‑RFF, RFNN‑MS‑RFF, and RFNN‑Sig. For RFF-based variants, Fourier features were generated by randomly sampling weights $w_k\sim\mathcal{N}(0,\sigma_w^2 I_d)$ and phases $b_k\sim\mathcal{U}[0,2\pi)$ according to Eq.~\eqref{eq:fourierrand}. The bandwidth parameter $\sigma_w$ was treated as a tunable hyperparameter. For MS-RFF-based variants, we employed the hierarchical multi-scale sampling discussed in Section~\ref{sb:RFNNdec_ms}. Specifically, we used $Q=10$ scales, drawing each scale parameter uniformly $\sigma_w\sim[0.001,\sigma_{UB})$, and then sampling $w_k \mid \sigma_w \sim \mathcal{N}(0,,\sigma_w^2 I_d)$ and $b_k\sim\mathcal{U}[0,2\pi)$ to form the feature set. The upper bound $\sigma_{UB}$ of the bandwidth distribution served as the tuning parameter. For Sig-based variants, following the ``parsimonious'' construction in \cite{galaris2022numerical},  
we sampled internal weights uniformly $w_k\sim\mathcal{U}[-c,c)$, and set the internal biases $b_k$ so that the inflection points of the sigmoid activations correspond to uniformly distributed centers within the training input domain $\mathcal{Y}_{tr}$. The weight bound $c$ was tuned. For all the above decoders, we considered feature numbers $P = N, N/2, N/4$ to evaluate the trade‑off between reconstruction accuracy and computational cost. An output bias term was included in every architecture, as it is essential for the RANDSMAP conservative schemes discussed in Section~\ref{sb:RANDSMAP}. 

Once the feature matrix $\Phi_S$ of the RANDSMAP and RFNN decoders was constructed from the training dataset $\mathcal{Y}_{tr}$, the coefficient matrix $A_S$ was computed by the closed-form solutions in Eq.~\eqref{eq:ShurAfinal0} and \eqref{eq:vanillaTikhonov}, respectively. The Tikhonov regularization parameter was set to $\lambda=10^{-3}$; this value provided mild numerical stabilization without significantly affecting the reconstruction error across benchmarks. For the RANDSMAP solution in Eq.~\eqref{eq:ShurAfinal0}, we truncated the SVD of $\Phi_S$, retaining the leading $tr$ modes with singular values satisfying $\sigma_{tr}>\delta_S\, \sigma_1$, where the tolerance was set to $\delta_S=10^{-8}$ to achieve single precision accuracy in conservation error, as per the bound in Corollary~\ref{cor:cons_err_bound}. For the RFNN solution, we solved the positive‑definite linear system in Eq.~\eqref{eq:vanillaTikhonov} via Cholesky decomposition.

The RANDSMAP and RFNN decoders involve random sampling of feature parameters $(w_k,b_k)$. To account for this stochasticity, we performed 100 independent runs for each hyperparameter configuration with different random initializations. All performance metrics were then aggregated across these repetitions.

For comparison to the RANDSMAP and RFNN decoders, we further constructed the numerical analysis-based DDM and $k$-NN decoders. For the DDM decoder (see Section~\ref{sec:decodingwithDMS}), we selected the number of DDM coordinates $r$ by retaining eigenvectors corresponding to eigenvalues $\lambda_r$ satisfying $\lambda_r>\delta \lambda_1$ where $\delta=N\|K^{(2)}\|_2\epsilon_u$ is a numerical tolerance, $\|K^{(2)}\|_2$ is the spectral norm of the DDM kernel $K^{(2)}$ and $\epsilon_u=2^{-53}\approx1.1\times 10^{-16}$ is the unit round-off in double precision arithmetic. The DDM weight $w_2$ for the construction of the ``second'' DM-like step kernel was tuned for each benchmark. For the $k$-NN decoder (see Appendix~\ref{app:kNN}), we used the trust-region method \cite{coleman1996interior} to solve the constrained optimization problem in Eq.~\eqref{eq:LiftkNN}, with function and optimality tolerances set to $10^{-8}$. The number of neighbors $k$ was tuned for each benchmark. 

As evident from the discussion above, all decoders were intentionally constructed with one tunable hyperparameter ($\sigma_w$ for RANDSMAP-/RFNN-RFF, $\sigma_{UB}$ for RANDSMAP-/RFNN-MS-RFF, $c$ for RANDSMAP-/RFNN-Sig, $\epsilon_2$ for DDM, and $k$ for $k$-NN decoders) to enable a fair comparison of computational cost. Hyperparameter tuning was performed on the validation set $\mathcal{Y}_{vl}$, using a grid search over $10$ values in specific, for each benchmark, ranges. The configuration yielding the lowest $L_2$ reconstruction error was selected. The decoder was then re‑trained on the training set $\mathcal{Y}_{tr}$ using this optimal hyperparameter value before final evaluation.
 
For assessing decoder performance, we computed the per-point relative $L_2$ and $L_\infty$ reconstruction errors, and, where relevant (i.e., when the original data are mass-preserving), the conservation errors 
\begin{equation}
    e_{2,i} = \dfrac{\|\tilde{x}_i - x_i\|_2 }{\|x_i\|_2}, \quad
    e_{\infty,i} = \dfrac{\|\tilde{x}_i - x_i\|_\infty }{\|x_i\|_\infty}, \quad e_{con,i} = \bigl| 1_M^\top\, \tilde{x}_{i} - 1 \bigr|,
    \label{eq:recon_errors}
\end{equation}
where $x_i$ is the true high-dimensional observation and $\tilde{x}_i$ is its reconstruction obtained from the decoder as $\tilde{x}_i=\Psi_N^{-1}(y_i)$, where $y_i=\Psi_N(x_i)$. The errors are evaluated both on the training set $\mathcal{X}_{tr}$ (in-sample data) and the testing set $\mathcal{X}_{ts}$ (out-of-sample data). For test points, $x_i$ in Eq.~\eqref{eq:recon_errors} corresponds to the unseen observation $x^*_i$ and $\tilde{x}_i$ to the reconstructed field $\tilde{x}^*_{i}$. 

%\clearpage
%\newpage
\renewcommand{\theequation}{E.\arabic{equation}}
\renewcommand{\thefigure}{E.\arabic{figure}}
\renewcommand{\thetable}{E.\arabic{table}}
\setcounter{equation}{0}
\setcounter{figure}{0}
\setcounter{table}{0}
\section{The Swiss Roll dataset (\texorpdfstring{$M=3$}{M=3})} \label{app:SR}

Here, we consider our first non-mass-preserving benchmark example to provide a baseline performance of the proposed non-mass-preserving (RFNN, DDM and $k$-NN) decoders. We consider a standard test example for manifold learning and nonlinear decoder evaluation, the Swiss Roll benchmark. The data lie on an intrinsic $d=2$-dimensional manifold $\mathcal{M}$ embedded in the $M=3$-dimensional ambient space; i.e., $x_i\in \mathbb{R}^3$. To generate the synthetic dataset, we sample $N_{obs}$ angles $\theta_i\sim \mathcal{U}[\pi, 4 \pi)$ and heights $z_i\sim \mathcal{U}[-5,5)$ uniformly. These are mapped to ambient coordinates via the Swiss Roll equations
\begin{equation}
x^1_i = (\theta_i + 0.1 z_i) \sin(\theta_i)/4 + \mathcal{N}(0,\sigma^2), \quad x^2_i = (\theta_i + 0.1 z_i) \cos(\theta_i)/4 + \mathcal{N}(0,\sigma^2), \quad x^3_i = z_i + \mathcal{N}(0,\sigma^2), \label{eq:SwissRolleq}
\end{equation}
with additive Gaussian noise $\mathcal{N}(0,\sigma^2)$ of magnitude $\sigma=0.05$. From the full dataset, we construct two distinct training/validation/testing splits: a sparse configuration with $N=1000$ points in $\mathcal{X}_{tr}$ and $\mathcal{X}_{vl}$ and a dense configuration with $N=2,000$; the remaining points are reserved for the testing sets $\mathcal{X}_{ts}$. Visualizations of the two training sets are provided in Figs.~\ref{fig:SR_data_1K} and \ref{fig:SR_data_2K}.

\begin{figure}[!b]
    \centering
    % First row
    \begin{subfigure}[b]{0.32\textwidth}
        \centering
        \includegraphics[width=\textwidth]{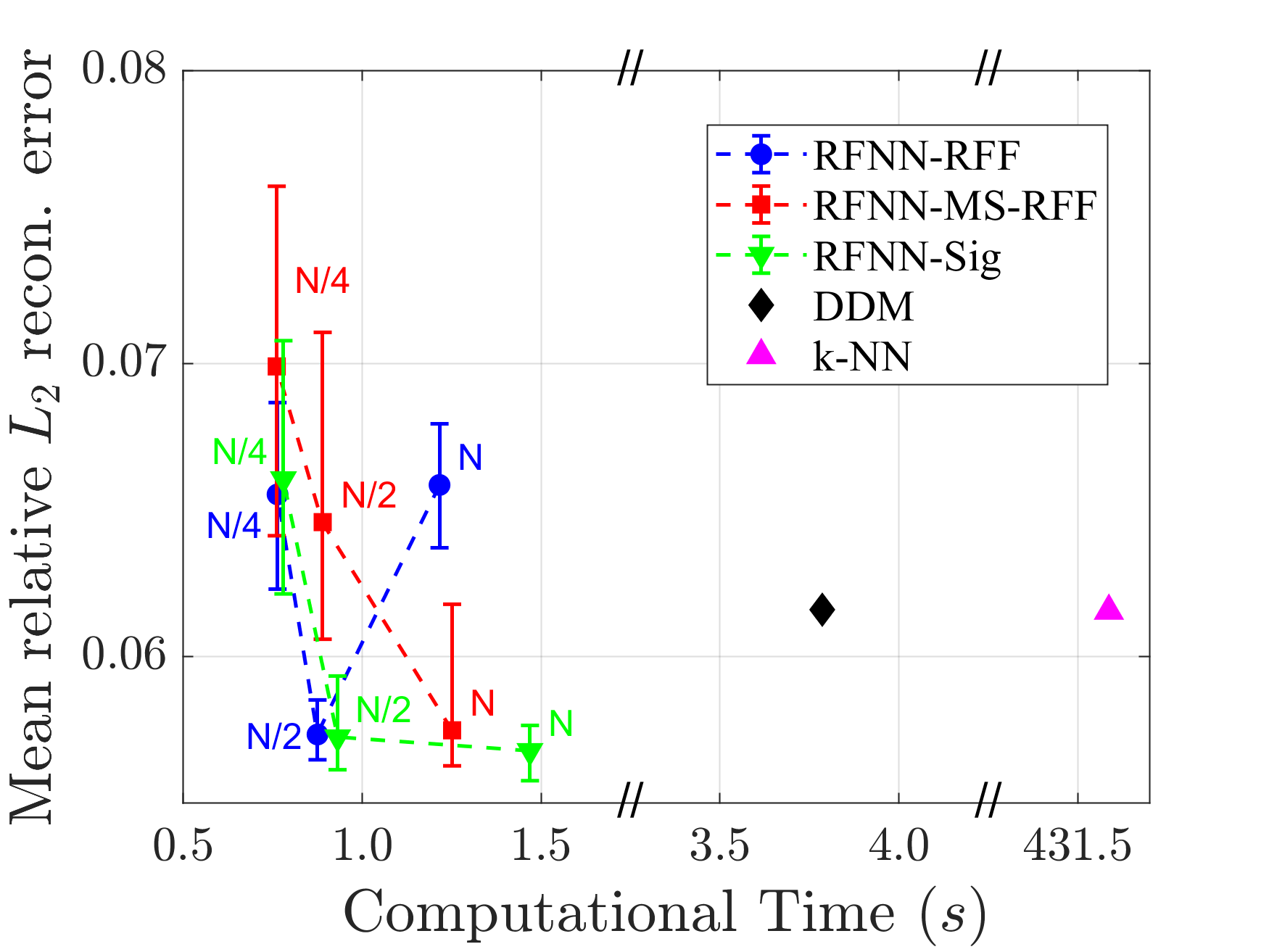}
        \caption{$N=1000$, training, $L_2$ error}
        \label{fig:SR_dec_tr1K}
    \end{subfigure}
    \hspace{2pt}
    \begin{subfigure}[b]{0.32\textwidth}
        \centering
        \includegraphics[width=\textwidth]{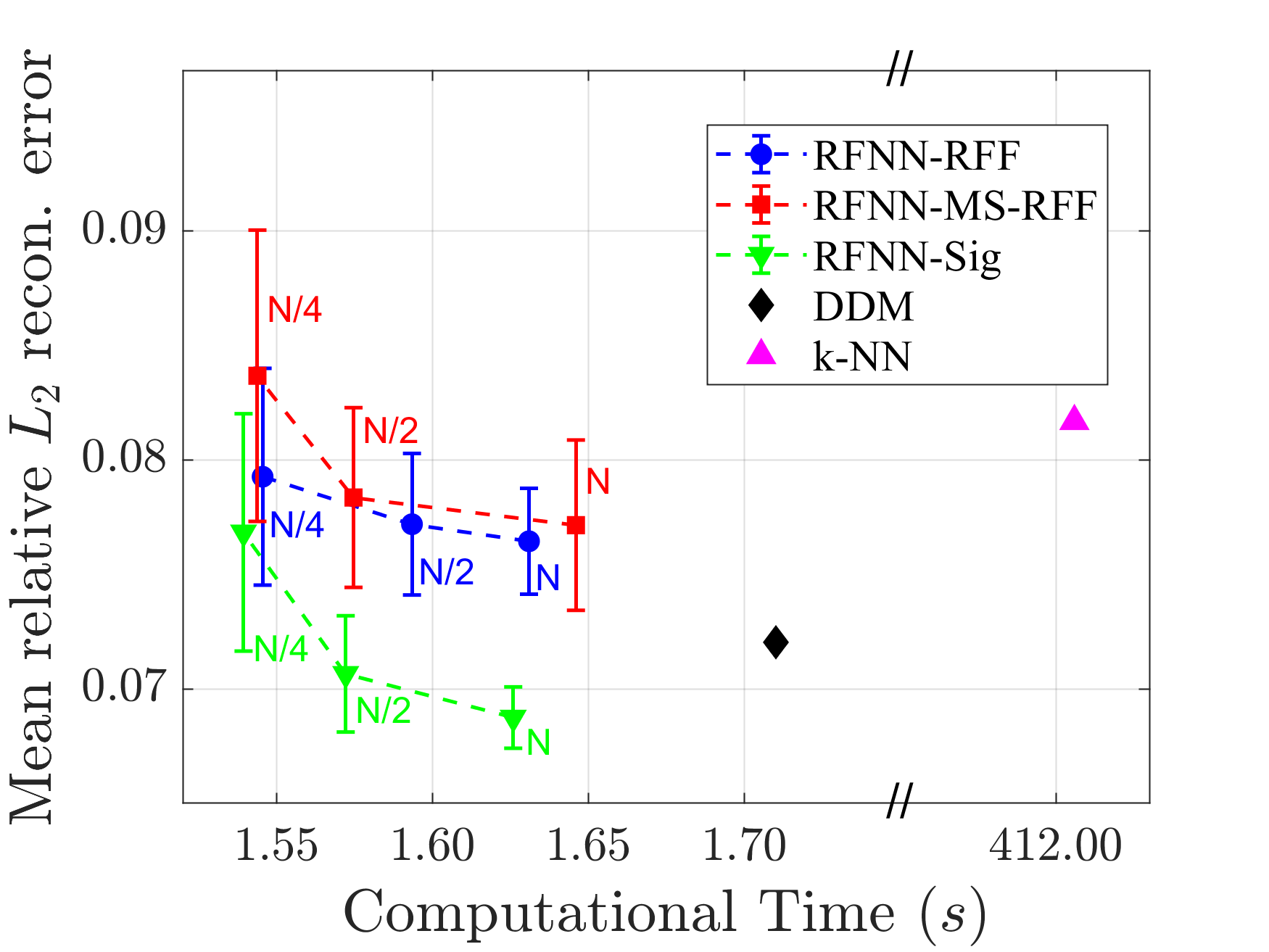}
        \caption{$N=1000$, testing, $L_2$ error}
        \label{fig:SR_dec_ts1K}
    \end{subfigure}
    \hspace{2pt}
    \begin{subfigure}[b]{0.32\textwidth}
        \centering
        \includegraphics[width=\textwidth]{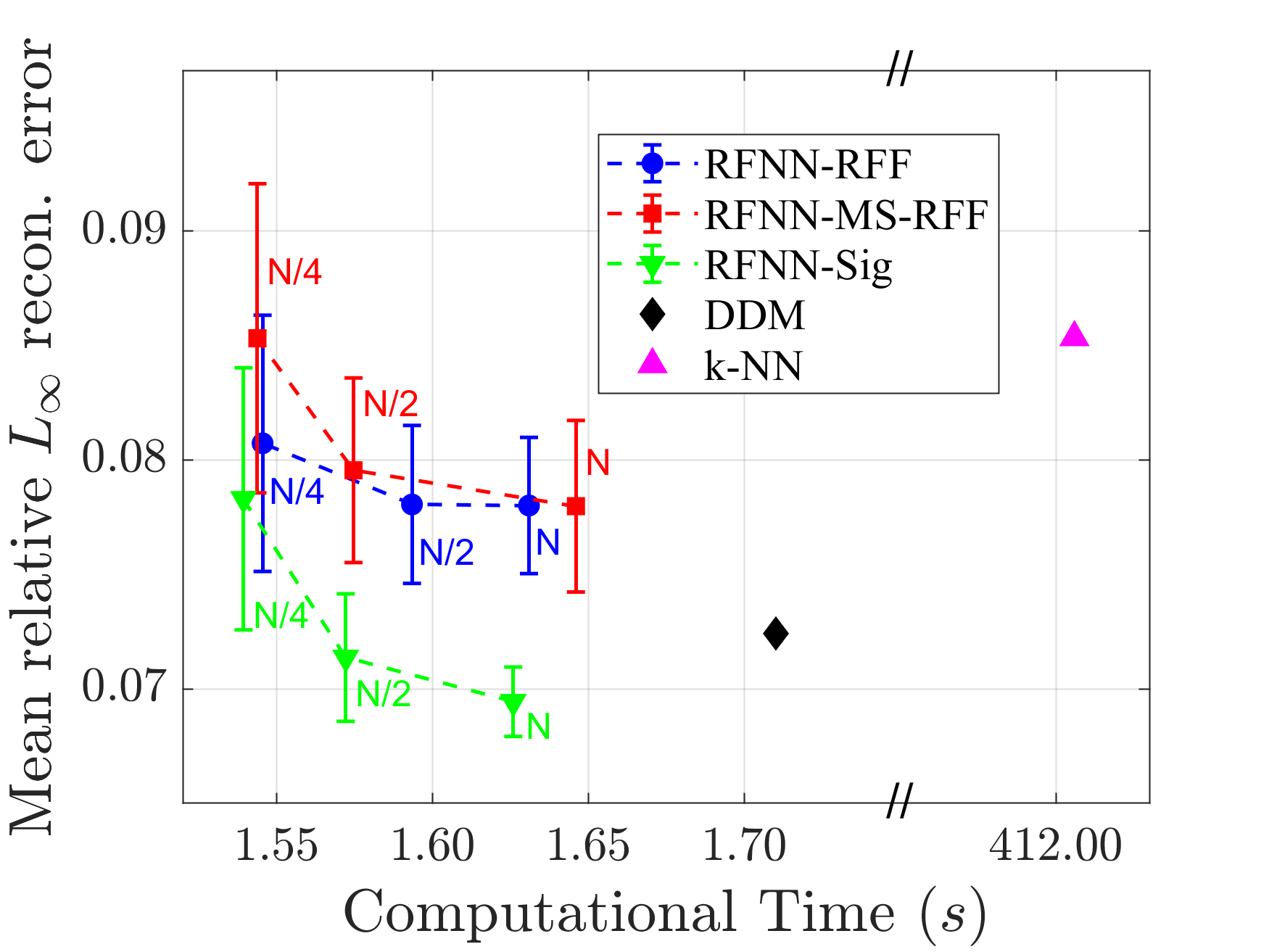}
        \caption{$N=1000$, testing, $L_\infty$ error}
        \label{fig:SR_dec_ts1K_linf}
    \end{subfigure}
    
    % Second row
    \begin{subfigure}[b]{0.32\textwidth}
        \centering
        \includegraphics[width=\textwidth]{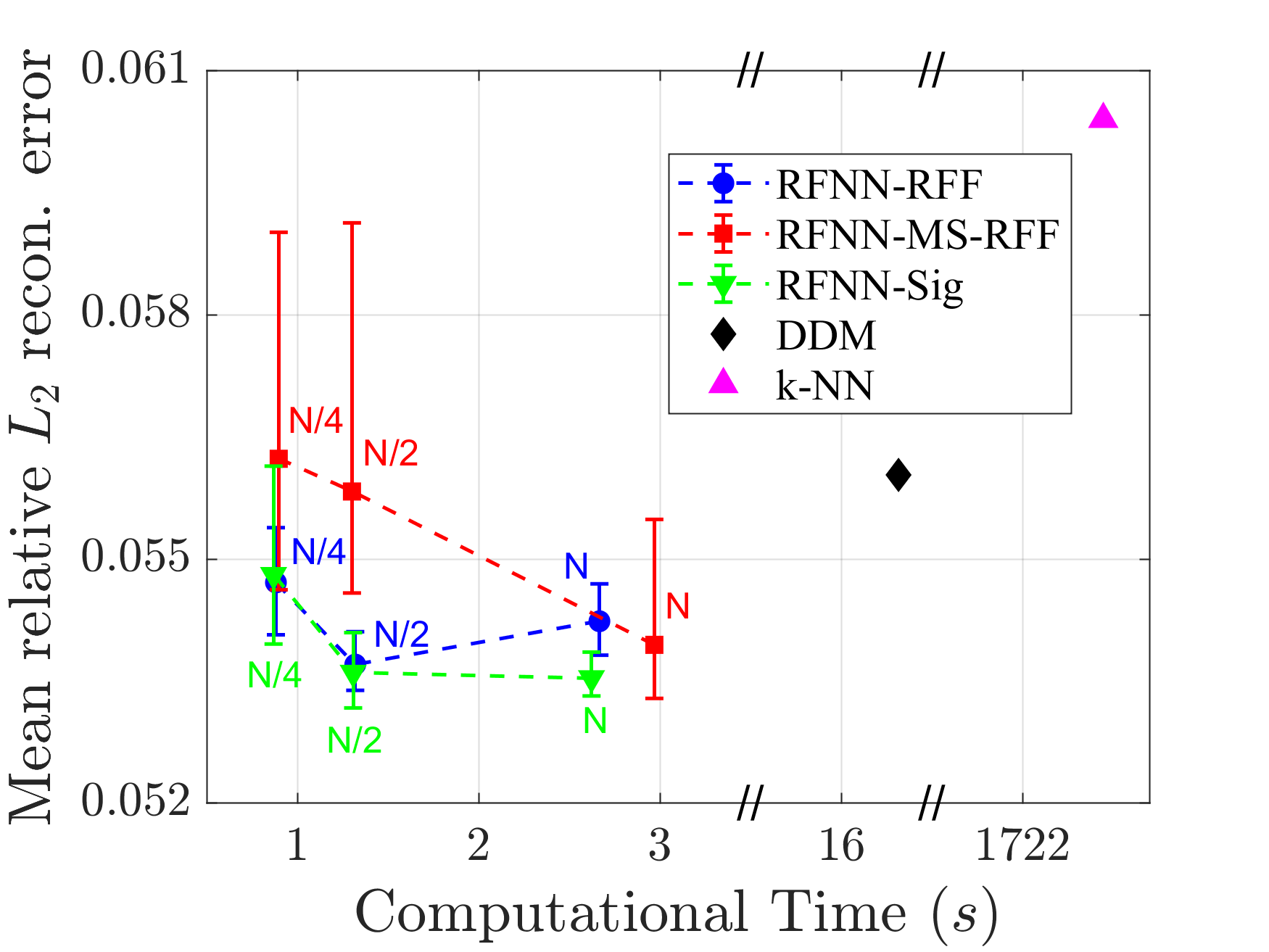}
        \caption{$N=2000$, training, $L_2$ error}
        \label{fig:SR_dec_tr2K}
    \end{subfigure}
    \hspace{2pt}
    \begin{subfigure}[b]{0.32\textwidth}
        \centering
        \includegraphics[width=\textwidth]{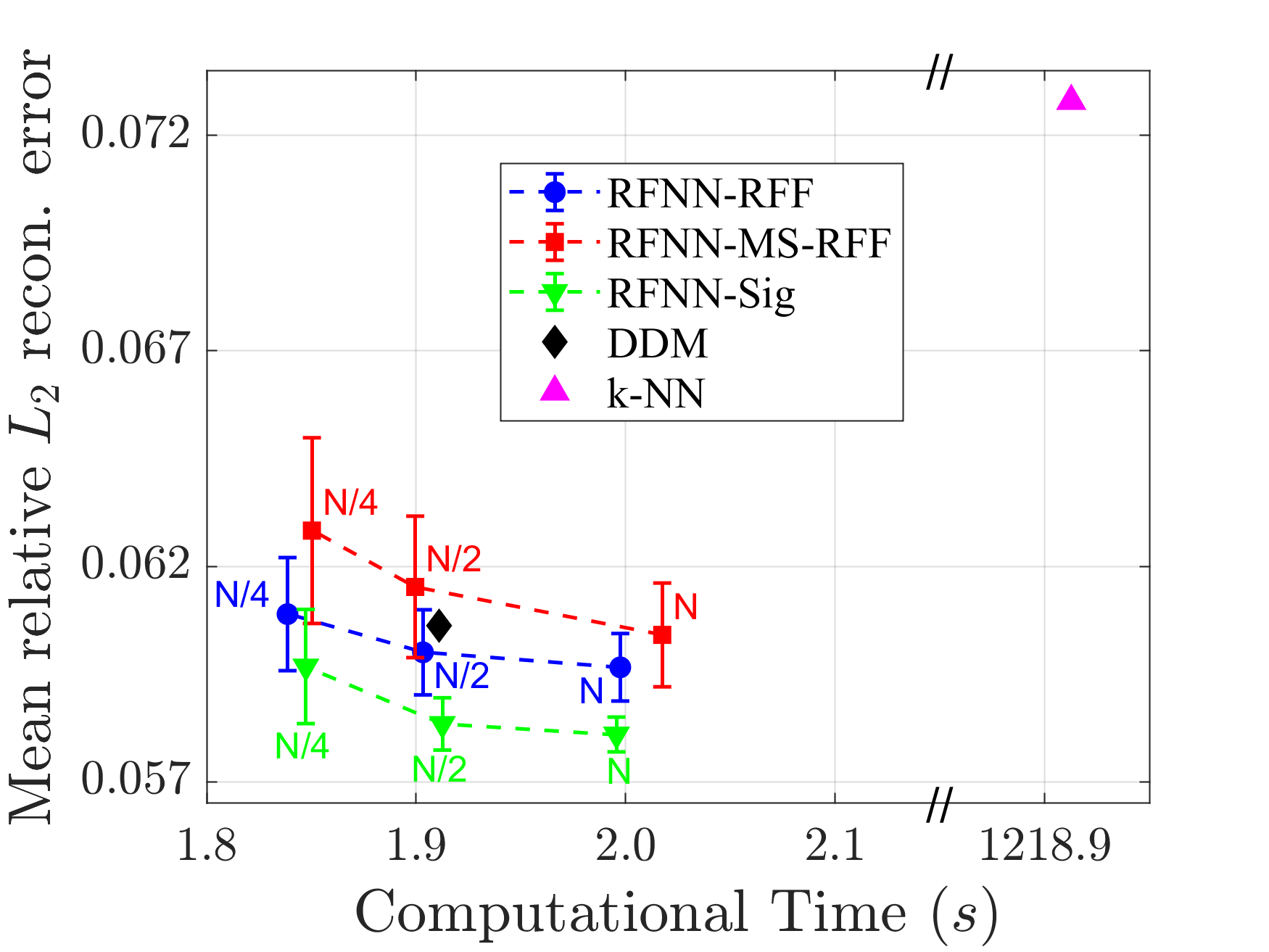}
        \caption{$N=2000$, testing, $L_2$ error}
        \label{fig:SR_dec_ts2K}
    \end{subfigure}
    \hspace{2pt}
    \begin{subfigure}[b]{0.32\textwidth}
        \centering
        \includegraphics[width=\textwidth]{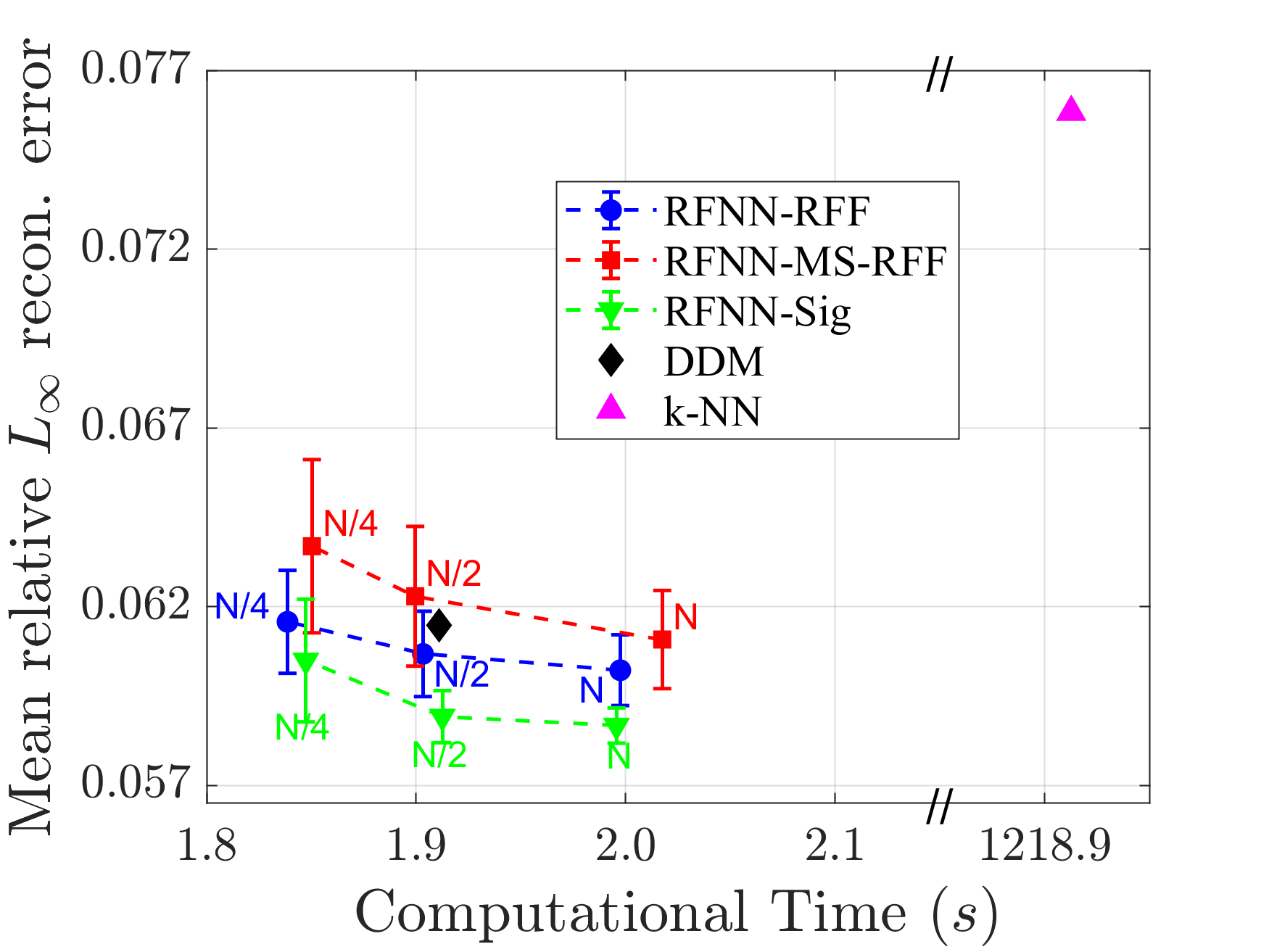}
        \caption{$N=2000$, testing, $L_\infty$ error}
        \label{fig:SR_dec_ts2K_linf}
    \end{subfigure}
    \caption{Reconstruction error vs. computational time for the Swiss Roll dataset ($M=3$). Results are shown for $N=1000$ (top row) and $N=2000$ (bottom row) training points. Panels~\ref{fig:SR_dec_tr1K}, \ref{fig:SR_dec_tr2K} show the mean relative $L_2$ error $e_{2,i}$ (Eq.~\eqref{eq:recon_errors}) on the training set vs. training time. Panels~\ref{fig:SR_dec_ts1K}, \ref{fig:SR_dec_ts2K} and \ref{fig:SR_dec_ts1K_linf}, \ref{fig:SR_dec_ts2K_linf} show the mean relative $L_2$ and $L_\infty$ error, $e_{2,i}$ and $e_{\infty,i}$, respectively, on the testing set vs. inference time. For the stochastic RFNN decoders (RFNN-RFF, RFNN-MS-RFF, RFNN-Sig) with $P=N,N/2,N/4$, points indicate the median over 100 runs; error bars show the 5–95\% percentile range. Deterministic decoders (DDM, $k$-NN) are shown without error bars. Detailed numerical results are provided in Tables~\ref{tab:SR_dec_tr_all} and \ref{tab:SR_dec_ts_all}.}
    \label{fig:SwissRoll_Dec}
\end{figure}

For each configuration, we apply the DM algorithm with $\alpha=1$ and DM weight $w_1=0.12$ (see Appendix~\ref{app:DDM_dec}) to obtain the intrinsic DM coordinates $y_i\in\mathbb{R}^2$. The resulting DM encoder, extended via the Nystr\"{o}m out-of-sample extension, yields the training, validation, and testing sets $\mathcal{Y}_{tr}$, $\mathcal{Y}_{vl}$, $\mathcal{Y}_{ts}$ for decoder training, tuning and evaluation. DM coordinate visualizations, colored by the angular coordinate $\theta$ are shown in Figs.~\ref{fig:SR_DMcoords_1K} and \ref{fig:SR_DMcoords_2K}.

The training and testing performance of all decoders is detailed in Tables~\ref{tab:SR_dec_tr_all} and \ref{tab:SR_dec_ts_all}, where the mean relative $L_2$ and $L_\infty$ reconstruction errors and the computational times required for training and inference are shown. A visual summary of the reconstruction error vs. computational time trade-off is shown in Fig.~\ref{fig:SwissRoll_Dec} (top: sparse $N=1000$, bottom: dense $N=2000$). 

We next train the RFNN with $P=N,N/2,N/4$ (downsampling was not required in this example, rendering $n=N$), DDM, and $k$-NN decoders following the procedure detailed in Appendix~\ref{app:ImpDet}. Specifically, to tune the hyperparameters of the RFNN-RFF, RFNN-MS-RFF and RFNN-Sig decoders we consider the ranges $\sigma_w\in[0.1,1]$, $\sigma_{UB}\in[1,10]$ and $c \in [1,20]$; for $P=N$ the optimal values were $\sigma_w=0.3$ (sparse) and $\sigma_w=0.2$ (dense), $\sigma_{UB}=6$ and $c=8$ (same in sparse and dense datasets). For tuning the DDM and $k$-NN decoders we consider the ranges $w_2\in[0.3,0.9]$ and $k\in[2,11]$; optimal values were $w_2=0.4$, and $k=6$ (sparse) and $k=9$ (dense). 

\begin{table}[htbp]
\centering
\caption{Training set performance for the Swiss Roll dataset ($M=3$). Detailed reconstruction metrics for $N=1000$ and $N=2000$ training points, using $d=2$-dim. DM embeddings. Decoders are compared based on the relative $L_2$ and $L_\infty$ mean reconstruction errors, $e_{2,i}$ and $e_{\infty,i}$ (Eq.~\eqref{eq:recon_errors}), and computational time (in seconds). For stochastic decoders (RFNN-RFF, RFNN-MS-RFF and RFNN-Sig), metrics show the median with 5-95\% percentiles in parentheses, computed over 100 random initializations. Deterministic decoders (DDM, $k$-NN) report single values.}
\label{tab:SR_dec_tr_all}
%\resizebox{\textwidth}{!}{
\begin{tabular}{@{}llccc@{}}
\multirow{2}{*}{Decoder} & \multirow{2}{*}{$P$} & \multicolumn{2}{c}{Mean reconstruction error ($\times 10^{-2}$)} & \multirow{2}{*}{Comp. Time (s)} \\
\cmidrule(lr){3-4}
& & relative $L_2$, $e_{2,i}$ &  relative $L_\infty$, $e_{\infty,i}$ &  \\
\midrule
\multicolumn{5}{l}{\textbf{$N=1000$}} \\
\midrule
\midrule

% DM-RFNNf variants
\multirow{3}{*}{RFNN-RFF} 
&	$N$	    & 6.586 (6.372--6.795) & 6.757 (6.503--7.008) & 1.216 (1.211--1.228) \\
&	$N/2$	& 5.735 (5.649--5.853) & 5.792 (5.704--5.919) & 0.876 (0.873--0.881) \\
&	$N/4$	& 6.554 (6.231--6.868) & 6.728 (6.329--7.066) & 0.764 (0.763--0.766) \\
\midrule

% DM-RFNNfMK variants  
\multirow{3}{*}{RFNN-MS-RFF} 
&	$N$	    & 5.749 (5.628--6.179) & 5.818 (5.655--6.296) & 1.252 (1.244--1.266) \\
&	$N/2$	& 6.459 (6.060--7.107) & 6.613 (6.161--7.342) & 0.890 (0.889--0.893) \\
&	$N/4$	& 6.991 (6.413--7.606) & 7.182 (6.519--7.900) & 0.762 (0.761--0.763) \\
\midrule

% DM-RFNNs variants  
\multirow{3}{*}{RFNN-Sig} 
& $N$   & 5.679 (5.577--5.766) & 5.719 (5.618--5.836) & 1.469 (1.465--1.483) \\
& $N/2$ & 5.725 (5.615--5.934) & 5.780 (5.638--6.037) & 0.932 (0.930--0.937) \\
& $N/4$ & 6.606 (6.215--7.079) & 6.792 (6.313--7.312) & 0.780 (0.780--0.783) \\
\midrule
% Deterministic methods
DDM  & - & 6.161 & 6.106 & 3.786   \\
$k$-NN & - & 6.156 & 6.377 & 431.587 \\

\midrule
\multicolumn{5}{l}{\textbf{$N=2000$}} \\
\midrule
\midrule

% DM-RFNNf variants
\multirow{3}{*}{RFNN-RFF} 
&	$N$	    & 5.424 (5.382--5.470) & 5.407 (5.349--5.454) & 2.667 (2.645--2.688) \\
&	$N/2$	& 5.371 (5.339--5.411) & 5.350 (5.307--5.394) & 1.319 (1.312--1.331) \\
&	$N/4$	& 5.471 (5.407--5.539) & 5.459 (5.376--5.538) & 0.882 (0.879--0.886) \\
\midrule

% DM-RFNNfMK variants  
\multirow{3}{*}{RFNN-MS-RFF} 
&	$N$	    & 5.395 (5.329--5.549) & 5.372 (5.297--5.547) & 2.970 (2.937--3.003) \\
&	$N/2$	& 5.583 (5.458--5.913) & 5.576 (5.439--5.957) & 1.302 (1.293--1.312) \\
&	$N/4$	& 5.623 (5.462--5.901) & 5.620 (5.442--5.934) & 0.897 (0.894--0.902) \\
\midrule

% DM-RFNNs variants  
\multirow{3}{*}{RFNN-Sig} 
&	$N$	    & 5.354 (5.332--5.386) & 5.326 (5.306--5.366) & 2.623 (2.599--2.652) \\
&	$N/2$	& 5.360 (5.317--5.410) & 5.337 (5.294--5.386) & 1.308 (1.297--1.322) \\
&	$N/4$	& 5.481 (5.396--5.614) & 5.479 (5.379--5.632) & 0.869 (0.864--0.875) \\
\midrule

% Deterministic methods
DDM  &	-	& 5.603 & 5.625 & 16.316   \\
$k$-NN &	-	& 6.041 & 6.214 & 1722.446 \\
\bottomrule
\end{tabular}
%}
\end{table}

\begin{table}[htbp]
\centering
\caption{Testing set performance for the Swiss Roll dataset ($M=3$). Detailed reconstruction metrics for $N=1000$ and $N=2000$ training points, using $d=2$-dim. DM embeddings. Decoders are compared based on the relative $L_2$ and $L_\infty$ mean reconstruction errors, $e_{2,i}$ and $e_{\infty,i}$ (Eq.~\eqref{eq:recon_errors}), and computational time (in seconds). For stochastic decoders (RFNN-RFF, RFNN-MS-RFF and RFNN-Sig), metrics show the median with 5-95\% percentiles in parentheses, computed over 100 random initializations. Deterministic decoders (DDM, $k$-NN) report single values.}
\label{tab:SR_dec_ts_all}
%\resizebox{\textwidth}{!}{
\begin{tabular}{@{}llccc@{}}
\multirow{2}{*}{Decoder} & \multirow{2}{*}{$P$} & \multicolumn{2}{c}{Mean reconstruction error ($\times 10^{-2}$)} & \multirow{2}{*}{Comp. Time (s)} \\
\cmidrule(lr){3-4}
& & relative $L_2$, $e_{2,i}$ &  relative $L_\infty$, $e_{\infty,i}$ &  \\
\midrule
\multicolumn{5}{l}{\textbf{$N=1000$}} \\
\midrule
\midrule

% DM-RFNNf variants
\multirow{3}{*}{RFNN-RFF} 
&	$N$	    & 7.645 (7.413--7.929) & 7.801 (7.504--8.123) & 1.631 (1.581--1.704) \\
&	$N/2$   & 7.719 (7.410--8.154) & 7.806 (7.461--8.244) & 1.593 (1.534--1.644) \\
&	$N/4$   & 7.927 (7.453--8.336) & 8.072 (7.513--8.527) & 1.546 (1.510--1.592) \\
\midrule

% DM-RFNNfMK variants  
\multirow{3}{*}{RFNN-MS-RFF} 
&	$N$	    & 7.715 (7.343--8.038) & 7.798 (7.423--8.126) & 1.646 (1.583--1.692) \\
&	$N/2$	& 7.836 (7.443--8.412) & 7.955 (7.552--8.612) & 1.575 (1.528--1.625) \\
&	$N/4$	& 8.367 (7.731--8.953) & 8.531 (7.856--9.149) & 1.544 (1.479--1.599) \\
\midrule

% DM-RFNNs variants  
\multirow{3}{*}{RFNN-Sig} 
& $N$       & 6.874 (6.740--7.033) & 6.944 (6.792--7.123) & 1.626 (1.562--1.687) \\
& $N/2$     & 7.065 (6.812--7.328) & 7.136 (6.858--7.434) & 1.572 (1.532--1.625) \\
& $N/4$     & 7.683 (7.165--8.267) & 7.830 (7.257--8.468) & 1.539 (1.493--1.581) \\
\midrule
% Deterministic methods
DDM  & -  & 7.203 & 7.242 & 1.710 \\
$k$-NN & -  & 8.169 & 8.535 & 412.006 \\

\midrule
\multicolumn{5}{l}{\textbf{$N=2000$}} \\
\midrule
\midrule

% DM-RFNNf variants
\multirow{3}{*}{RFNN-RFF} 
&	$N$	    &  5.966 (5.887--6.071) & 6.022 (5.923--6.148) & 1.997 (1.943--2.049) \\
&	$N/2$	&  6.000 (5.902--6.155) & 6.068 (5.949--6.242) & 1.903 (1.858--1.950) \\
&	$N/4$	&  6.089 (5.958--6.221) & 6.158 (6.014--6.306) & 1.839 (1.798--1.889) \\
\midrule

% DM-RFNNfMK variants  
\multirow{3}{*}{RFNN-MS-RFF} 
&	$N$	    &  6.041 (5.920--6.158) & 6.108 (5.971--6.239) & 2.017 (1.975--2.071) \\
&	$N/2$	&  6.152 (5.988--6.481) & 6.229 (6.034--6.631) & 1.900 (1.851--1.942) \\
&	$N/4$	&  6.283 (6.067--6.646) & 6.369 (6.127--6.799) & 1.850 (1.780--1.901) \\
\midrule

% DM-RFNNs variants  
\multirow{3}{*}{RFNN-Sig} 
&	$N$	    & 5.810 (5.769--5.858) & 5.868 (5.818--5.925) & 1.996 (1.943--2.046) \\
&	$N/2$	& 5.834 (5.774--5.918) & 5.893 (5.820--5.991) & 1.913 (1.866--1.961) \\
&	$N/4$	& 5.967 (5.835--6.168) & 6.050 (5.878--6.281) & 1.847 (1.792--1.888) \\
\midrule

% Deterministic methods
DDM  &	-	 & 6.062 & 6.148 & 1.911 \\
$k$-NN &	-	 & 7.279 & 7.584 & 1218.913 \\
\bottomrule
\end{tabular}
%}
\end{table}

Across both sparse and dense data configurations, RFNN decoders are consistently faster in both training and inference than the DDM and especially the $k$-NN decoders, while generally matching or exceeding their accuracy. In training, RFNN with few features ($P=N/4$) are less accurate, but as $P$ increases their accuracy improves, eventually surpassing that of the DDM decoder. RFNN-MS-RFF and RFNN-Sig exhibit a clear monotonic improvement with $P$, whereas RFNN‑RFF shows a less regular trend. In testing, all RFNN variants display a clear accuracy gain as $P$ grows, and the similar behavior of $L_2$ and $L_\infty$ errors indicates uniformly good pointwise reconstruction. In the sparse case ($N=1000$), RFNN‑Sig is the most accurate RFNN variant; RFNN‑RFF and RFNN-MS-RFF are slightly less accurate than DDM. In the dense case ($N=2000$), RFNN‑RFF surpasses the accuracy of DDM, though the inference times of all RFNN variants increase and become comparable to those of DDM. Increasing $N$ from $1000$ to $2000$ reduces the reconstruction error of all kernel‑based decoders (RFNN and DDM) to about 2\%, whereas $k$-NN improves less markedly, widening the performance gap in favor of the kernel methods. Visualizations of the Swiss Roll reconstructions are provided for the testing sets in Figs.~\ref{fig:SwissRoll_1K} and~\ref{fig:SwissRoll_2K}. The reconstructed datasets confirm that RFNN decoders preserve the underlying geometry with high fidelity, matching the quantitative trends reported above.

\begin{figure}[!htbp]
   \centering
  \begin{subfigure}[b]{0.32\textwidth}
    \centering
    \includegraphics[width=\textwidth]{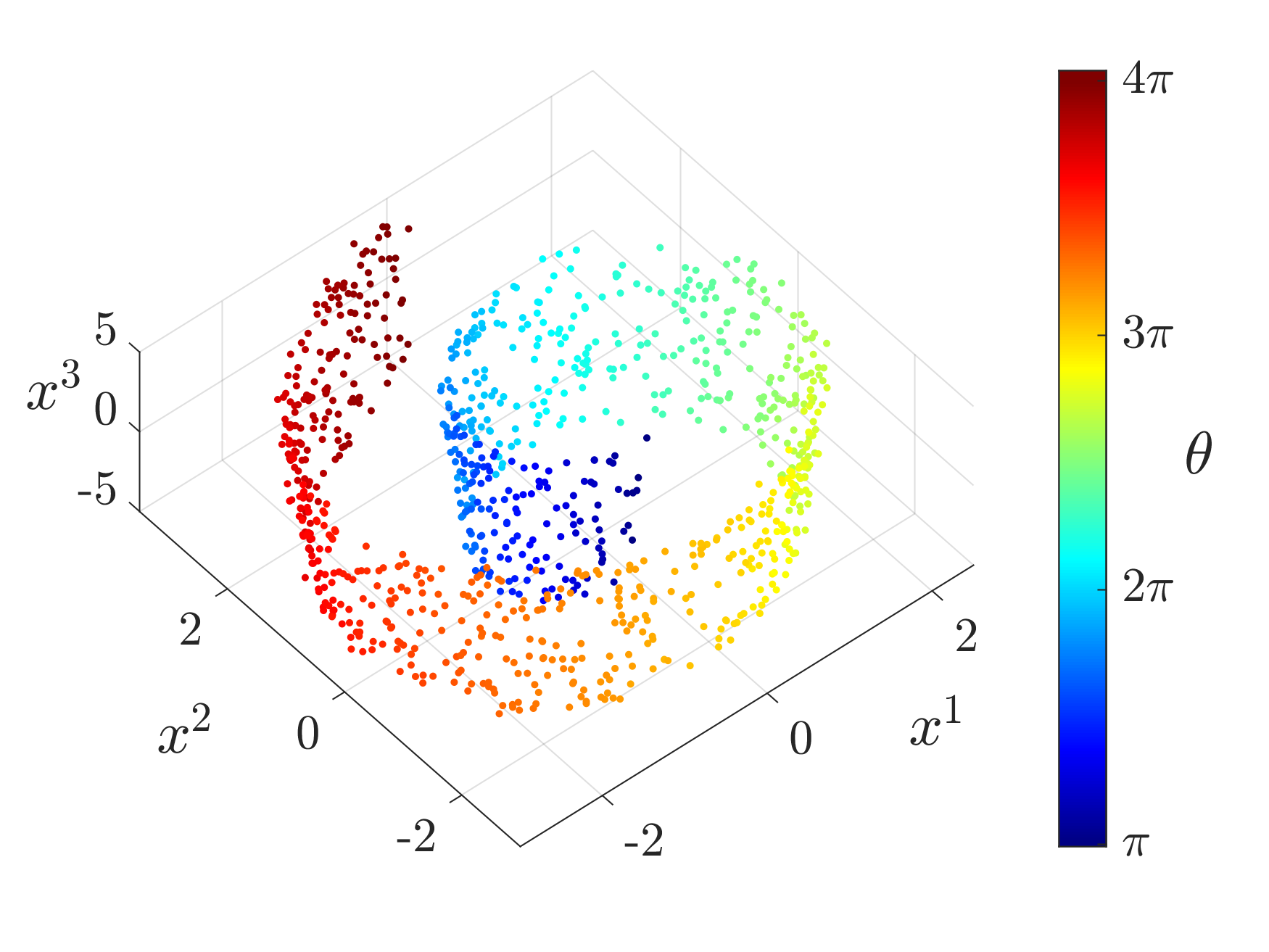}
    \caption{Data set}
    \label{fig:SR_data_1K}
  \end{subfigure}\hspace{0.02\textwidth}%
  \begin{subfigure}[b]{0.32\textwidth}
    \centering
    \includegraphics[width=\textwidth]{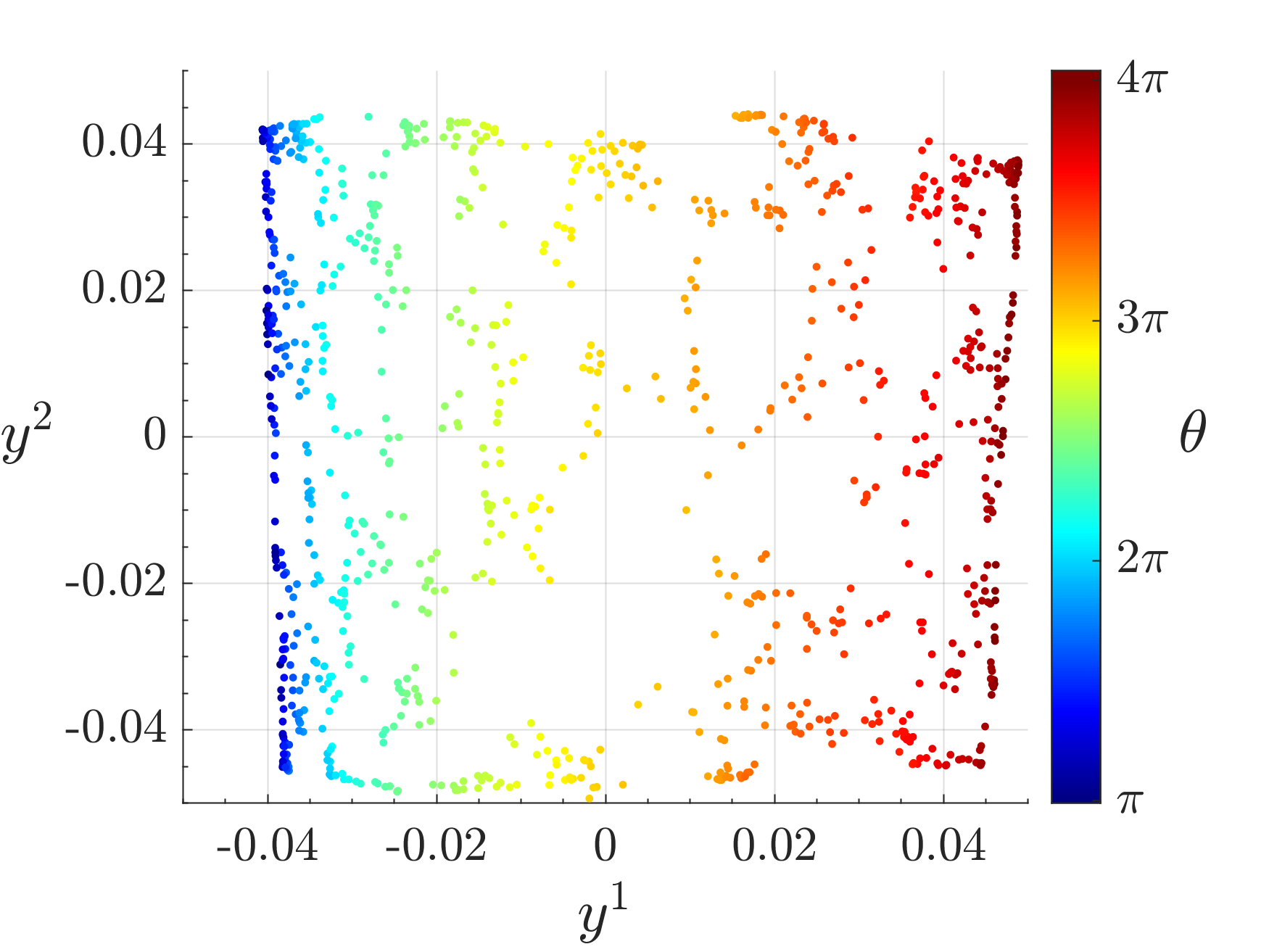}
    \caption{DM coordinates}
    \label{fig:SR_DMcoords_1K}
  \end{subfigure}
  %\caption{(a) Data set and (b) DM coordinates.}

    % Second row
    \begin{subfigure}[b]{0.32\textwidth}
        \centering
        \includegraphics[width=\textwidth]{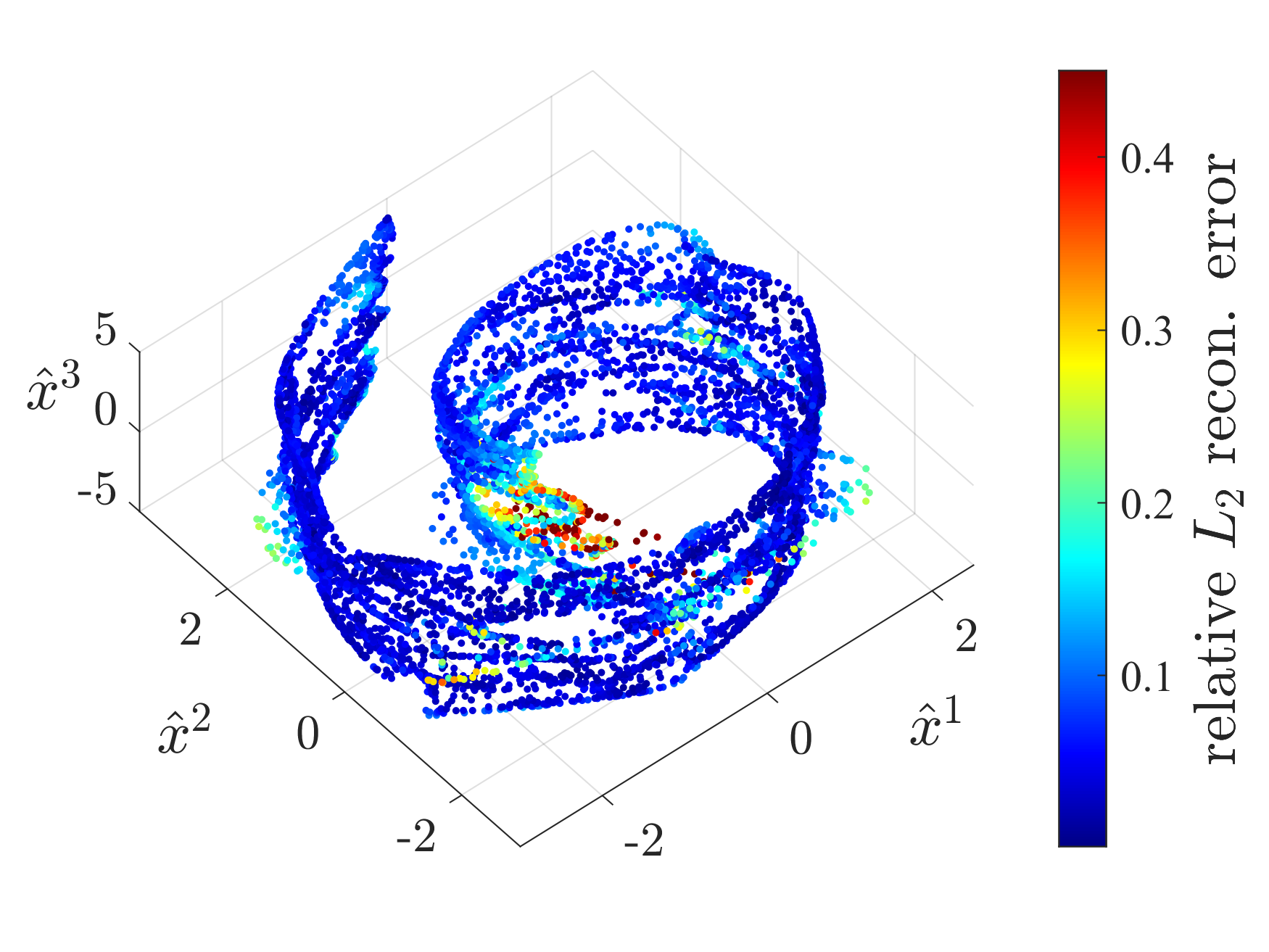}
        \caption{RFNN-RFF ($P=N/2$)}
        \label{fig:SR_DM_RFNNf_recon_1K}
    \end{subfigure}
    \hfill
    \begin{subfigure}[b]{0.32\textwidth}
        \centering
        \includegraphics[width=\textwidth]{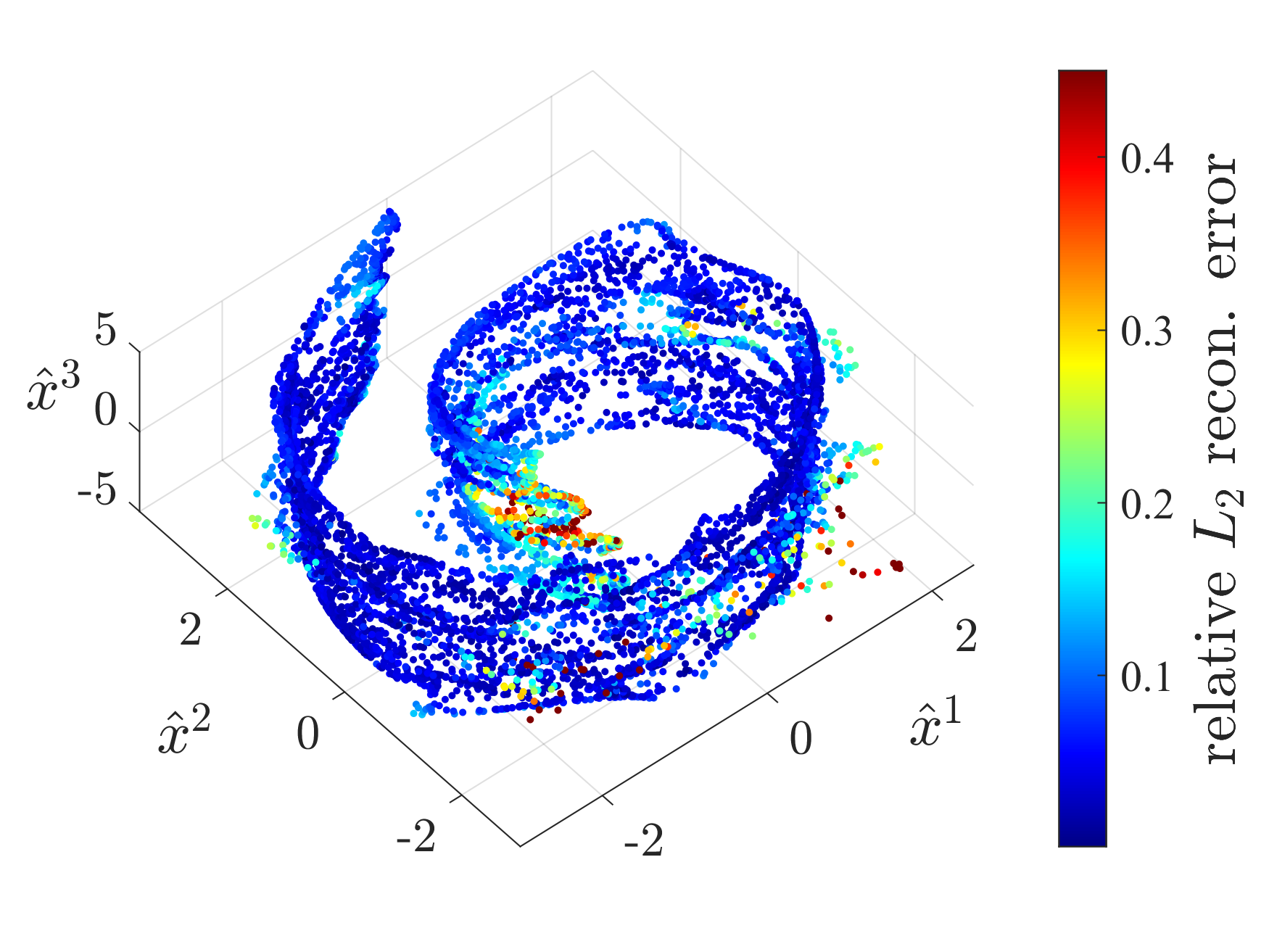}
        \caption{RFNN-MS-RFF ($P=N$)}
        \label{fig:SR_DM_RFNNfMK_recon_1K}
    \end{subfigure}
    \hfill
    \begin{subfigure}[b]{0.32\textwidth}
        \centering
        \includegraphics[width=\textwidth]{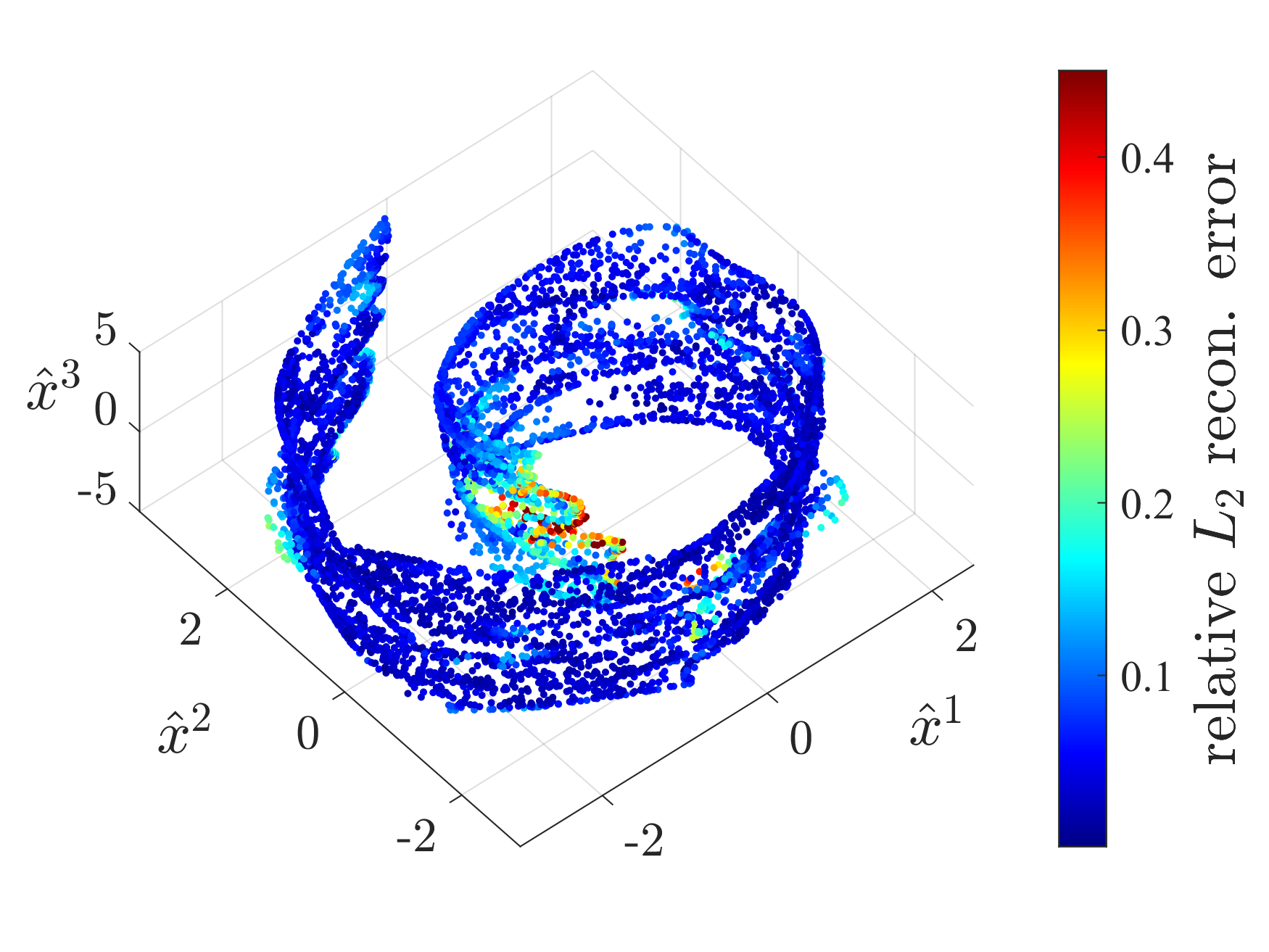}
        \caption{RFNN-Sig ($P=N$)}
        \label{fig:SR_DM_RFNNs_recon_1K}
    \end{subfigure}
    
    % Third row
    \begin{subfigure}[b]{0.32\textwidth}
        \centering
        \includegraphics[width=\textwidth]{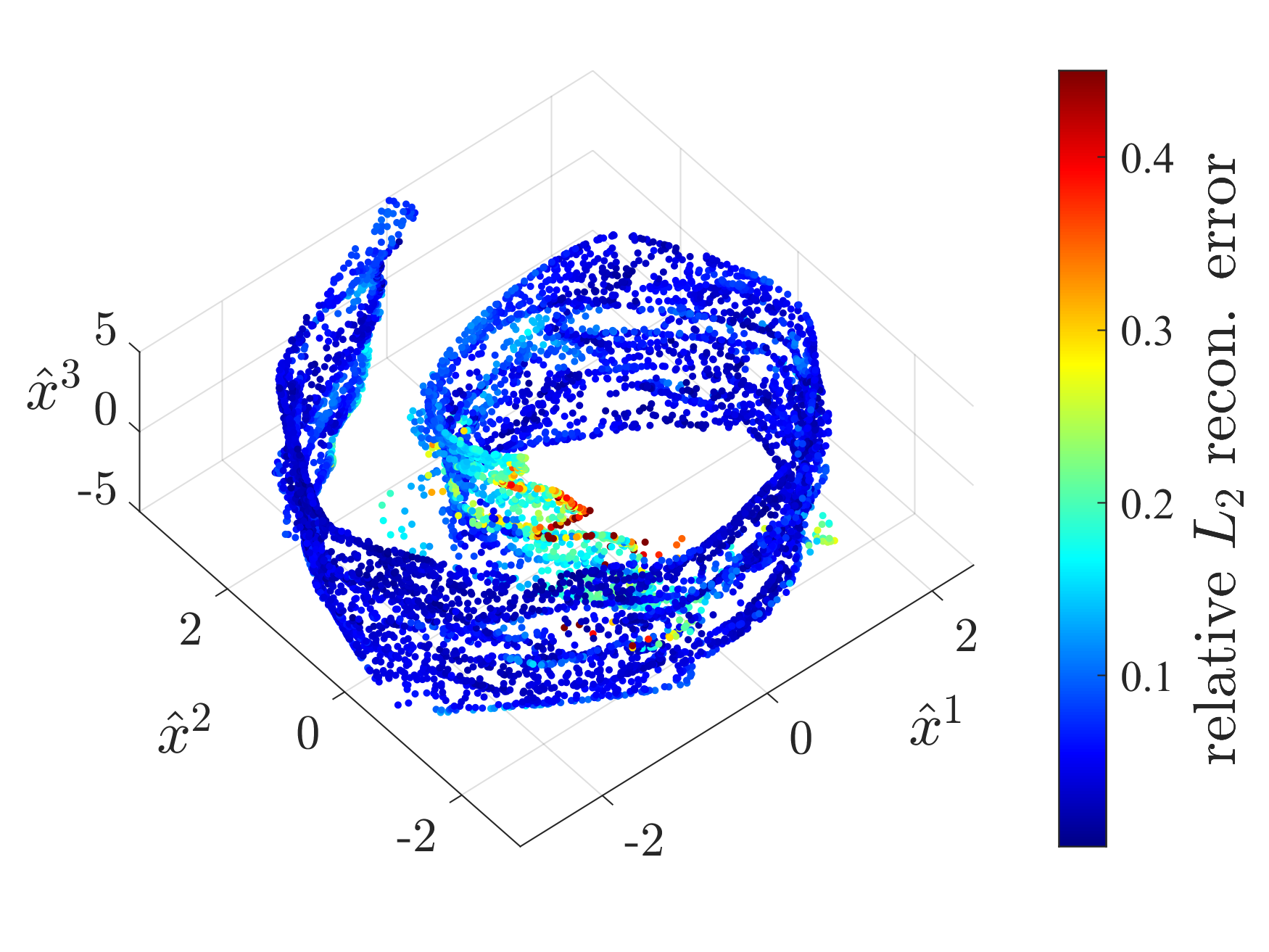}
        \caption{DDM}
        \label{fig:SR_DM_GH_recon_1K}
    \end{subfigure}\hspace{0.02\textwidth}
    \begin{subfigure}[b]{0.32\textwidth}
        \centering
        \includegraphics[width=\textwidth]{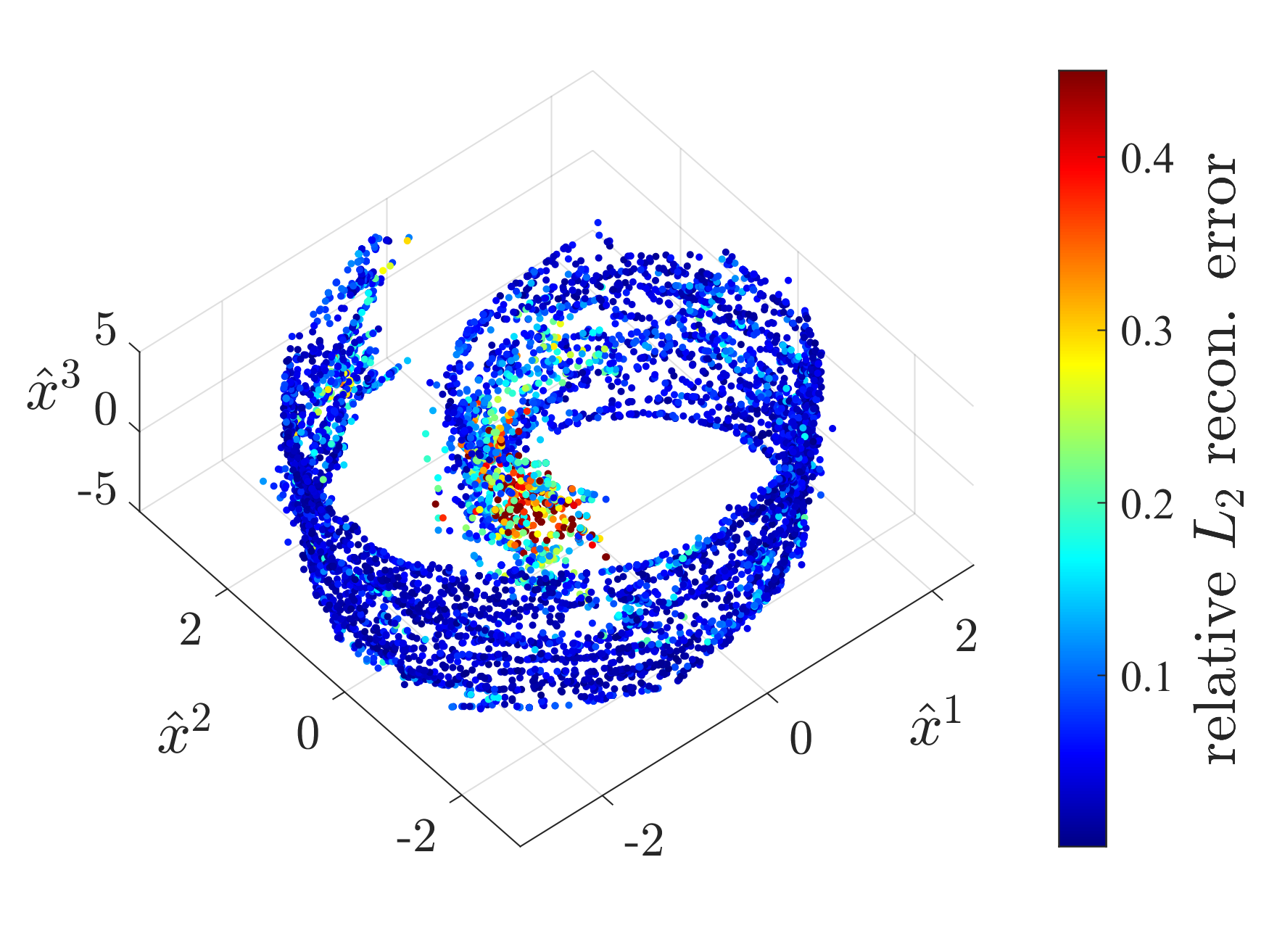}
        \caption{$k$-NN}
        \label{fig:SR_DM_kNN_recon_1K}
    \end{subfigure}
    \caption{Swiss Roll dataset ($M=3$) with $N=1000$ training points. Panel~\ref{fig:SR_data_1K} shows the training data, while panel~\ref{fig:SR_DMcoords_1K} shows the $d=2$ DM coordinates, colored by the intrinsic angular coordinate $\theta$. Panels~\ref{fig:SR_DM_RFNNf_recon_1K}-\ref{fig:SR_DM_kNN_recon_1K} show the reconstructed data $\widehat{x}=[\widehat x^1,\widehat x^2, \widehat x^3]^\top$ and the per-point relative $L_2$ errors $e_{2,i}$ (Eq.~\eqref{eq:recon_errors}) for the five decoders: RFNN-RFF with $P=N/2$, RFNN-MS-RFF with $P=N$, RFNN-Sig with $P=N$, DDM and $k$-NN. For the RFNN decoders, the displayed configuration is the one achieving the lowest training reconstruction error among all tested variants (see Fig.~\ref{fig:SR_dec_tr1K} for the full comparison).}
    \label{fig:SwissRoll_1K}
\end{figure}

\begin{figure}[!htbp]
   \centering
  \begin{subfigure}[b]{0.32\textwidth}
    \centering
    \includegraphics[width=\textwidth]{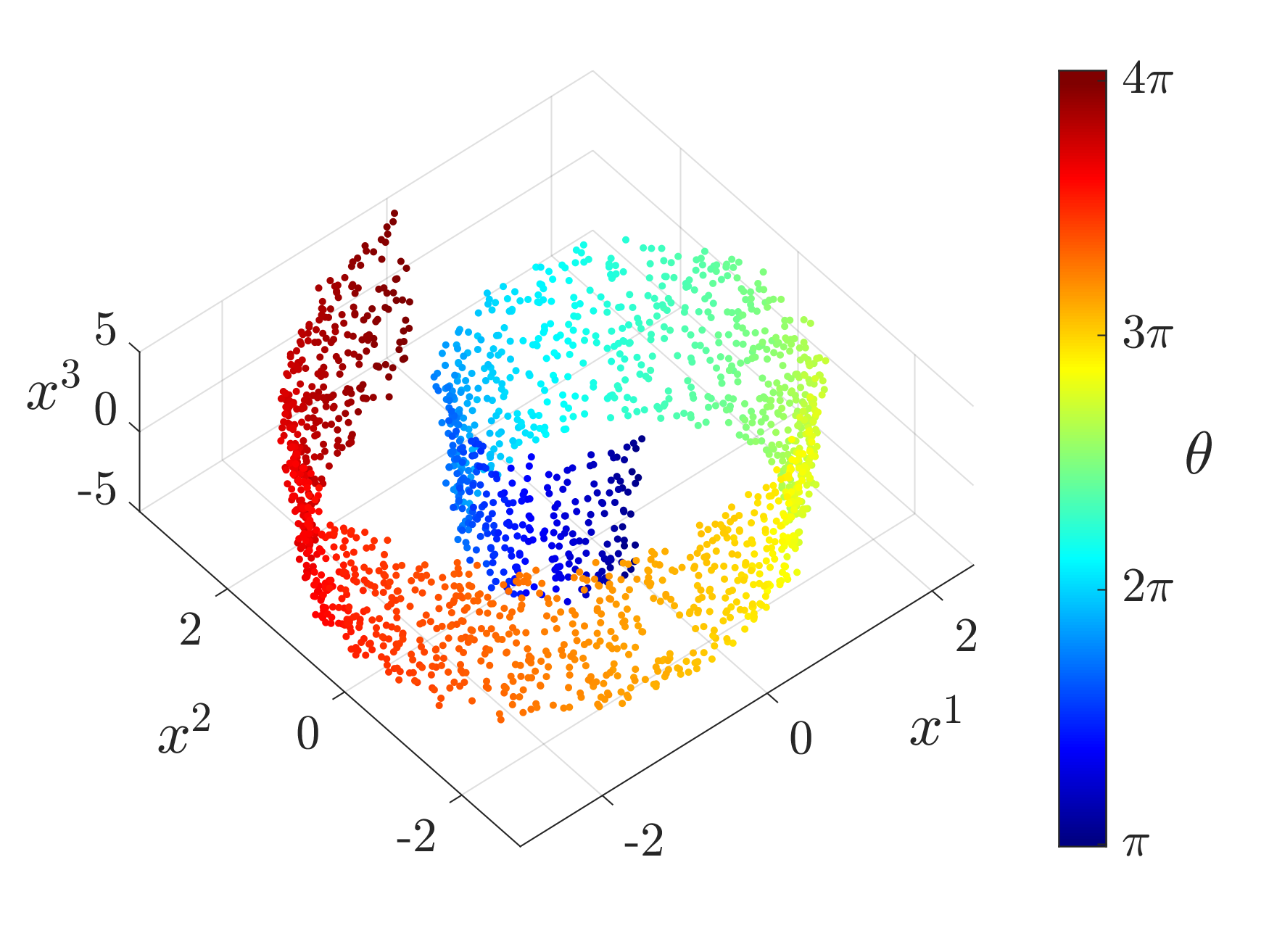}
    \caption{Data set}
    \label{fig:SR_data_2K}
  \end{subfigure}\hspace{0.02\textwidth}%
  \begin{subfigure}[b]{0.32\textwidth}
    \centering
    \includegraphics[width=\textwidth]{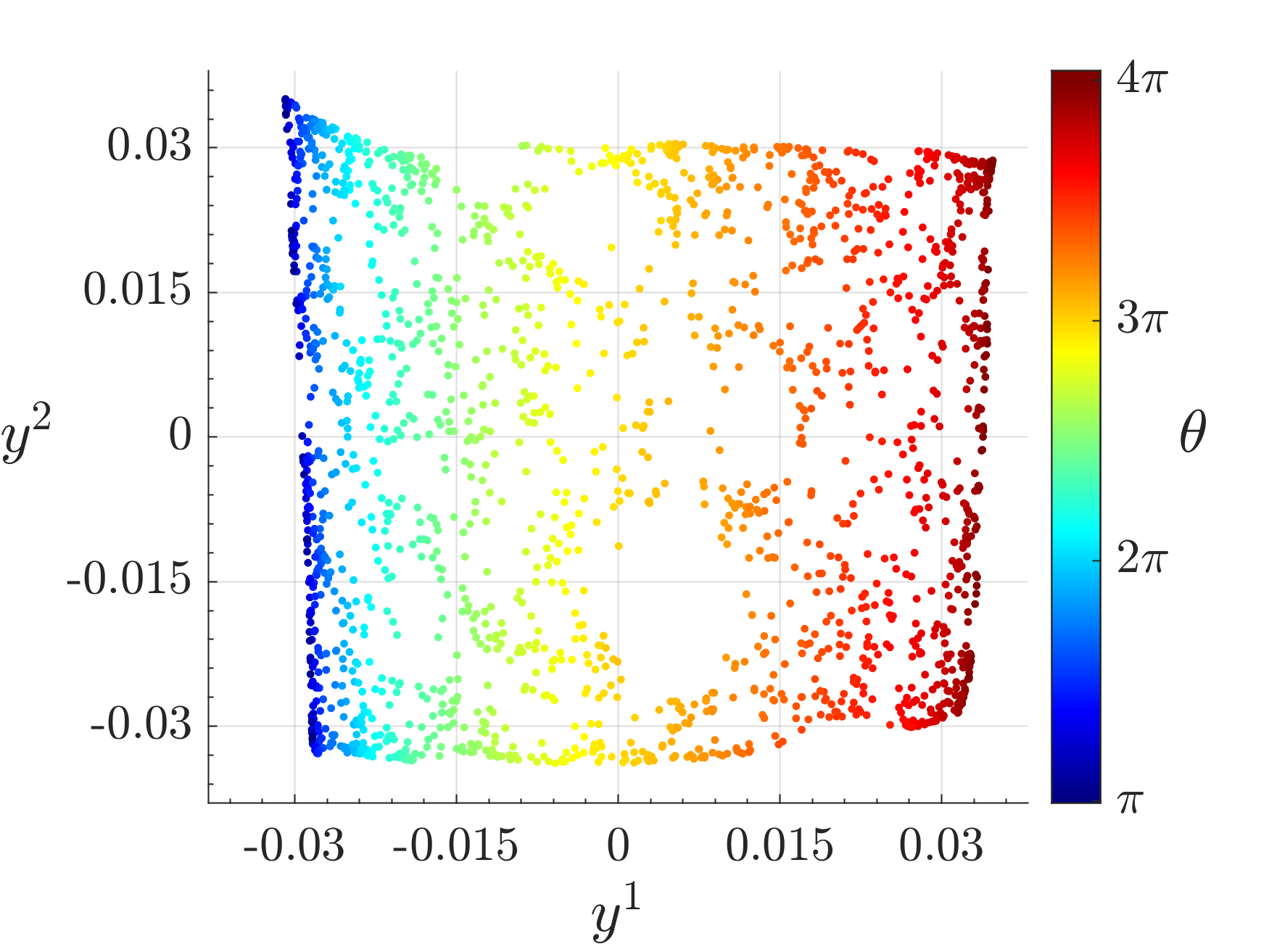}
    \caption{DM coordinates}
    \label{fig:SR_DMcoords_2K}
  \end{subfigure}
  %\caption{(a) Data set and (b) DM coordinates.}

    % Second row
    \begin{subfigure}[b]{0.32\textwidth}
        \centering
        \includegraphics[width=\textwidth]{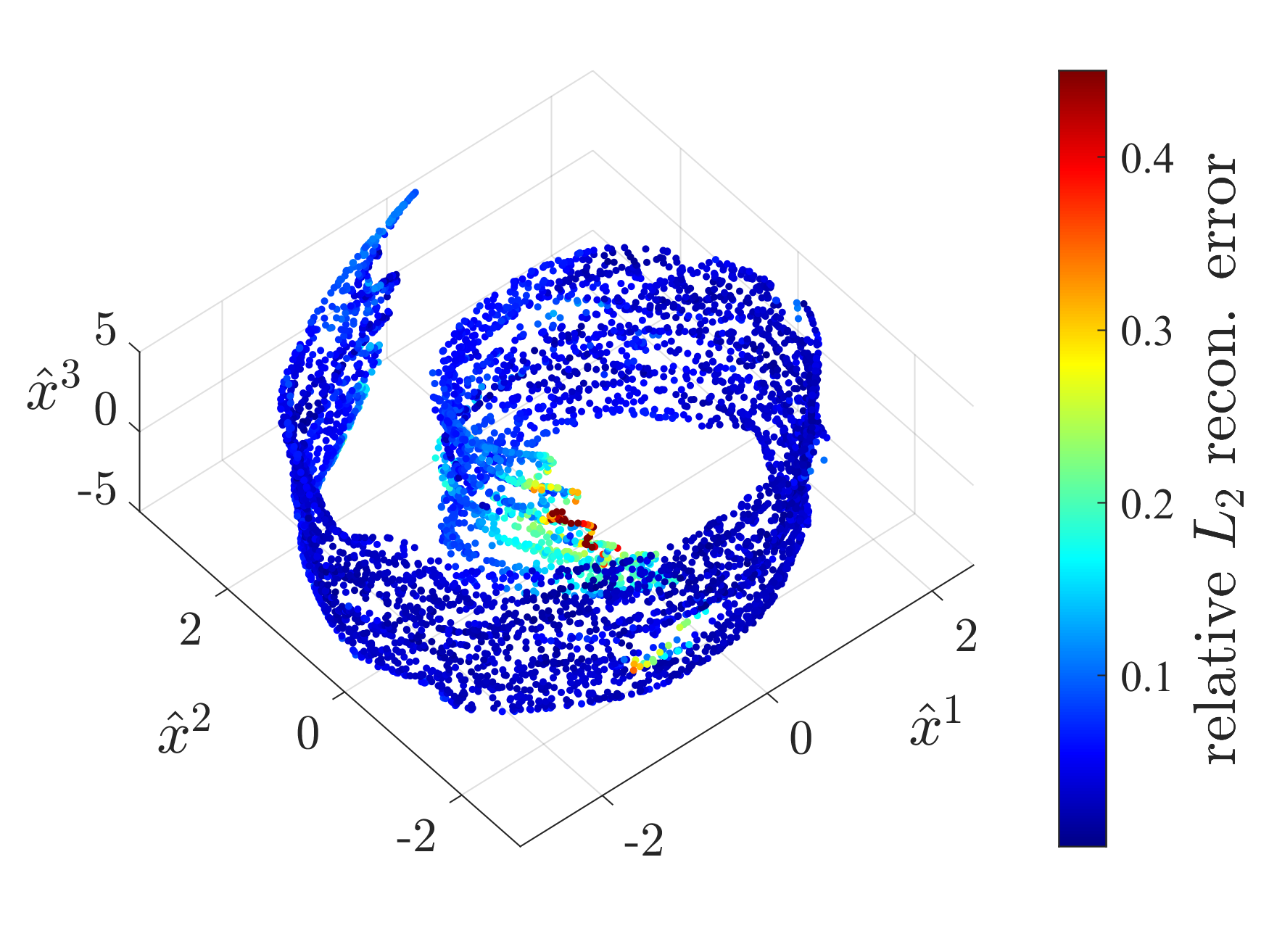}
        \caption{RFNN-RFF ($P=N/2$)}
        \label{fig:SR_DM_RFNNf_recon_2K}
    \end{subfigure}
    \hfill
    \begin{subfigure}[b]{0.32\textwidth}
        \centering
        \includegraphics[width=\textwidth]{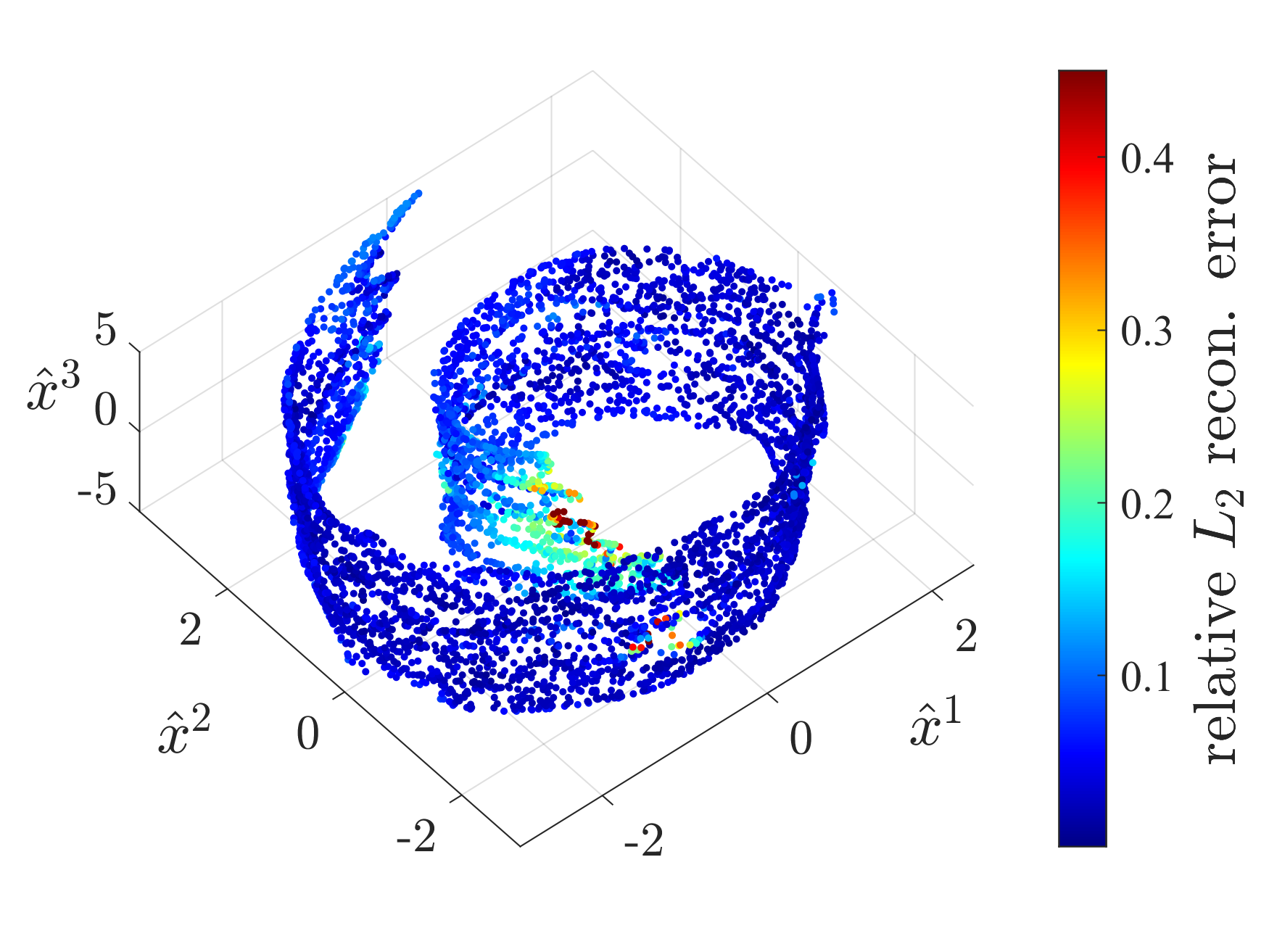}
        \caption{RFNN-MS-RFF ($P=N$)}
        \label{fig:SR_DM_RFNNfMK_recon_2K}
    \end{subfigure}
    \hfill
    \begin{subfigure}[b]{0.32\textwidth}
        \centering
        \includegraphics[width=\textwidth]{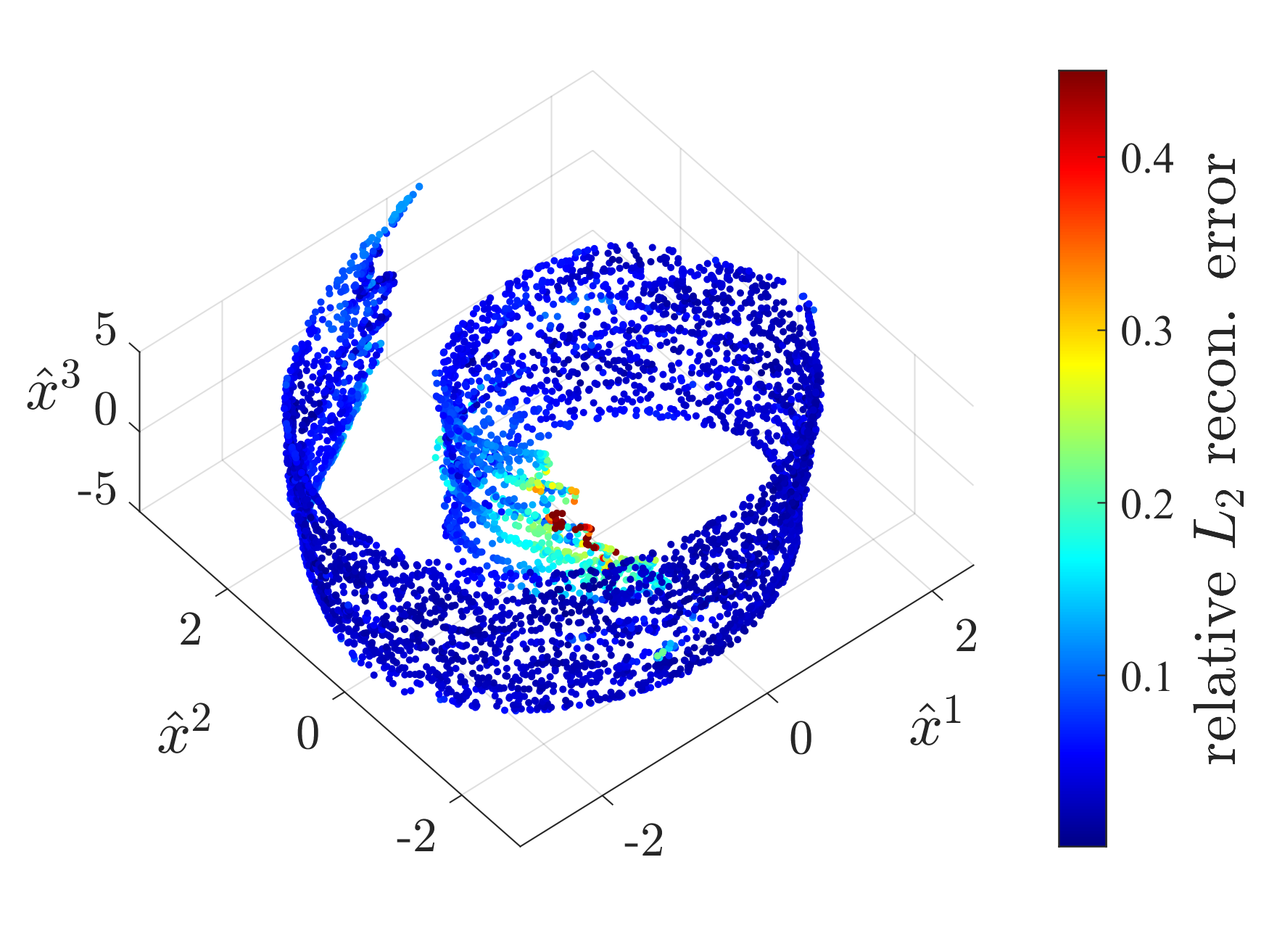}
        \caption{RFNN-Sig ($P=N$)}
        \label{fig:SR_DM_RFNNs_recon_2K}
    \end{subfigure}
    
    % Third row
    \begin{subfigure}[b]{0.32\textwidth}
        \centering
        \includegraphics[width=\textwidth]{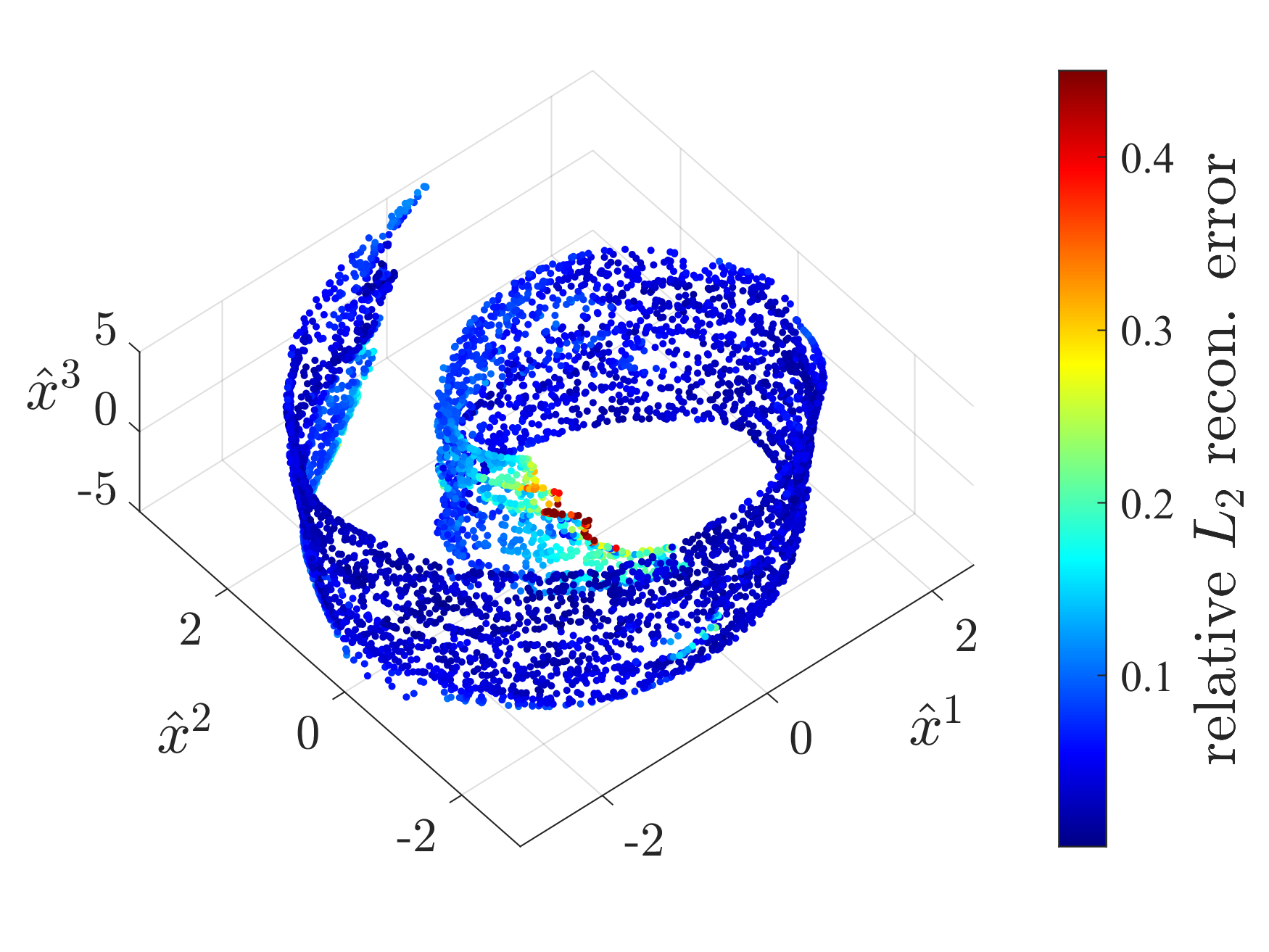}
        \caption{DDM}
        \label{fig:SR_DM_GH_recon_2K}
    \end{subfigure}\hspace{0.02\textwidth}
    \begin{subfigure}[b]{0.32\textwidth}
        \centering
        \includegraphics[width=\textwidth]{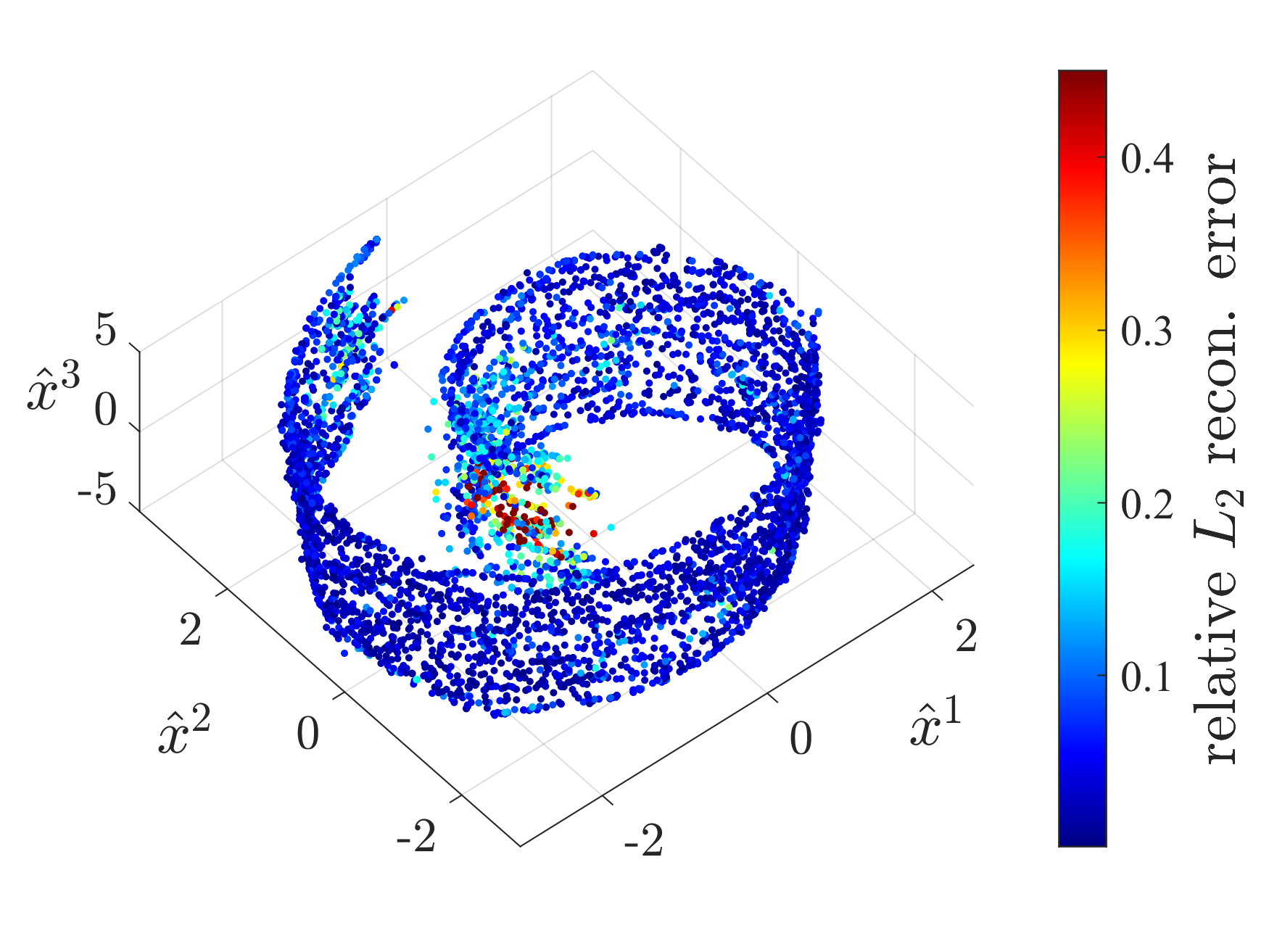}
        \caption{$k$-NN}
        \label{fig:SR_DM_kNN_recon_2K}
    \end{subfigure}
    \caption{Swiss Roll dataset ($M=3$) with $N=2000$ training points. Panel~\ref{fig:SR_data_2K} shows the training data, while panel~\ref{fig:SR_DMcoords_2K} shows the $d=2$ DM coordinates, colored by the intrinsic angular coordinate $\theta$. Panels~\ref{fig:SR_DM_RFNNf_recon_2K}-\ref{fig:SR_DM_kNN_recon_2K} show the reconstructed data $\widehat{x}=[\widehat x^1,\widehat x^2, \widehat x^3]^\top$ and the per-point relative $L_2$ errors $e_{2,i}$ (Eq.~\eqref{eq:recon_errors}) for the five decoders: RFNN-RFF with $P=N/2$, RFNN-MS-RFF with $P=N$, RFNN-Sig with $P=N$, DDM and $k$-NN. For the RFNN decoders, the displayed configuration is the one achieving the lowest training reconstruction error among all tested variants (see Fig.~\ref{fig:SR_dec_tr2K} for the full comparison).}
    \label{fig:SwissRoll_2K}
\end{figure}

\clearpage
\newpage
\renewcommand{\theequation}{F.\arabic{equation}}
\renewcommand{\thefigure}{F.\arabic{figure}}
\renewcommand{\thetable}{F.\arabic{table}}
\setcounter{equation}{0}
\setcounter{figure}{0}
\setcounter{table}{0}
\section{The S-curve dataset (\texorpdfstring{$M=20$}{M=20})}
\label{app:SC20}

Here, we consider our second non-mass-preserving benchmark example to evaluate the performance of the proposed decoders. We consider the well-known S-curve benchmark, for which the data lie on an intrinsic $d=2$-dimensional manifold $\mathcal{M}$ embedded in the $3$-dimensional ambient space, forming a curve of shape $S$. To test the generalizability in higher ambient dimensions, we project the $3$-dimensional data into an $M=20$-dimensional space. To generate the synthetic dataset, we uniformly sample $N_{obs}$ angles $\theta_i\sim \mathcal{U}[-3\pi/2, 3\pi/2)$ and widths $w_i\sim \mathcal{U}[0, 1)$. These are mapped to the $3$-dimensional space via the S-curve equations
\begin{equation}
z^1_i = \sin(\theta_i) + \mathcal{N}(0,\sigma^2), \quad z^2_i = w_i + \mathcal{N}(0,\sigma^2), \quad z^3_i = \text{sign}(\theta_i) (\cos(\theta_i) - 1) + \mathcal{N}(0,\sigma^2). \label{eq:SCurveeq}
\end{equation}
with additive Gaussian noise $\mathcal{N}(0,\sigma^2)$ of magnitude $\sigma=0.01$; $\text{sign}(\theta_i)$ denotes the sign of $\theta_i$, producing a fold in the S-curve. Each point $z_i=[z^1_i,z^2_i,z^3_i]^\top\in \mathbb{R}^3$ is then linearly projected into $x_i=R\, z_i\in \mathbb{R}^{20}$, where $R\in\mathbb{R}^{20\times 3}$ is a random orthogonal matrix constructed via the SVD of a Gaussian random matrix. From the resulting full dataset, we construct a training/validation/testing split configuration with $N=1000$ points in $\mathcal{X}_{tr}$ and $\mathcal{X}_{vl}$, with the remaining points reserved for the testing set $\mathcal{X}_{ts}$. A visualization of the training set is provided in Fig.~\ref{fig:SC20_data}. 

Using the training dataset, we then employ the DM algorithm with $\alpha=1$ and $w_1=0.2$ to obtain the intrinsic DM coordinates $y_i\in \mathbb{R}^2$, which are depicted in Fig.~\ref{fig:SC20_DMcoords} colored by the angle $\theta_i$. The extended by Nystr\"{o}m extension DM encoder yields the training, validation, and testing sets $\mathcal{Y}_{tr}$, $\mathcal{Y}_{vl}$, $\mathcal{Y}_{ts}$ for decoder training, tuning and evaluation.

Next, we train the RFNN decoder with $P=N,N/2,N/4$ (again without downsampling, using $n=N$) and the DDM and $k$-NN decoders as described in Appendix~\ref{app:ImpDet}. To tune the hyperparameters of the RFNN-RFF, RFNN-MS-RFF and RFNN-Sig decoders, we considered the ranges $\sigma_w\in[0.1,1]$, $\sigma_{UB}\in[2,10]$ and $c \in [5,18]$; for $P=N$ the optimal values were $\sigma_w=0.2$, $\sigma_{UB}=5$ and $c=9$. For tuning the DDM and $k$-NN decoders we consider the ranges $w_2\in[0.3,0.9]$ and $k\in[2,11]$; optimal values were $w_2=0.4$ and $k=7$.

Training and testing performance (mean relative $L_2$ and $L_\infty$ reconstruction errors and computational times) are detailed in Tables~\ref{tab:SC20_dec_tr_all} and \ref{tab:SC20_dec_ts_all}; a visual summary of the reconstruction error vs. computational time trade-off is shown in Fig.~\ref{fig:SCurve20_Dec}.
\begin{table}[htbp]
\centering
\caption{Training set performance for the S-Curve dataset ($M=20$). Detailed reconstruction metrics for $N=1000$ training points, using $d=2$-dim. DM embeddings. Decoders are compared based on the relative $L_2$ and $L_\infty$ mean reconstruction errors, $e_{2,i}$ and $e_{\infty,i}$ (Eq.~\eqref{eq:recon_errors}), and computational time (in seconds). For stochastic decoders (RFNN-RFF, RFNN-MS-RFF and RFNN-Sig), metrics show the median with 5-95\% percentiles in parentheses, computed over 100 random initializations. Deterministic decoders (DDM, $k$-NN) report single values.}
\label{tab:SC20_dec_tr_all}
%\resizebox{\textwidth}{!}{%
\begin{tabular}{@{}llccc@{}}
\multirow{2}{*}{Decoder} & \multirow{2}{*}{$P$} & \multicolumn{2}{c}{Mean reconstruction error ($\times 10^{-2}$)} & \multirow{2}{*}{Comp. Time (s)} \\
\cmidrule(lr){3-4}
& & relative $L_2$, $e_{2,i}$ &  relative $L_\infty$, $e_{\infty,i}$ &  \\
\midrule
\midrule

% DM-RFNNf variants
\multirow{3}{*}{RFNN-RFF} 
&	$N$	    & 5.890 (5.794--6.019) & 6.395 (6.284--6.524) & 0.721 (0.718--0.734) \\
&	$N/2$	& 6.034 (5.932--6.186) & 6.547 (6.429--6.673) & 0.420 (0.418--0.425) \\
&	$N/4$	& 6.298 (6.087--6.550) & 6.817 (6.591--7.162) & 0.254 (0.252--0.256) \\
\midrule

% DM-RFNNfMK variants  
\multirow{3}{*}{RFNN-MS-RFF} 
&	$N$	    & 5.967 (5.770--6.255) & 6.472 (6.283--6.775) & 0.856 (0.848--0.866) \\
&	$N/2$	& 6.058 (5.873--6.464) & 6.602 (6.390--7.013) & 0.366 (0.362--0.368) \\
&	$N/4$	& 6.654 (6.290--7.777) & 7.227 (6.813--8.335) & 0.266 (0.265--0.268) \\
\midrule

% DM-RFNNs variants  
\multirow{3}{*}{RFNN-Sig} 
&	$N$	    & 5.850 (5.795--5.908) & 6.343 (6.281--6.403) & 0.752 (0.748--0.767) \\
&	$N/2$	& 5.841 (5.743--5.941) & 6.348 (6.247--6.453) & 0.404 (0.403--0.410) \\
&	$N/4$	& 6.343 (6.151--6.528) & 6.910 (6.690--7.131) & 0.262 (0.261--0.264) \\
\midrule

% Deterministic methods
DDM  &	-	&	8.437 & 9.078 & 3.119  \\
$k$-NN &	-	&	5.391 & 5.913 & 401.370  \\

\bottomrule
\end{tabular}%
%}
\end{table}

\begin{table}[htbp]
\centering
\caption{Testing set performance for the S-Curve dataset ($M=20$). Detailed reconstruction metrics for $N=1000$ training points, using $d=2$-dim. DM embeddings. Decoders are compared based on the relative $L_2$ and $L_\infty$ mean reconstruction errors, $e_{2,i}$ and $e_{\infty,i}$ (Eq.~\eqref{eq:recon_errors}), and computational time (in seconds). For stochastic decoders (RFNN-RFF, RFNN-MS-RFF and RFNN-Sig), metrics show the median with 5-95\% percentiles in parentheses, computed over 100 random initializations. Deterministic decoders (DDM, $k$-NN) report single values.}
\label{tab:SC20_dec_ts_all}
%\resizebox{\textwidth}{!}{%
\begin{tabular}{@{}llccc@{}}
\multirow{2}{*}{Decoder} & \multirow{2}{*}{$P$} & \multicolumn{2}{c}{Mean reconstruction error ($\times 10^{-2}$)} & \multirow{2}{*}{Comp. Time (s)} \\
\cmidrule(lr){3-4}
& & relative $L_2$, $e_{2,i}$ &  relative $L_\infty$, $e_{\infty,i}$ &  \\
\midrule
\midrule

% DM-RFNNf variants
\multirow{3}{*}{RFNN-RFF} 
&	$N$	    & 6.733 (6.522--6.914) & 7.253 (7.024--7.521) & 1.635 (1.578--1.677) \\
&	$N/2$	& 6.852 (6.634--7.098) & 7.393 (7.173--7.667) & 1.573 (1.529--1.621) \\
&	$N/4$	& 7.083 (6.813--7.438) & 7.628 (7.342--8.069) & 1.544 (1.501--1.579) \\
\midrule

% DM-RFNNfMK variants  
\multirow{3}{*}{RFNN-MS-RFF} 
&	$N$	    & 6.843 (6.589--7.114) & 7.400 (7.150--7.696) & 1.633 (1.583--1.675) \\
&	$N/2$	& 7.000 (6.729--7.358) & 7.564 (7.233--7.972) & 1.567 (1.513--1.626) \\
&	$N/4$	& 7.445 (6.916--8.575) & 8.038 (7.458--9.173) & 1.554 (1.498--1.615) \\
\midrule

% DM-RFNNs variants  
\multirow{3}{*}{RFNN-Sig} 
&	$N$	    & 6.475 (6.350--6.570) & 6.951 (6.834--7.060) & 1.655 (1.599--1.701) \\
&	$N/2$	& 6.742 (6.522--6.968) & 7.281 (7.051--7.525) & 1.584 (1.530--1.643) \\
&	$N/4$	& 7.141 (6.873--7.450) & 7.750 (7.472--8.043) & 1.555 (1.507--1.592) \\
\midrule

% Deterministic methods
DDM  &	-	& 8.631 & 9.288 & 1.632 \\
$k$-NN &	-	& 6.971 & 7.565 & 475.168 \\
\bottomrule
\end{tabular}%
%}
\end{table}

\begin{figure}[!htbp]
    \centering
    \begin{subfigure}[b]{0.32\textwidth}
        \centering
        \includegraphics[width=\textwidth]{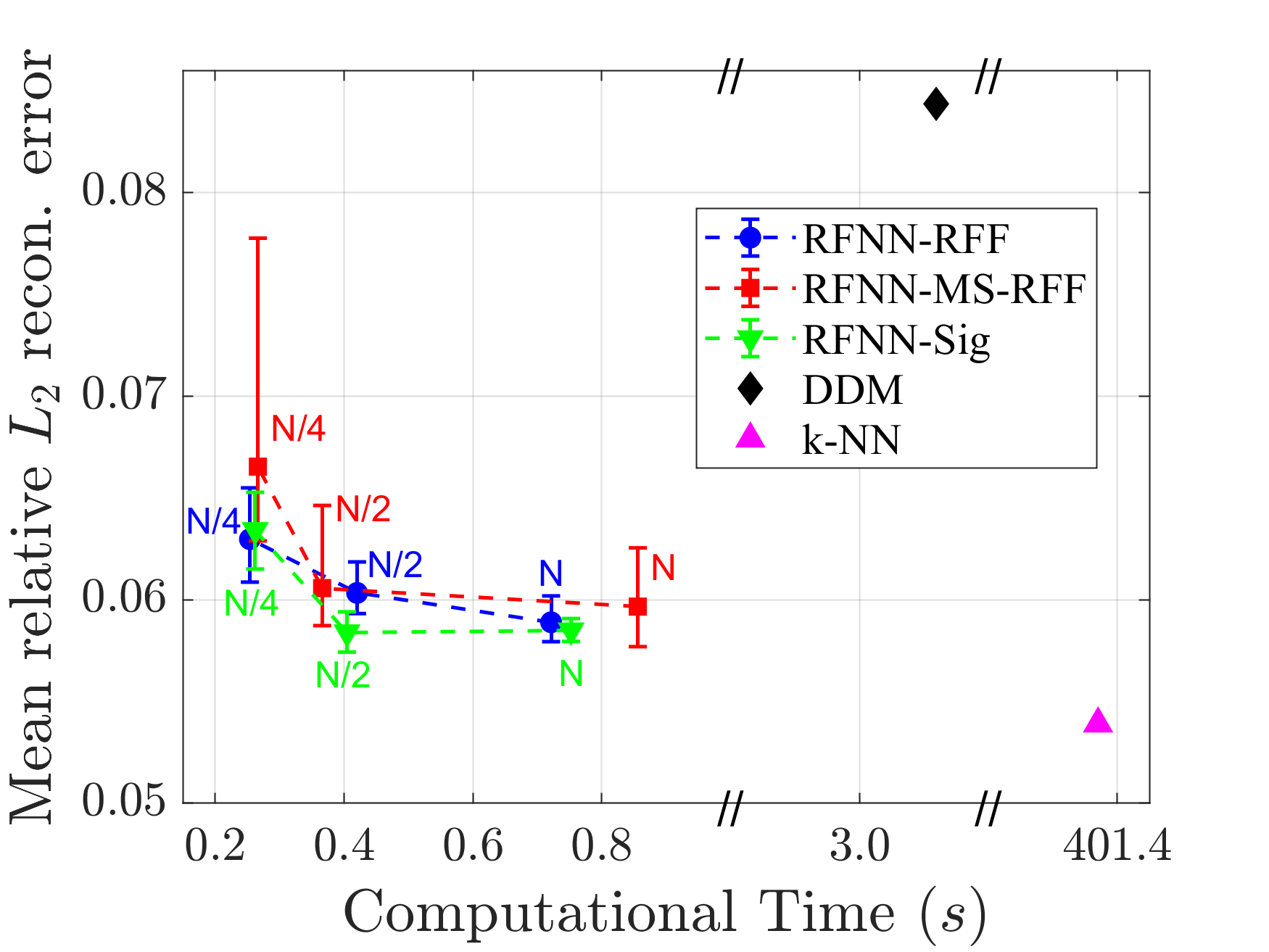}
        \caption{Training, $L_2$ error}
        \label{fig:SC20_dec_tr}
    \end{subfigure}
    \hspace{2pt}
    \begin{subfigure}[b]{0.32\textwidth}
        \centering
        \includegraphics[width=\textwidth]{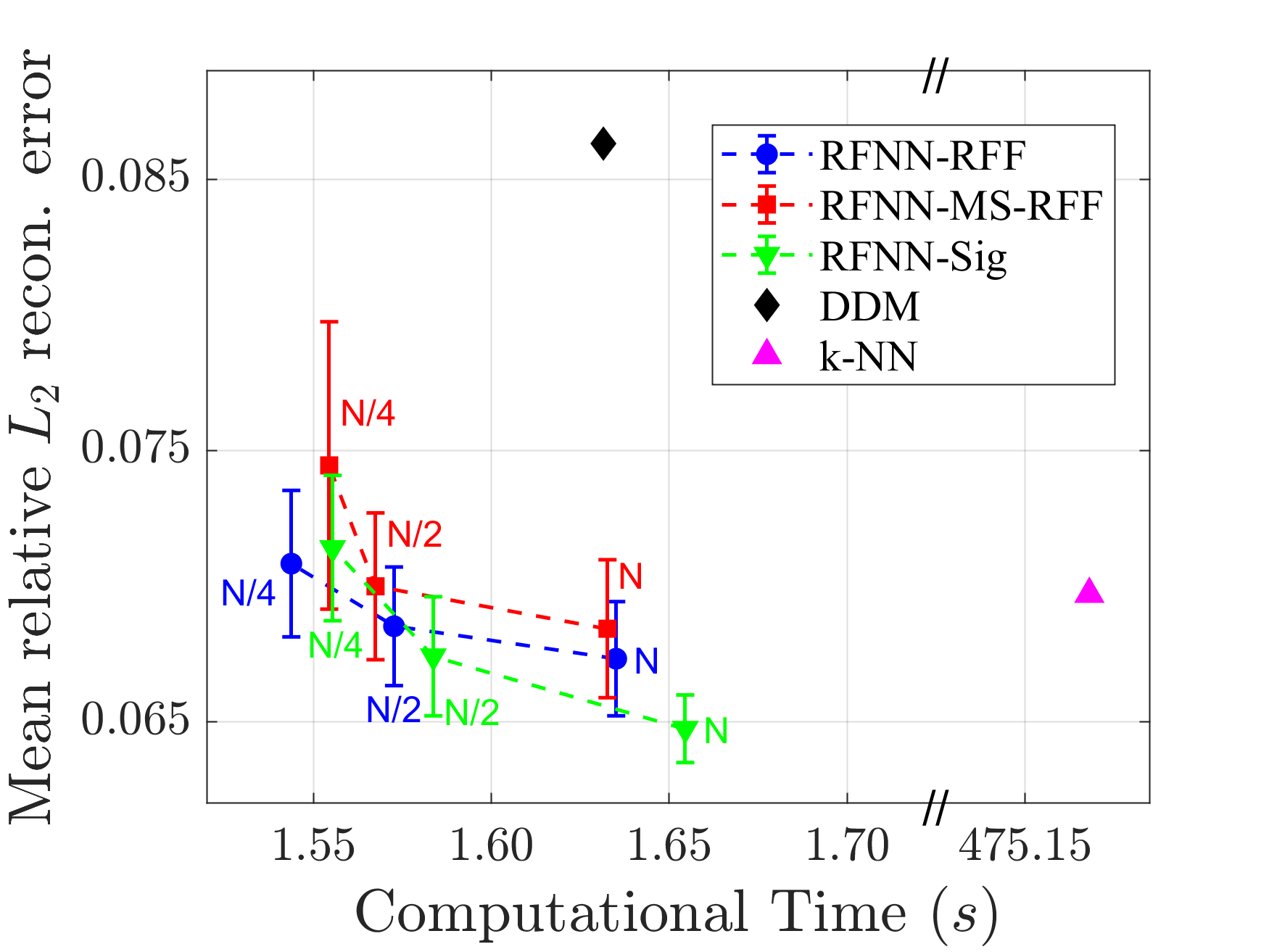}
        \caption{Testing, $L_2$ error}
        \label{fig:SC20_dec_ts}
    \end{subfigure}
    \hspace{2pt}
    \begin{subfigure}[b]{0.32\textwidth}
        \centering
        \includegraphics[width=\textwidth]{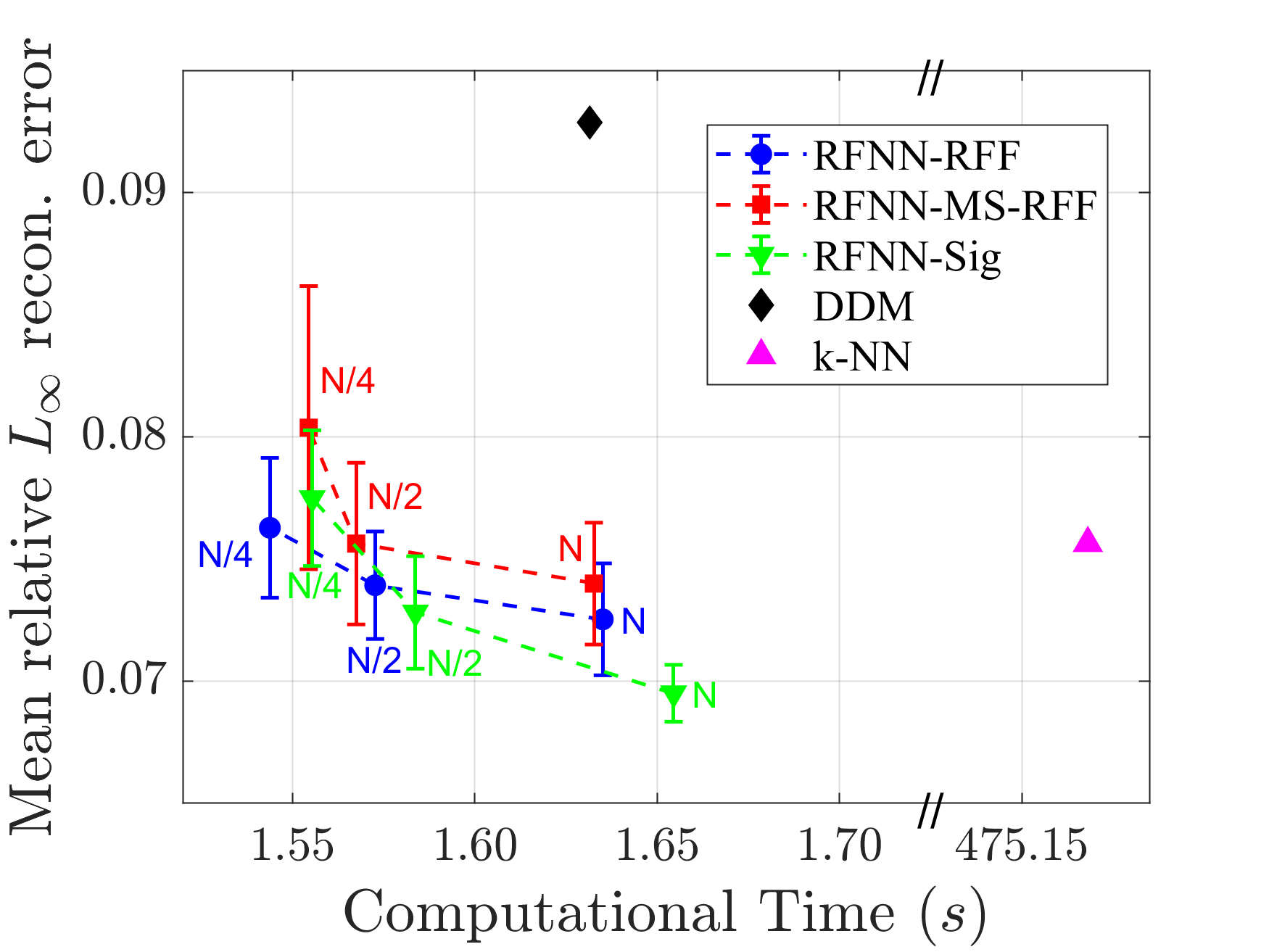}
        \caption{Testing, $L_\infty$ error}
        \label{fig:SC20_dec_ts_linf}
    \end{subfigure}
    \caption{Reconstruction error vs. computational time for the S-Curve dataset ($M=20$). Results are shown for $N=1000$ training points. Panel~\ref{fig:SC20_dec_tr} shows the mean relative $L_2$ error $e_{2,i}$ (Eq.~\eqref{eq:recon_errors}) on the training set vs. training time. Panels~\ref{fig:SC20_dec_ts} and \ref{fig:SC20_dec_ts_linf} show mean relative $L_2$ and $L_\infty$ errors, $e_{2,i}$ and $e_{\infty,i}$, respectively, on the testing set vs. inference time. For the stochastic RFNN decoders (RFNN-RFF, RFNN-MS-RFF, RFNN-Sig) with $P=N/4,N/2,N$, points indicate the median over 100 runs; error bars show the 5–95\% percentile range. Deterministic decoders (DDM, $k$-NN) are shown without error bars. %Detailed numerical results are provided in Tables~\ref{tab:SC20_dec_tr_all} and \ref{tab:SC20_dec_ts_all}.
    }
    \label{fig:SCurve20_Dec}
\end{figure}

All RFNN decoder variants are faster in both training and inference than the DDM and especially the $k$-NN decoders, while generally matching or surpassing their accuracy. In training, RFNN decoders with few features ($P=N/4$) are less accurate, but as $P$ increases, their accuracy improves, exhibiting a clear monotonic trend. All of them surpass the accuracy of the DDM decoder (which was exhibiting higher reconstruction accuracy in the Swiss Roll example in Appendix~\ref{app:SR}), while being slightly less accurate than the $k$-NN one. In testing (out-of-sample data), all RFNN variants show a clear accuracy gain as $P$ increases. The similar behavior of $L_2$ and $L_\infty$ errors further indicates uniformly good pointwise reconstruction. While RFNN-Sig is the most accurate RFNN variant for $P=N,N/2$ with RFNN-RFF and RFNN-MS-RFF following, when the number of features is small at $P=N/4$, the RFNN-RFF becomes more accurate. More importantly, all RFNN variants are more accurate than DDM and $k$-NN decoders in out-of-sample data, though their inference times become comparable to those of DDM. Visualizations of the S-Curve reconstructions, projected onto a $3$-dimensional space are provided for the testing sets in Fig.~\ref{fig:SCurve20}. The reconstructed datasets confirm that RFNN decoders preserve the underlying geometry with high fidelity, matching the quantitative trends reported above.

Considering together the non-mass-preserving examples in Appendices~\ref{app:SR} and \ref{app:SC20}, the results demonstrate that RFNN decoders, particularly the RFNN‑Sig one, consistently achieve higher reconstruction accuracy at lower computational cost compared to both numerical analysis-based DDM and $k$-NN approaches.

\begin{figure}[!htbp]
    \centering
    % First row
    \begin{subfigure}[b]{0.32\textwidth}
        \centering
        \includegraphics[width=\textwidth]{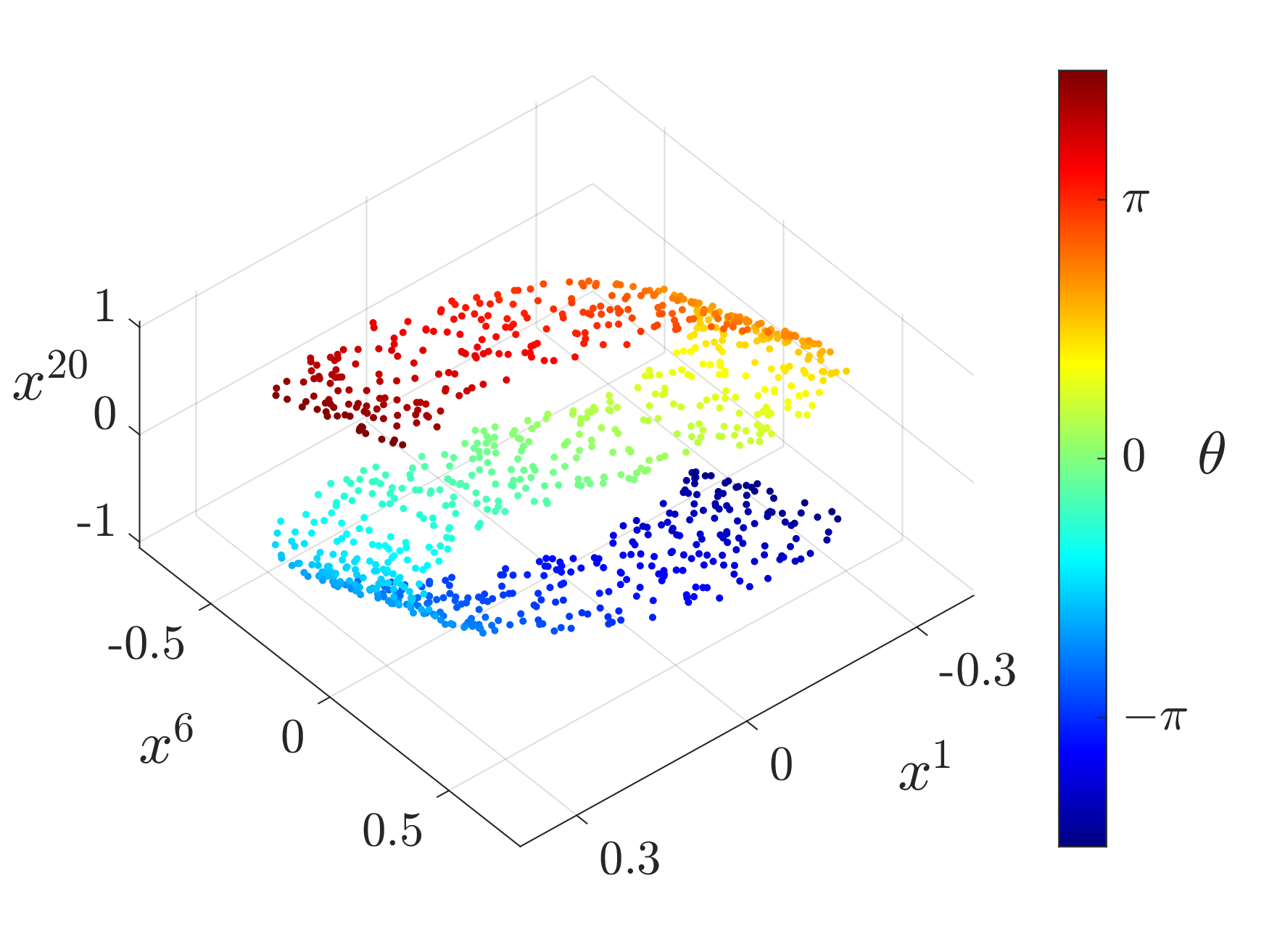}
        \caption{Data set}
        \label{fig:SC20_data}
    \end{subfigure} \hspace{0.02\textwidth}
    \begin{subfigure}[b]{0.32\textwidth}
        \centering
        \includegraphics[width=\textwidth]{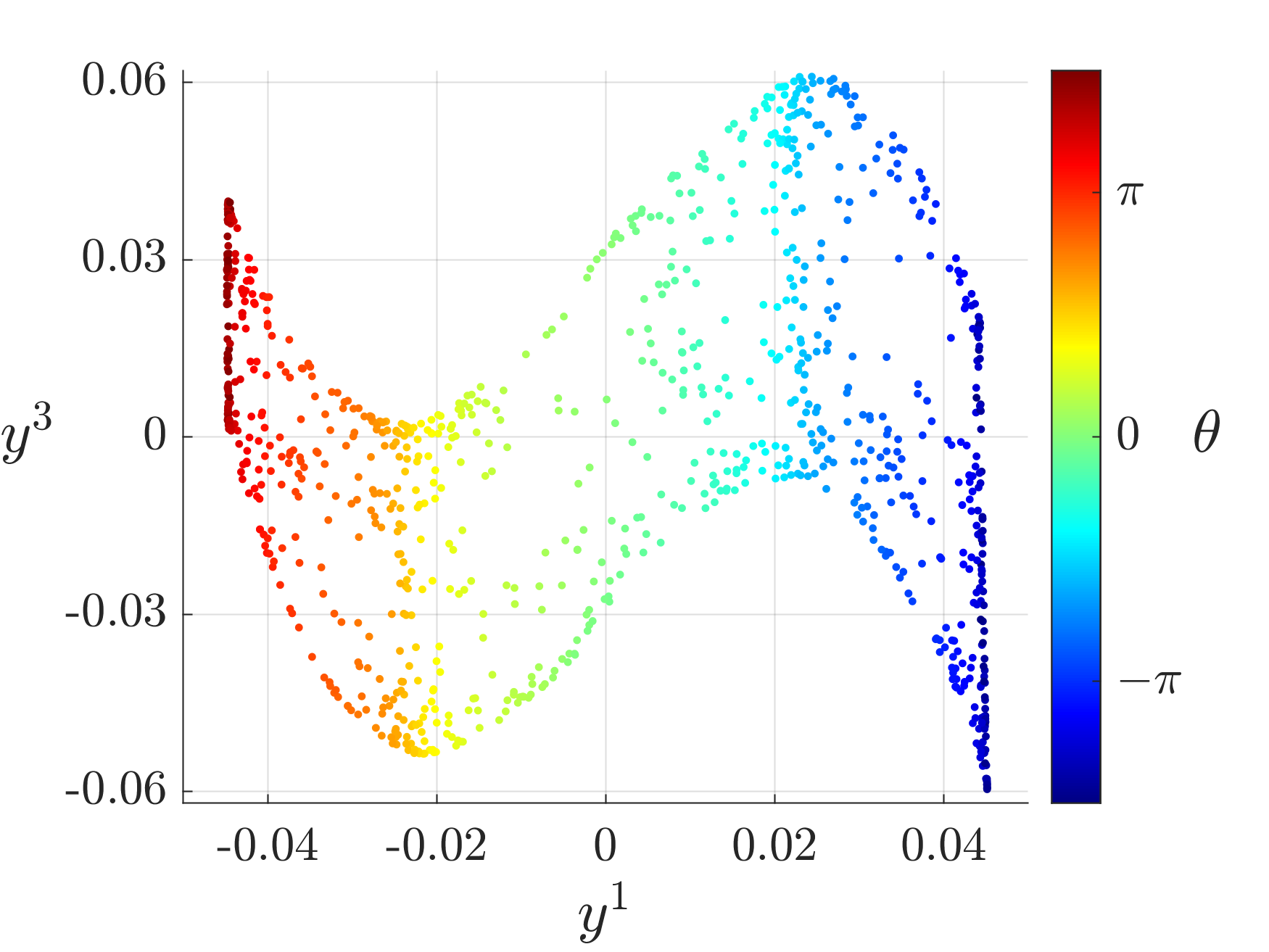}
        \caption{DM coordinates}
        \label{fig:SC20_DMcoords}
    \end{subfigure}

    % Second row
    \begin{subfigure}[b]{0.32\textwidth}
        \centering
        \includegraphics[width=\textwidth]{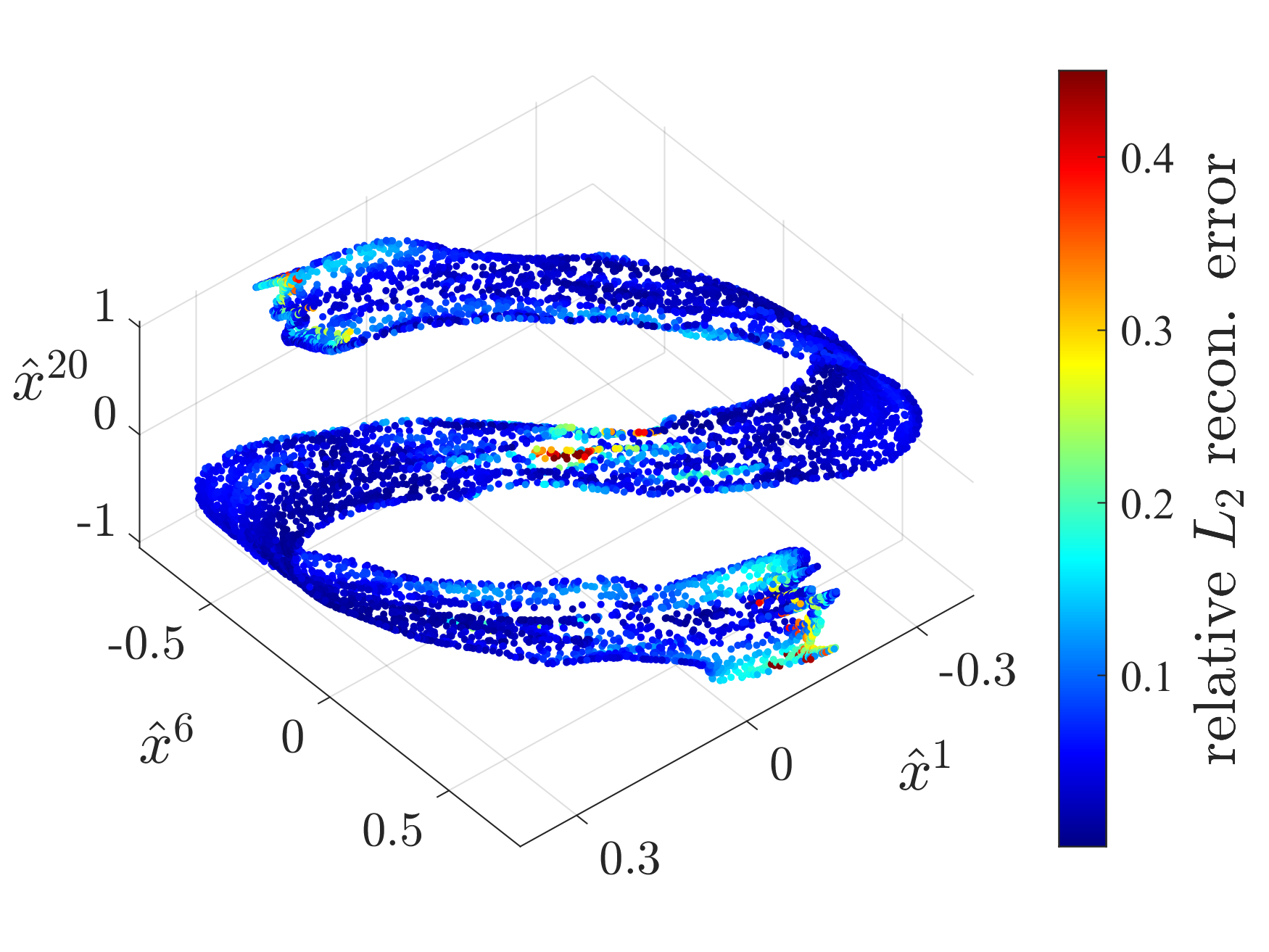}
        \caption{RFNN-RFF ($P=N$)}
        \label{fig:SC20_DM_RFNNf_recon}
    \end{subfigure}
    \hfill
    \begin{subfigure}[b]{0.32\textwidth}
        \centering
        \includegraphics[width=\textwidth]{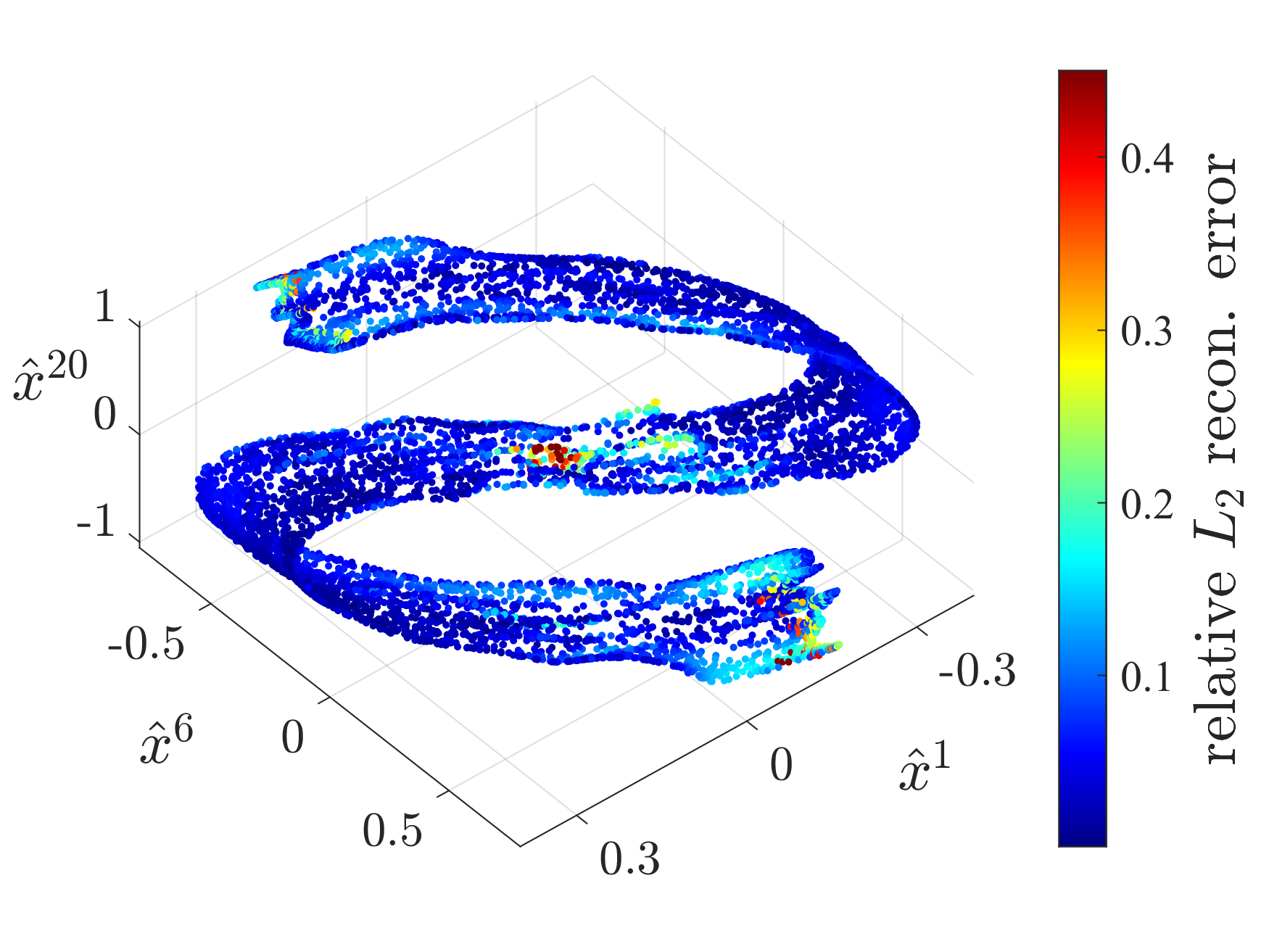}
        \caption{RFNN-MS-RFF ($P=N$)}
        \label{fig:SC20_DM_RFNNfMK_recon}
    \end{subfigure}
    \hfill
    \begin{subfigure}[b]{0.32\textwidth}
        \centering
        \includegraphics[width=\textwidth]{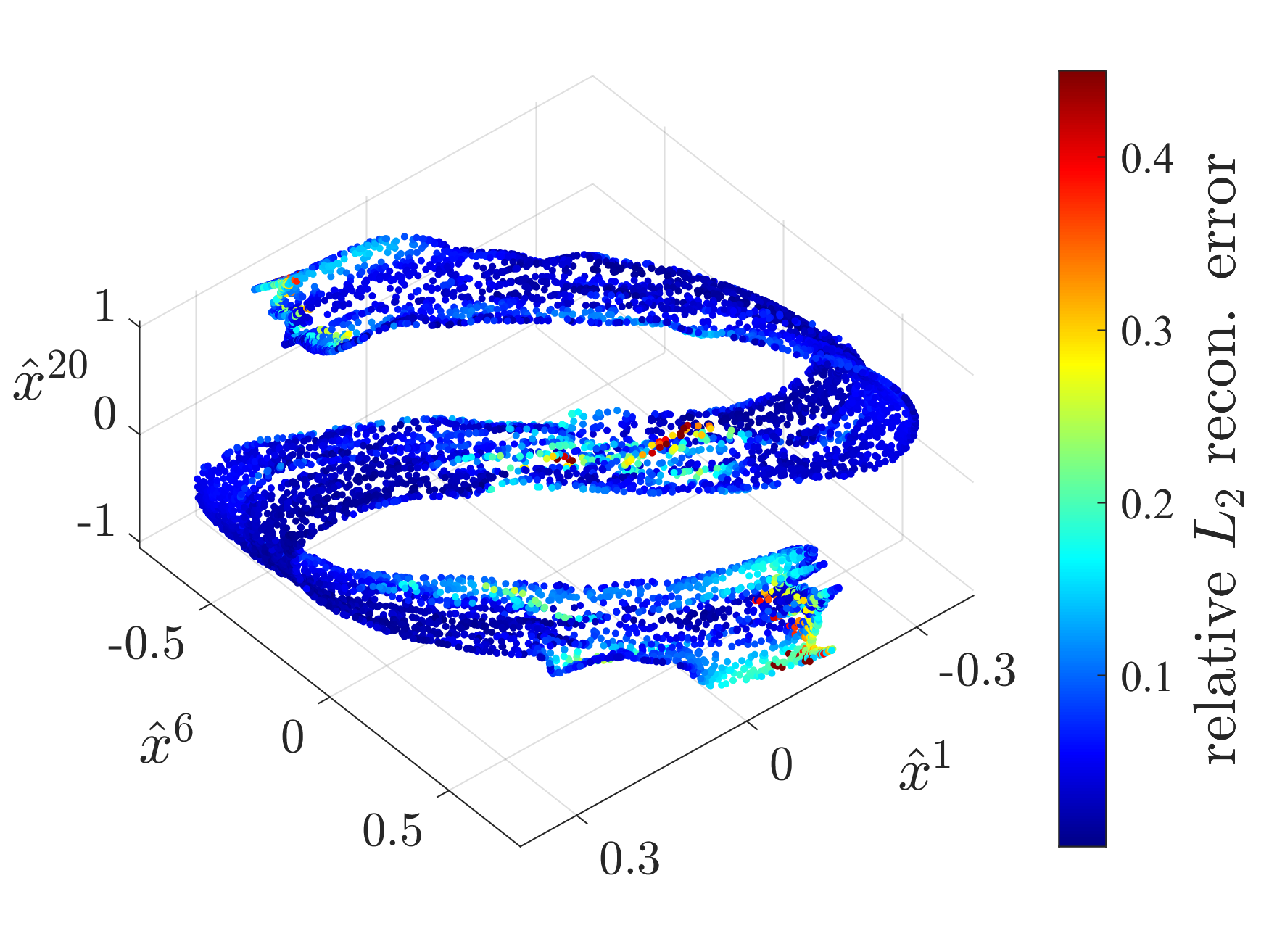}
        \caption{RFNN-Sig ($P=N/2$)}
        \label{fig:SC20_DM_RFNNs_recon}
    \end{subfigure}

    % Third row
    \begin{subfigure}[b]{0.32\textwidth}
        \centering
        \includegraphics[width=\textwidth]{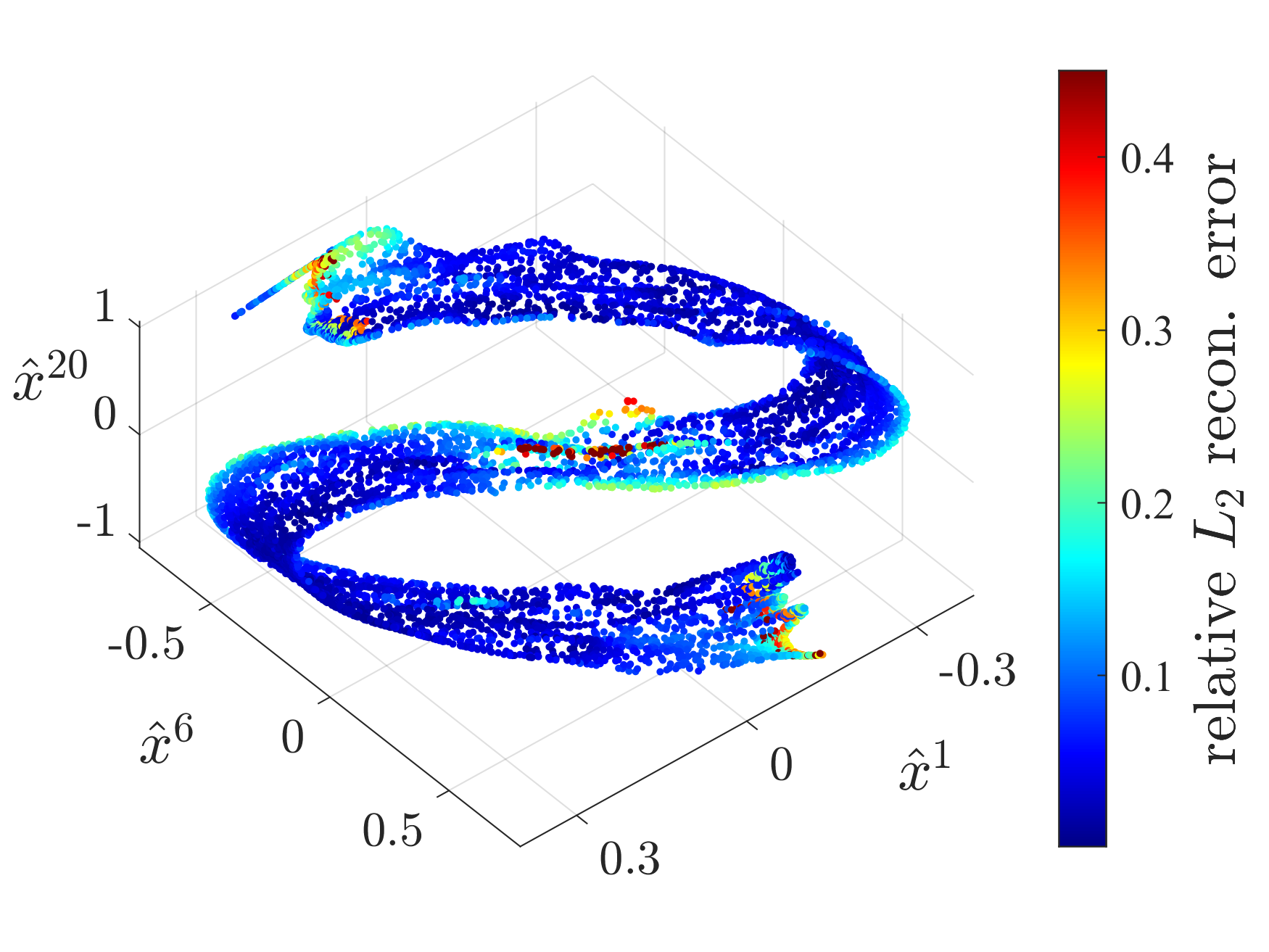}
        \caption{DDM}
        \label{fig:SC20_DM_GH_recon}
    \end{subfigure}\hspace{0.02\textwidth}
    \begin{subfigure}[b]{0.32\textwidth}
        \centering
        \includegraphics[width=\textwidth]{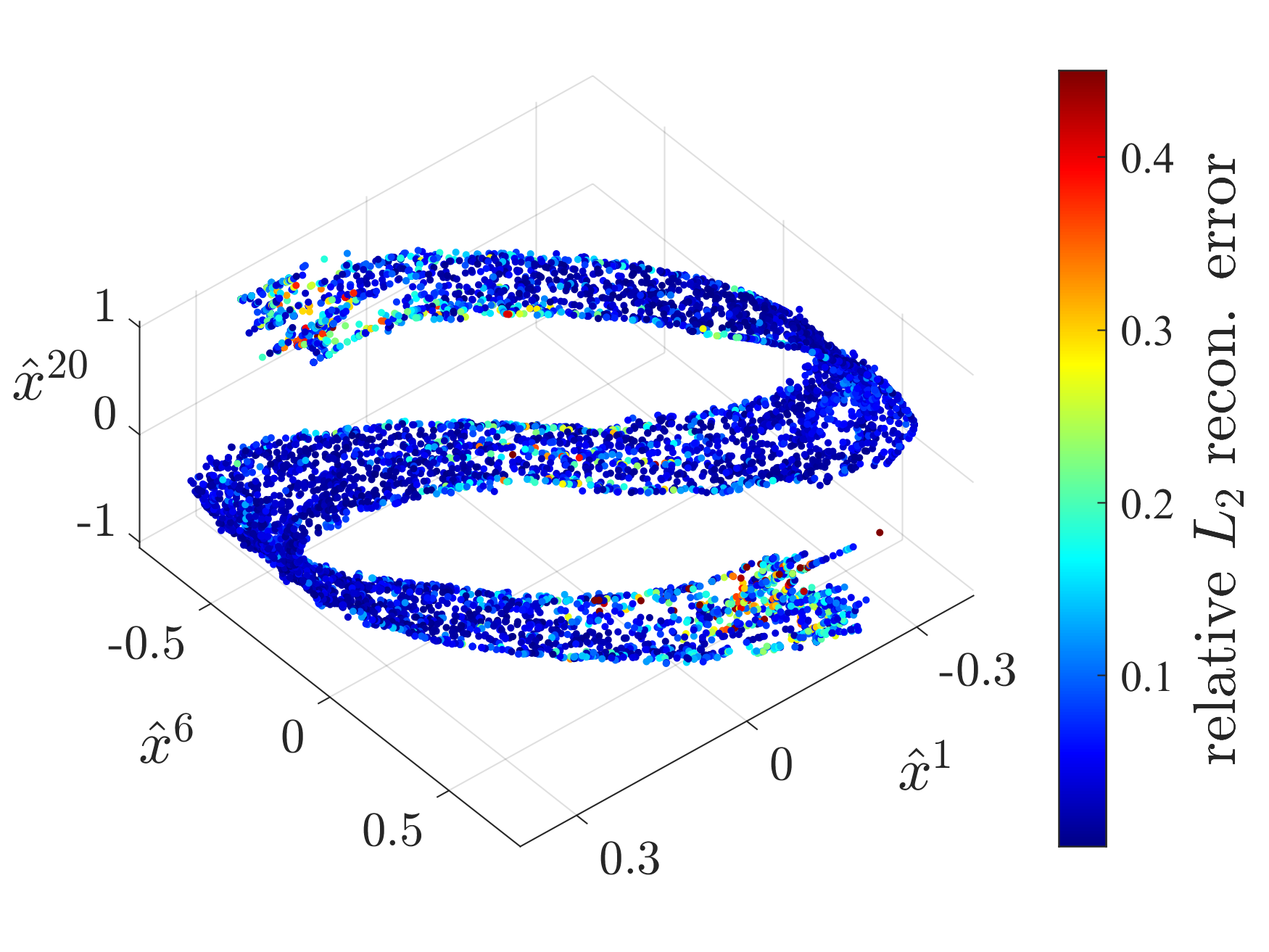}
        \caption{$k$-NN}
        \label{fig:SC20_DM_kNN_recon}
    \end{subfigure}
    \caption{S-Curve dataset ($M=20$) with $N=1000$ training points. Panel~\ref{fig:SC20_data} shows the training data projected onto the coordinates $[x^1,x^6,x^{20}]^\top$, while panel~\ref{fig:SC20_DMcoords} shows the $d=2$ DM coordinates $[y^1, y^3]^\top$; both panels are colored by the intrinsic angular coordinate $\theta$. Panels~\ref{fig:SC20_DM_RFNNf_recon}-\ref{fig:SC20_DM_kNN_recon} show the reconstructed data projected onto $[\widehat x^1,\widehat x^6,\widehat x^{20}]^\top$ and the per-point relative $L_2$ errors $e_{2,i}$ (Eq.~\eqref{eq:recon_errors} for the five decoders: RFNN-RFF with $P=N$, RFNN-MS-RFF with $P=N$, RFNN-Sig with $P=N/2$, DDM and $k$-NN. For the RFNN decoders, the displayed configuration is the one achieving the lowest training reconstruction error among all tested variants (see Fig.~\ref{fig:SC20_dec_tr} for the full comparison).}
    \label{fig:SCurve20}
\end{figure}

\clearpage
\newpage
\renewcommand{\theequation}{G.\arabic{equation}}
\renewcommand{\thefigure}{G.\arabic{figure}}
\renewcommand{\thetable}{G.\arabic{table}}
\setcounter{equation}{0}
\setcounter{figure}{0}
\setcounter{table}{0}
\section{Reconstruction results of the LWR 1D Traffic Model Dataset (\texorpdfstring{$M=400$}{M=400})} \label{app:LWR}

Here, we provide detailed reconstruction results for the LWR 1D traffic model dataset in Section~\ref{sb:LWR}. Tables~\ref{tab:LWR_dec_tr_all} and \ref{tab:LWR_dec_ts_all} enlist the training and testing set performance of all decoders considered for this benchmark.
\begin{table}[htbp]
\centering
\caption{Training set performance for the LWR 1D traffic model dataset ($M=400$). Detailed reconstruction metrics for $N=1000$ training points, using $d=2$-dim. DM embeddings. Decoders are compared based on the relative $L_2$ and $L_\infty$ mean reconstruction errors, $e_{2,i}$ and $e_{\infty,i}$ (Eq.~\eqref{eq:recon_errors}), computational time (in seconds), and mean conservation error $e_{con,i}$. For stochastic decoders (RANDSMAP-RFF, RANDSMAP-MS-RFF, RANDSMAP-Sig, and RFNN-Sig), metrics show the median with 5-95\% percentiles in parentheses, computed over 100 random initializations. Deterministic decoders (DDM, $k$-NN) report single values.}
\label{tab:LWR_dec_tr_all}
\resizebox{\textwidth}{!}{%
\begin{tabular}{@{}llcccc@{}}
\multirow{2}{*}{Decoder} & \multirow{2}{*}{$P$} & \multicolumn{2}{c}{Mean reconstruction error ($\times 10^{-1}$)} & Mean conservation & \multirow{2}{*}{Comp. Time (s)} \\
\cmidrule(lr){3-4}
& & relative $L_2$, $e_{2,i}$ &  relative $L_\infty$, $e_{\infty,i}$ & error, $e_{con,i}$ ($\times 10^{-8}$) & \\
\midrule
\midrule

% DM-RFNNf variants
\multirow{3}{*}{RANDSMAP-RFF} 
&	$N$	    & 1.372 (1.361--1.382) & 3.630 (3.617--3.644) & 2.328 (1.999--2.701) & 11.702 (11.665--11.733)\\
&	$N/2$	& 1.365 (1.354--1.379) & 3.617 (3.600--3.635) & 2.128 (1.811--2.487) & 4.624 (4.607--4.658)   \\   
&	$N/4$	& 1.381 (1.367--1.393) & 3.629 (3.610--3.648) & 1.729 (1.305--2.031) & 2.636 (2.630--2.659)   \\   
\midrule

% DM-RFNNfMK variants  
\multirow{3}{*}{RANDSMAP-MS-RFF} 
&	$N$	     & 1.282 (1.223--1.371) & 3.472 (3.341--3.621) & 2.475 (1.801--3.131) & 11.757 (11.701--11.811) \\
&	$N/2$	 & 1.353 (1.291--1.440) & 3.574 (3.466--3.683) & 1.874 (1.303--2.423) & 4.698 (4.677--4.730)    \\
&	$N/4$	 & 1.391 (1.317--1.474) & 3.594 (3.473--3.732) & 1.369 (0.832--1.877) & 2.655 (2.645--2.680)    \\
\midrule

% DM-RFNNs variants  
\multirow{3}{*}{RANDSMAP-Sig} 
&	$N$	    & 1.300 (1.295--1.305) & 3.544 (3.537--3.552) & 0.524 (0.424--0.628) & 14.449 (14.394--14.520) \\
&	$N/2$	& 1.327 (1.322--1.334) & 3.580 (3.571--3.590) & 0.647 (0.501--0.830) & 5.183 (5.165--5.212)    \\
&	$N/4$	& 1.367 (1.356--1.381) & 3.616 (3.602--3.633) & 0.829 (0.535--1.363) & 2.819 (2.810--2.840)    \\
\midrule

% DM-RFNNs variants without conservation
\multirow{3}{*}{RFNN-Sig} 
&	$N$	    & 1.300 (1.295--1.305) & 3.544 (3.537--3.552) & 0.365 (0.337--0.387) $\times 10^{3}$ & 3.723 (3.714--3.737) \\
&	$N/2$	& 1.327 (1.322--1.334) & 3.580 (3.571--3.590) & 0.779 (0.718--0.827) $\times 10^{3}$ & 2.013 (2.011--2.021) \\
&	$N/4$	& 1.367 (1.356--1.381) & 3.616 (3.602--3.633) & 1.909 (1.716--2.147) $\times 10^{3}$ & 1.469 (1.468--1.473) \\
\midrule

% Deterministic methods
DDM  &	-	& 6.384 & 6.842 & 5.157 $\times 10^{6}$ & 10.369 \\
$k$-NN &	-	& 1.343 & 3.690 & 4.030 $\times 10^{-8}$ & 1283.531 \\

\bottomrule
\end{tabular}%
}
\end{table}

\begin{table}[htbp]
\centering
\caption{Testing set performance for the LWR 1D traffic model dataset ($M=400$). Detailed reconstruction metrics for $N=1000$ training points, using $d=2$-dim. DM embeddings. Decoders are compared based on the relative $L_2$ and $L_\infty$ mean reconstruction errors, $e_{2,i}$ and $e_{\infty,i}$ (Eq.~\eqref{eq:recon_errors}), computational time (in seconds), and mean conservation error, $e_{con,i}$. For stochastic decoders (RANDSMAP-RFF, RANDSMAP-MS-RFF, RANDSMAP-Sig, and RFNN-Sig), metrics show the median with 5-95\% percentiles in parentheses, computed over 100 random initializations. Deterministic decoders (DDM, $k$-NN) report single values.}
\label{tab:LWR_dec_ts_all}
\resizebox{\textwidth}{!}{%
\begin{tabular}{@{}llcccc@{}}
\multirow{2}{*}{Decoder} & \multirow{2}{*}{$P$} & \multicolumn{2}{c}{Mean reconstruction error ($\times 10^{-1}$)} & Mean conservation & \multirow{2}{*}{Comp. Time (s)} \\
\cmidrule(lr){3-4}
& & relative $L_2$, $e_{2,i}$ &  relative $L_\infty$, $e_{\infty,i}$ & error, $e_{con,i}$ ($\times 10^{-8}$) & \\
\midrule
\midrule

% DM-RFNNf variants
\multirow{3}{*}{RANDSMAP-RFF} 
&	$N$	    & 1.622 (1.618--1.626) & 4.236 (4.230--4.243) & 2.646 (2.261--3.096) & 3.179 (3.114--3.239) \\
&	$N/2$	& 1.623 (1.620--1.628) & 4.240 (4.232--4.247) & 2.421 (2.045--2.870) & 3.053 (2.989--3.106) \\   
&	$N/4$	& 1.629 (1.625--1.634) & 4.231 (4.222--4.237) & 1.925 (1.475--2.327) & 2.939 (2.868--3.011) \\   
\midrule

% DM-RFNNfMK variants  
\multirow{3}{*}{RANDSMAP-MS-RFF} 
&	$N$	     & 1.628 (1.619--1.654) & 4.296 (4.226--4.368) & 3.022 (2.027--4.038) & 3.179 (3.115--3.247) \\
&	$N/2$	 & 1.629 (1.623--1.651) & 4.241 (4.198--4.292) & 2.149 (1.454--2.864) & 3.044 (2.973--3.110) \\
&	$N/4$	 & 1.648 (1.637--1.681) & 4.217 (4.180--4.248) & 1.512 (0.870--2.134) & 2.956 (2.898--3.023) \\
\midrule

% DM-RFNNs variants  
\multirow{3}{*}{RANDSMAP-Sig} 
&	$N$	     & 1.586 (1.583--1.588) & 4.263 (4.258--4.267) & 4.100 (3.089--5.653)  & 3.179 (3.118--3.250) \\
&	$N/2$	 & 1.592 (1.588--1.598) & 4.247 (4.241--4.255) & 3.556 (2.645--4.839)  & 3.033 (2.962--3.117) \\
&	$N/4$	 & 1.617 (1.610--1.627) & 4.241 (4.230--4.253) & 2.178 (1.169--4.188)  & 2.943 (2.865--3.030) \\
\midrule

% DM-RFNNs variants without conservation 
\multirow{3}{*}{RFNN-Sig} 
&	$N$	     & 1.586 (1.583--1.588) & 4.263 (4.258--4.267) & 0.495 (0.449--0.535) $\times 10^{3}$ & 3.198 (3.142--3.275) \\
&	$N/2$	 & 1.592 (1.588--1.598) & 4.247 (4.241--4.255) & 1.008 (0.909--1.123) $\times 10^{3}$ & 3.077 (3.014--3.128) \\
&	$N/4$	 & 1.617 (1.610--1.627) & 4.241 (4.230--4.253) & 2.408 (2.152--2.685) $\times 10^{3}$ & 2.984 (2.939--3.048) \\
\midrule

% Deterministic methods
DDM  &	-	& 6.367 & 6.827 & 5.118 $\times 10^{6}$ & 4.415 \\
$k$-NN &	-	& 1.599 & 4.419 & 3.916 $\times 10^{-8}$ & 724.774 \\

\bottomrule
\end{tabular}%
}
\end{table}

Figure~\ref{fig:LWR_recon} shows a shock-including representative density profile and its reconstructions with the five decoders considered, while Fig.~\ref{fig:LWR_DM_GH} demonstrates that the poor performance of the DDM decoder is not due to poor tuning, via an exhaustive 2D grid search. Finally, Figure~\ref{fig:LWR} demonstrates per-density profile reconstructions provided by all decoders, projected in the summary statistics space $(\mu_x,\sigma^2_x,\gamma_x)$.
\begin{figure}[!htbp]
    \centering
    % First row
    \begin{subfigure}[b]{0.32\textwidth}
        \centering
        \includegraphics[width=\textwidth]{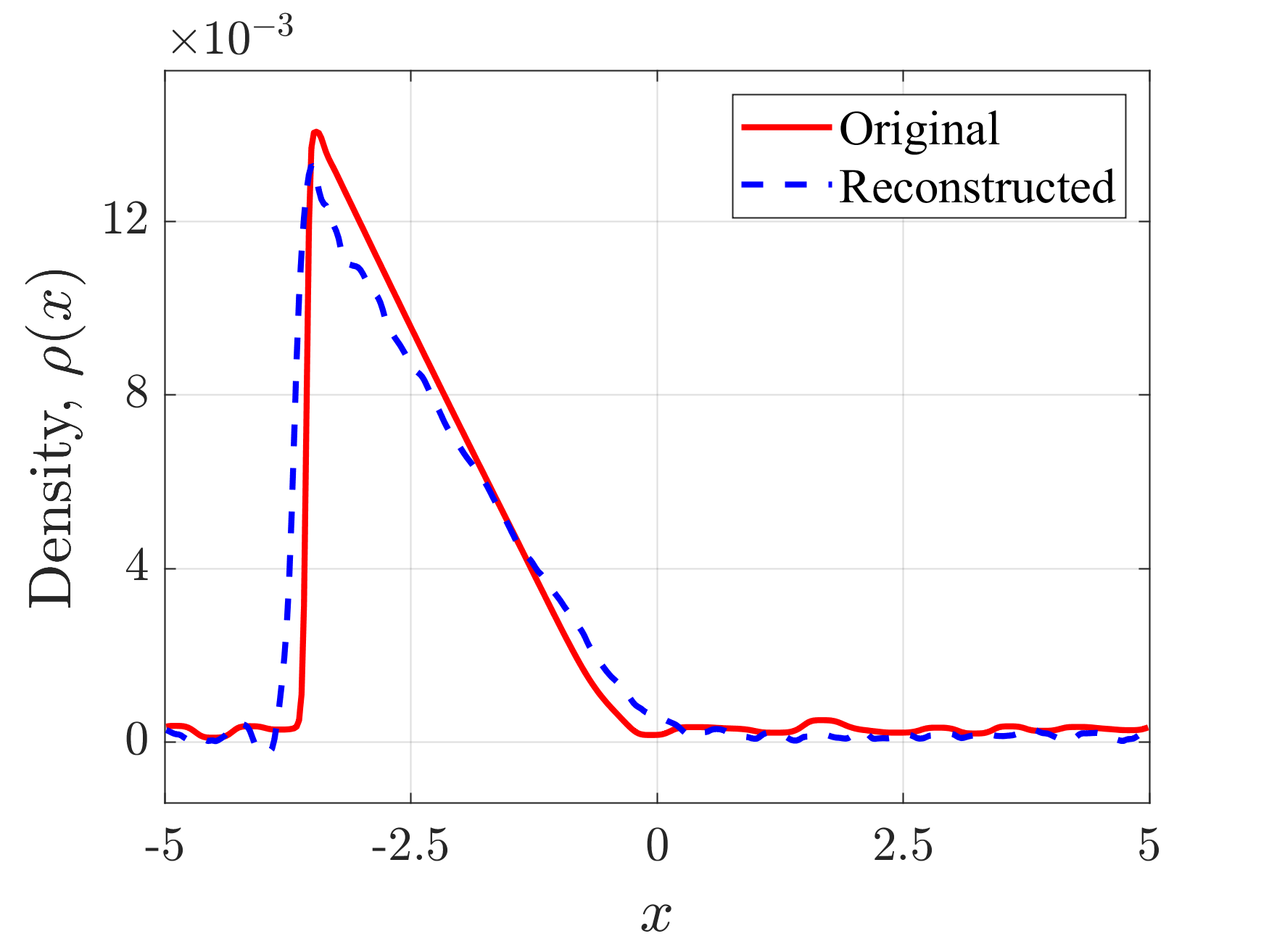}
        \caption{RANDSMAP-RFF ($P=N/2$)}
        \label{fig:LWR_DM_RFNNf_recon_sample}
    \end{subfigure}
    \hfill
    \begin{subfigure}[b]{0.32\textwidth}
        \centering
        \includegraphics[width=\textwidth]{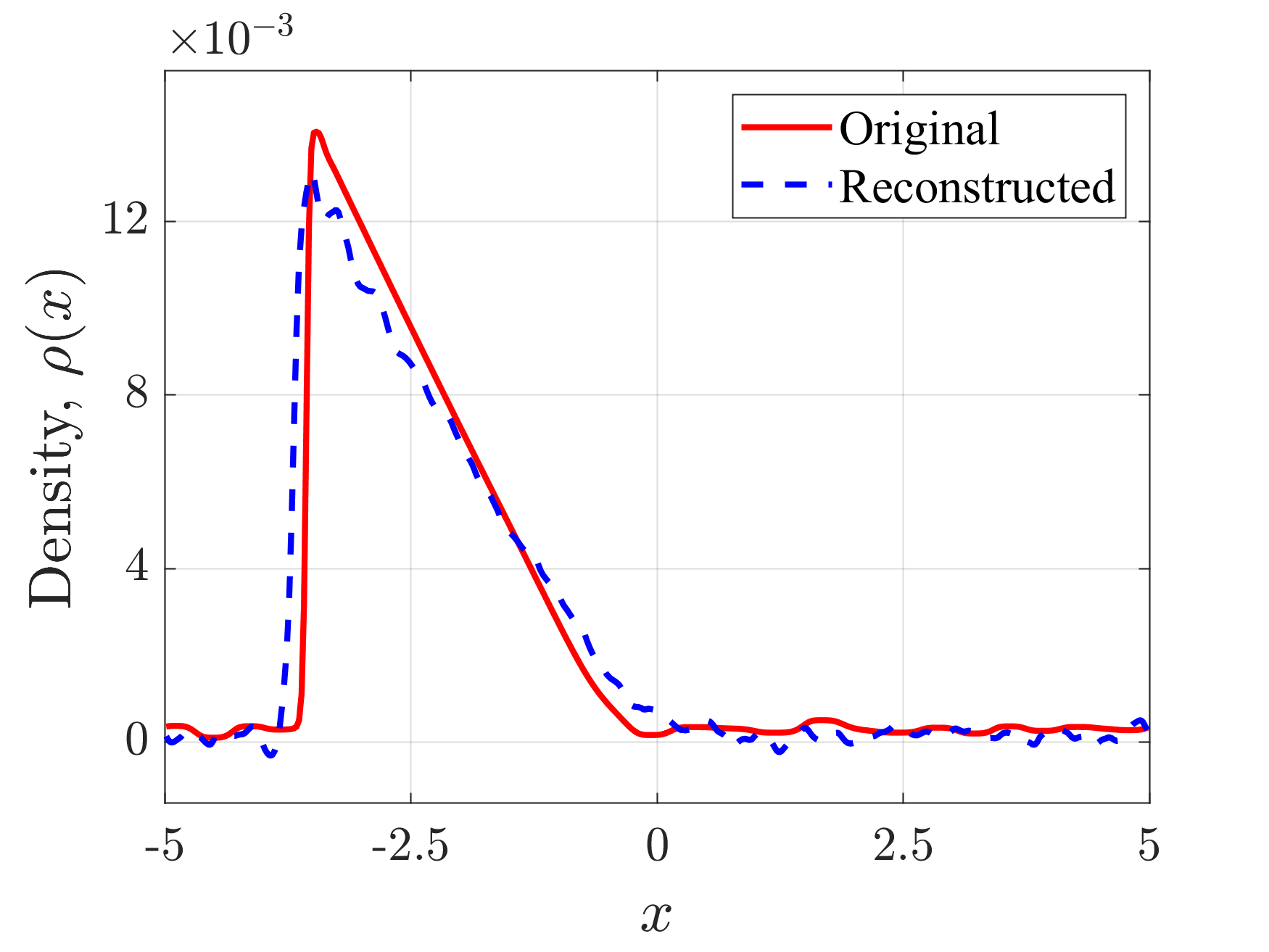}
        \caption{RANDSMAP-MS-RFF ($P=N$)}
        \label{fig:LWR_DM_RFNNfMK_recon_sample}
    \end{subfigure}
    \hfill
    \begin{subfigure}[b]{0.32\textwidth}
        \centering
        \includegraphics[width=\textwidth]{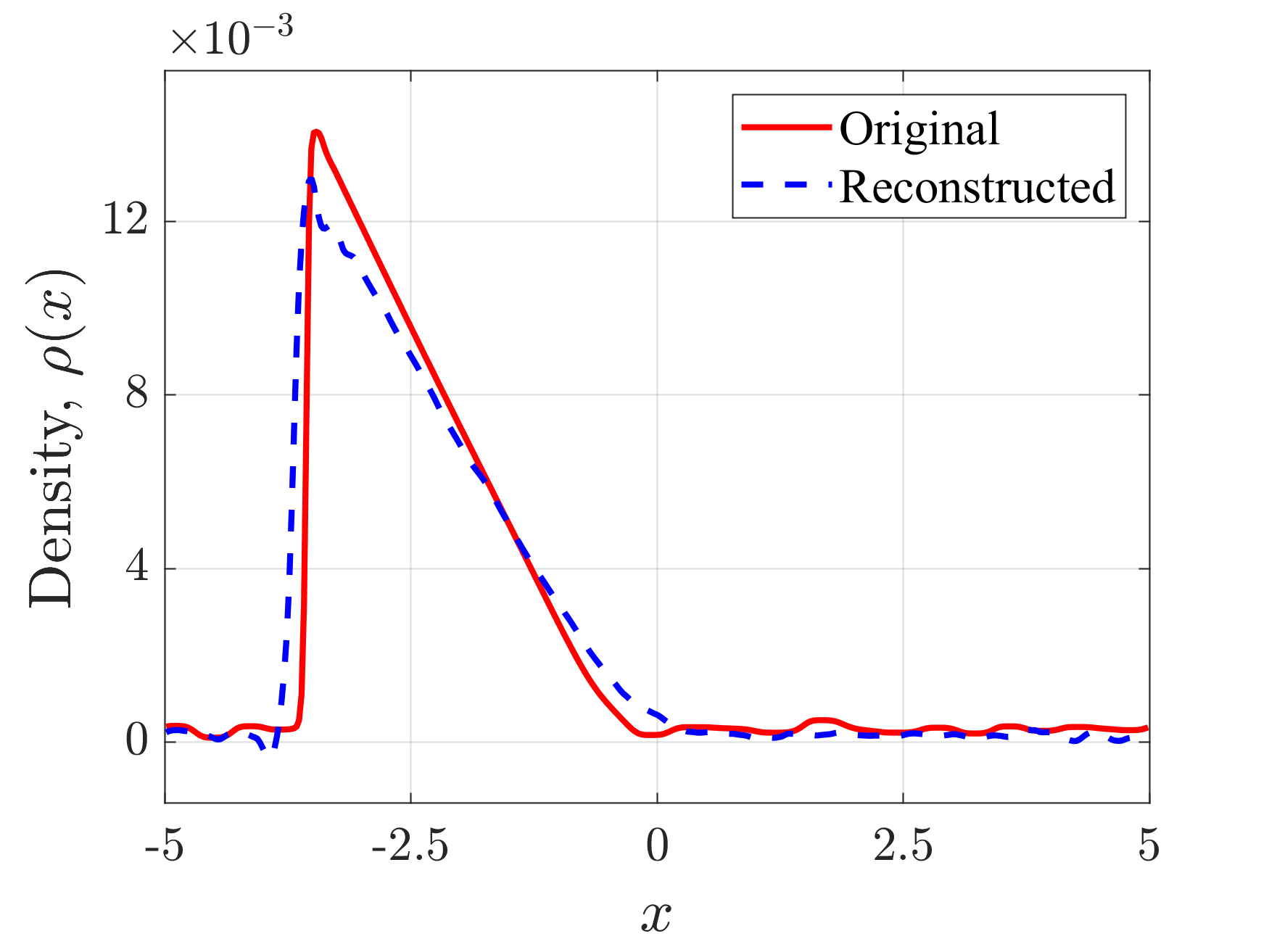}
        \caption{RANDSMAP-Sig ($P=N$)}
        \label{fig:LWR_DM_RFNNs_recon_sample}
    \end{subfigure}

    % Second row
    \begin{subfigure}[b]{0.32\textwidth}
        \centering
        \includegraphics[width=\textwidth]{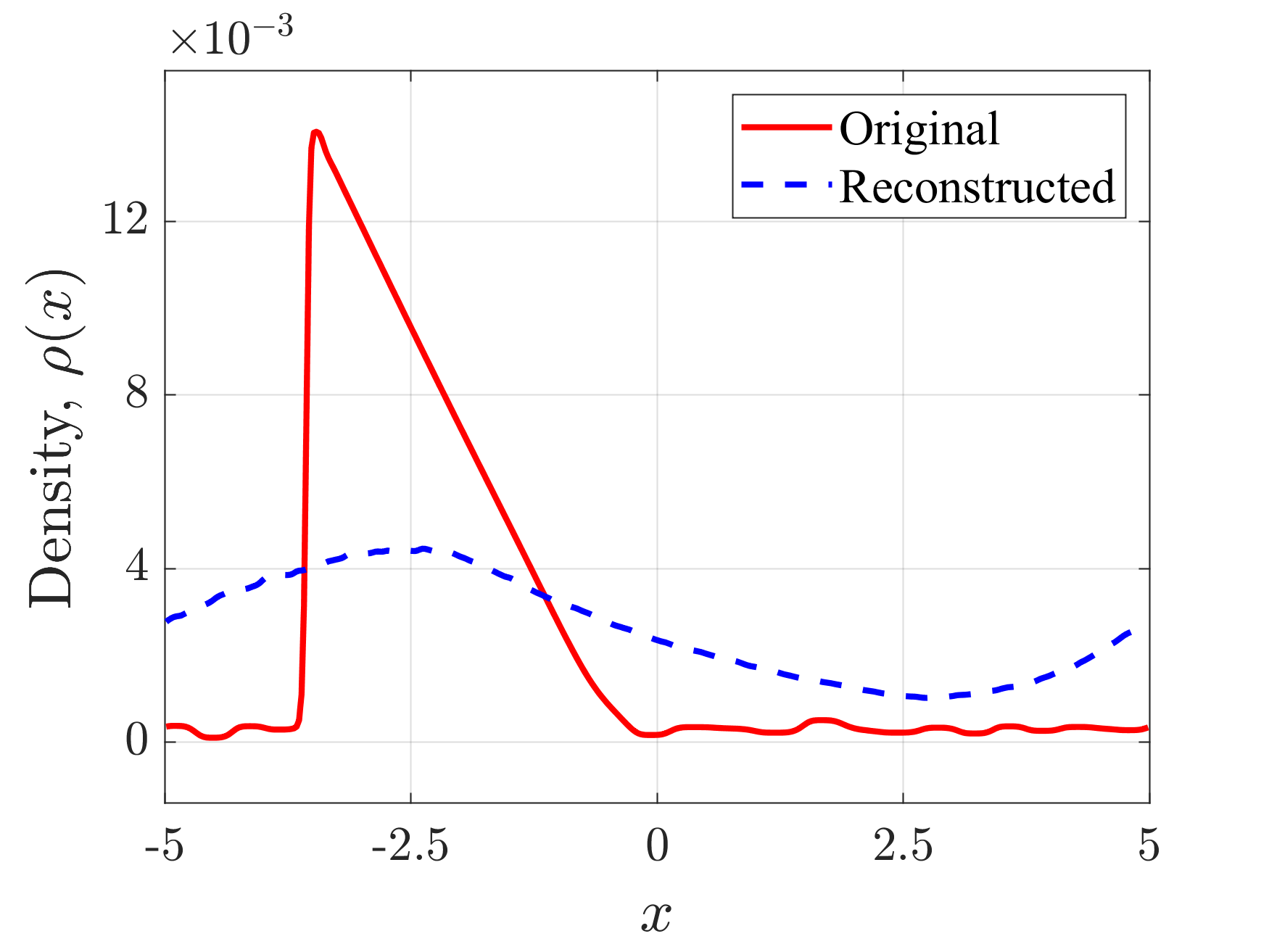}
        \caption{DDM}
        \label{fig:LWR_DM_GH_recon_sample}
    \end{subfigure}\hspace{0.02\textwidth}
    \begin{subfigure}[b]{0.32\textwidth}
        \centering
        \includegraphics[width=\textwidth]{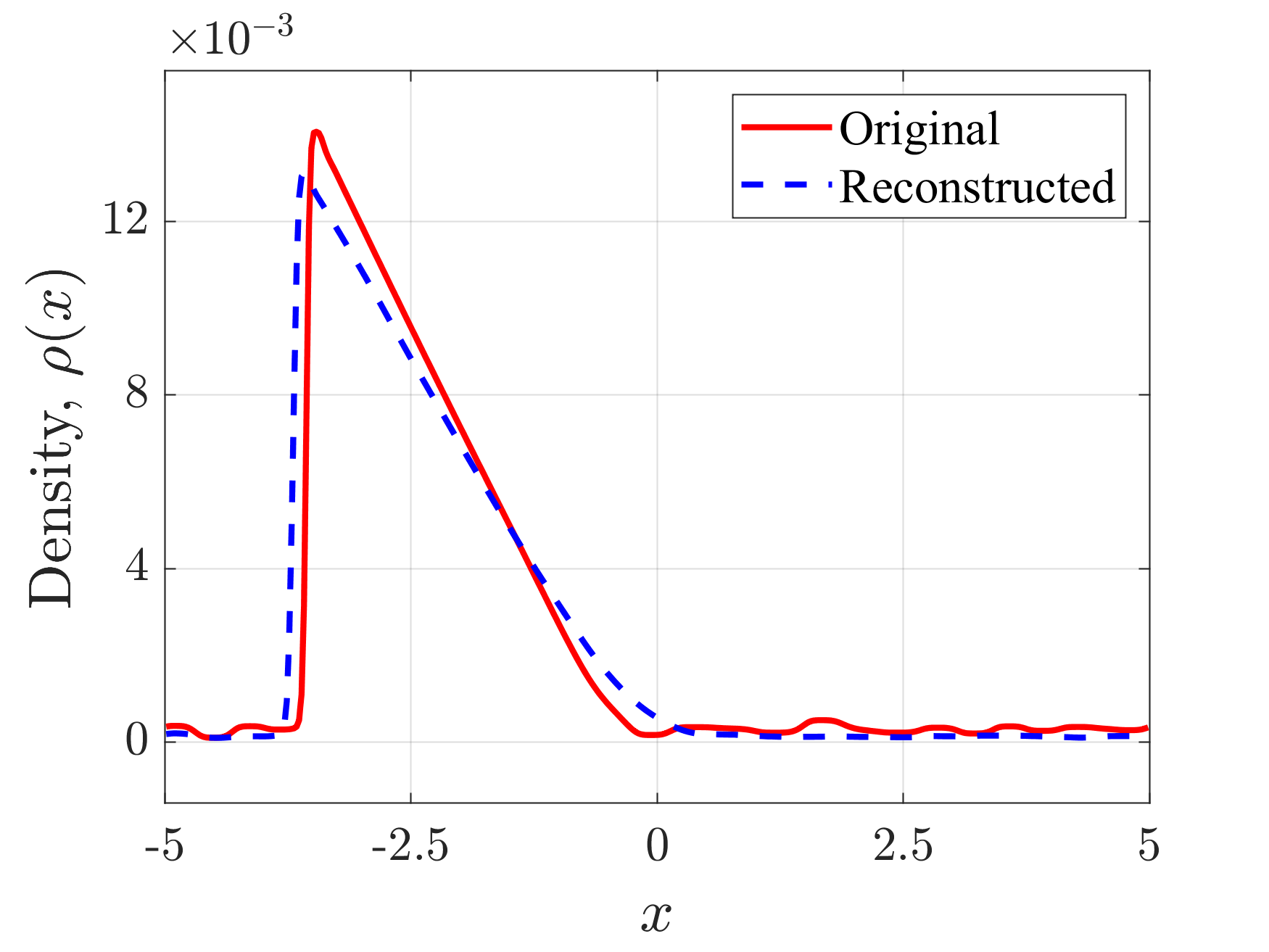}
        \caption{$k$-NN}
        \label{fig:LWR_DM_kNN_recon_sample}
    \end{subfigure}
    \caption{Representative original and reconstructed density profiles for the LWR 1D traffic model dataset ($M=400$). Panels~\ref{fig:LWR_DM_RFNNf_recon_sample}-\ref{fig:LWR_DM_kNN_recon} show the reconstructions for the five decoders: RANDSMAP-RFF with $P=N/2$, RANDSMAP-MS-RFF with $P=N$, RANDSMAP-Sig with $P=N$, DDM and $k$-NN. For the RANDSMAP decoders, the displayed configuration is the one achieving the lowest training reconstruction error among all variants of that type (see Fig.~\ref{fig:LWR_dec_tr} for the full comparison).}
    \label{fig:LWR_recon}
\end{figure}

\begin{figure}[!htbp]
    \centering
    % First row
    \begin{subfigure}[b]{0.32\textwidth}
        \centering
        \includegraphics[width=\textwidth]{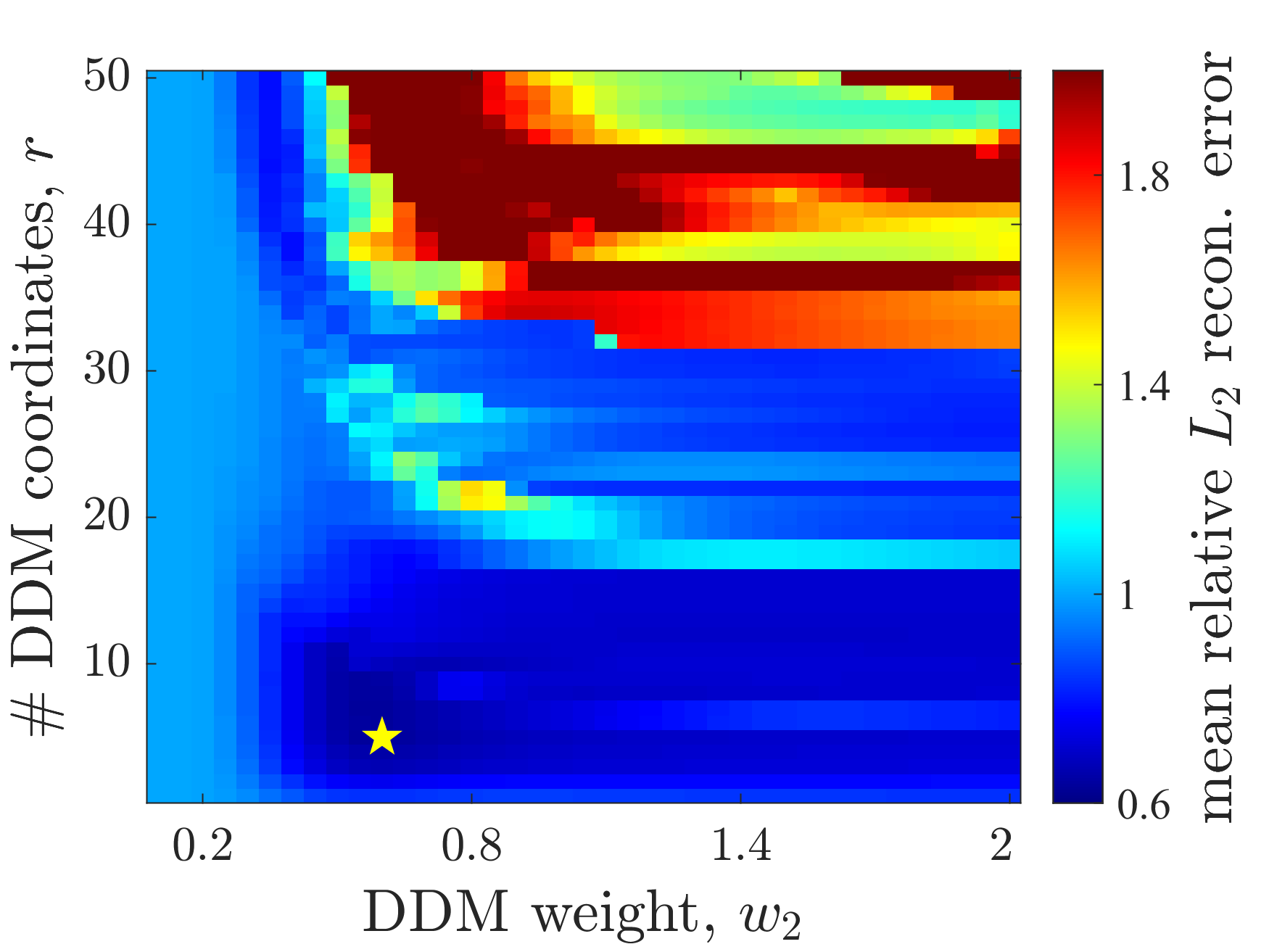}
        \caption{Exhaustive ($w_2,r$) tuning}
        \label{fig:LWR_DM_GH_tuning}
    \end{subfigure}
    \hfill
    \begin{subfigure}[b]{0.32\textwidth}
        \centering
        \includegraphics[width=\textwidth]{FigsLWR/ReconSample_DM_GH.png}
        \caption{Reconstruction, $(w_2,r)=(0.6,5)$}
        \label{fig:LWR_DM_GH_k5}
    \end{subfigure}
    \hfill
    \begin{subfigure}[b]{0.32\textwidth}
        \centering
        \includegraphics[width=\textwidth]{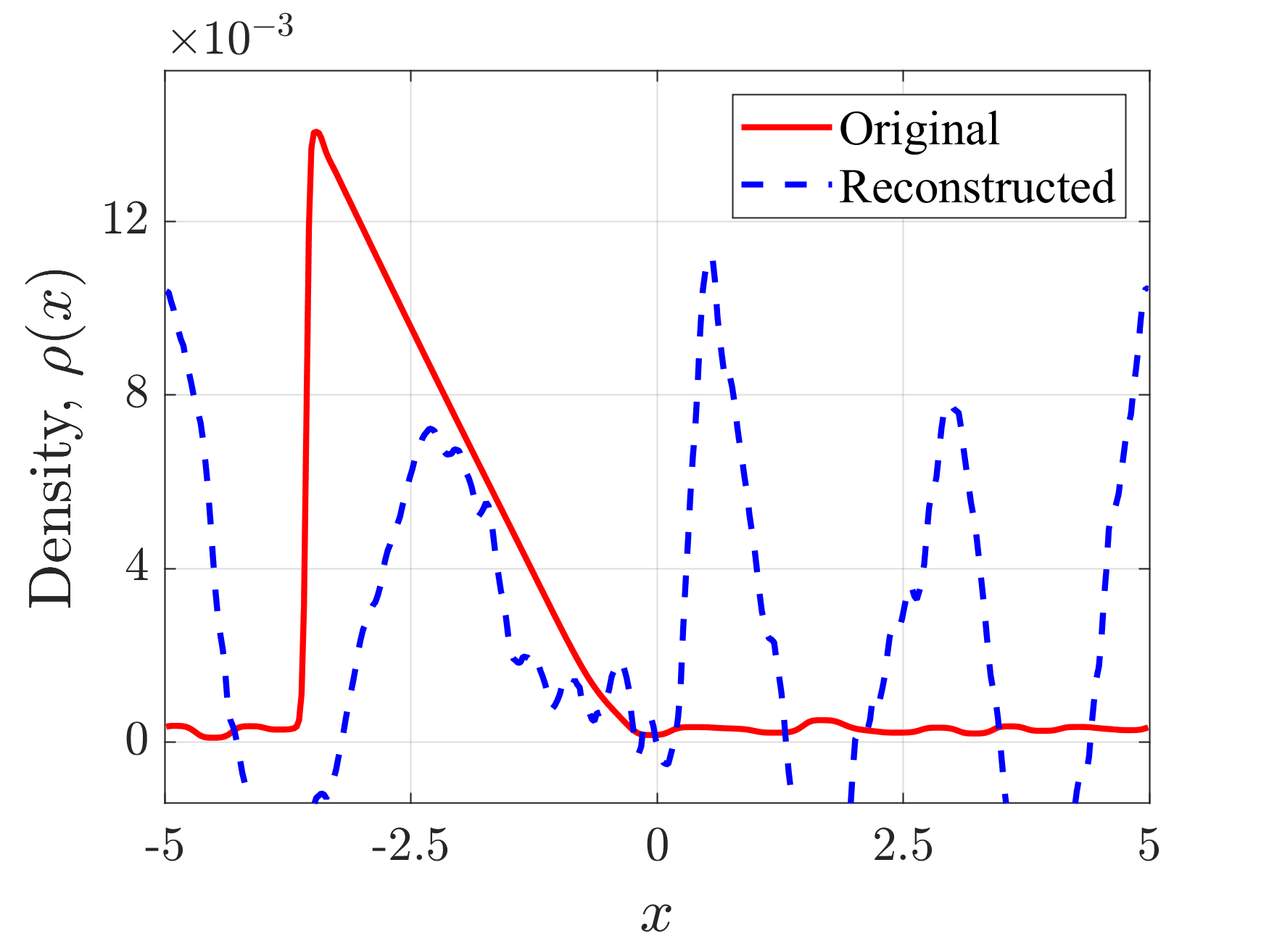}
        \caption{Reconstruction, $(w_2,r)=(0.6,25)$}
        \label{fig:LWR_DM_GH_k25}
    \end{subfigure}
    \caption{Reconstruction with the DDM decoder for the LWR 1D traffic model dataset ($M=400$). Panel~\ref{fig:LWR_DM_GH_tuning} shows exhaustive hyperparameter tuning over the DDM weight $w_2$ and the number of DDM coordinates $r$. The yellow star marks the optimal combination $(w_2,r) = (0.6, 5)$. Panels \ref{fig:LWR_DM_GH_k5} and \ref{fig:LWR_DM_GH_k25} show the reconstructed density profiles (corresponding to the original profile in Fig.~\ref{fig:LWR_recon}) obtained with $w_2=0.6$ and $r=5$ (optimal) or $r=25$ DDM coordinates, respectively.}
    \label{fig:LWR_DM_GH}
\end{figure}

\begin{figure}[!htbp]
    \centering
    % First row
    \begin{subfigure}[b]{0.32\textwidth}
        \centering
        \includegraphics[width=\textwidth]{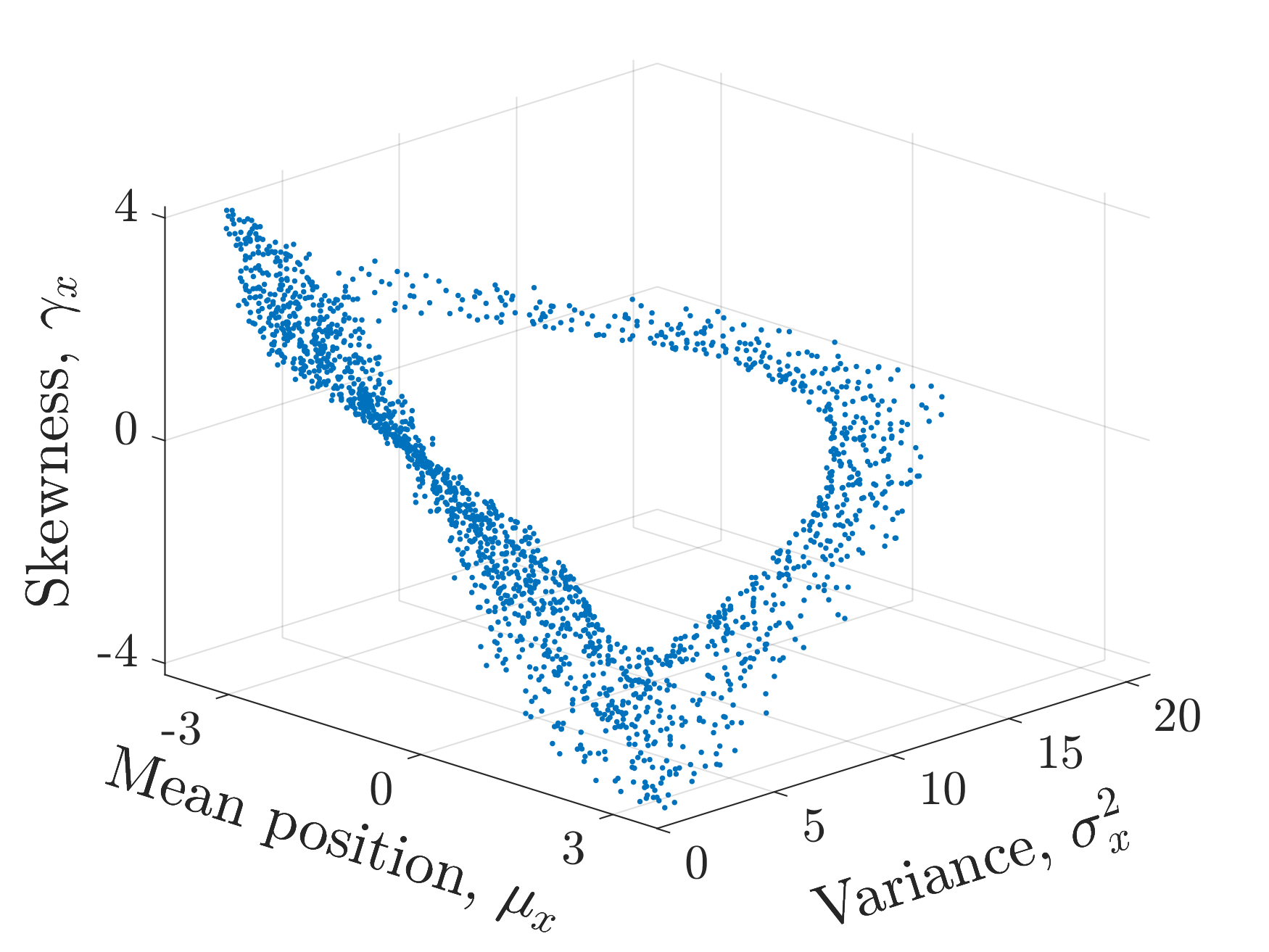}
        \caption{Data set}
        \label{fig:LWR_data}
    \end{subfigure} \hspace{0.02\textwidth}
    \begin{subfigure}[b]{0.32\textwidth}
        \centering
        \includegraphics[width=\textwidth]{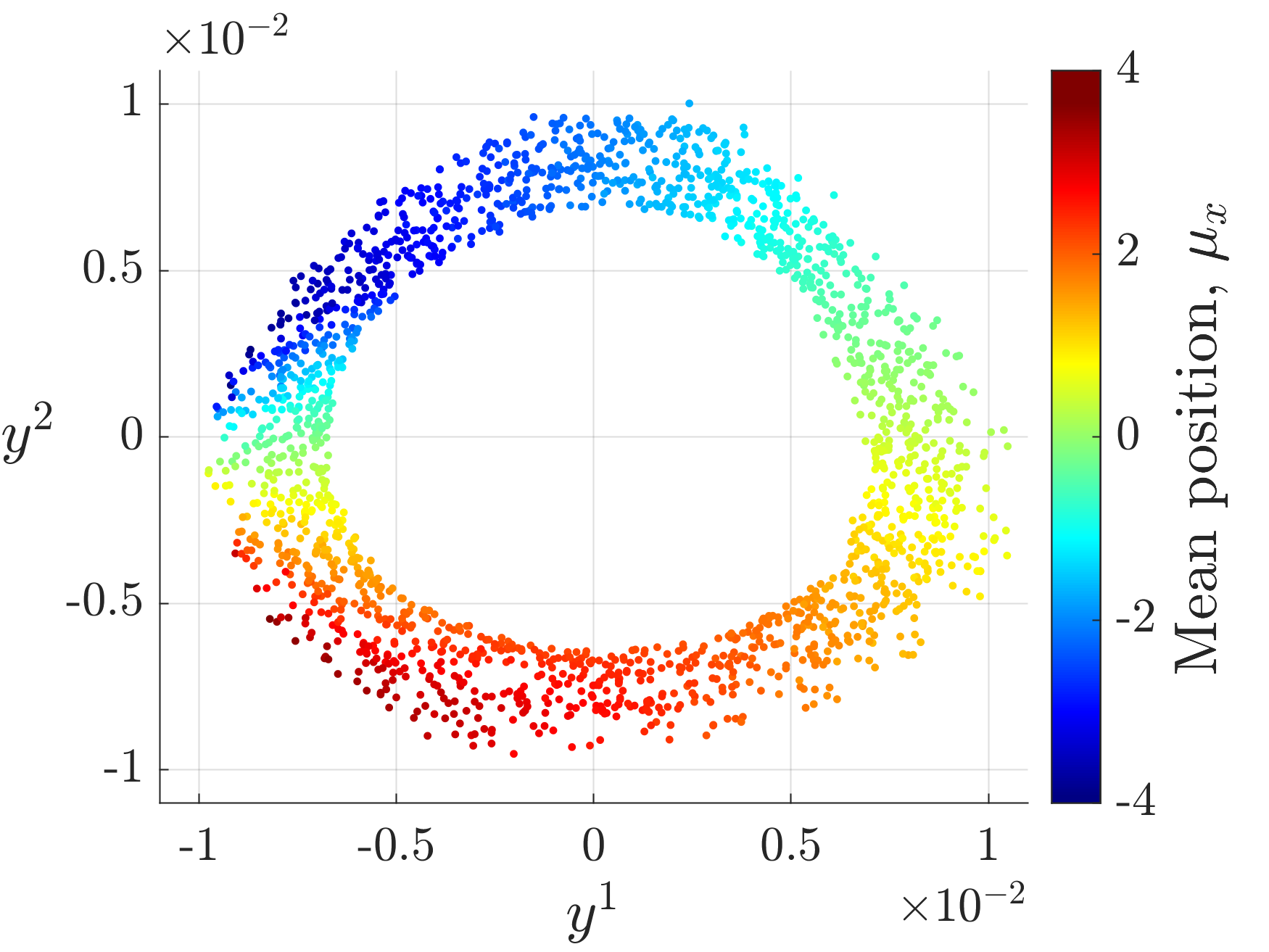}
        \caption{DM coordinates}
        \label{fig:LWR_DMcoords}
    \end{subfigure}

    % Second row
    \begin{subfigure}[b]{0.32\textwidth}
        \centering
        \includegraphics[width=\textwidth]{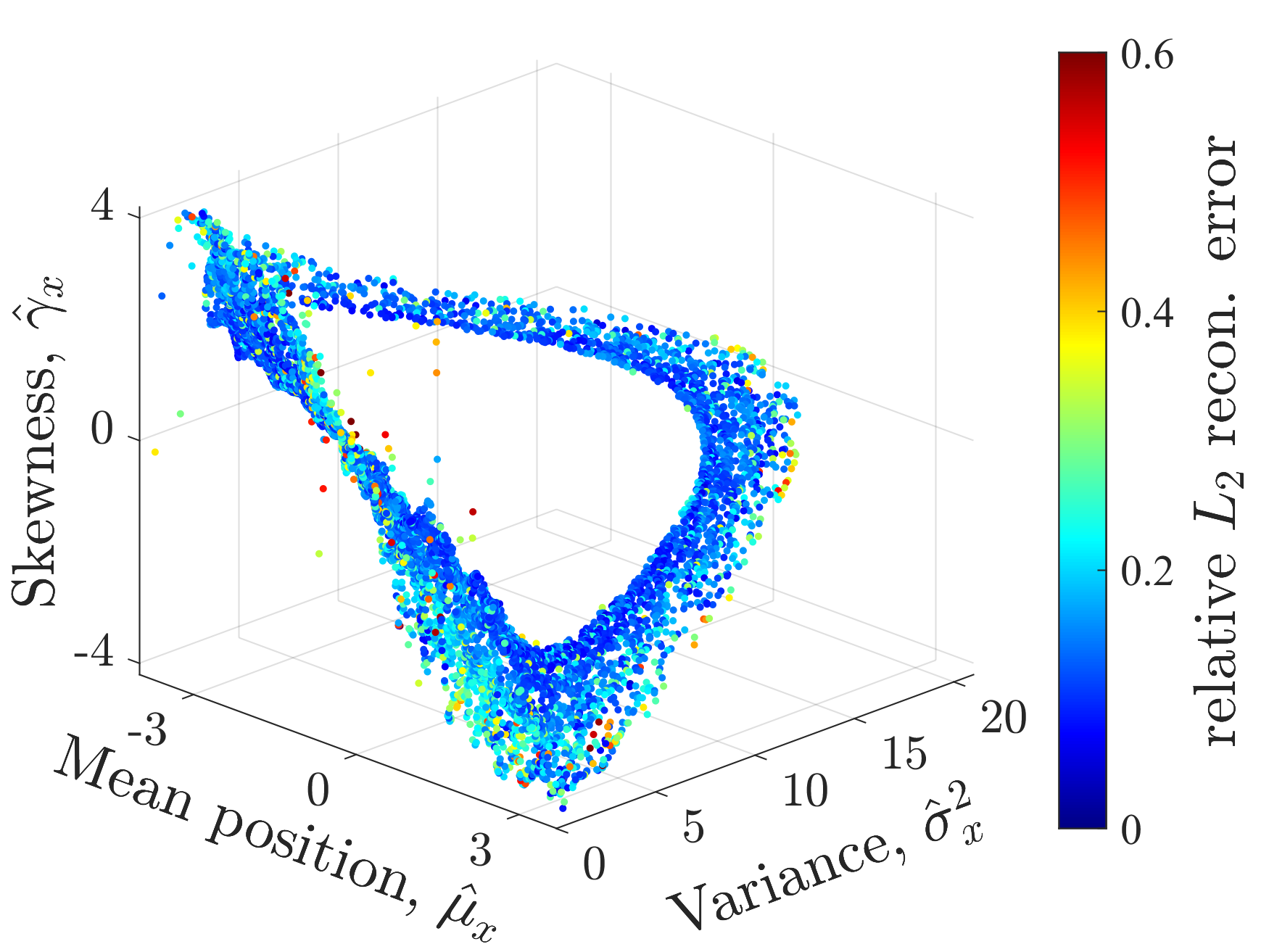}
        \caption{RANDSMAP-RFF ($P=N/2$)}
        \label{fig:LWR_DM_RFNNf_recon}
    \end{subfigure}
    \hfill
    \begin{subfigure}[b]{0.32\textwidth}
        \centering
        \includegraphics[width=\textwidth]{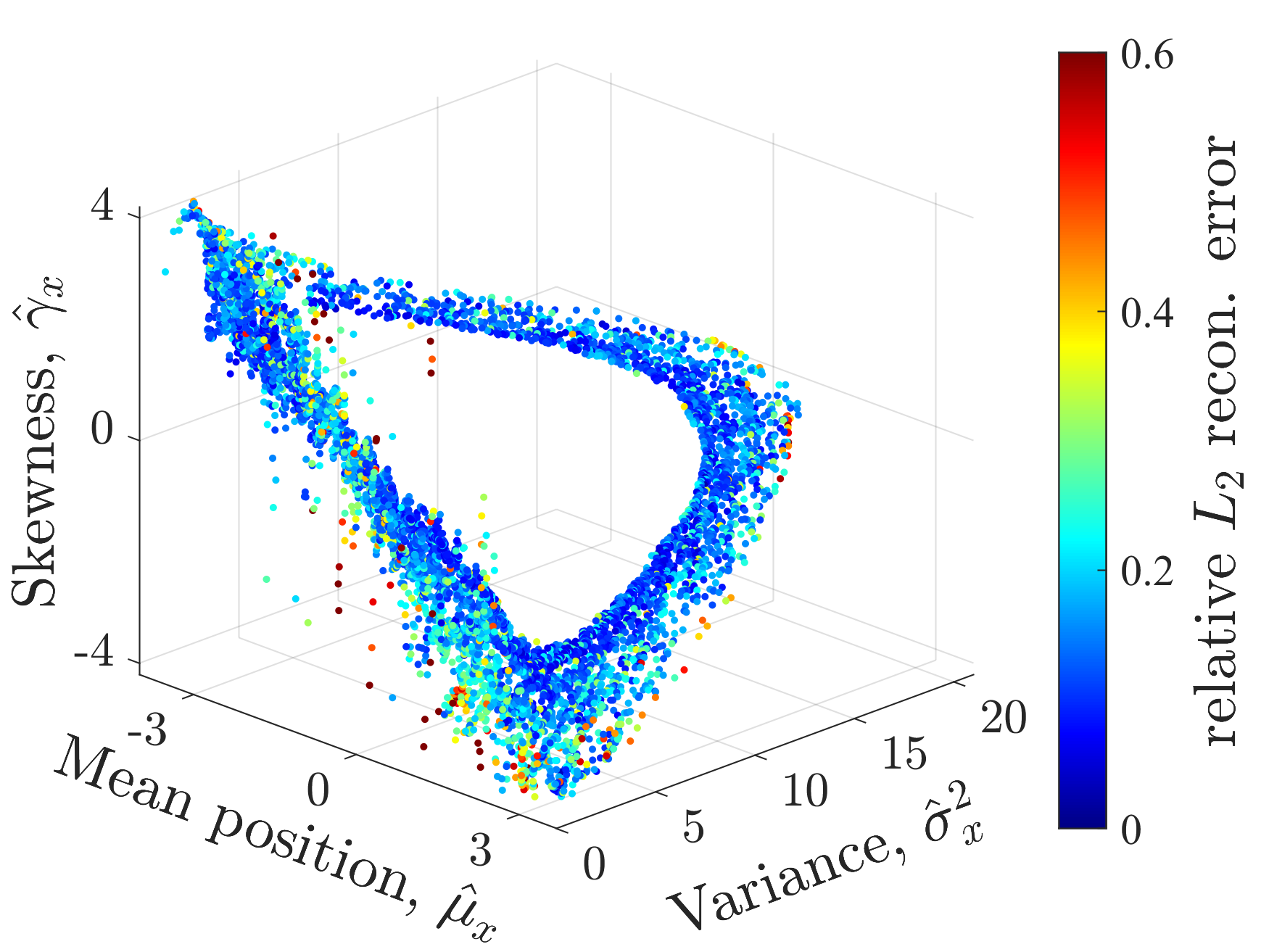}
        \caption{RANDSMAP-MS-RFF ($P=N$)}
        \label{fig:LWR_DM_RFNNfMK_recon}
    \end{subfigure}
    \hfill
    \begin{subfigure}[b]{0.32\textwidth}
        \centering
        \includegraphics[width=\textwidth]{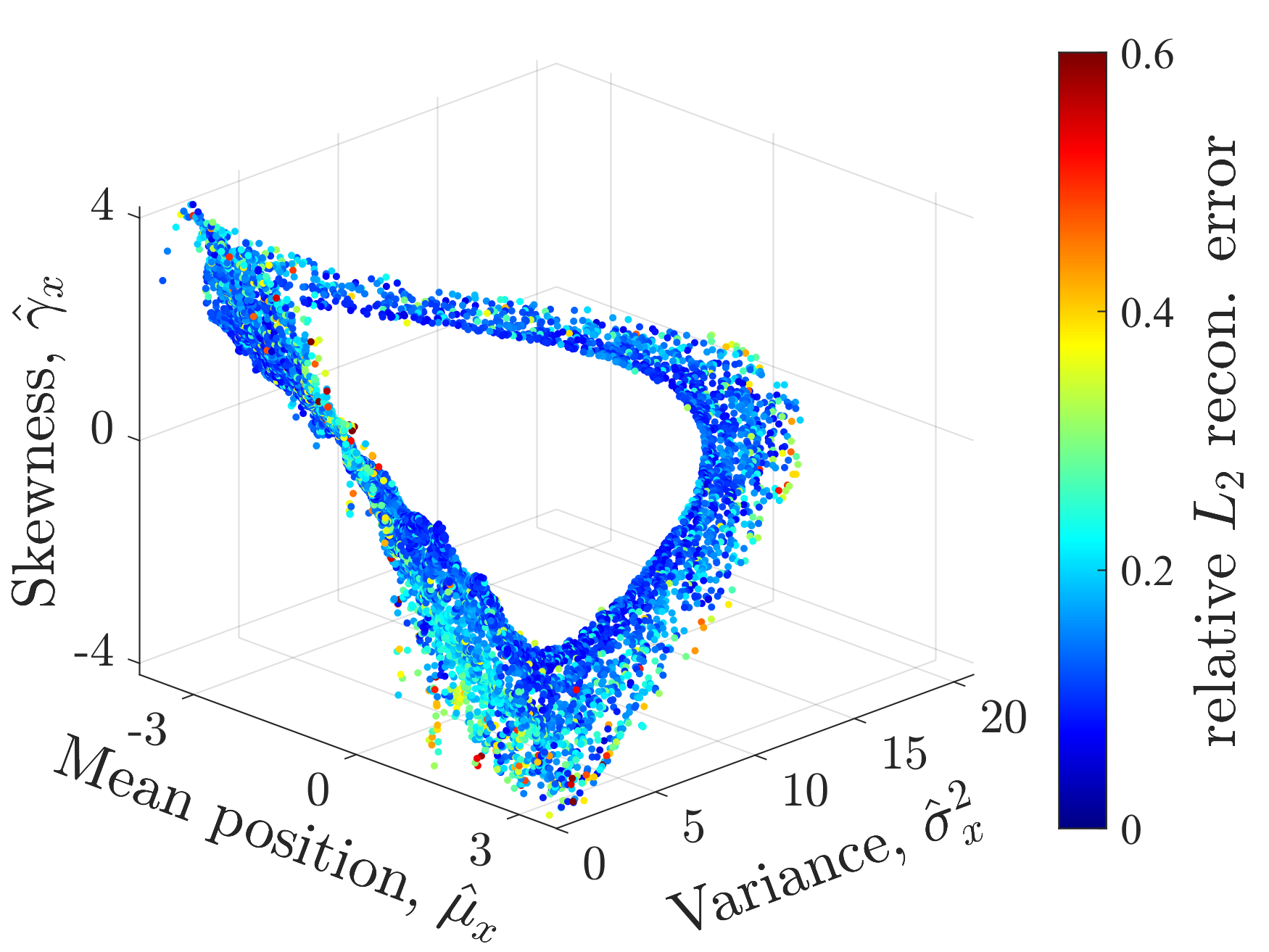}
        \caption{RANDSMAP-Sig ($P=N$)}
        \label{fig:LWR_DM_RFNNs_recon}
    \end{subfigure}

    % Third row
    \begin{subfigure}[b]{0.32\textwidth}
        \centering
        \includegraphics[width=\textwidth]{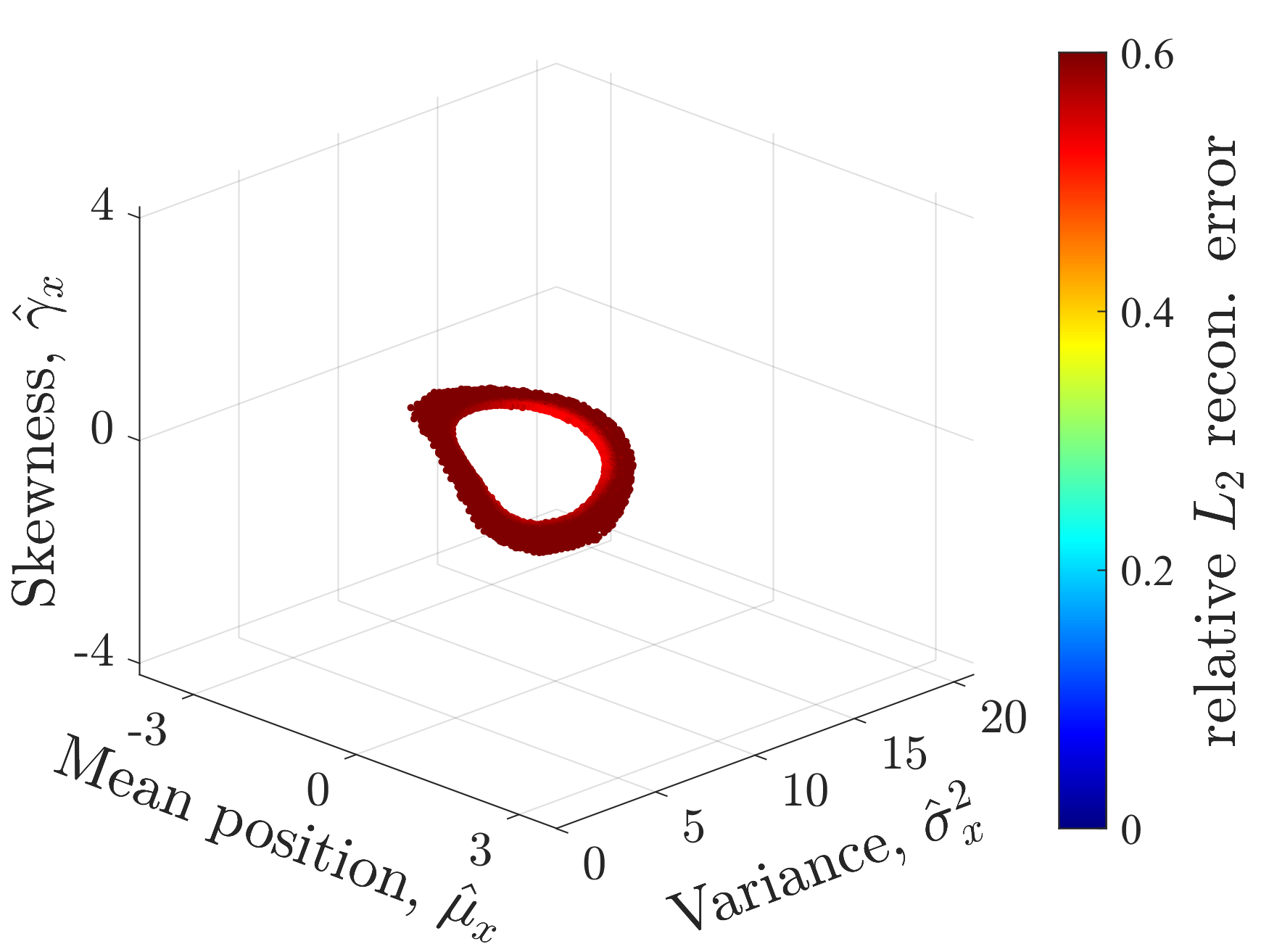}
        \caption{DDM}
        \label{fig:LWR_DM_GH_recon}
    \end{subfigure}\hspace{0.02\textwidth}
    \begin{subfigure}[b]{0.32\textwidth}
        \centering
        \includegraphics[width=\textwidth]{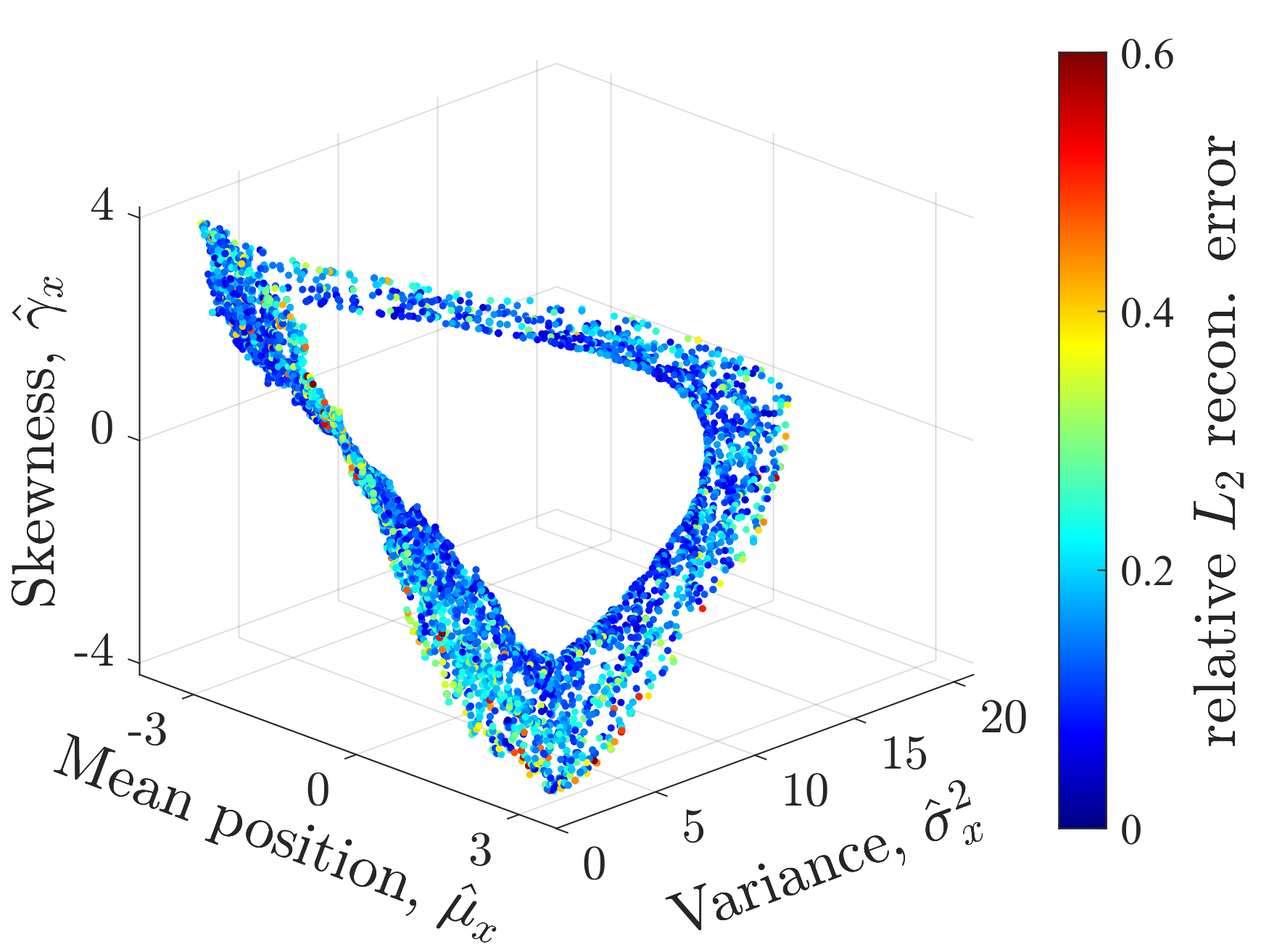}
        \caption{$k$-NN}
        \label{fig:LWR_DM_kNN_recon}
    \end{subfigure}
    \caption{LWR 1D traffic model dataset ($M=400$) with $N=2000$ training points. Because the density profiles are 400‑dimensional, we visualize them via three summary statistics: mean position $\mu_x$, variance $\sigma_x^2$, and skewness $\gamma_x$. Panel~\ref{fig:LWR_data} shows the training data represented in this statistics space, while panel~\ref{fig:LWR_DMcoords} shows the $d=2$ DM coordinates $[y^1, y^2]^\top$, colored by the mean position $\mu_x$ of each profile. Panels~\ref{fig:LWR_DM_RFNNf_recon}-\ref{fig:LWR_DM_kNN_recon} show the reconstructed density profiles projected onto the statistic space ($\mu_x$, $\sigma_x^2$, $\gamma_x$) and their per‑profile relative $L_2$ error $e_{2,i}$ (Eq.~\eqref{eq:recon_errors} for the five decoders: RANDSMAP-RFF with $P=N/2$, RANDSMAP-MS-RFF with $P=N$, RANDSMAP-Sig with $P=N$, DDM and $k$-NN. For the RANDSMAP decoders, the displayed configuration is the one achieving the lowest training reconstruction error among all tested variants (see Fig.~\ref{fig:LWR_dec_tr} for the full comparison).}
    \label{fig:LWR}
\end{figure}

\clearpage
\newpage
\renewcommand{\theequation}{H.\arabic{equation}}
\renewcommand{\thefigure}{H.\arabic{figure}}
\renewcommand{\thetable}{H.\arabic{table}}
\setcounter{equation}{0}
\setcounter{figure}{0}
\setcounter{table}{0}
\section{Reconstruction results of the 2D Rotated MRI Image Dataset (\texorpdfstring{$M=128\times128$}{M=128 x 128})} \label{app:MRI}

Here, we provide detailed reconstruction results for the 2D rotated MRI image dataset in Section~\ref{sb:MRI}. Figure~\ref{fig:MRIimage_data} provides indicative images of belonging to the generating dataset. Tables~\ref{tab:MRI_dec_tr_all} and \ref{tab:MRI_dec_ts_all} enlist the training and testing set performance of all decoders considered for this benchmark. Figure~\ref{fig:MRI_DM_GH} demonstrates that the poor performance of the DDM decoder is not due to poor tuning, via an exhaustive 2D grid search. Figure~\ref{fig:MRI} demonstrates per-image profile reconstructions provided by all decoders, ordered by the rotation angle $\theta$ used for constructing each rotated image. Finally, Fig.~\ref{fig:MRI_recon} shows a representative MRI image and its reconstructions with the five decoders considered, alongside absolute errors.

\begin{figure}[!htbp]
    \centering
    \includegraphics[width=\textwidth]{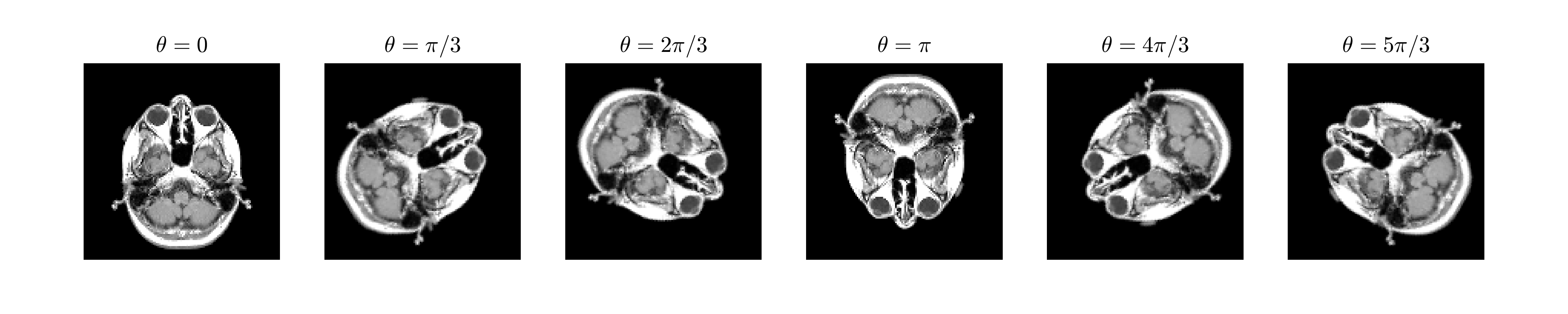}
    \caption{Representative rotated MRI images from the 2D rotated MRI images dataset ($M=128 \times 128$). Shown are samples at rotation angles $\theta = 0,\ \pi/3,\ 2\pi/3,\ \pi,\ 4\pi/3,\ 5\pi/3$.}
    \label{fig:MRIimage_data}
\end{figure}

\begin{table}[htbp]
\centering
\caption{Training set performance for the 2D rotated MRI images dataset ($M=128 \times 128$). Detailed reconstruction metrics for $N=720$ training points, using $d=2$-dim. DM embeddings. Decoders are compared based on the relative $L_2$ and $L_\infty$ mean reconstruction errors, $e_{2,i}$ and $e_{\infty,i}$ (Eq.~\eqref{eq:recon_errors}), computational time (in seconds), and mean conservation error $e_{con,i}$. For stochastic decoders (RANDSMAP-RFF, RANDSMAP-MS-RFF, RANDSMAP-Sig, and RFNN-Sig), metrics show the median with 5-95\% percentiles in parentheses, computed over 100 random initializations. Deterministic decoders (DDM, $k$-NN) report single values.}
\label{tab:MRI_dec_tr_all}
\resizebox{\textwidth}{!}{%
\begin{tabular}{@{}llcccc@{}}
\multirow{2}{*}{Decoder} & \multirow{2}{*}{$P$} & \multicolumn{2}{c}{Mean reconstruction error ($\times 10^{-2}$)} & Mean conservation & \multirow{2}{*}{Comp. Time (s)} \\
\cmidrule(lr){3-4}
& & relative $L_2$, $e_{2,i}$ &  relative $L_\infty$, $e_{\infty,i}$ & error, $e_{con,i}$ ($\times 10^{-7}$) & \\
\midrule
\midrule
% DM-RFNNf variants
\multirow{3}{*}{RANDSMAP-RFF} 
&	$N$	    & 3.296 (3.201--3.405) & 12.152 (11.583--12.636) & 5.066 (3.770--6.224)                    & 6.634 (6.604--6.653) \\
&	$N/2$	& 4.052 (3.933--4.254) & 13.607 (12.996--14.207) & 6.326 (2.242--28.527) $\times 10^{-6}$  & 4.457 (4.449--4.481) \\
&	$N/4$	& 5.717 (5.398--6.342) & 20.929 (20.365--21.479) & 2.413 (2.121--2.941) $\times 10^{-8}$   & 3.410 (3.405--3.416) \\
\midrule

% DM-RFNNfMK variants  
\multirow{3}{*}{RANDSMAP-MS-RFF} 
&	$N$	     & 2.782 (2.687--3.464)  &  8.981 (8.660--10.465)  & 4.216 (2.321--20.225)                  & 6.709 (6.685--6.736) \\
&	$N/2$	 & 3.810 (3.605--4.864)  & 12.025 (11.153--14.395) & 0.131 (0.003--13.461) $\times 10^{-1}$ & 4.525 (4.513--4.549) \\
&	$N/4$	 & 6.079 (5.497--10.978) & 22.550 (20.315--27.916) & 0.759 (0.022--5.087)  $\times 10^{-6}$ & 3.591 (3.580--3.613) \\
\midrule

% DM-RFNNs variants  
\multirow{3}{*}{RANDSMAP-Sig} 
&	$N$	    & 2.882 (2.704--3.043) & 11.537 (10.712--12.263) & 7.288 (2.981--15.268)                  & 6.945 (6.928--6.971) \\
&	$N/2$	& 4.127 (3.872--4.351) & 16.557 (15.512--17.429) & 2.588 (0.341--6.724) $\times 10^{-1}$  & 4.474 (4.465--4.498) \\
&	$N/4$	& 6.222 (5.939--6.561) & 26.943 (25.810--28.007) & 6.652 (0.089--20.264) $\times 10^{-2}$ & 3.438 (3.430--3.464) \\
\midrule

% DM-RFNNs variants  
\multirow{3}{*}{RFNN-Sig} 
&	$N$	     & 2.882 (2.704--3.043) & 11.537 (10.712--12.263) & 1.860 (1.404--2.184) $\times 10^3$ & 5.404 (5.385--5.428) \\
&	$N/2$	 & 4.127 (3.872--4.351) & 16.557 (15.512--17.429) & 1.488 (1.107--1.821) $\times 10^3$ & 3.438 (3.433--3.450) \\
&	$N/4$	 & 6.222 (5.939--6.561) & 26.943 (25.810--28.007) & 1.233 (0.853--1.820) $\times 10^3$ & 2.544 (2.541--2.553) \\
\midrule

% Deterministic methods
DDM  &	-	& 34.940 & 82.110 & 5.964 $\times 10^{5}$  & 4.629    \\
$k$-NN &	-	& 2.230  & 9.686  & 6.273 $\times 10^{-9}$ & 1857.641 \\

\bottomrule
\end{tabular}%
}
\end{table}

\begin{table}[htbp]
\centering
\caption{Testing set performance for the 2D rotated MRI images dataset ($M=128 \times 128$). Detailed reconstruction metrics for $N=720$ training points, using $d=2$-dim. DM embeddings. Decoders are compared based on the relative $L_2$ and $L_\infty$ mean reconstruction errors, $e_{2,i}$ and $e_{\infty,i}$ (Eq.~\eqref{eq:recon_errors}), computational time (in seconds), and mean conservation error $e_{con,i}$. For stochastic decoders (RANDSMAP-RFF, RANDSMAP-MS-RFF, RANDSMAP-Sig, and RFNN-Sig), metrics show the median with 5-95\% percentiles in parentheses, computed over 100 random initializations. Deterministic decoders (DDM, $k$-NN) report single values.}
\label{tab:MRI_dec_ts_all}
\resizebox{\textwidth}{!}{%
\begin{tabular}{@{}llcccc@{}}
\multirow{2}{*}{Decoder} & \multirow{2}{*}{$P$} & \multicolumn{2}{c}{Mean reconstruction error ($\times 10^{-2}$)} & Mean conservation & \multirow{2}{*}{Comp. Time (s)} \\
\cmidrule(lr){3-4}
& & relative $L_2$, $e_{2,i}$ &  relative $L_\infty$, $e_{\infty,i}$ & error, $e_{con,i}$ ($\times 10^{-7}$) & \\
\midrule
\midrule

% DM-RFNNf variants
\multirow{3}{*}{RANDSMAP-RFF} 
&	$N$	    & 4.015 (3.929--4.109) & 16.305 (15.912--16.633) & 6.985 (4.804--9.561)                   & 1.479 (1.393--1.534) \\
&	$N/2$	& 4.903 (4.787--5.087) & 17.727 (17.190--18.264) & 7.874 (2.653--34.676) $\times 10^{-6}$ & 1.290 (1.233--1.353) \\
&	$N/4$	& 6.882 (6.543--7.583) & 26.490 (26.011--27.036) & 2.479 (2.233--3.035) $\times 10^{-8}$  & 1.205 (1.144--1.273) \\
\midrule

% DM-RFNNfMK variants  
\multirow{3}{*}{RANDSMAP-MS-RFF} 
&	$N$	     & 3.910 (3.742--5.029)   & 13.696 (13.270--14.968) & 6.441 (3.075--33.768)                  & 1.511 (1.421--1.591) \\
&	$N/2$	 & 5.117 (4.748--6.791)   & 17.259 (16.048--21.134) & 0.193 (0.004--18.154) $\times 10^{-1}$ & 1.329 (1.258--1.400) \\
&	$N/4$	 & 7.603 (6.832--14.111)  & 29.002 (26.491--34.696) & 0.937 (0.023--6.634) $\times 10^{-6}$  & 1.220 (1.141--1.287) \\
\midrule

% DM-RFNNs variants  
\multirow{3}{*}{RANDSMAP-Sig} 
&	$N$	     & 3.644 (3.482--3.802) & 15.699 (15.064--16.188) & 10.991 (4.381--25.037)                 & 1.501 (1.423--1.562) \\
&	$N/2$	 & 5.034 (4.804--5.260) & 20.820 (19.899--21.598) & 2.724 (0.340--6.816) $\times 10^{-1}$  & 1.315 (1.234--1.376) \\
&	$N/4$	 & 7.367 (7.086--7.726) & 32.468 (31.451--33.700) & 7.165 (0.101--19.934) $\times 10^{-2}$ & 1.206 (1.143--1.273) \\
\midrule

% DM-RFNNs variants  
\multirow{3}{*}{RFNN-Sig} 
&	$N$	     & 3.644 (3.482--3.802) & 15.699 (15.064--16.188) & 2.153 (1.708--2.527) $\times 10^3$ & 1.504 (1.439--1.567) \\
&	$N/2$	 & 5.034 (4.804--5.260) & 20.820 (19.899--21.598) & 1.720 (1.358--2.072) $\times 10^3$ & 1.337 (1.257--1.409) \\
&	$N/4$	 & 7.367 (7.086--7.726) & 32.468 (31.451--33.700) & 1.393 (0.985--2.144) $\times 10^3$ & 1.245 (1.174--1.319) \\
\midrule

% Deterministic methods
DDM  &	-	& 35.561 & 83.406 & 6.747 $\times 10^{5}$  & 1.310 \\
$k$-NN &	-	& 4.318  & 20.196 & 5.791 $\times 10^{-9}$ & 253.217 \\

\bottomrule
\end{tabular}%
}
\end{table}

\begin{figure}[!htbp]
    \centering
    % First row
    \begin{subfigure}[b]{0.32\textwidth}
        \centering
        \includegraphics[width=\textwidth]{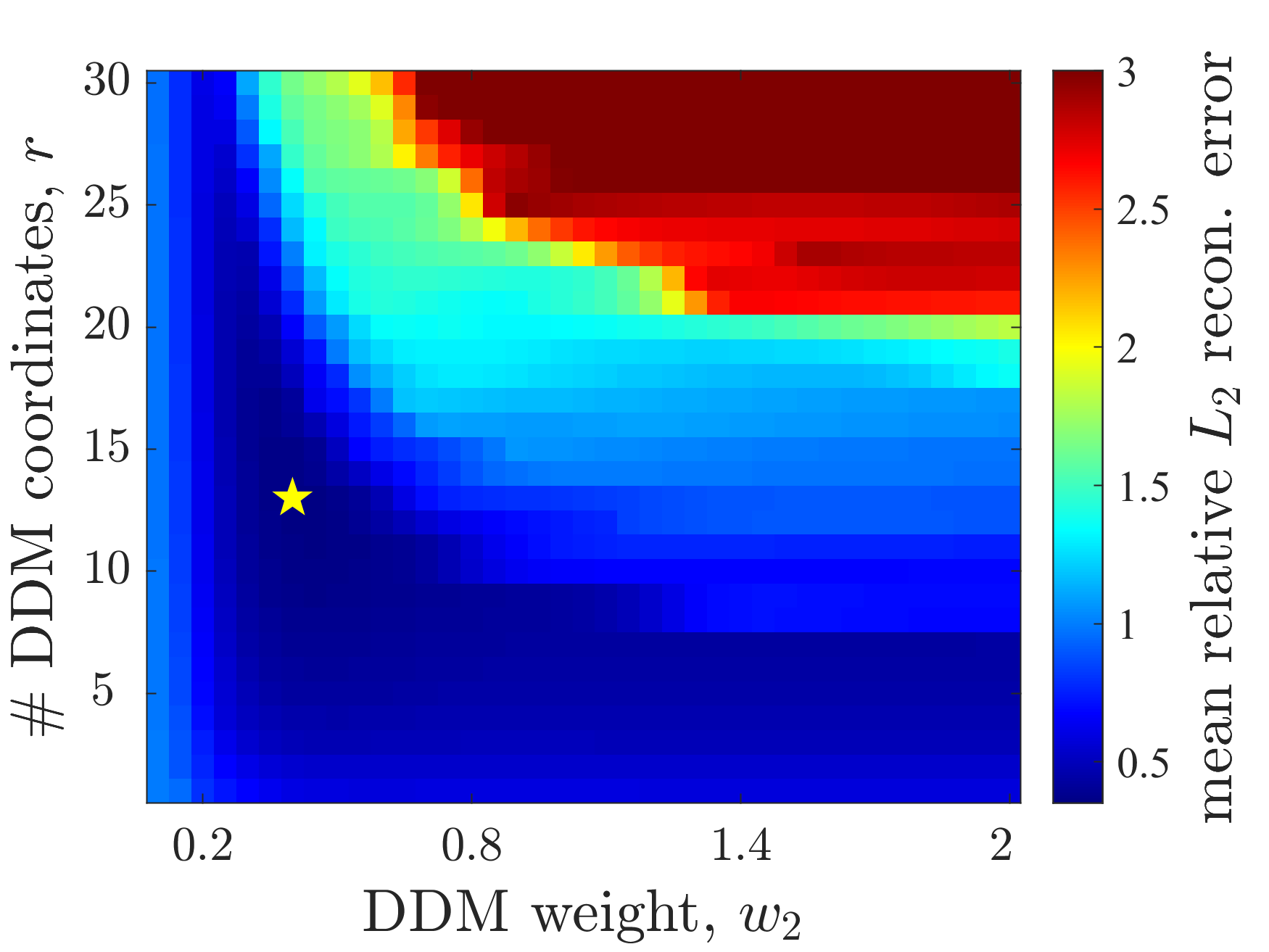}
        \caption{Exhaustive ($w_2,r$) tuning}
        \label{fig:MRI_DM_GH_tuning}
    \end{subfigure}
    \hfill
    \begin{subfigure}[b]{0.32\textwidth}
        \centering
        \includegraphics[width=\textwidth]{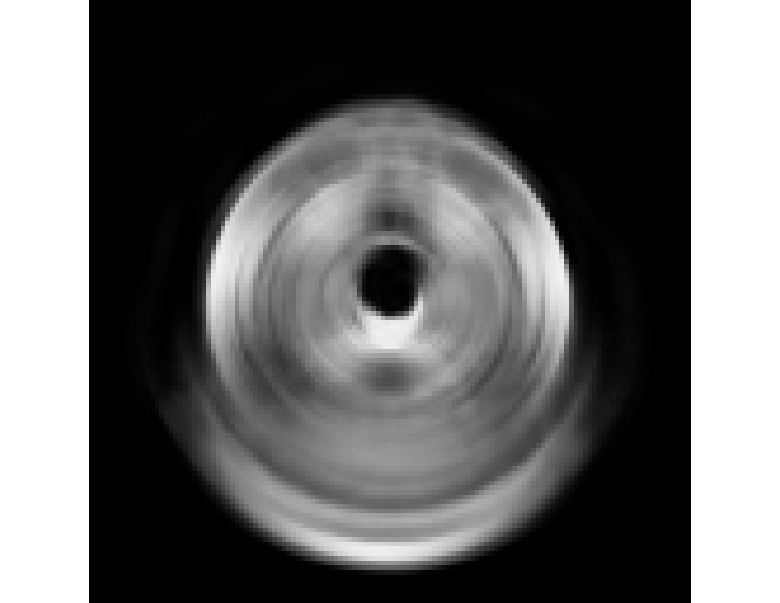}
        \caption{Reconstruction, $(w_2,r)=(0.4,13)$}
        \label{fig:MRI_DM_GH_k13}
    \end{subfigure}
    \hfill
    \begin{subfigure}[b]{0.32\textwidth}
        \centering
        \includegraphics[width=\textwidth]{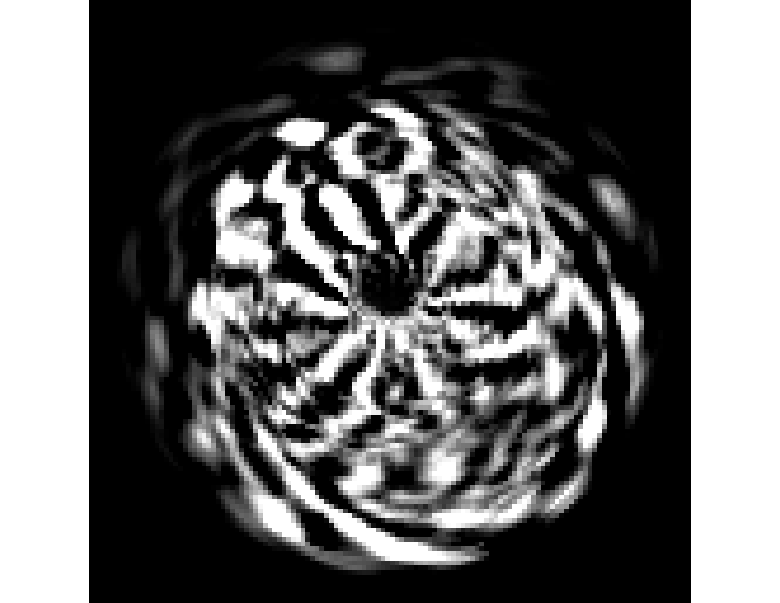}
        \caption{Reconstruction, $(w_2,r)=(0.4,25)$}
        \label{fig:MRI_DM_GH_k25}
    \end{subfigure}
    \caption{Reconstruction with the DDM decoder for the 2D rotated MRI images dataset ($M=128 \times 128$). Panel~\ref{fig:MRI_DM_GH_tuning} shows exhaustive hyperparameter tuning over the DDM weight $w_2$ and the number of DDM coordinates $r$. The yellow star indicates the optimal combination $(w_2,r)=(0.4,13)$. Panels \ref{fig:MRI_DM_GH_k13} and \ref{fig:MRI_DM_GH_k25} show the reconstructed MRI slice (corresponding to the original MRI slice in Fig.~\ref{fig:MRIslice_orig}) obtained with $w_2=0.4$ and $r=13$ (optimal) or $r=25$ DDM coordinates, respectively.}
    \label{fig:MRI_DM_GH}
\end{figure}

\begin{figure}[!htbp]
    \centering
    % First row
    \begin{subfigure}[b]{0.32\textwidth}
        \centering
        \includegraphics[width=\textwidth]{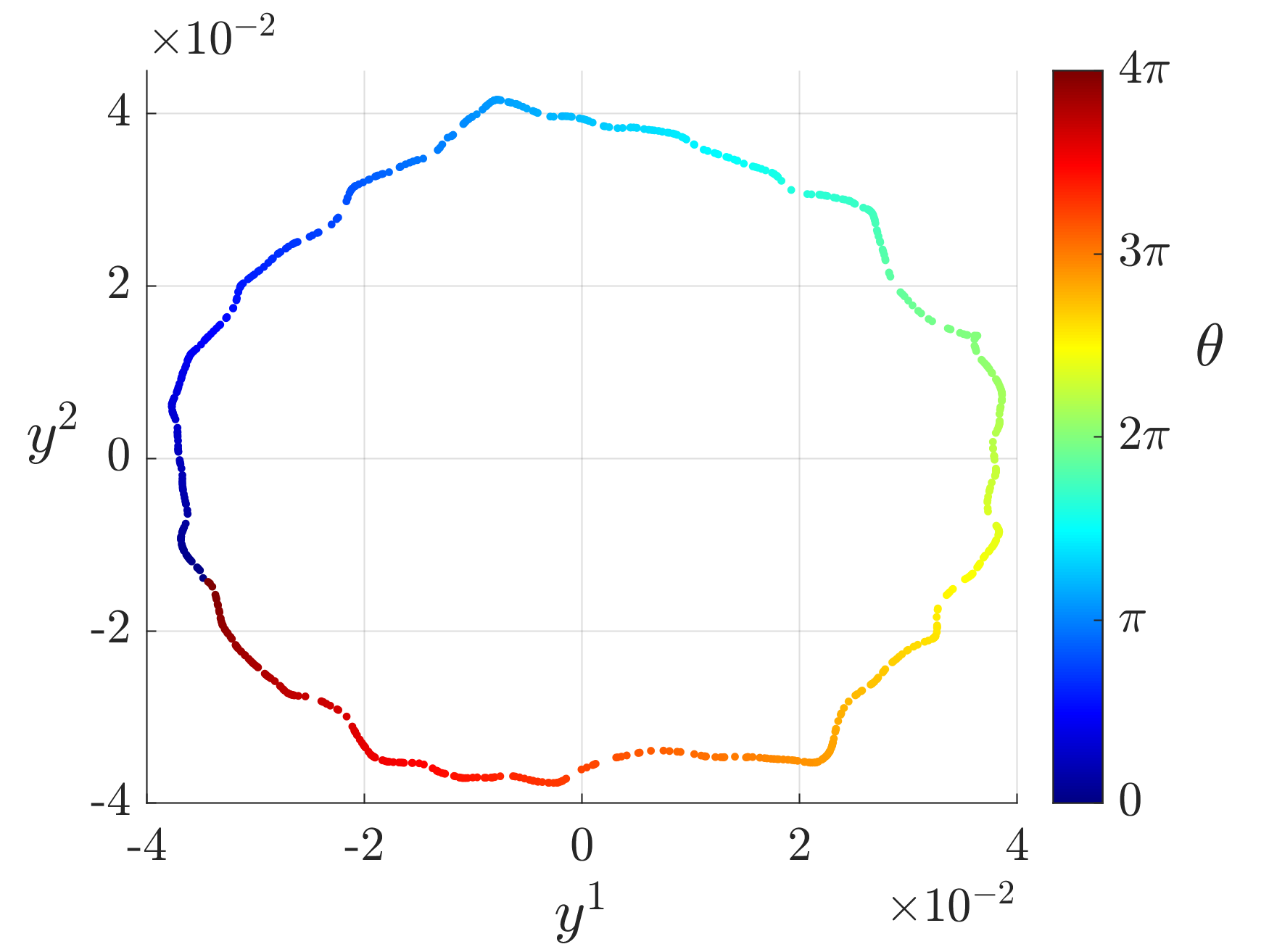}
        \caption{DM coordinates}
        \label{fig:MRI_DMcoords}
    \end{subfigure} 
    \hfill
    \begin{subfigure}[b]{0.32\textwidth}
        \centering
        \includegraphics[width=\textwidth]{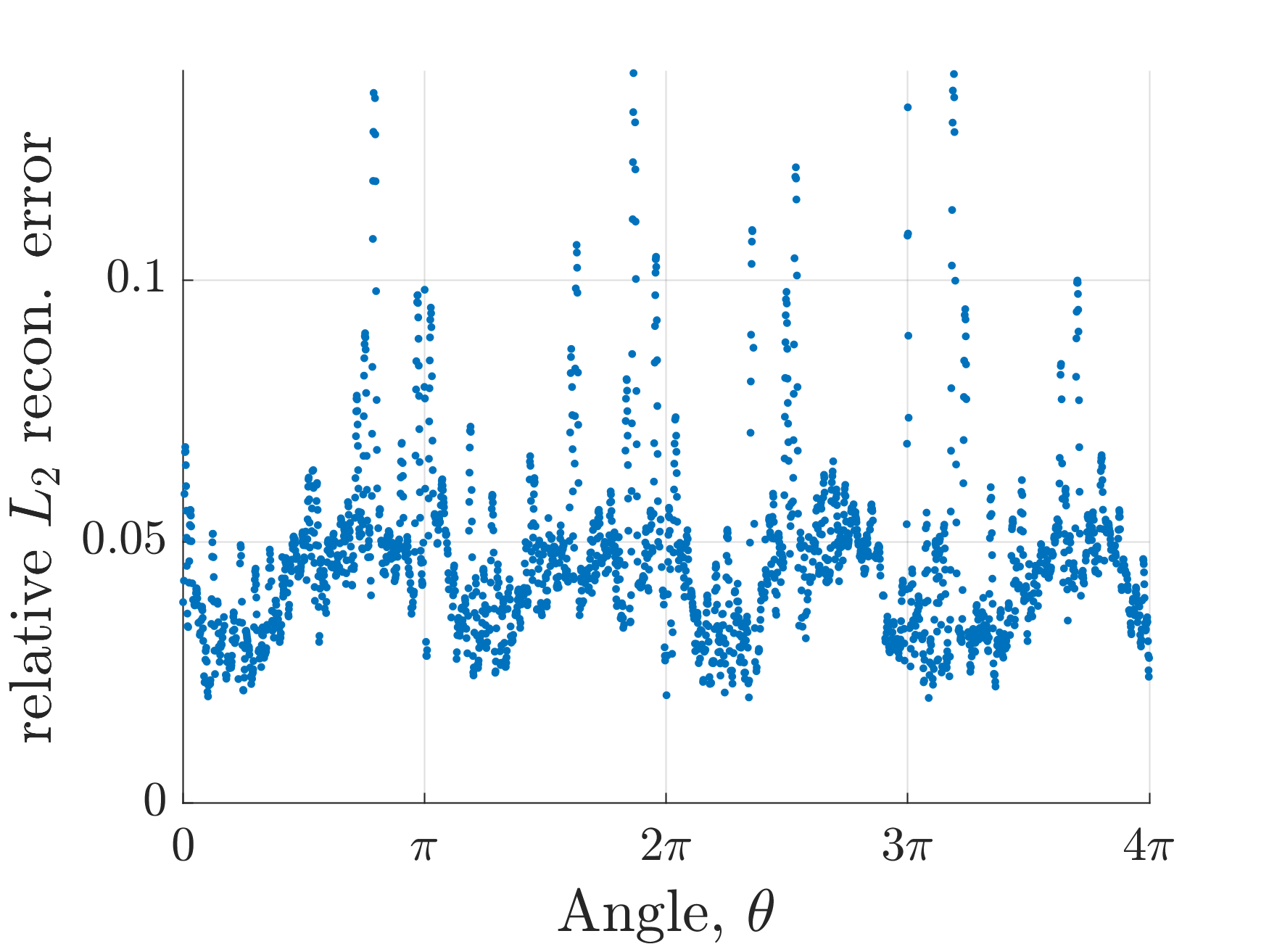}
        \caption{RANDSMAP-RFF ($P=N$)}
        \label{fig:MRI_DM_RFNNf_recon}
    \end{subfigure}
    \hfill
    \begin{subfigure}[b]{0.32\textwidth}
        \centering
        \includegraphics[width=\textwidth]{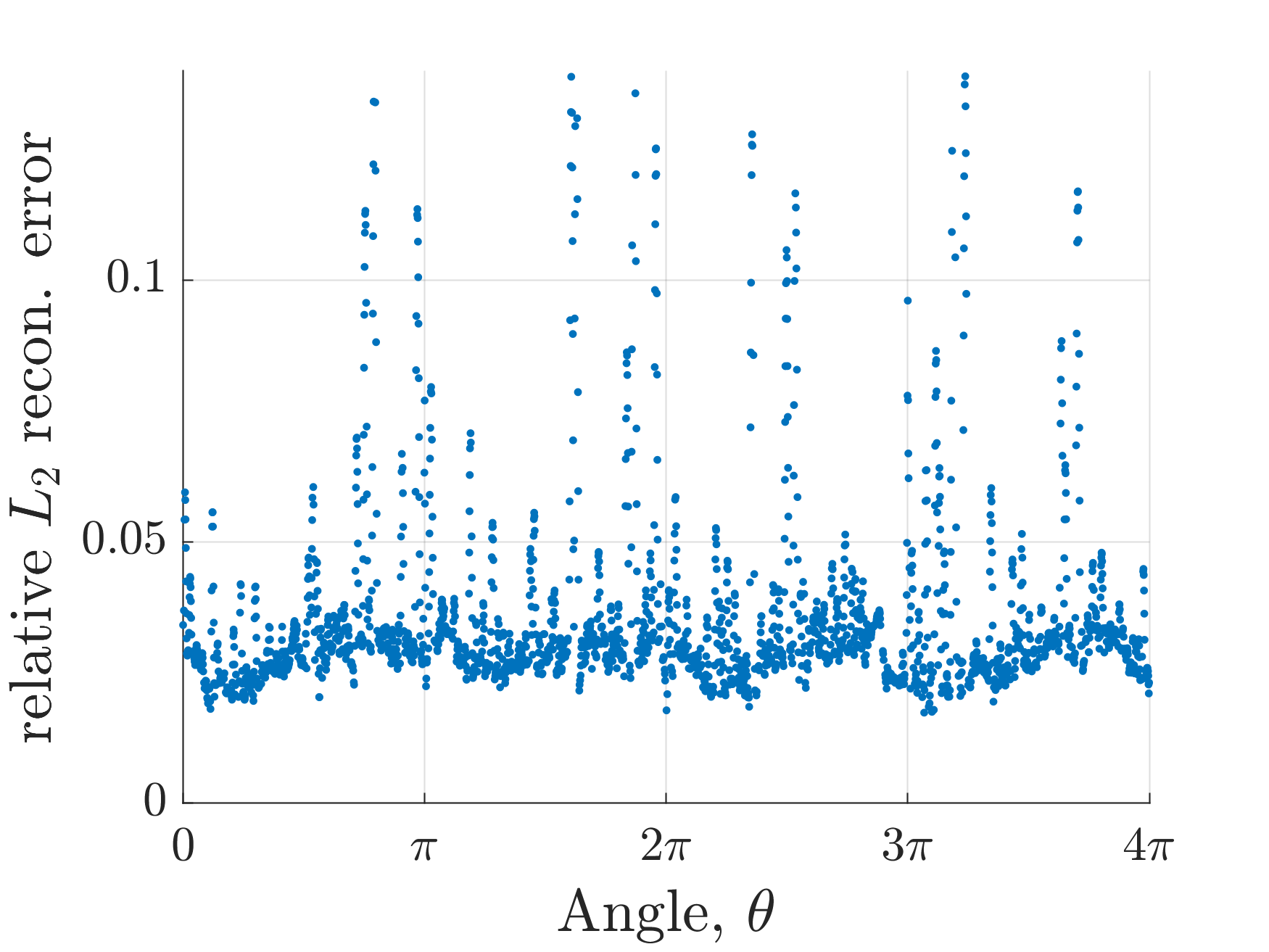}
        \caption{RANDSMAP-MS-RFF ($P=N$)}
        \label{fig:MRI_DM_RFNNfMK_recon}
    \end{subfigure}

    % Second row
    \begin{subfigure}[b]{0.32\textwidth}
        \centering
        \includegraphics[width=\textwidth]{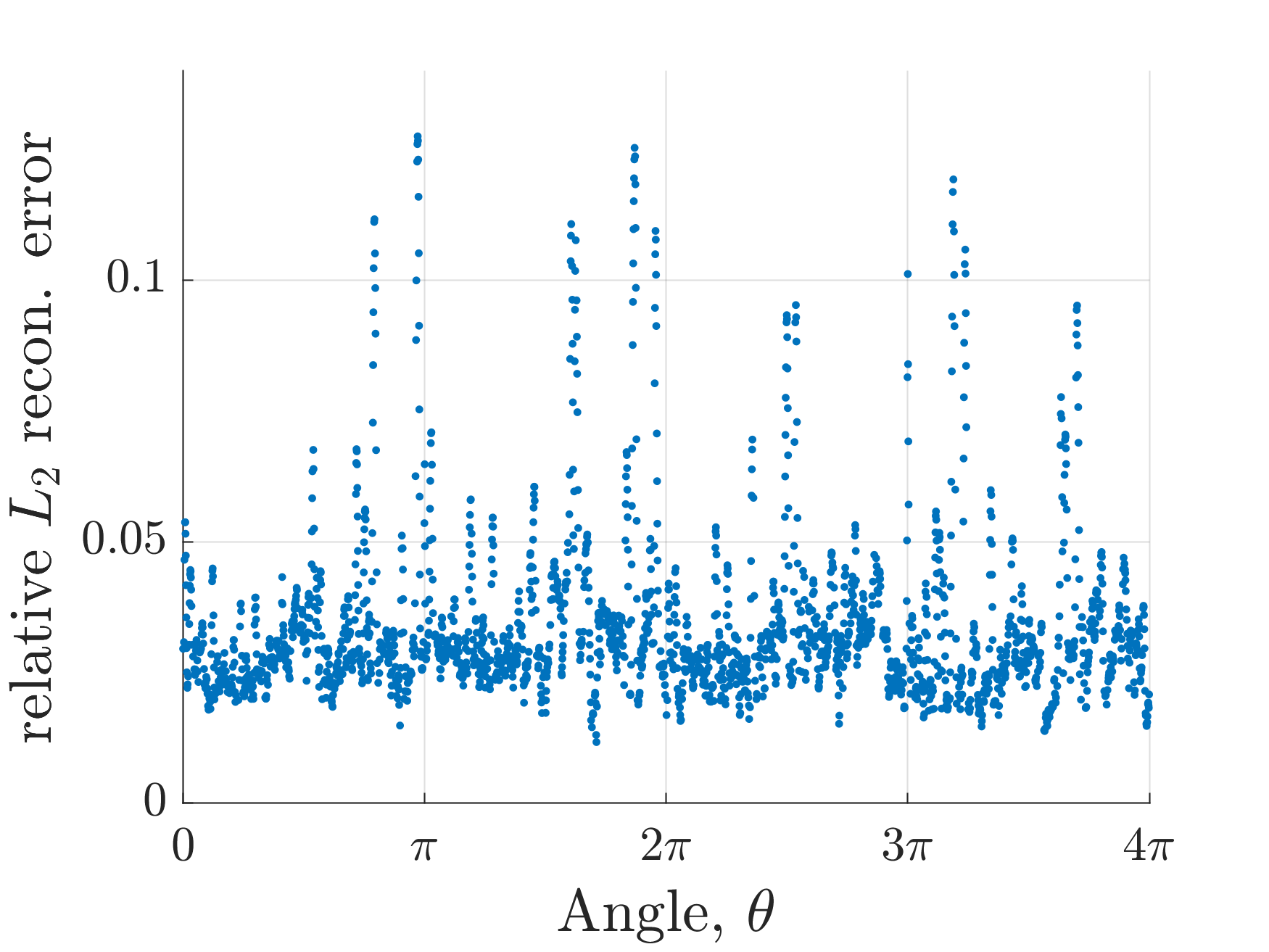}
        \caption{RANDSMAP-Sig ($P=N$)}
        \label{fig:MRI_DM_RFNNs_recon}
    \end{subfigure}
    \hfill
    \begin{subfigure}[b]{0.32\textwidth}
        \centering
        \includegraphics[width=\textwidth]{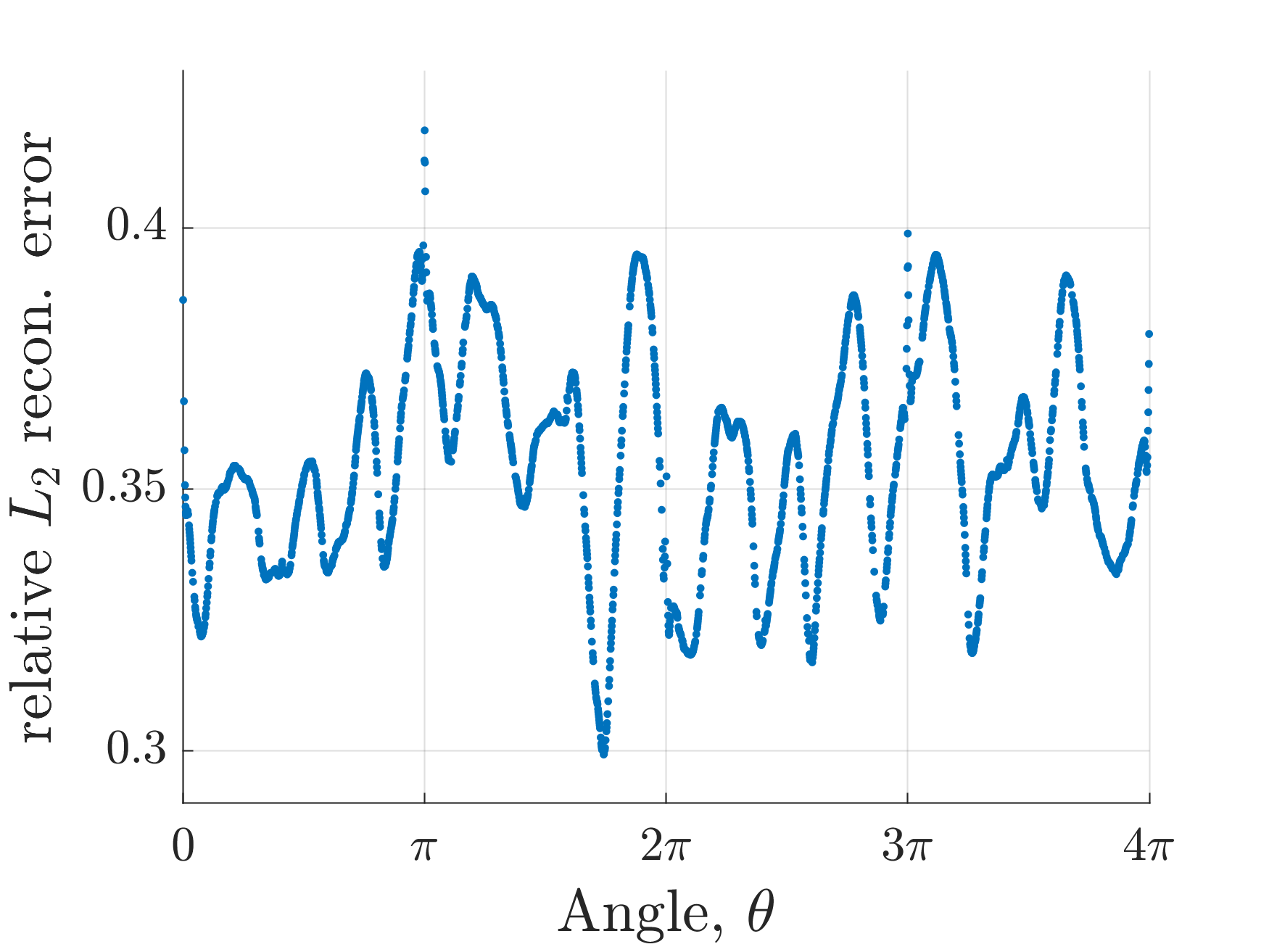}
        \caption{DDM}
        \label{fig:MRI_DM_GH_recon}
    \end{subfigure}
    \hfill
    \begin{subfigure}[b]{0.32\textwidth}
        \centering
        \includegraphics[width=\textwidth]{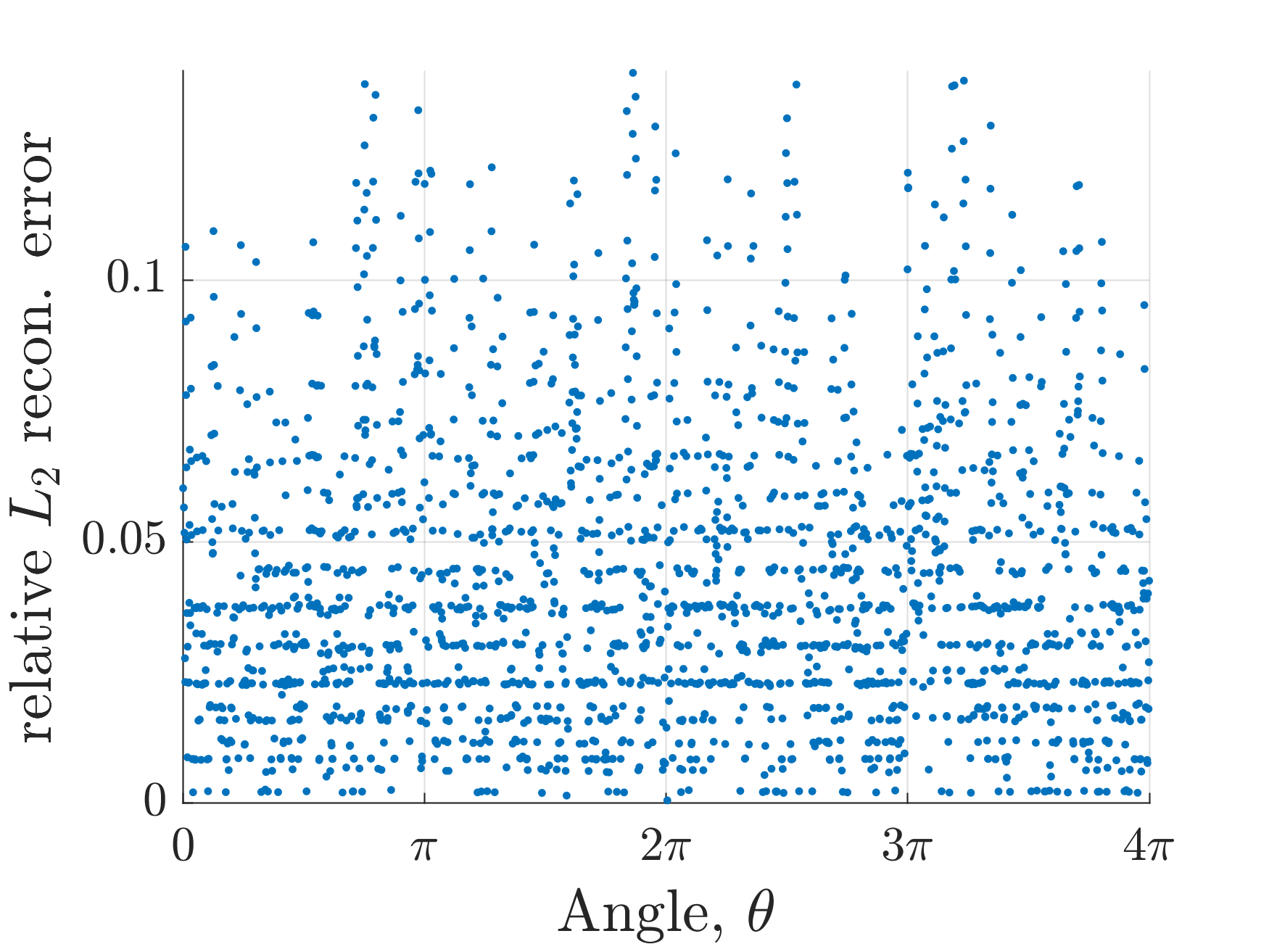}
        \caption{$k$-NN}
        \label{fig:MRI_DM_kNN_recon}
    \end{subfigure}
    \caption{2D rotated MRI images dataset ($M=128 \times 128$) with $N=720$ training points. Panel~\ref{fig:MRI_DMcoords} shows the $d=2$ DM coordinates $[y^1, y^2]^\top$, colored by the rotation angle $\theta$ of each MRI image. Panels~\ref{fig:MRI_DM_RFNNf_recon}-\ref{fig:MRI_DM_kNN_recon} show the per-image relative $L_2$ reconstruction errors versus the rotation angle $\theta$ for the five decoders: RANDSMAP-RFF with $P=N$, RANDSMAP-MS-RFF $P=N$, RANDSMAP-Sig with $P=N$, DDM, and $k$-NN. For the RANDSMAP decoders, the displayed configuration is the one achieving the lowest training reconstruction error among all tested variants (see Fig.~\ref{fig:MRI_dec_tr} for the full comparison).}
    \label{fig:MRI}
\end{figure}

\begin{figure}[!htbp]
    \centering
    % First row
    \begin{subfigure}[b]{0.3\textwidth}
        \centering
        \includegraphics[width=0.8\textwidth]{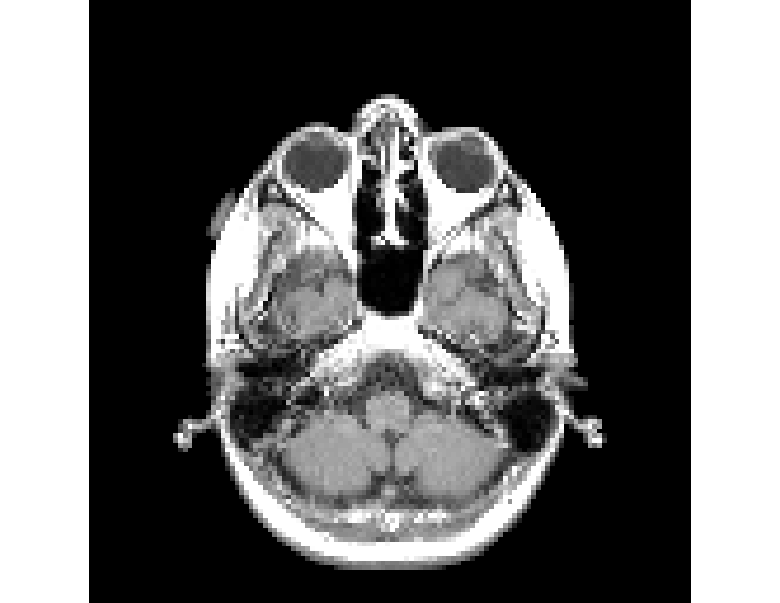}
        \caption{Original slice}
        \label{fig:MRIslice_orig}
    \end{subfigure}
    \hspace{2pt}
    \begin{subfigure}[b]{0.3\textwidth}
        \centering
        \includegraphics[width=0.8\textwidth]{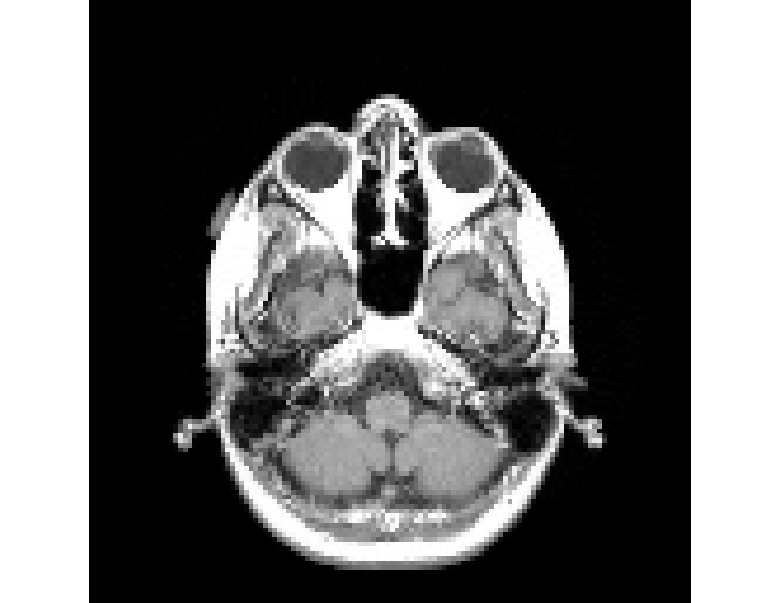}
        \caption{RANDSMAP-RFF ($P=N$)}
        \label{fig:MRIslice_RFNNf_recon}
    \end{subfigure}
    \hspace{2pt}
    \begin{subfigure}[b]{0.3\textwidth}
        \centering
        \includegraphics[width=0.8\textwidth]{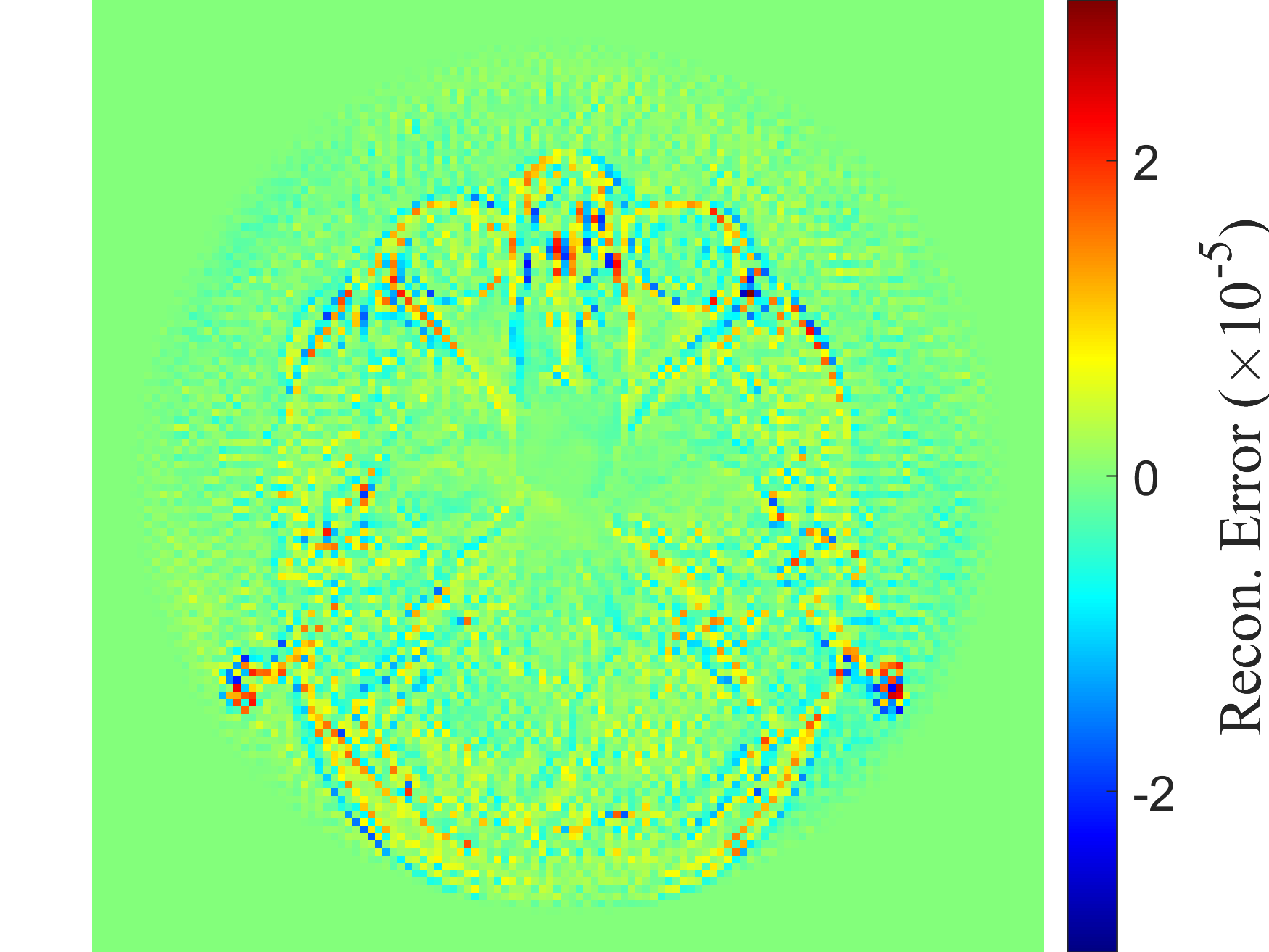}
        \caption{Absolute recon. error ($\times 10^{-5}$)}
        \label{fig:MRIslice_RFNNf_recon_err}
    \end{subfigure}
    \\
    % Second row
    \begin{subfigure}[b]{0.3\textwidth}
        \centering
        \includegraphics[width=0.8\textwidth]{FigsMRI/SliceOrig.png}
        \caption{Original slice}
    \end{subfigure}
    \hspace{2pt}
    \begin{subfigure}[b]{0.3\textwidth}
        \centering
        \includegraphics[width=0.8\textwidth]{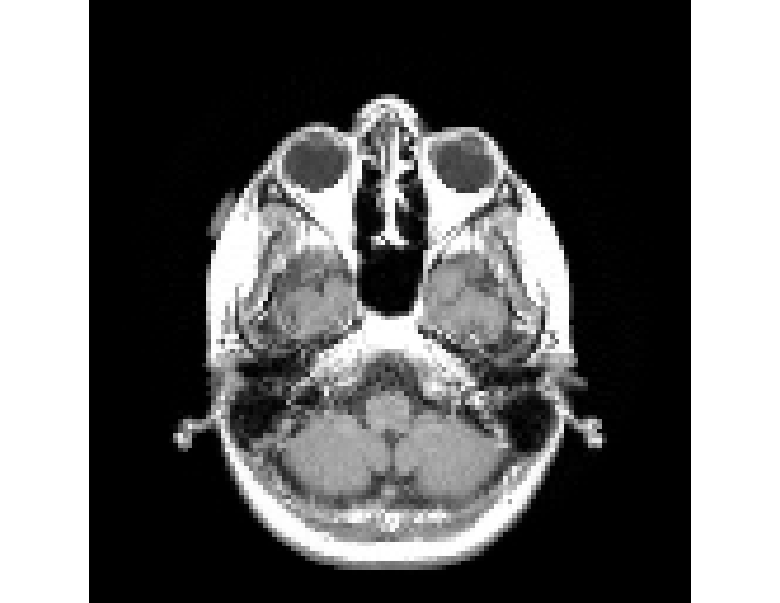}
        \caption{RANDSMAP-MS-RFF ($P=N$)}
        \label{fig:MRIslice_RFNNfMK_recon}
    \end{subfigure}
    \hspace{2pt}
    \begin{subfigure}[b]{0.3\textwidth}
        \centering
        \includegraphics[width=0.8\textwidth]{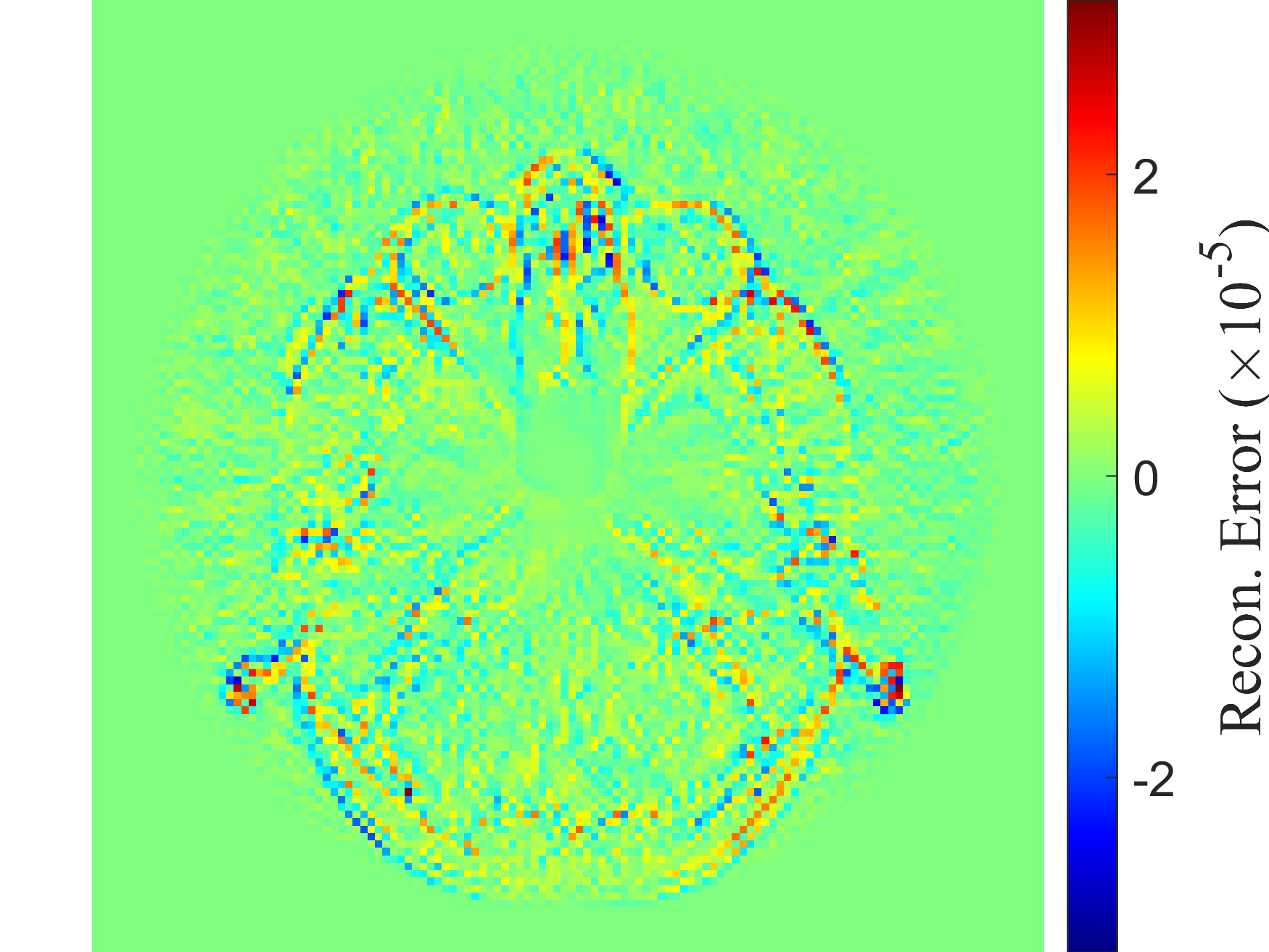}
        \caption{Absolute recon. error ($\times 10^{-5}$)}
        \label{fig:MRIslice_RFNNfMK_recon_err}
    \end{subfigure}
    \\
    % Third row
    \begin{subfigure}[b]{0.3\textwidth}
        \centering
        \includegraphics[width=0.8\textwidth]{FigsMRI/SliceOrig.png}
        \caption{Original slice}
    \end{subfigure}
    \hspace{2pt}
    \begin{subfigure}[b]{0.3\textwidth}
        \centering
        \includegraphics[width=0.8\textwidth]{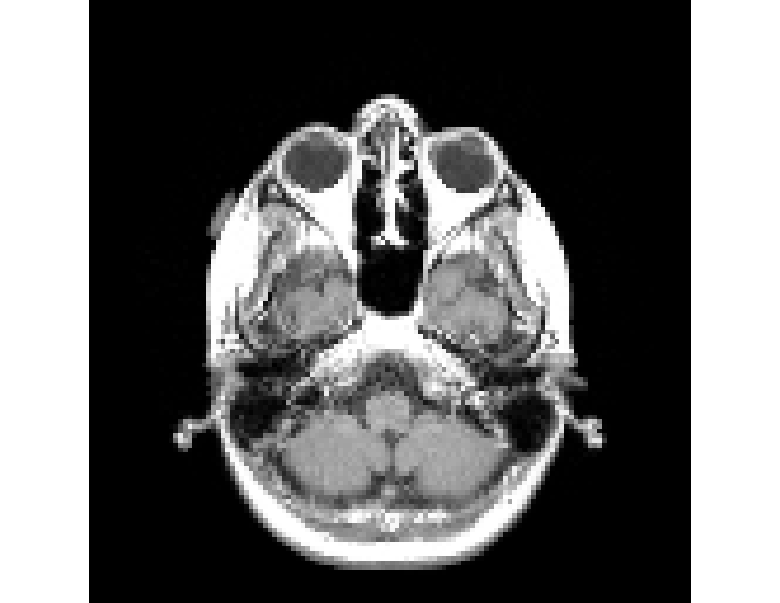}
        \caption{RANDSMAP-Sig ($P=N$)}
        \label{fig:MRIslice_RFNNs_recon}
    \end{subfigure}
    \hspace{2pt}
    \begin{subfigure}[b]{0.3\textwidth}
        \centering
        \includegraphics[width=0.8\textwidth]{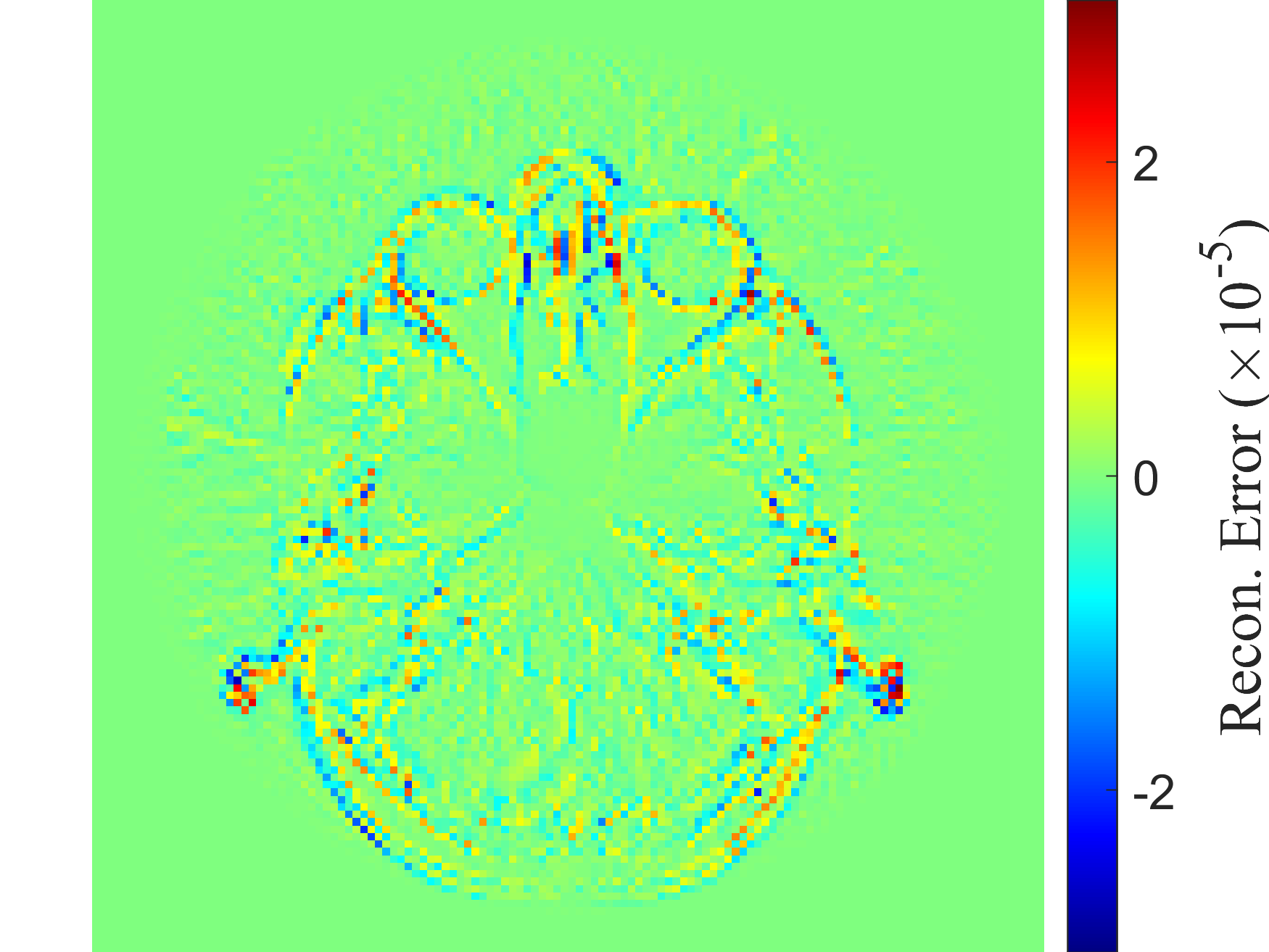}
        \caption{Absolute recon. error ($\times 10^{-5}$)}
        \label{fig:MRIslice_RFNNs_recon_err}
    \end{subfigure}
    \\
    % Fourth row
    \begin{subfigure}[b]{0.3\textwidth}
        \centering
        \includegraphics[width=0.8\textwidth]{FigsMRI/SliceOrig.png}
        \caption{Original slice}
    \end{subfigure}
    \hspace{2pt}
    \begin{subfigure}[b]{0.3\textwidth}
        \centering
        \includegraphics[width=0.8\textwidth]{FigsMRI/SliceReconGH.png}
        \caption{DDM}
        \label{fig:MRIslice_GH_recon}
    \end{subfigure}
    \hspace{2pt}
    \begin{subfigure}[b]{0.3\textwidth}
        \centering
        \includegraphics[width=0.8\textwidth]{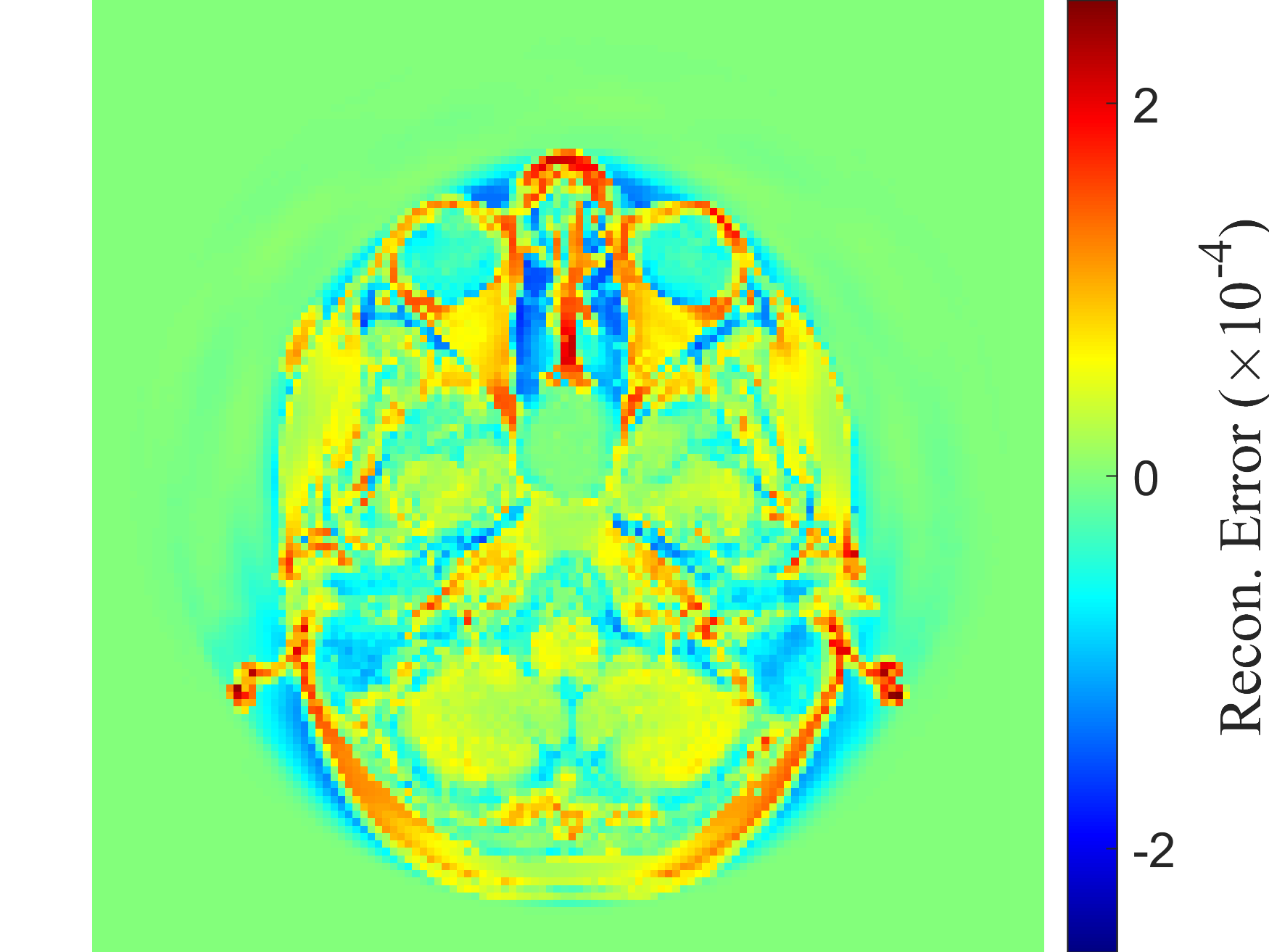}
        \caption{Absolute recon. error ($\times 10^{-4}$)}
        \label{fig:MRIslice_GH_recon_err}
    \end{subfigure}
    \\
    % Fifth row
    \begin{subfigure}[b]{0.3\textwidth}
        \centering
        \includegraphics[width=0.8\textwidth]{FigsMRI/SliceOrig.png}
        \caption{Original slice}
    \end{subfigure}
    \hspace{2pt}
    \begin{subfigure}[b]{0.3\textwidth}
        \centering
        \includegraphics[width=0.8\textwidth]{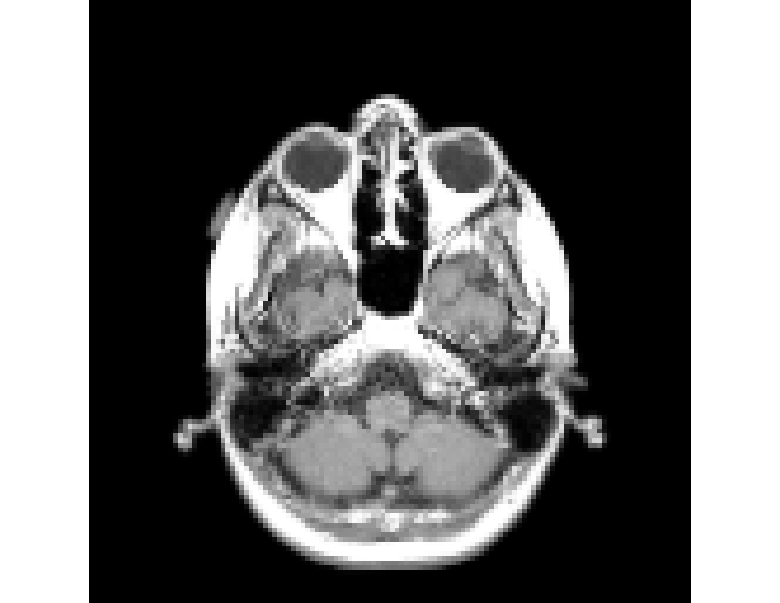}
        \caption{$k$-NN}
        \label{fig:MRIslice_kNN_recon}
    \end{subfigure}
    \hspace{2pt}
    \begin{subfigure}[b]{0.3\textwidth}
        \centering
        \includegraphics[width=0.8\textwidth]{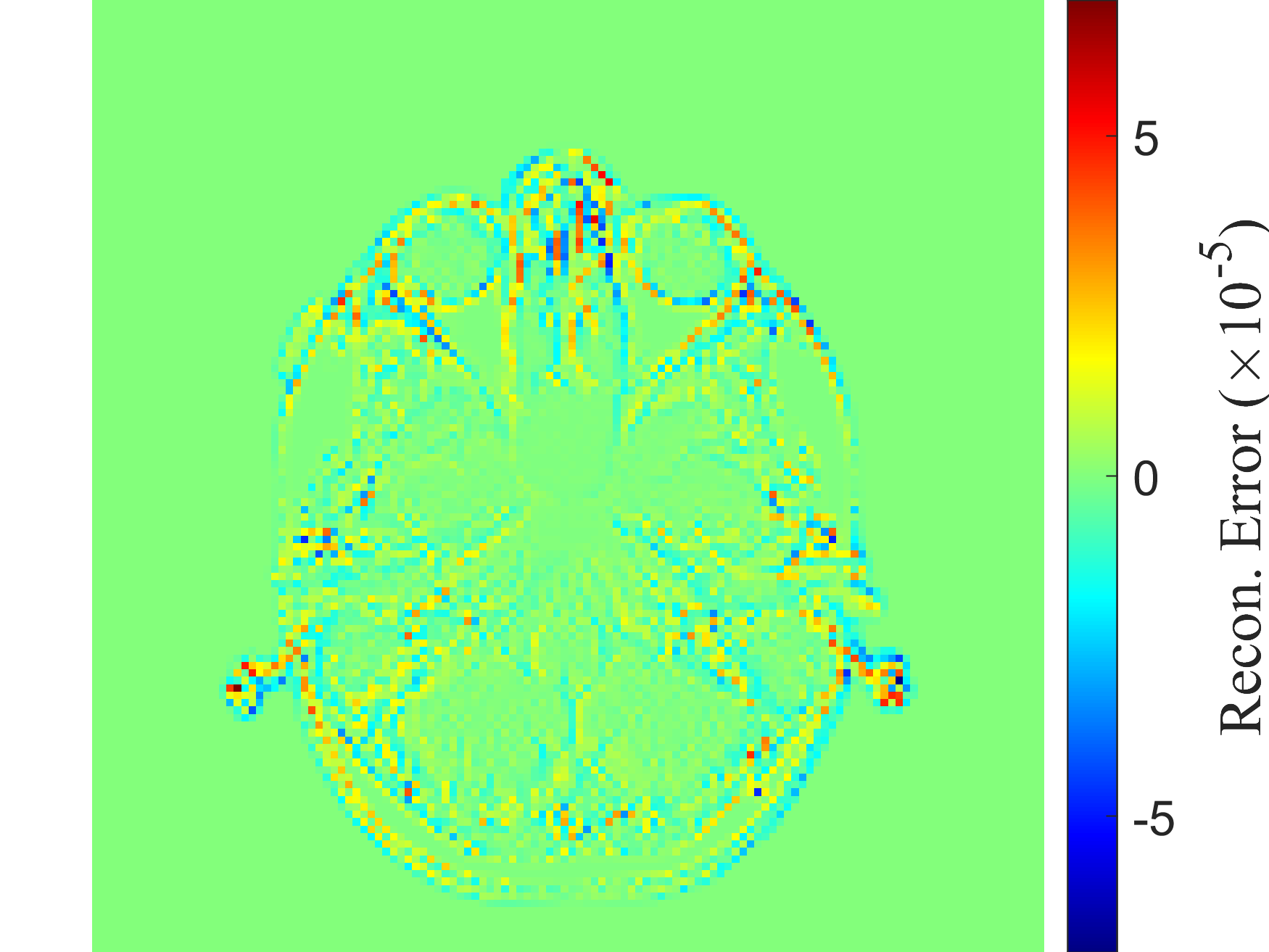}
        \caption{Absolute recon. error ($\times 10^{-5}$)}
        \label{fig:MRIslice_kNN_recon_err}
    \end{subfigure}
    \caption{Representative MRI image reconstruction for the 2D rotated MRI images dataset ($M=128 \times 128$). The original MRI slice in the left panels (corresponding to $\theta=0$ in Fig.~\ref{fig:MRIimage_data}) is compared to the reconstructed images in panels~\ref{fig:MRIslice_RFNNf_recon}, \ref{fig:MRIslice_RFNNfMK_recon}, \ref{fig:MRIslice_RFNNs_recon}, \ref{fig:MRIslice_GH_recon} and \ref{fig:MRIslice_kNN_recon} obtained with the five decoders RANDSMAP-RFF with $P=N$, RANDSMAP-MS-RFF $P=N$, RANDSMAP-Sig with $P=N$, DDM, and $k$-NN, respectively. The corresponding absolute pixel‑wise reconstruction errors are shown in panels~\ref{fig:MRIslice_RFNNf_recon_err}, \ref{fig:MRIslice_RFNNfMK_recon_err}, \ref{fig:MRIslice_RFNNs_recon_err}, \ref{fig:MRIslice_GH_recon_err} and \ref{fig:MRIslice_kNN_recon_err}, where the color‑bar range indicates the error magnitude. For the RANDSMAP decoders, the configuration shown is the one that achieved the lowest training reconstruction error among all variants of that type (see full comparison in Fig.~\ref{fig:MRI_dec_tr}).}
    \label{fig:MRI_recon}
\end{figure}

\clearpage
\newpage
\renewcommand{\theequation}{I.\arabic{equation}}
\renewcommand{\thefigure}{I.\arabic{figure}}
\renewcommand{\thetable}{I.\arabic{table}}
\setcounter{equation}{0}
\setcounter{figure}{0}
\setcounter{table}{0}
\section{Reconstruction results of the Hughes 2D Pedestrian Model Dataset (\texorpdfstring{$M=200\times50$}{M=200 x 50}) reconstruction results} \label{app:Hughes}

Here, we provide detailed reconstruction results for the Hughes 2D pedestrian model dataset in Section~\ref{sb:Hughes}. Tables~\ref{tab:Hughes_dec_tr_all} and \ref{tab:Hughes_dec_ts_all} enlist the training and testing set performance of all decoders considered for this benchmark. 

\begin{table}[htbp]
\centering
\caption{Training set performance for the Hughes 2D pedestrian model dataset ($M=200 \times 50$). Detailed reconstruction metrics for $N=5000$ training points, using $d=10$-dim. DM embeddings. Decoders are compared based on the relative $L_2$ and $L_\infty$ mean reconstruction errors, $e_{2,i}$ and $e_{\infty,i}$ (Eq.~\eqref{eq:recon_errors}), computational time (in seconds), and mean conservation error $e_{con,i}$. For stochastic decoders (RANDSMAP-RFF, RANDSMAP-MS-RFF, RANDSMAP-Sig, and RFNN-Sig), metrics show the median with 5-95\% percentiles in parentheses, computed over 100 random initializations. Deterministic decoders (DDM, $k$-NN) report single values.}
\label{tab:Hughes_dec_tr_all}
\resizebox{\textwidth}{!}{%
\begin{tabular}{@{}llcccc@{}}
\multirow{2}{*}{Decoder} & \multirow{2}{*}{$P$} & \multicolumn{2}{c}{Mean reconstruction error ($\times 10^{-2}$)} & Mean conservation & Comp. Time (s) \\
\cmidrule(lr){3-4}
& & relative $L_2$, $e_{2,i}$ &  relative $L_\infty$, $e_{\infty,i}$ & error, $e_{con,i}$ ($\times 10^{-7}$) & ($\times 10^2$) \\
\midrule
\midrule
% DM-RFNNf variants
\multirow{3}{*}{RANDSMAP-RFF} 
&	$N$	    & 3.582 (3.571--3.597) & 6.434 (6.419--6.452) & 3.474 (3.293--3.697) $\times 10^{-1}$ & 2.686 (2.678--2.712) \\
&	$N/2$	& 3.744 (3.722--3.764) & 6.633 (6.608--6.658) & 2.216 (1.873--2.548) $\times 10^{-2}$ & 0.967 (0.965--0.969) \\
&	$N/4$	& 4.225 (4.195--4.247) & 7.440 (7.411--7.470) & 3.111 (2.513--3.741) $\times 10^{-4}$ & 0.521 (0.520--0.522) \\
\midrule

% DM-RFNNfMK variants  
\multirow{3}{*}{RANDSMAP-MS-RFF} 
&	$N$	    & 3.868 (3.665--4.155) & 6.831 (6.467--7.234) & 2.833 (2.099--3.358) $\times 10^{-1}$  & 2.682 (2.670--2.697) \\
&	$N/2$	& 4.074 (3.860--4.422) & 7.096 (6.751--7.614) & 1.795 (1.125--2.169) $\times 10^{-1}$  & 0.954 (0.951--0.957) \\
&	$N/4$	& 4.464 (4.212--4.799) & 7.579 (7.236--7.961) & 7.305 (0.003--10.918) $\times 10^{-2}$ & 0.523 (0.522--0.524) \\
\midrule

% DM-RFNNs variants  
\multirow{3}{*}{RANDSMAP-Sig} 
&	$N$	    & 2.476 (2.468--2.486) & 5.025 (5.011--5.041) & 4.601 (4.343--4.955) $\times 10^{-2}$  & 2.793 (2.778--2.811) \\
&	$N/2$	& 2.843 (2.825--2.855) & 5.562 (5.544--5.581) & 9.558 (7.640--11.694) $\times 10^{-2}$ & 1.000 (0.999--1.002) \\
&	$N/4$	& 3.383 (3.353--3.424) & 6.304 (6.275--6.347) & 1.287 (0.789--1.847) $\times 10^{-1}$  & 0.535 (0.535--0.536) \\
\midrule

% DM-RFNNs variants  
\multirow{3}{*}{RFNN-Sig} 
&	$N$	    & 2.476 (2.468--2.486) & 5.025 (5.011--5.041) & 1.863 (1.798--1.934) $\times 10^1$ & 0.921 (0.918--0.952) \\
&	$N/2$	& 2.843 (2.825--2.855) & 5.562 (5.544--5.581) & 3.911 (3.706--4.105) $\times 10^1$ & 0.546 (0.546--0.547) \\
&	$N/4$	& 3.383 (3.353--3.424) & 6.304 (6.275--6.347) & 8.888 (8.179--9.490) $\times 10^1$ & 0.369 (0.369--0.369) \\
\midrule

% Deterministic methods
DDM  &	-	& 19.424 & 27.649 & 1.879 $\times 10^{5}$  & 1.735   \\
$k$-NN &	-	& 7.250  & 11.198 & 1.799 $\times 10^{-8}$ & 572.722 \\

\bottomrule
\end{tabular}%
}
\end{table}

\begin{table}[htbp]
\centering
\caption{Testing set performance for the Hughes 2D pedestrian model dataset ($M=200 \times 50$). Detailed reconstruction metrics for $N=5000$ training points, using $d=10$-dimensional DM embeddings. Decoders are compared based on the relative $L_2$ and $L_\infty$ mean reconstruction errors, $e_{2,i}$ and $e_{\infty,i}$ (Eq.~\eqref{eq:recon_errors}), computational time (in seconds), and mean conservation error $e_{con,i}$. For stochastic decoders (RANDSMAP-RFF, RANDSMAP-MS-RFF, RANDSMAP-Sig, and RFNN-Sig), metrics show the median with 5-95\% percentiles in parentheses, computed over 100 random initializations. Deterministic decoders (DDM, $k$-NN) report single values.}
\label{tab:Hughes_dec_ts_all}
\resizebox{\textwidth}{!}{%
\begin{tabular}{@{}llcccc@{}}
\multirow{2}{*}{Decoder} & \multirow{2}{*}{$P$} & \multicolumn{2}{c}{Mean reconstruction error ($\times 10^{-2}$)} & Mean conservation & Comp. Time (s) \\
\cmidrule(lr){3-4}
& & relative $L_2$, $e_{2,i}$ &  relative $L_\infty$, $e_{\infty,i}$ & error, $e_{con,i}$ ($\times 10^{-7}$) & ($\times 10^2$) \\
\midrule
\midrule
% DM-RFNNf variants
\multirow{3}{*}{RANDSMAP-RFF} 
&	$N$	    & 5.088 (5.053--5.141) & 9.157 (9.115--9.219) & 11.693 (6.698--19.590)                & 0.327 (0.286--0.364) \\
&	$N/2$	& 5.247 (5.202--5.303) & 9.288 (9.226--9.370) & 2.485 (2.131--2.891) $\times 10^{-2}$ & 0.241 (0.237--0.244) \\
&	$N/4$	& 5.286 (5.232--5.354) & 9.396 (9.337--9.461) & 3.434 (2.758--4.142) $\times 10^{-4}$ & 0.223 (0.220--0.227) \\
\midrule

% DM-RFNNfMK variants  
\multirow{3}{*}{RANDSMAP-MS-RFF} 
&	$N$	    & 5.304 (5.142--5.450) & 9.411 (9.245--9.585)  & 6.761 (2.818--20.996)                  & 0.339 (0.304--0.417) \\
&	$N/2$	& 5.466 (5.303--5.687) & 9.557 (9.386--9.768)  & 1.125 (0.485--3.004)                   & 0.254 (0.252--0.256) \\
&	$N/4$	& 5.754 (5.497--6.009) & 9.817 (9.600--10.065) & 1.203 (0.003--2.307) $\times 10^{-1}$  & 0.232 (0.231--0.235) \\
\midrule

% DM-RFNNs variants  
\multirow{3}{*}{RANDSMAP-Sig} 
&	$N$	     & 3.909 (3.848--4.004) & 7.752 (7.663--7.861) & 22.476 (13.719--38.500) & 0.328 (0.289--0.380) \\
&	$N/2$	 & 4.235 (4.162--4.340) & 8.143 (8.060--8.297) & 8.364 (5.359--16.079)   & 0.242 (0.240--0.247) \\
&	$N/4$	 & 4.701 (4.604--4.797) & 8.690 (8.597--8.870) & 1.007 (0.480--2.070)    & 0.224 (0.220--0.228) \\
\midrule

% DM-RFNNs variants without conservation
\multirow{3}{*}{RFNN-Sig} 
&	$N$	     & 3.909 (3.848--4.004) & 7.752 (7.663--7.861) & 4.556 (3.712--6.882) $\times 10^1$ & 0.330 (0.303--0.416) \\
&	$N/2$	 & 4.235 (4.162--4.340) & 8.143 (8.060--8.297) & 7.876 (6.640--10.887) $\times 10^1$ & 0.252 (0.250--0.255) \\
&	$N/4$	 & 4.701 (4.604--4.797) & 8.690 (8.597--8.870) & 14.348 (12.677--17.731) $\times 10^1$ & 0.230 (0.228--0.233) \\
\midrule

% Deterministic methods
DDM  &	-	& 19.259 & 27.634 & 1.872 $\times 10^{5}$  & 0.324   \\
$k$-NN &	-	& 8.543  & 13.530 & 1.801 $\times 10^{-8}$ & 299.823 \\

\bottomrule
\end{tabular}%
}
\end{table}

Figure~\ref{fig:Hughes_DM_GH} demonstrates that the fair performance of the DDM decoder is not due to poor tuning, via an exhaustive 2D grid search. Figure~\ref{fig:Hughes} demonstrates per-density profile reconstructions provided by all decoders, projected in the summary statistics space $(\mu_x,\mu_y,\sigma^2_x)$. Finally, Fig.~\ref{fig:Hughes_recon} shows a representative density profile, in which the crowd flow avoids the obstacle,  and its reconstructions with the five decoders considered.

\begin{figure}[!htbp]
    \centering
    % First row
    \begin{subfigure}[b]{0.49\textwidth}
        \centering
        \includegraphics[width=\textwidth]{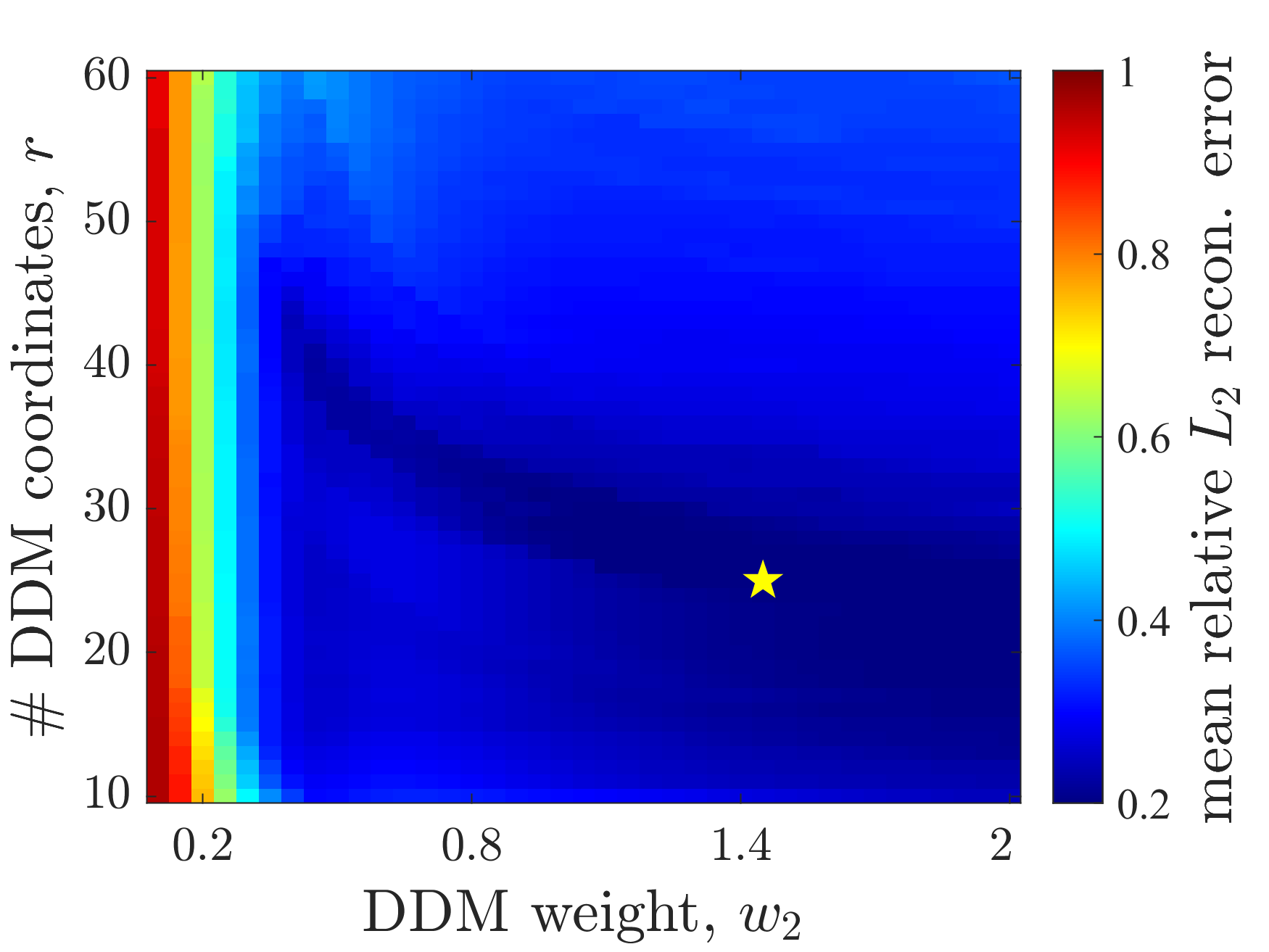}
        \caption{Exhaustive ($w_2,r$) tuning}
        \label{fig:Hughes_DM_GH_tuning}
    \end{subfigure}
    \hfill
    \begin{subfigure}[b]{0.49\textwidth}
        \centering
        \begin{subfigure}[b]{\textwidth}
            \centering
            \includegraphics[width=\textwidth]{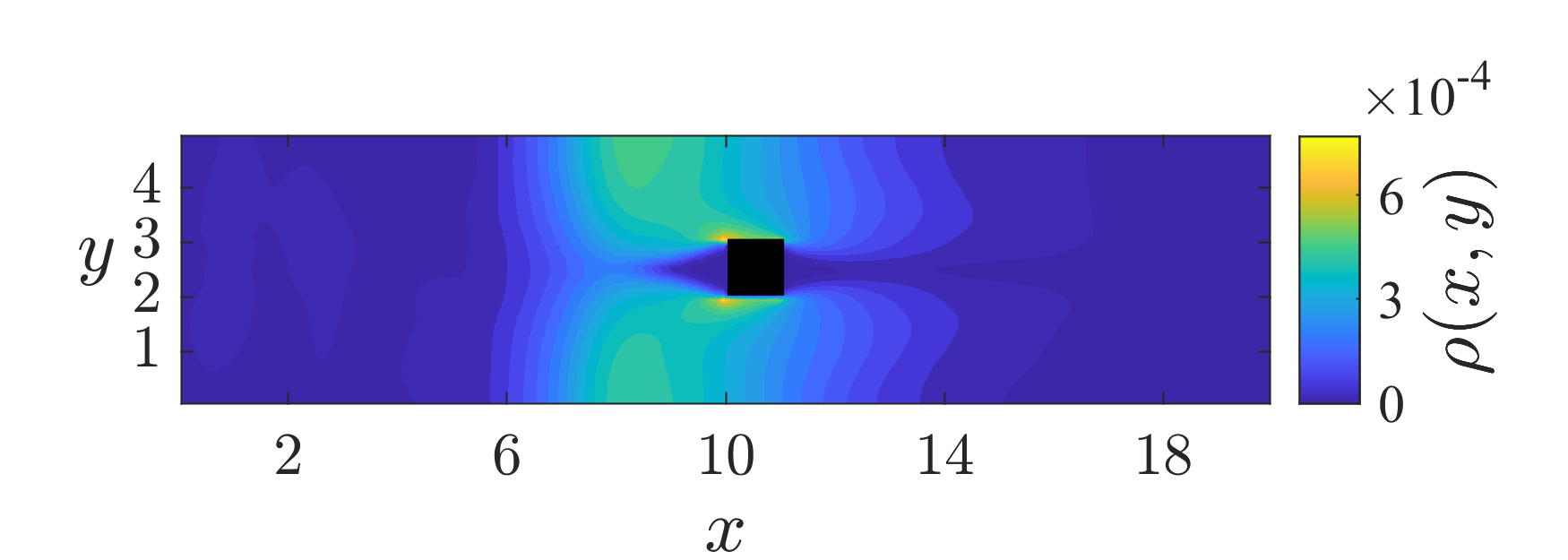}
            \caption{Reconstruction, $(w_2,r)=(1.45,25)$}
            \label{fig:Hughes_DM_GH_k25}
        \end{subfigure}
        
        %\vspace{5pt} % Adjust spacing between the two bottom panels
        
        \begin{subfigure}[b]{\textwidth}
            \centering
            \includegraphics[width=\textwidth]{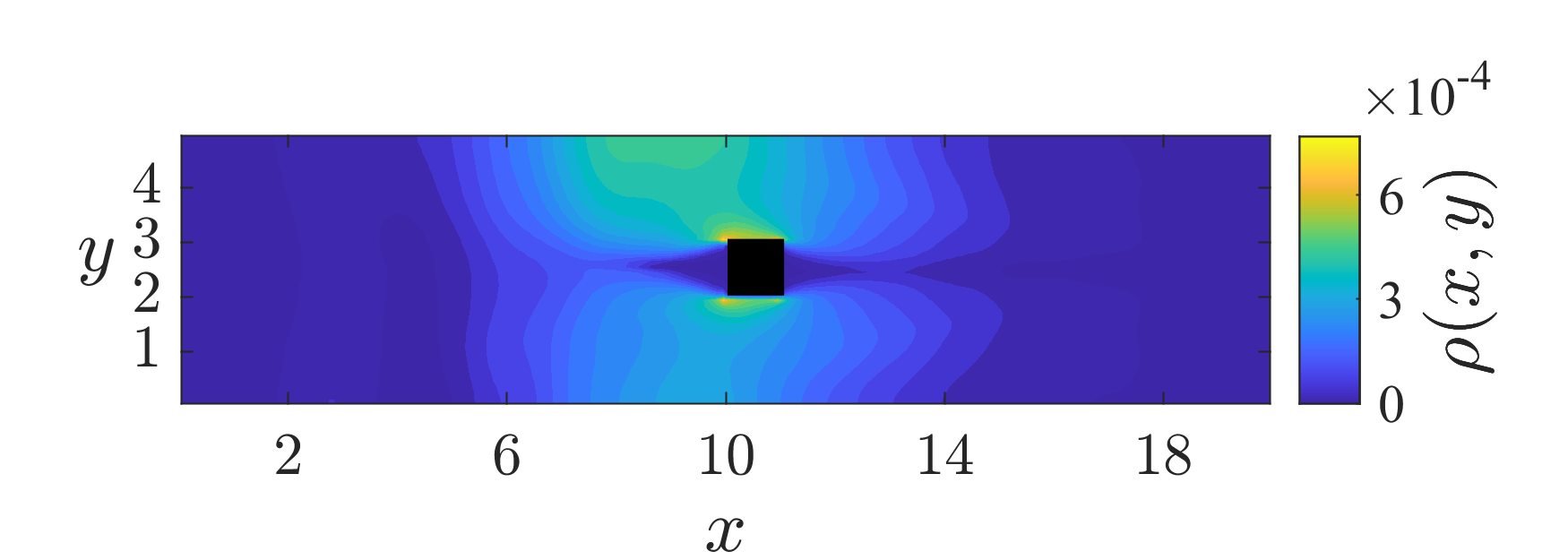}
            \caption{Reconstruction, $(w_2,r)=(1.45,50)$}
            \label{fig:Hughes_DM_GH_k50}
        \end{subfigure}
    \end{subfigure}
    \caption{Reconstruction with the DDM decoder for the Hughes 2D pedestrian model dataset ($M=200 \times 50$). Panel~\ref{fig:Hughes_DM_GH_tuning} shows exhaustive hyperparameter tuning over the DDM weight $w_2$ and the number of DDM coordinates $r$. The yellow star indicates the optimal combination $(w_2,r)=(1.45,25)$. Panels \ref{fig:Hughes_DM_GH_k25} and \ref{fig:Hughes_DM_GH_k50} show the reconstructed density profiles (corresponding to the original profile in Fig.~\ref{fig:Hughes_den_orig}) obtained with $w_2=1.45$ and $r=25$ (optimal) or $r=50$ DDM coordinates, respectively.}
    \label{fig:Hughes_DM_GH}
\end{figure}

\begin{figure}[!htbp]
    \centering
    % First row
    \begin{subfigure}[b]{0.32\textwidth}
        \centering
        \includegraphics[width=\textwidth]{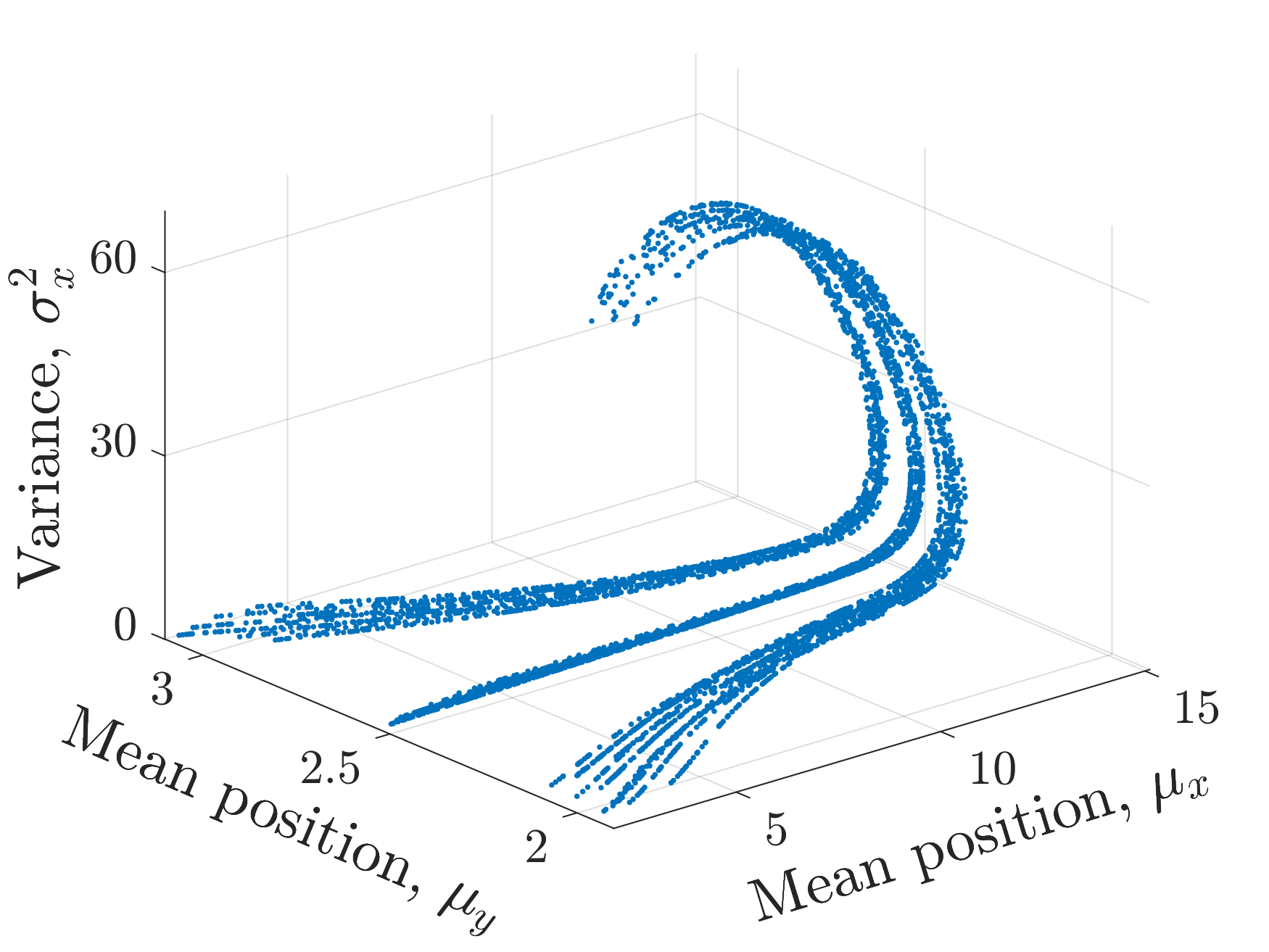}
        \caption{Data set}
        \label{fig:Hughes_data}
    \end{subfigure} \hspace{0.02\textwidth}
    \begin{subfigure}[b]{0.32\textwidth}
        \centering
        \includegraphics[width=\textwidth]{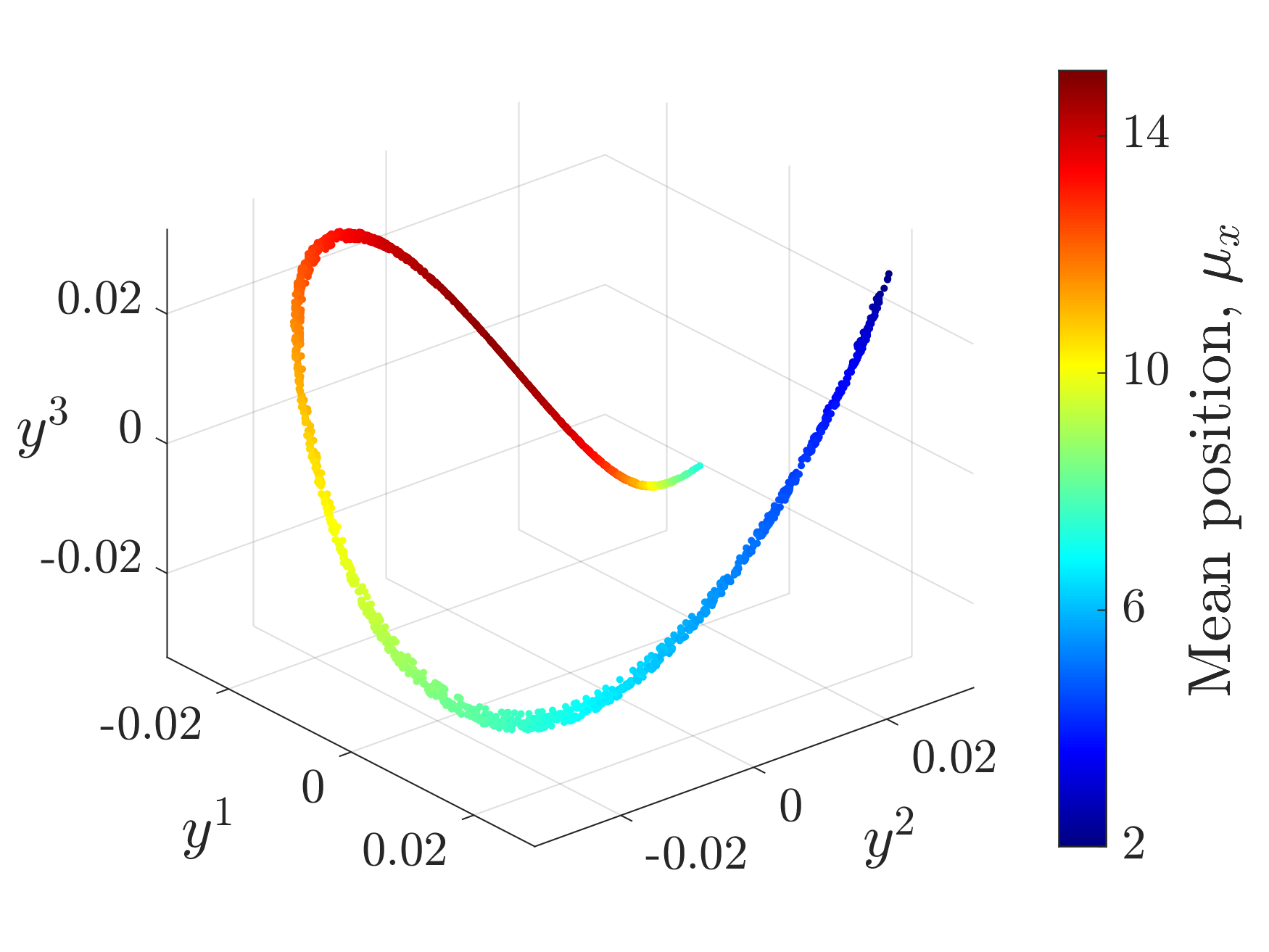}
        \caption{DM coordinates}
        \label{fig:Hughes_DMcoords}
    \end{subfigure}

    % Second row
    \begin{subfigure}[b]{0.32\textwidth}
        \centering
        \includegraphics[width=\textwidth]{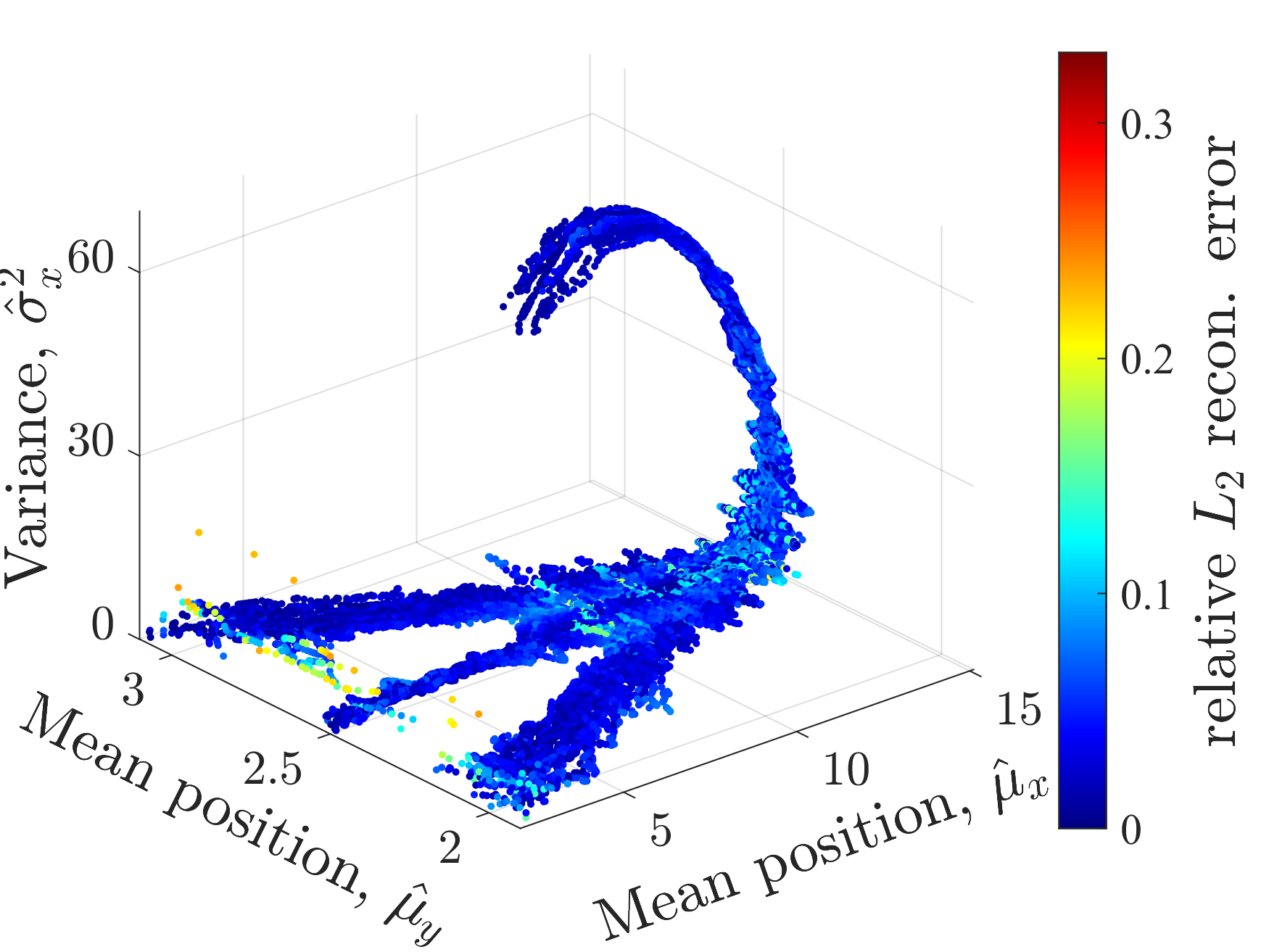}
        \caption{RANDSMAP-RFF ($P=N$)}
        \label{fig:Hughes_DM_RFNNf_recon}
    \end{subfigure}
    \hfill
    \begin{subfigure}[b]{0.32\textwidth}
        \centering
        \includegraphics[width=\textwidth]{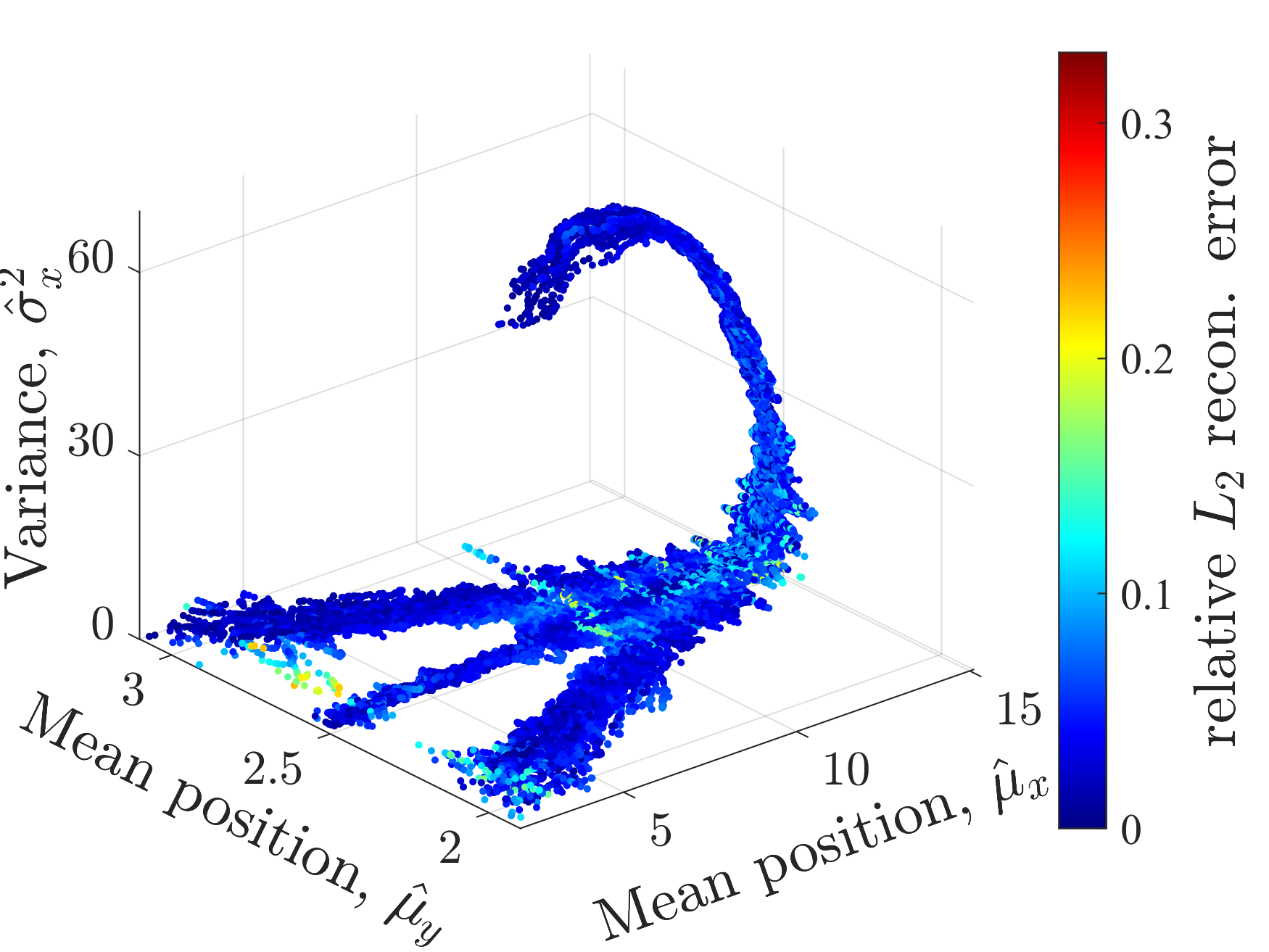}
        \caption{RANDSMAP-MS-RFF ($P=N$)}
        \label{fig:Hughes_DM_RFNNfMK_recon}
    \end{subfigure}
    \hfill
    \begin{subfigure}[b]{0.32\textwidth}
        \centering
        \includegraphics[width=\textwidth]{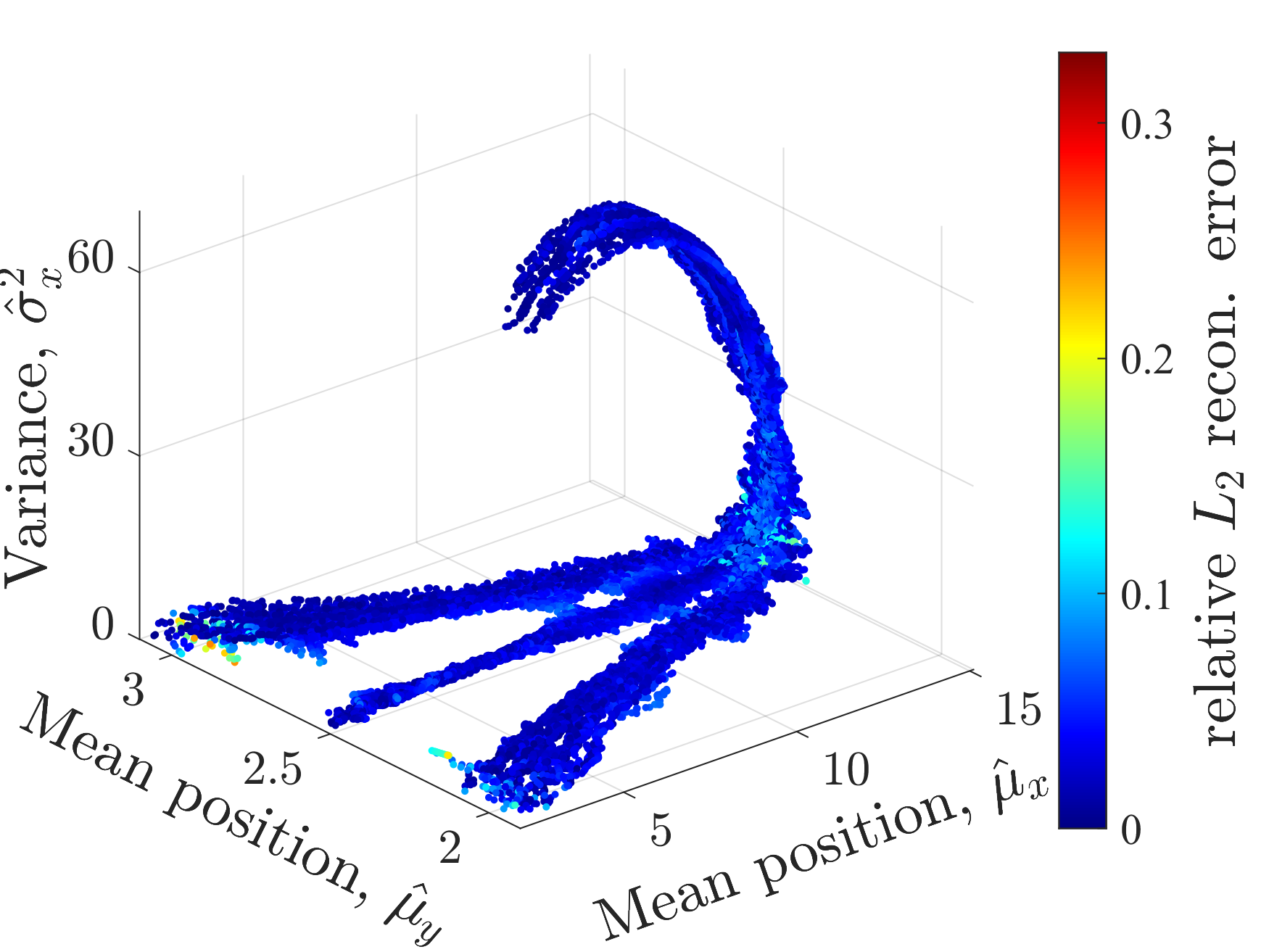}
        \caption{RANDSMAP-Sig ($P=N$)}
        \label{fig:Hughes_DM_RFNNs_recon}
    \end{subfigure}

    % Third row
    \begin{subfigure}[b]{0.32\textwidth}
        \centering
        \includegraphics[width=\textwidth]{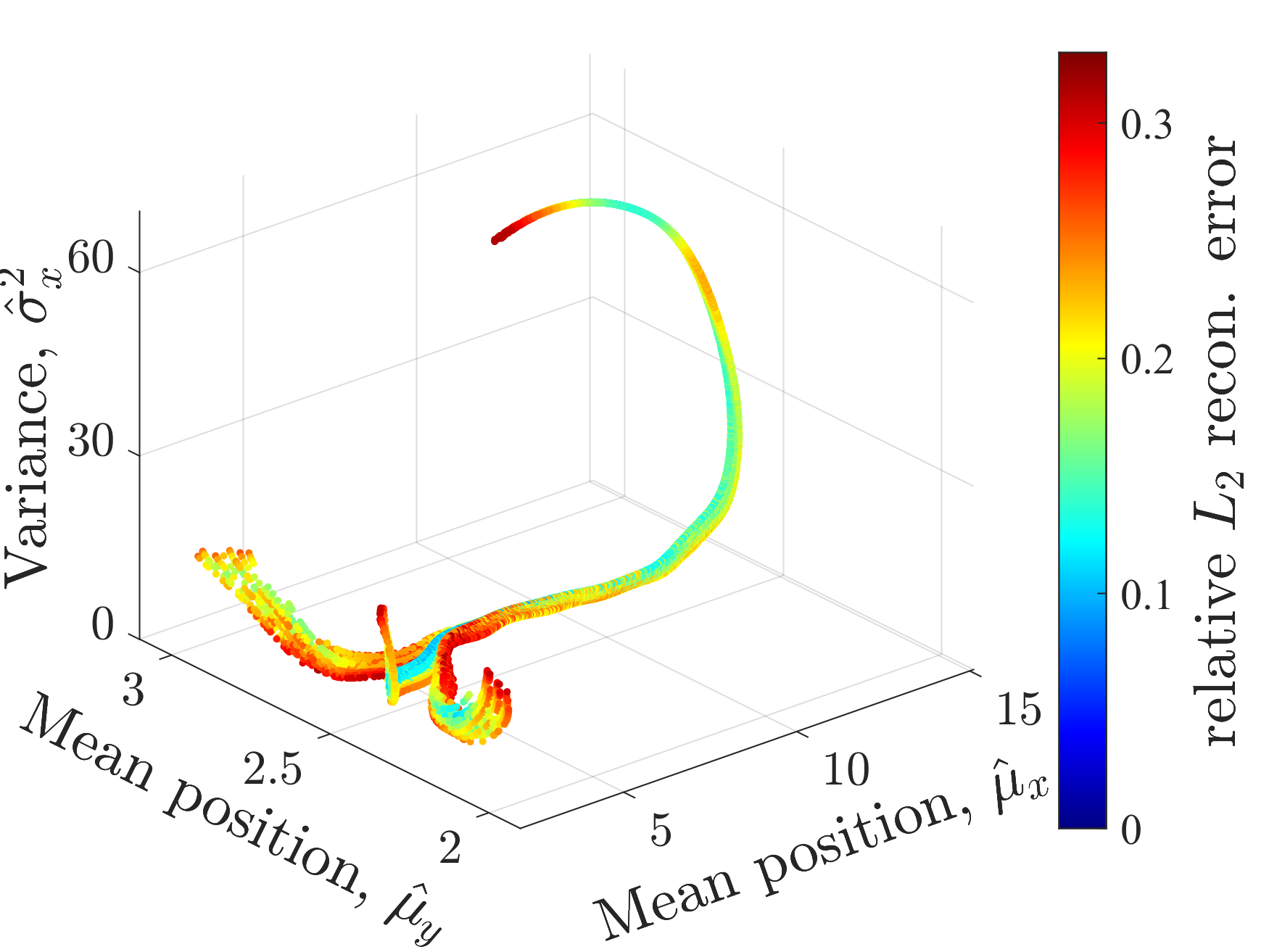}
        \caption{DDM}
        \label{fig:Hughes_DM_GH_recon}
    \end{subfigure}\hspace{0.02\textwidth}
    \begin{subfigure}[b]{0.32\textwidth}
        \centering
        \includegraphics[width=\textwidth]{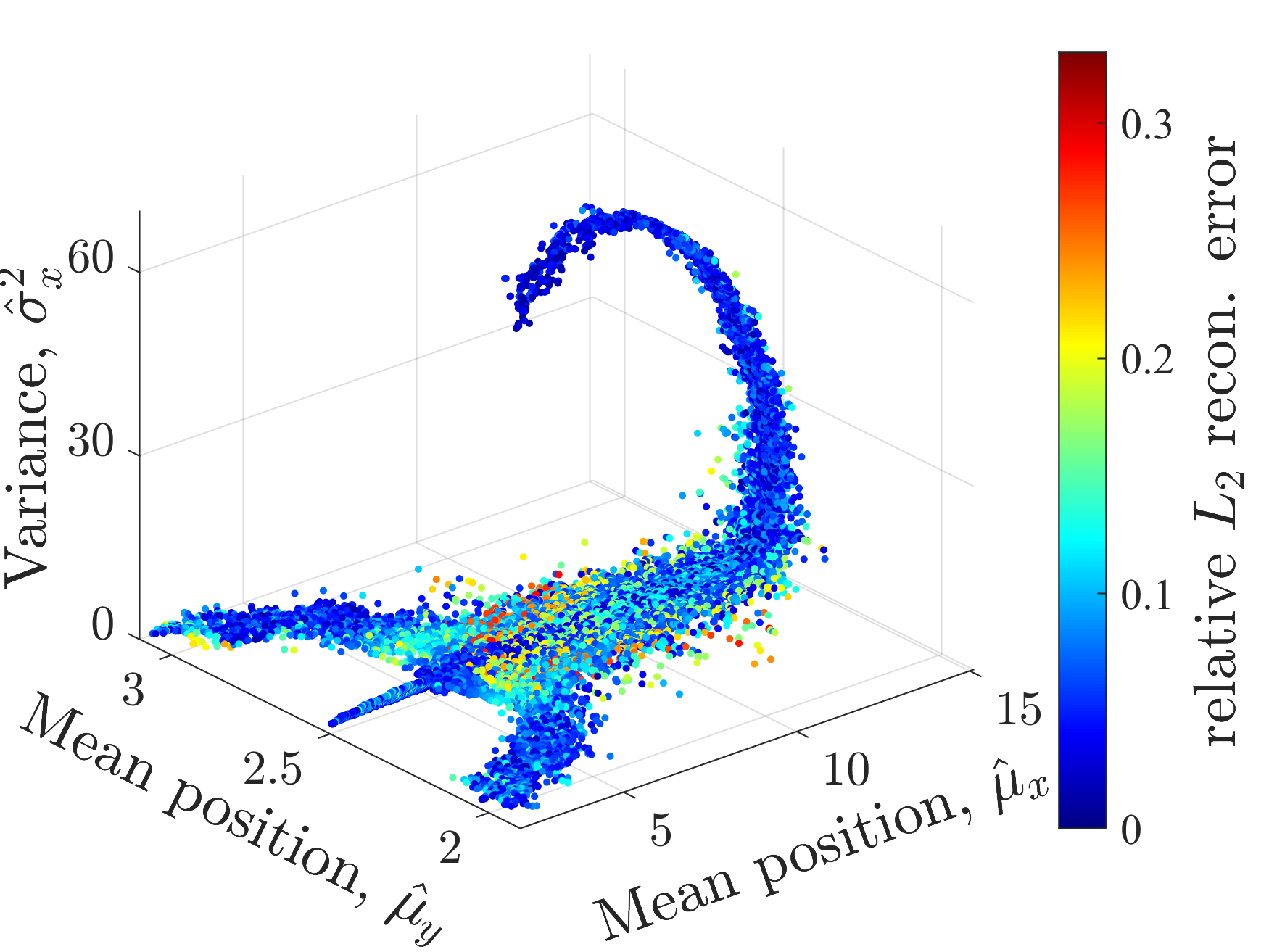}
        \caption{$k$-NN}
        \label{fig:Hughes_DM_kNN_recon}
    \end{subfigure}
    \caption{Hughes 2D pedestrian model dataset ($M=200 \times 50$) with $N=5000$ training points. Because the density fields are $10.000$-dimensional, we visualize them via three summary statistics: mean positions $(\mu_x, \mu_y)$ and variance $\sigma_x^2$. Panel~\ref{fig:Hughes_data} shows the training data represented in this statistics space, while panel~\ref{fig:Hughes_DMcoords} shows the projection onto the first three DM coordinates $[y^1, y^2, y^3]^\top$ (from a $d=10$-dimensional embedding), colored by the mean x-position $\mu_x$ for each density profile. Panels~\ref{fig:Hughes_DM_RFNNf_recon}-\ref{fig:Hughes_DM_kNN_recon} show the reconstructed density profiles, projected onto the statistics space ($\mu_x$, $\mu_y$, $\sigma^2_x$), and their per-profile relative $L_2$ errors for the five decoders: RANDSMAP-RFF with $P=N$, RANDSMAP-MS-RFF with $P=N$, RANDSMAP-Sig with $P=N$, DDM and $k$-NN. For the RFNN decoders, the displayed configuration is the one achieving the lowest training reconstruction error among all tested variants (see Fig.~\ref{fig:Hughes_dec_tr} for the full comparison).}
    \label{fig:Hughes}
\end{figure}

\begin{figure}[!htbp]
    \centering
    % First row
    \begin{subfigure}[b]{0.49\textwidth}
        \centering
        \includegraphics[width=\textwidth]{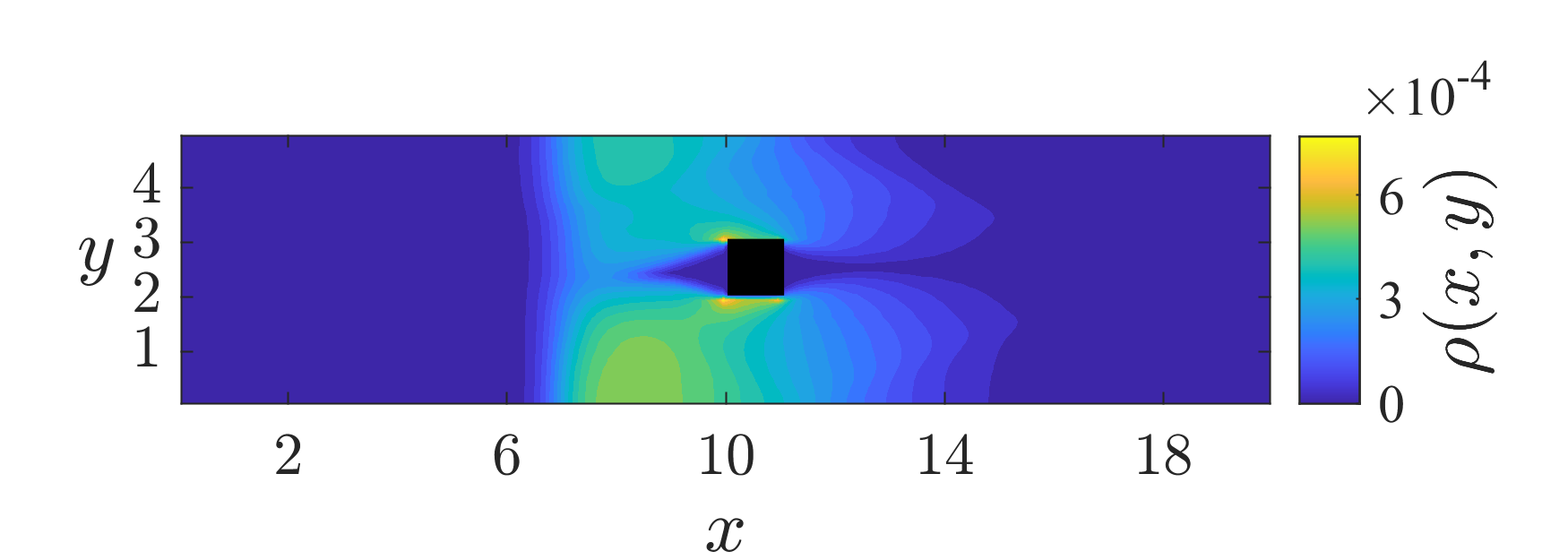}
        \caption{Original density profile}
        \label{fig:Hughes_den_orig}
    \end{subfigure}
    \hfill
    \begin{subfigure}[b]{0.49\textwidth}
        \centering
        \includegraphics[width=\textwidth]{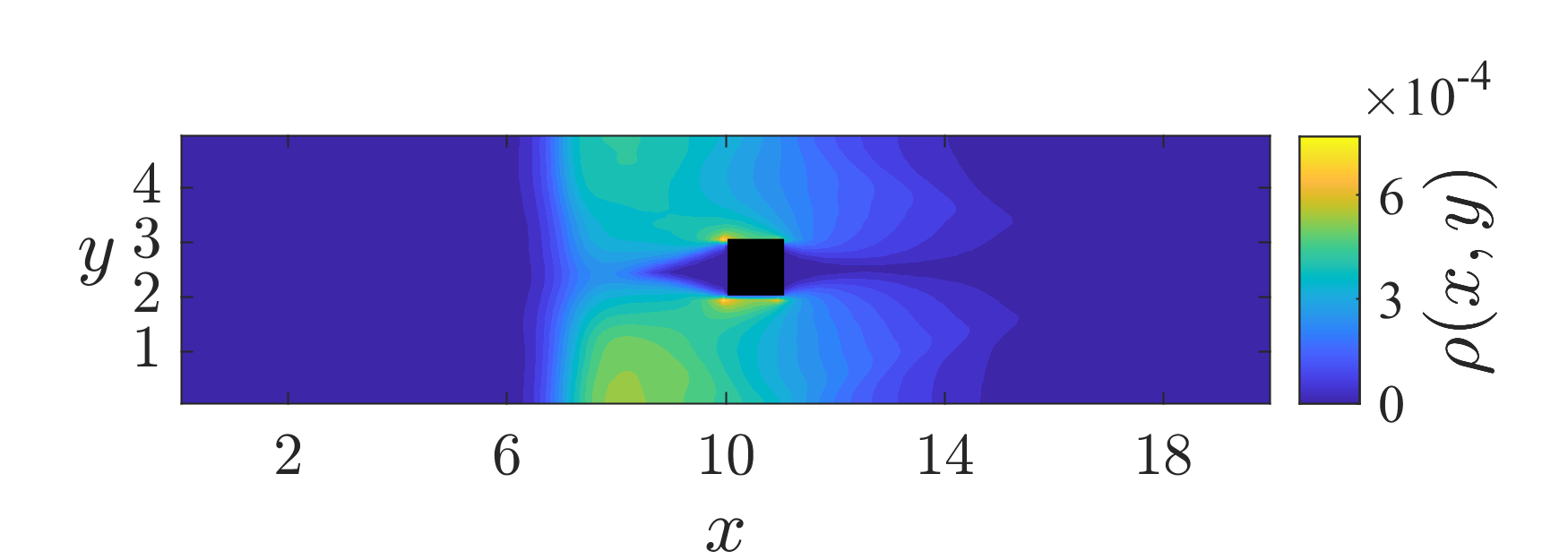}
        \caption{RANDSMAP-RFF ($P=N$)}
        \label{fig:Hughes_den_RFNNf_recon}
    \end{subfigure}
    \\
    % Second row
    \begin{subfigure}[b]{0.49\textwidth}
        \centering
        \includegraphics[width=\textwidth]{FigsHughes/DenProf_Orig.png}
        \caption{Original density profile}
    \end{subfigure}
    \hfill
    \begin{subfigure}[b]{0.49\textwidth}
        \centering
        \includegraphics[width=\textwidth]{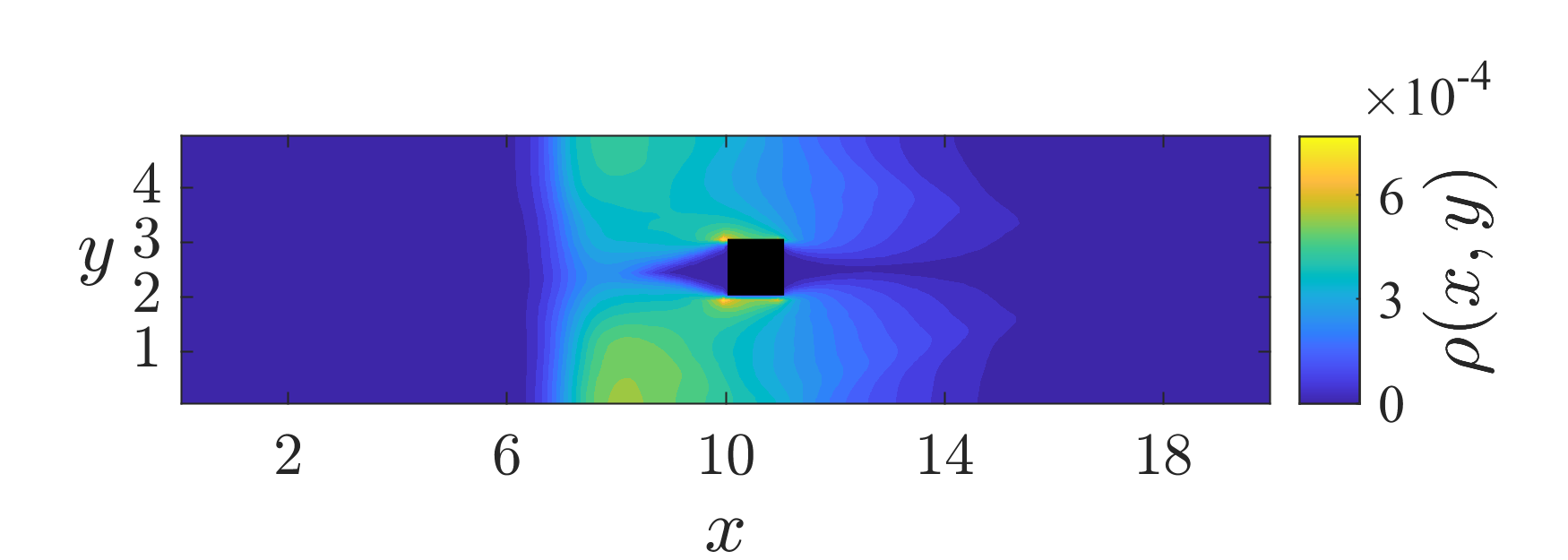}
        \caption{RANDSMAP-MS-RFF ($P=N$)}
        \label{fig:Hughes_den_RFNNfMK_recon}
    \end{subfigure}
    \\
    % Third row
    \begin{subfigure}[b]{0.49\textwidth}
        \centering
        \includegraphics[width=\textwidth]{FigsHughes/DenProf_Orig.png}
        \caption{Original density profile}
    \end{subfigure}
    \hfill
    \begin{subfigure}[b]{0.49\textwidth}
        \centering
        \includegraphics[width=\textwidth]{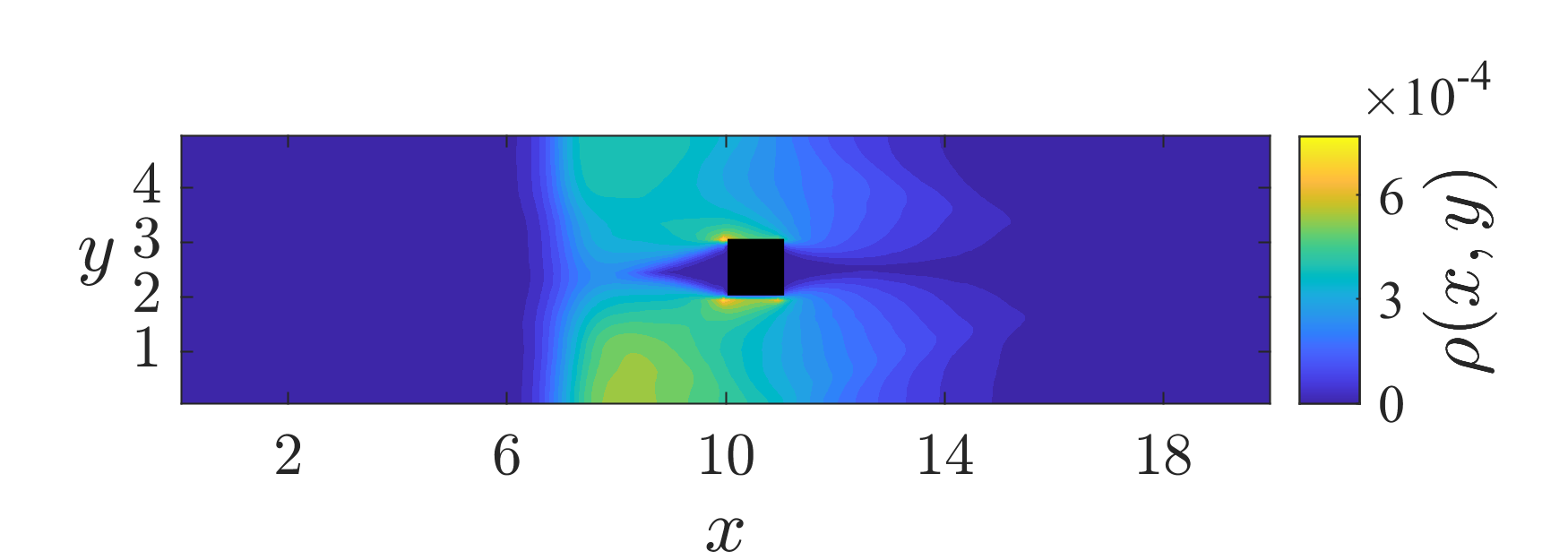}
        \caption{RANDSMAP-Sig ($P=N$)}
        \label{fig:Hughes_den_RFNNs_recon}
    \end{subfigure}
    \\
    % Fourth row
    \begin{subfigure}[b]{0.49\textwidth}
        \centering
        \includegraphics[width=\textwidth]{FigsHughes/DenProf_Orig.png}
        \caption{Original density profile}
    \end{subfigure}
    \hfill
    \begin{subfigure}[b]{0.49\textwidth}
        \centering
        \includegraphics[width=\textwidth]{FigsHughes/DenProf_GH.png}
        \caption{DDM}
        \label{fig:Hughes_den_GH_recon}
    \end{subfigure}
    \\
    % Fifth row
    \begin{subfigure}[b]{0.49\textwidth}
        \centering
        \includegraphics[width=\textwidth]{FigsHughes/DenProf_Orig.png}
        \caption{Original density profile}
    \end{subfigure}
    \hfill
    \begin{subfigure}[b]{0.49\textwidth}
        \centering
        \includegraphics[width=\textwidth]{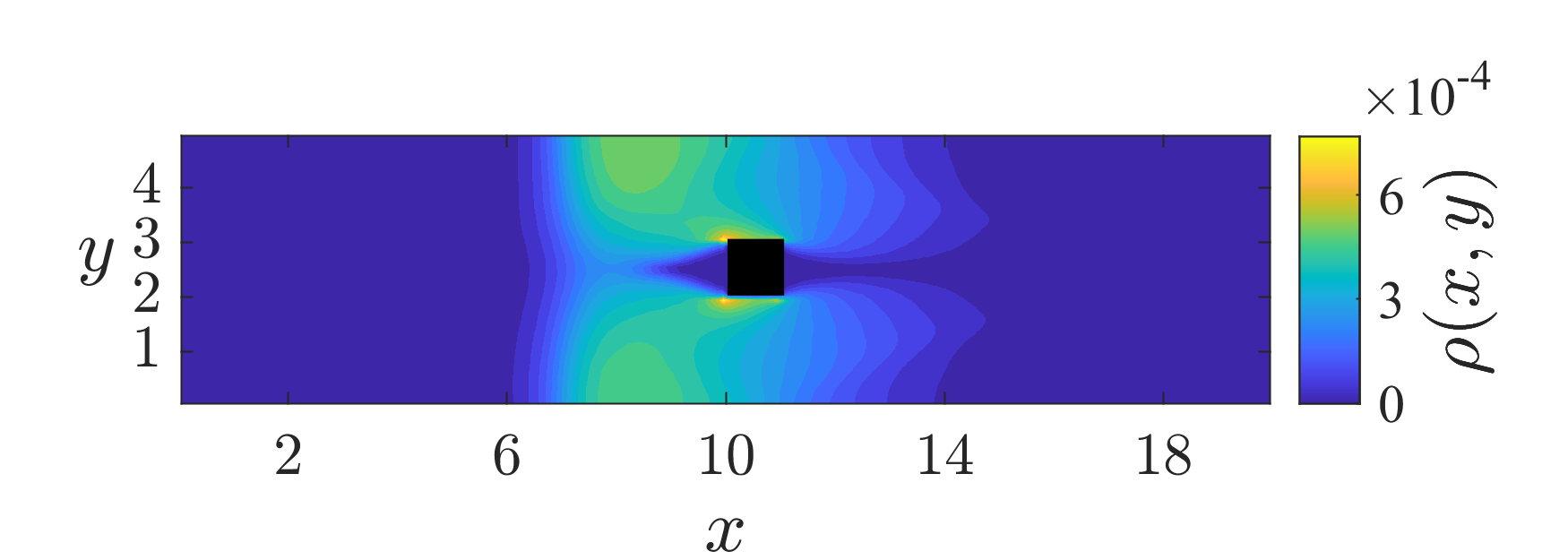}
        \caption{$k$-NN}
        \label{fig:Hughes_den_kNN_recon}
    \end{subfigure}
    \caption{Representative original and reconstructed density profiles for the Hughes 2D pedestrian model dataset ($M=200 \times 50$). Left column panels show the original density profile, while panels~\ref{fig:Hughes_den_RFNNf_recon}, \ref{fig:Hughes_den_RFNNfMK_recon}, \ref{fig:Hughes_den_RFNNs_recon}, \ref{fig:Hughes_den_GH_recon} and \ref{fig:Hughes_den_kNN_recon} show its reconstructions, obtained with the five decoders: RANDSMAP-RFF with $P=N$, RANDSMAP-MS-RFF with $P=N$, RANDSMAP-Sig with $P=N$, DDM and $k$-NN, respectively. For the RANDSMAP decoders, the displayed configuration is the one achieving the lowest training error among all variants of that type (see full comparison in Fig.~\ref{fig:Hughes_dec_tr}).}
    \label{fig:Hughes_recon}
\end{figure}

\end{document}